\documentclass[reqno]{amsart}

\usepackage{hyperref}

\newtheorem{theorem}{Theorem}
\newtheorem{lemma}[theorem]{Lemma}
\newtheorem{proposition}[theorem]{Proposition}
\newtheorem{corollary}[theorem]{Corollary}
\newtheorem{example}[theorem]{Example}
\theoremstyle{exercise}
\newtheorem{exercise}[theorem]{Exercise}
\theoremstyle{definition}
\newtheorem{definition}[theorem]{Definition}
\usepackage{graphicx}
\usepackage{pstricks, enumerate, pst-node, pst-text, pst-plot}
\theoremstyle{remark}
\newtheorem{remark}[theorem]{Remark}

\numberwithin{equation}{section}

\newcommand{\intav}[1]{\mathchoice {\mathop{\vrule width 6pt height 3 pt depth  -2.5pt
\kern -8pt \intop}\nolimits_{\kern -6pt#1}} {\mathop{\vrule width
5pt height 3  pt depth -2.6pt \kern -6pt \intop}\nolimits_{#1}}
{\mathop{\vrule width 5pt height 3 pt depth -2.6pt \kern -6pt
\intop}\nolimits_{#1}} {\mathop{\vrule width 5pt height 3 pt depth
-2.6pt \kern -6pt \intop}\nolimits_{#1}}}

\newcommand{\intavl}[1]{\mathchoice {\mathop{\vrule width 6pt height 3 pt depth  -2.5pt
\kern -8pt \intop}\limits_{\kern -6pt#1}} {\mathop{\vrule width 5pt
height 3  pt depth -2.6pt \kern -6pt \intop}\nolimits_{#1}}
{\mathop{\vrule width 5pt height 3 pt depth -2.6pt \kern -6pt
\intop}\nolimits_{#1}} {\mathop{\vrule width 5pt height 3 pt depth
-2.6pt \kern -6pt \intop}\nolimits_{#1}}}

\newcommand{\R}{\mathbb{R}}

\begin{document}

\title[Introduction to Teichm\"uller dynamics]{Introduction to Teichm\"uller theory and its applications to dynamics of interval exchange transformations, 
flows on surfaces and billiards}

\author[Giovanni Forni]{Giovanni Forni}
\address{Department of Mathematics, University of Maryland, College Park, MD 20742-4015, USA.}
\email{gforni@math.umd.edu}

\author[Carlos Matheus]{Carlos Matheus}
\address{CNRS, LAGA, Institut Galil\'ee, Universit\'e Paris 13, 99, Av. Jean-Baptiste Cl\'ement 93430, Villetaneuse, France.}
\email{matheus@impa.br}
\thanks{The authors are thankful to their coauthors A. Avila, J.-C. Yoccoz and A. Zorich for the pleasure of working with them on some of the topics discussed in these notes. C.M. is grateful to his wife Aline Cerqueira for the immense help with several figures in this text. C.M. was partially supported by the Balzan Research Project of J. Palis.}

\date{\today}


\begin{abstract}
This text is an expanded version of the lecture notes of a minicourse (with the same title of this text) delivered by the authors in the Bedlewo school ``Modern Dynamics and its Interaction with Analysis, Geometry and Number Theory'' (from 4 to 16 July, 2011). 

In the first part of this text, i.e., from Sections 1 to 5, we discuss the Teichm\"uller and moduli space of translation surfaces, the Teichm\"uller flow and the $SL(2,\mathbb{R})$-action on these moduli spaces
and  the Kontsevich--Zorich cocycle over the Teichm\"uller geodesic flow.  We sketch two applications of the ergodic properties of the Teichm\"uller flow and Kontsevich--Zorich cocycle, with respect to Masur--Veech measures, to the unique ergodicity, deviation of ergodic averages and weak mixing properties of typical interval exchange transformations and translation flows. These applications are based on the fundamental fact that the Teichm\"uller flow and the Kontsevich--Zorich cocycle work as \emph{renormalization dynamics} for interval exchange transformations and translation flows. 

In the second part, i.e., from Sections 6 to 9, we start by pointing out that it is interesting to study the ergodic properties of the Kontsevich--Zorich cocycle with respect to invariant measures  other than the Masur--Veech ones, in view of potential applications to the investigation of billiards in rational polygons (for instance). We then study some examples of measures for which the ergodic properties of the Kontsevich--Zorich cocycle are very different from the case of Masur--Veech measures. Finally, we end these notes by constructing some examples of closed $SL(2,\mathbb{R})$-orbits such that the restriction of the Teichm\"uller flow to them has arbitrary small rate of exponential mixing, or, equivalently, the naturally associated unitary $SL(2,\mathbb{R})$-representation has arbitrarily small spectral gap 
(and in particular it has complementary series).
\end{abstract}

\maketitle

\tableofcontents


\section{Quick review of basic elements of Teichm\"uller theory}\label{s.intro}

The long-term goal of these lecture notes is the study of the so-called \emph{Teichm\"uller geodesic flow} and its noble cousin the \emph{Kontsevich--Zorich cocycle}, and some of its applications to interval exchange transformations, translation flows and billiards. As any respectable geodesic flow, the 
Teichm\"uller flow acts naturally in a certain \emph{unit cotangent bundle}. More precisely, the phase space of the Teichm\"uller geodesic flow is the unit cotangent bundle of the \emph{moduli space of Riemann surfaces}.

In this initial section, we'll briefly recall some basic results of Teichm\"uller theory leading to the conclusion that the unit cotangent bundle of the moduli space of Riemann surfaces (i.e., the phase space of the Teichm\"uller flow) is naturally identified to the \emph{moduli space of quadratic differentials}. As we'll see later in this text, the relevance of this identification resides in the fact that it makes apparent the existence of a natural $SL(2,\mathbb{R})$ action on the moduli space of quadratic differentials which extends the action of the Teichm\"uller flow, in the sense that the Teichm\"uller flow corresponds to the sub-action of the diagonal subgroup $g_t:=\left(\begin{array}{cc}e^t & 0 \\ 0 & e^{-t}\end{array}\right)$ of $SL(2,\mathbb{R})$. In any event, the basic reference for this section is J. Hubbard's book~\cite{Hu}.

\subsection{Deformation of Riemann surfaces: moduli and Teichm\"uller spaces of curves} Let us consider two Riemann surface structures $M_0$ and $M_1$ on a fixed (topological) compact surface 
$S$ of genus $g\geq 1$. If $M_0$ and $M_1$ are not biholomorphic (i.e., they are ``distinct''), there is no way to produce a \emph{conformal map} (i.e., holomorphic map with non-vanishing derivative) $f:M_0\to M_1$. However, we can try to produce differentiable maps $f:M_0\to M_1$ which are as ``nearly conformal'' as possible. To do so, we need a reasonable way to ``measure'' the amount of ``non-conformality'' of $f$. A fairly standard procedure is the following one. Given a point $x\in M_0$ and some local coordinates around $x$ and $f(x)$, we write the derivative $Df(x)$ of $f$ at $x$ as $Df(x)u=\frac{\partial f}{\partial z}(x)u+\frac{\partial f}{\partial \overline{z}}(x)\overline{u}$, so that $Df(x)$ sends infinitesimal circles into infinitesimal ellipses of eccentricity
$$\frac{\left|\frac{\partial f}{\partial z}(x)\right|+\left|\frac{\partial f}{\partial \overline{z}}(x)\right|}{\left|\frac{\partial f}{\partial z}(x)\right|-\left|\frac{\partial f}{\partial\overline{z}}(x)\right|} = \frac{1+k(f,x)}{1-k(f,x)}:=K(f,x),$$
where $k(f,x):=\frac{\left|\frac{\partial f}{\partial\overline{z}}(x)\right|}{\left|\frac{\partial f}{\partial z}(x)\right|}$. This is illustrated in the figure below:

\begin{figure}[h!]\label{f.qc}
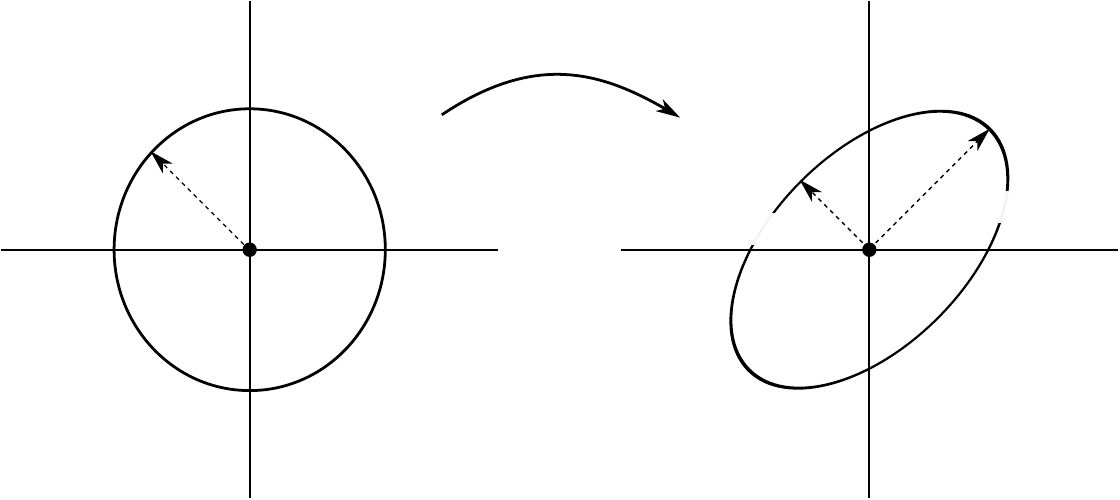
\end{figure}

We say that $K(f,x)$ is the \emph{eccentricity coefficient} of $f$ at $x$, while
$$K(f):=\sup\limits_{x\in M_0}K(f,x)$$
is the \emph{eccentricity coefficient} of $f$. Note that, by definition, $K(f)\geq 1$ and $f$ is a conformal map if and only if $K(f)=1$ (or equivalently $k(f,x)=0$ for every $x\in M_0$). Hence, $K(f)$ accomplishes the task of measuring the amount of non-conformality of $f$. We call $f:M_0\to M_1$ quasiconformal whenever $K=K(f)<\infty$.

In the next subsection, we'll see that quasiconformal maps are useful to compare distinct Riemann structures on a given topological compact surface $S$. In a more advanced language, we consider the \emph{moduli space} $\mathcal{M}(S)$ of Riemann surface structures on $S$ \emph{modulo conformal maps} and the Teichm\"uller space $\mathcal{T}(S)$ of Riemann surface structures on $S$ \emph{modulo conformal maps isotopic to the identity}. It follows that $\mathcal{M}(S)$ is the quotient of $\mathcal{T}(S)$ by the so-called \emph{modular group} (or \emph{mapping class group}) $\Gamma(S):=\Gamma_g:=\textrm{Diff}^+(S)/\textrm{Diff}_0^+(S)$ of isotopy classes of diffeomorphisms of $S$ (here $\textrm{Diff}^+(S)$ is the set of orientation-preserving diffeomorphisms and $\textrm{Diff}_0^+(S)$ is the set of orientation-preserving diffeomorphisms isotopic to the identity). Therefore, the problem of studying deformations of Riemann surface structures corresponds to the study of the nature of the moduli space $\mathcal{M}(S)$ (and of the Teichm\"uller space $\mathcal{T}(S)$).

\subsection{Beltrami differentials and Teichm\"uller metric} Let's come back to the definition of $K(f)$ in order to investigate the nature of the quantities $k(f,x):= \frac{\left|\frac{\partial f}{\partial\overline{z}}(x)\right|}{\left|\frac{\partial f}{\partial z}(x)\right|}$. Since we are dealing with Riemann surfaces (and we used local charts to perform calculations), $k(f,x)$ \emph{doesn't} provide a \emph{globally defined} function on $M_0$. Instead, by looking at how $k(f,x)$ transforms under changes of coordinates, one can check that the quantities $k(f,x)$ can be collected to globally define a \emph{tensor} $\mu$ (of \emph{type} $(-1,1)$) via the formula:
$$\mu(x)=\frac{\frac{\partial f}{\partial\overline{z}}(x) d\overline{z}}{\frac{\partial f}{\partial z}(x)dz}\,.$$
In the literature, $\mu$ is called a \emph{Beltrami differential}. Note that $\|\mu\|_{L^{\infty}}<1$ when $f$ is an orientation-preserving quasiconformal map. The intimate relationship between quasiconformal maps and Beltrami differentials is revealed by the following profound theorem of L.~Ahlfors and L.~Bers:

\begin{theorem}[Measurable Riemann mapping theorem]Let $U\subset\mathbb{C}$ be an open subset and consider $\mu\in L^{\infty}(U)$ verifying $\|\mu\|_{L^{\infty}}<1$. Then, there exists a quasiconformal mapping $f:U\to\mathbb{C}$ such that the \emph{Beltrami equation}
$$\frac{\partial f}{\partial \overline{z}} = \mu \frac{\partial f}{\partial z}$$
is satisfied in the sense of distributions. Furthermore, $f$ is unique modulo composition with conformal maps: if $g$ is another solution of Beltrami equation above, then there exists an injective conformal map $\phi: f(U)\to\mathbb{C}$ such that $g=\phi\circ f$.
\end{theorem}

A direct consequence of this striking result for the deformation of Riemann surface structures is the following proposition (whose proof is left as an exercise to the reader):

\begin{proposition}Let $X$ be a Riemann surface and $\mu$ a Beltrami differential on $X$. Given an atlas $\{(U_i, \phi_i)\}$ of $X$ (where $\phi_i: U_i \to \mathbb{C}$), denote by $\mu_i$ the function on $V_i:=\phi_i(U_i)\subset\mathbb{C}$ defined by
$$\mu|_{U_i}=\phi_i^*\left(\mu_i\frac{d\overline{z}}{dz}\right).$$
Then, there is a family of mappings $\psi_i(\mu):V_i\to\mathbb{C}$ solving the Beltrami equations
$$\frac{\partial \psi_i(\mu)}{\partial \overline{z}} = \mu_i \frac{\partial \psi_i(\mu)}{\partial z}$$
such that $\psi_i(\mu)$ are homeomorphisms from $V_i$ to $\psi_i(\mu)(V_i)$.

Moreover, $\psi_i\circ\phi_i:U_i\to\mathbb{C}$ form an atlas giving a well-defined Riemann surface structure $X_\mu$ in the sense that it is independent of the initial choice of the atlas $\phi_i:U_i\to\mathbb{C}$ and the choice of $\phi_i$ verifying the corresponding Beltrami equations.
\end{proposition}

In other words, the measurable Riemann mapping theorem of Alhfors and Bers implies that one can use Beltrami differentials to naturally deform Riemann surfaces through quasiconformal mappings. Of course, we can ask to what extend this is a general phenomena: namely, given two Riemann surface structures $M_0$ and $M_1$ on $S$, can we relate them by quasiconformal mappings? The answer to this question is provided by the remarkable theorem of O. Teichm\"uller:

\begin{theorem}[O. Teichm\"uller]Given two Riemann surfaces structures $M_0$ and $M_1$ on a compact topological surface $S$ of genus $g\geq 1$, there exists a quasiconformal mapping $f:M_0\to M_1$ \emph{minimizing} the eccentricity coefficient $K(g)$ among all quasiconformal maps $g:M_0\to M_1$ isotopic to the identity map $\textrm{id}:S\to S$. Furthermore, whenever a quasiconformal map $f:M_0\to M_1$ minimizes the eccentricity coefficient in the isotopy class of a given orientation-preserving diffeomorphism $h:S\to S$, we have that the eccentricity coefficient of $f$ at any point $x\in M_0$ is constant, i.e.,
$$K(f,x)=K(f)$$
\emph{except} for a finite number of points $x_1,\dots,x_n\in M_0$. Also, quasiconformal mappings minimizing the eccentricity coefficient in a given isotopy class are unique modulo (pre and post) composition with conformal maps.
\end{theorem}

In the literature, any such minimizing quasiconformal map in a given isotopy class is called an \emph{extremal map}. By using the extremal quasiconformal mappings, we can naturally introduce a \emph{distance} between two Riemann surface structures $M_0$ and $M_1$ by
$$d(M_0,M_1)=\frac{1}{2}\ln K(f)$$
where $f:M_0\to M_1$ is an extremal map isotopic to the identity. The metric $d$ is called \emph{Teichm\"uller metric}. The main focus of these notes is the study of the geodesic flow associated to the Teichm\"uller metric on the moduli space of Riemann surfaces. As we anticipated in the introduction, it is quite convenient to regard a geodesic flow as a flow defined on the cotangent bundle of the underlying space. The discussion of the cotangent bundle of $\mathcal{T}(S)$ is the subject of the
next subsection.

\subsection{Quadratic differentials and the cotangent bundle of the moduli space of curves} The results of the previous subsection show that the Teichm\"uller space is modeled on the space of Beltrami differentials. Recall that Beltrami differentials are measurable tensors $\mu$ of type $(-1,1)$ such that $\|\mu\|_{L^{\infty}}<1$. It follows that the tangent bundle to $\mathcal{T}(S)$ is modeled on the space of measurable and essentially bounded ($L^{\infty}$) tensors of type $(-1,1)$ (because Beltrami differentials form the unit ball of this Banach space). Hence, the cotangent bundle to $\mathcal{T}(S)$ can be identified with the space $\mathcal{Q}(S)$ of \emph{integrable quadratic differentials} on $S$, i.e., the space of (integrable) tensors $q$ of type $(2,0)$ (that is, $q$ is written as $q(z)dz^2$ in a local coordinate $z$). In fact, we can determine the cotangent bundle once we can find an object (a tensor of some type) such that the pairing
$$\langle\mu,q\rangle=\int_S q\mu$$
is well-defined and continuous; when $\mu$ is a tensor of type $(-1,1)$ and $q$ is a tensor of type $(2,0)$, we can write $q\mu = q(z) \mu(z) dz^2 \frac{d\overline{z}}{dz} = q(z)\mu(z)dz\, d\overline{z} = q(z)\mu(z) |dz|^2$ in local coordinates, i.e., we obtain a tensor of type $(1,1)$, that is, an area form. Therefore,
since the Beltrami differential $\mu$ is  locally given by essentially bounded functions, we see that the requirement that this pairing makes sense is equivalent to ask that the tensor $q$ of type $(2,0)$ is integrable.

Next, let's see how the geodesic flow associated to the Teichm\"uller metric looks like after the identification of the cotangent bundle of $\mathcal{T}(S)$ with the space $\mathcal{Q}(S)$ of integrable quadratic differentials. Firstly, we need to investigate more closely the geometry of extremal quasiconformal maps between two Riemann surfaces. To do so, we recall another notable theorem of O. Teichm\"uller:

\begin{theorem}[O. Teichm\"uller]\label{t.Teichmuller}Given an extremal map $f:M_0\to M_1$, there is an atlas $\{(U_i,\phi_i)\}$ (where $\phi_i: U_i \to \mathbb C$) on $M_0$ compatible with the underlying complex structure such that
\begin{itemize}
\item the changes of coordinates $\phi_i \circ \phi_j^{-1}: \phi_j(U_i\cap U_j) \to \phi_i(U_i\cap U_j)$ 
are all of the form $z\mapsto \pm z+c$, $c\in\mathbb{C}$, outside the neighborhoods of a finite number 
of points on $S$,
\item the \emph{horizontal} (resp., \emph{vertical}) foliation $\{\Im \phi_i = 0\}$ (resp., $\{\Re\phi_i=0\}$) is tangent to the major (resp.minor) axis of the infinitesimal ellipses obtained as the images of infinitesimal circles under the derivative $Df$, and
\item in terms of these coordinates, $f$ expands the horizontal direction by the constant factor of $\sqrt{K}$ and $f$ contracts the vertical direction by the constant factor of $1/\sqrt{K}$.
\end{itemize}
\end{theorem}

An atlas $\{(U_i,\phi_i)\}$ satisfying the property of the first item of Teichm\"uller theorem above is called 
a \emph{half-translation structure}. In this language, Teichm\"uller's theorem says that extremal maps $f:M_0\to M_1$ (i.e., deformations of Riemann surface structures) can be easily understood in terms of half-translation structures: it suffices to expand (resp., contract) the corresponding horizontal (resp., vertical) foliation on $M_0$ by a constant factor equal to $e^{d(M_0,M_1)}$ in order to get a horizontal (resp., vertical) foliation of a half-translation structure compatible with the Riemann surface structure of $M_1$. This provides a simple way to describe the Teichm\"uller geodesic flow in terms of half-translation structures.
Thus, it remains to relate half-translation structures with quadratic differentials to get a pleasant formulation of this geodesic flow. While we could accomplish this task right now, we'll postpone this discussion to the third section of these notes for two reasons:
\begin{itemize}
\item Teichm\"uller geodesic flow is naturally embedded into a $SL(2,\mathbb{R})$-action (as a consequence of this relationship between half-translation structures and quadratic differentials), so that it is preferable to give a unified treatment of this fact later;
\item for pedagogical motivations, once we know that quadratic differentials is the correct object to study, it seems more reasonable to introduce the fine structures of the space $\mathcal{Q}(S)$ \emph{before} introducing the dynamics on this space (than the other way around).
\end{itemize}

In particular, we'll proceed as follows: for the remainder of this subsection, we'll briefly sketch the \emph{bijective correspondence} between half-translation structures and quadratic differentials; after that, we make some remarks on the Teichm\"uller metric (and other metric structures on $\mathcal{Q}(S)$) and we pass to the next subsection where we work out the particular (but important) case of genus 1 surfaces; then, in the spirit of the two items above, we devote Section~\ref{s.orbifold} to the fine structures of 
$\mathcal{Q}(S)$, and Section~\ref{s.general-dynamics} to the dynamics on $\mathcal{Q}(S)$.

Given a half-translation structure $\{(U_i, \phi_i)\}$ (where $\phi_i:U_i\to\mathbb{C}$) on a Riemann surface $S$, one can easily construct a quadratic differential $q$ on $S$ by pulling back the quadratic differential $dz^2$ on $\mathbb{C}$ through the map $\phi_i$ on every $U_i\subset S$: indeed, this procedure leads to a well-defined global quadratic differential on $S$ because we are assuming that the changes of coordinates (outside the neighborhoods of finitely many points) have the form $z\mapsto \pm z + c$. Conversely, given a quadratic differential $q$ on a Riemann surface $S$, we take an atlas $\{(U_i, \phi_i)\}$ (where 
$\phi_i:U_i\to\mathbb{C}$) such that $q|_{U_i}=\phi_i^*(dz^2)$ outside the neighborhoods of finitely many singularities of $q$. Note that the fact that $q$ is obtained by pulling back the quadratic differential $dz^2$ on $\mathbb{C}$ means that the associated changes of coordinates $z\mapsto z'$ send the quadratic differential $dz^2$ to $(dz')^2$. Thus, our changes of coordinates outside the neighborhoods of the singularities of $q$ have the form $z\mapsto z'=\pm z+c$, i.e., $\{(U_i,\phi_i)\}$ is a half-translation structure. 

\begin{remark} Generally speaking, a quadratic differential on a Riemann surface is either \emph{orientable} or \emph{non-orientable}. More precisely, given a quadratic differential $q$, consider the underlying half-translation structure $\{(U_i,\phi_i)\}$ and define two foliations by $\{\Im\phi_i=c\}$ and $\{\Re\phi_i=c\}$ (these are called the \emph{horizontal} and \emph{vertical} foliations associated to $q$). We say that $q$ is orientable if these foliations are orientable and $q$ is non-orientable otherwise. Alternatively, we say that $q$ is orientable if the changes of coordinates of the underlying half-translation structure $\{(U_i,\phi_i)\}$ outside the singularities of $q$ on $S$ have the form $z\mapsto z+c$. Equivalently, $q$ is orientable if it is the global square of a holomorphic $1$-form, i.e., $q=\omega^2$, where $\omega$ is a holomorphic $1$-form, that is, an \emph{Abelian differential}. For the sake of simplicity of the exposition, from now on, we'll deal \emph{exclusively} with \emph{orientable} quadratic differentials $q$, or, more precisely, we'll restrict our attention to \emph{Abelian differentials}. The reason to doing so is two-fold: firstly, most of our computations below become easier and clearer in the orientable setting, and secondly, \emph{usually} (but not always) some results about Abelian differentials can be extended to the non-orientable setting by a \emph{double cover construction}, that is, one consider a (canonical) double cover of the initial Riemann surface, equipped with a non-orientable quadratic differential $q$, such that a global square of the lift of $q$ is well-defined. In the sequel, we denote the space of Abelian differentials on a compact surface $S$ of genus $g$ by $\mathcal{H}(S)$ 
or $\mathcal{H}_g$.
\end{remark} 

\begin{remark}
In Subsection~\ref{ss.translation-surfaces}, we will come back to the correspondence between quadratic differentials and half-translation structures in the context of Abelian differentials: more precisely, we will see there that \emph{Abelian differentials} bijectively correspond to the so-called \emph{translation structures}.
\end{remark}

We close this subsection with the following comments.

\begin{remark} The Teichm\"uller metric is induced by a family of norms on the fibers of the cotangent bundle $\mathcal{Q}(S)$ of Teichm\"uller space $\mathcal{T}(S)$ given by the $L^1$ norm of quadratic differentials (see Theorem 6.6.5 of~\cite{Hu}). However, these norms \emph{do not} come from  inner
products, hence the Teichm\"uller metric is \emph{not Riemannian}. In fact, it is only a reversible \emph{Finsler metric}, i.e., it is defined by a family of norms depending continuously on the base point.
\end{remark}

\begin{remark} The Teichm\"uller space $\mathcal{T}(S)$ of a compact surface $S$ of genus $g\geq 2$ is a nice pseudo-convex \emph{complex-analytic manifold} of complex dimension $3g-3$ and it is homeomorphic to the unit open ball of $\mathbb{C}^{3g-3}$, while the moduli space $\mathcal{M}(S)$ is a \emph{complex orbifold} in general. In fact, we are going to face this phenomenon in the next subsection (when we review the particular important case of genus $1$ curves).
\end{remark}

\begin{remark} Another important metric on Teichm\"uller spaces whose geometrical and dynamical properties are the subject of several recent papers is the so-called \emph{Weil-Petersson metric}. It is the metric coming from the Hermitian inner product $\langle q_1, q_2\rangle_{WP}:=\int_S \frac{\overline{q_1} q_2}{\rho_S^2}$ on $\mathcal{Q}(S)$, where $\rho_S$ is the hyperbolic metric of the Riemann surface $S$ and $\rho_S^2$ is the associated area form. A profound result says that Weil-Petersson metric is a \emph{K\"ahler metric}, i.e., the $2$-form $\Im\langle.,.\rangle_{WP}$ given by the imaginary part of the Weil-Petersson metric is closed. Furthermore, a beautiful theorem of S.~Wolpert states that this $2$-form admits a simple expression in terms of the \emph{Fenchel-Nielsen coordinates} on Teichm\"uller space. Other important facts about the Weil--Petersson geodesic flow (i.e., the geodesic flow associated to $\langle .,.\rangle_{WP}$) are:
\begin{itemize}
\item it is a natural example of singular hyperbolic dynamics, since the Weil--Petersson metric is a negatively curved, \emph{incomplete} metric with unbounded sectional curvatures (i.e., the sectional curvatures can approach $0$ and/or $-\infty$ in general);
\item it  is defined for all times on a set of full measure of $\mathcal{Q}(S)$ (S.~Wolpert);
\item  it is transitive, it has a dense set of periodic orbits and it has infinite topological entropy (J.~Brock, H.~Masur and Y.~Minsky);
\item  it is ergodic with respect to the Weil--Petersson volume form
(K.~Burns, H.~Masur and A.~Wilkinson, building on important previous work of S.~Wolpert and 
C.~McMullen).
\end{itemize}
We refer to the excellent introduction of the paper~\cite{BMW} (and references therein) of K.~Burns, 
H.~Masur and A.~Wilkinson for a nice account of the Weil-Petersson metric. Ending this remark, we note that the basic difference between the Teichm\"uller metric and the Weil-Petersson metric is the following: as we already indicated, the Teichm\"uller metric is related to flat (half-translation) structures, while the Weil-Petersson metric can be better understood in terms of hyperbolic structures.
\end{remark}

\subsection{An example: Teichm\"uller and moduli spaces of elliptic curves (torii)}\label{ss.1.torii}
The goal of this subsection is the illustration of the role of the several objects introduced previously in the concrete case of genus $1$ surfaces (elliptic curves). Indeed, we'll see that, in this particular case, one can do ``everything'' by hand.

We begin by recalling that an elliptic curve, i.e., a Riemann surface of genus $1$, is uniformized by the complex plane. In other words, any elliptic curve is biholomorphic to a quotient $\mathbb{C}/\Lambda$ where $\Lambda\subset\mathbb{C}$ is a \emph{lattice}. Given a lattice $\Lambda\subset\mathbb{C}$ generated by two elements $w_1$ and $w_2$, that is, $\Lambda=\mathbb{Z}w_1\oplus\mathbb{Z}w_2$, we see that the multiplication by $1/w_1$ or $1/w_2$ provides a biholomorphism \emph{isotopic to the identity} between $\mathbb{C}/\Lambda$ and $\mathbb{C}/\Lambda(w)$, where $\Lambda(w):=\mathbb{Z}\oplus\mathbb{Z}w$ is the lattice generated by $1$ and $w\in\mathbb{H}\subset\mathbb{C}$ (the upper-half plane of the complex plane). In fact, $w=w_2/w_1$ or $w=w_1/w_2$ here. Next, we observe that any biholomorphism $f$ between $\mathbb{C}/\Lambda(w')$ and $\mathbb{C}/\Lambda(w)$ can be lifted to an automorphism $F$ of the complex plane $\mathbb{C}$. This implies that $F$ has the form $F(z)=Az+B$ for some $A,B\in\mathbb{C}$. On the other hand, since $F$ is a lift of $f$, we can find $\alpha,\beta,\gamma,\delta\in\mathbb{Z}$ such that
\begin{equation*}
\left\{ \begin{array}{c} F(z+1)-F(z) = \delta+\gamma w \\ F(z+w')-F(z) = \beta+\alpha w \end{array} \right. .
\end{equation*}
Expanding these equations by using the fact that $F(z)=Az+B$, we get
$$w'=\frac{\alpha w+\beta}{\gamma w+\delta}\,.$$
Also, since we're dealing with invertible maps ($f$ and $F$), it is not hard to check that $\alpha\delta-\beta\gamma= \pm 1$ (because it is an  integer whose inverse is also an integer) and, in fact, $\alpha\delta-\beta\gamma= 1$ since $w$, $w'\in  \mathbb H$.  In other words, recalling that $SL(2,\mathbb{R})$ acts on $\mathbb{H}$ via
$$\left(\begin{array}{cc} a & b \\ c & d\end{array}\right)\in SL(2,\mathbb{R}) \longleftrightarrow z\mapsto\frac{az+b}{cz+d},$$
we see that the torii $\mathbb{C}/\Lambda(w)$ and $\mathbb{C}/\Lambda(w')$ are biholomorphic if and only if $w'\in SL(2,\mathbb{Z})\cdot w$.

For example, we show below the torii $\mathbb{C}/\Lambda(i)$ (on the left) and $\mathbb{C}/\Lambda(1+i)$ (in the middle). 
\begin{figure}[htb!]
\includegraphics[scale=0.4]{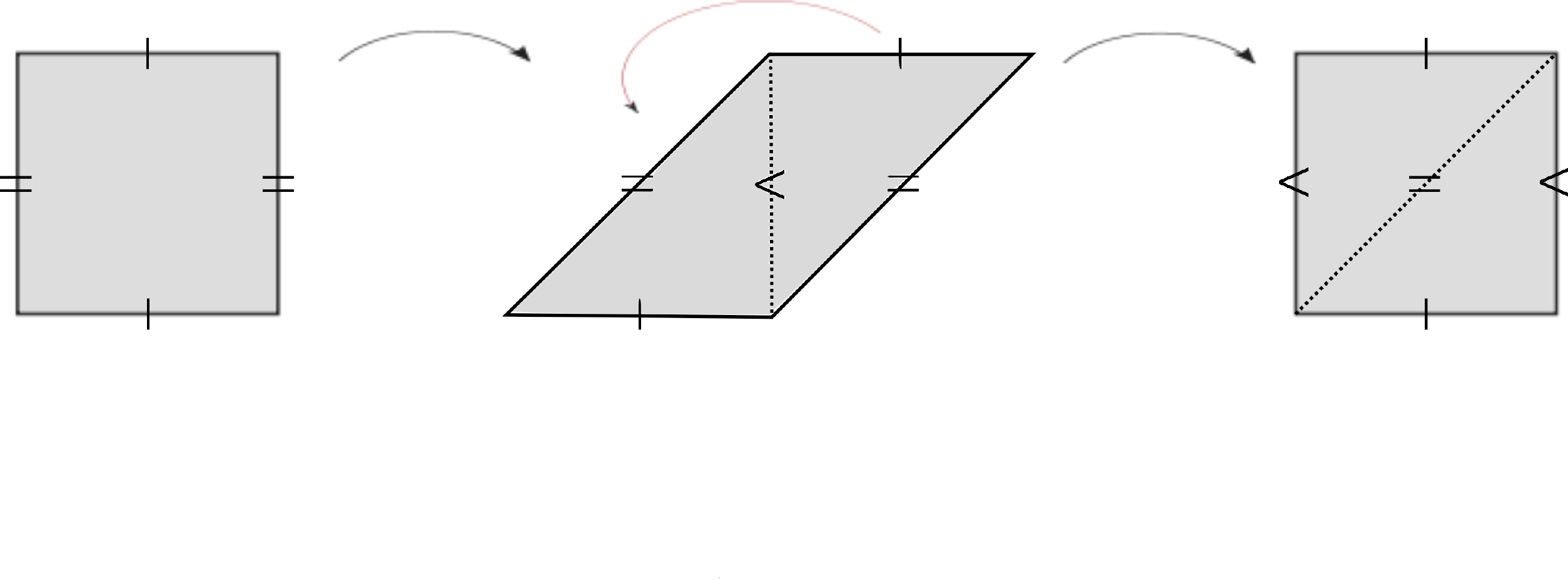}
\end{figure}

Since $1+i$ is deduced from $i$ via the action of $T=\left(\begin{array}{cc} 1 & 1 \\ 0 & 1\end{array}\right)\in SL(2,\mathbb{Z})$ on $\mathbb{H}$, we see that these torii are biholomorphic and hence they represent the \emph{same} point in the \emph{moduli space} $\mathcal{M}_{1,1}$ (see the right hand side part of the figure above). On the other hand, they represent \emph{distinct} points in the \emph{Teichm\"uller space} $\mathcal{T}_{1,1}$ (because $T\neq\textrm{id}$, hence $T$ is not isotopic to the identity).

Our discussion so far implies that the Teichm\"uller space $\mathcal{T}_{1,1}$ of elliptic curves with a marked point is naturally identified with the upper-half plane $\mathbb{H}$ and the moduli space 
$\mathcal{M}_{1,1}$ of elliptic curves with a marked point is naturally identified with $\mathbb{H}/SL(2,\mathbb{Z})$. Furthermore, it is possible to show that, under this identification, the Teichm\"uller metric on $\mathcal{T}_{1,1}$ corresponds to the hyperbolic metric (of constant curvature $-1$) on $\mathbb{H}$, so that the Teichm\"uller geodesic flow on $\mathcal{T}_{1,1}$ and $\mathcal{M}_{1,1}$ are the geodesic flows of the hyperbolic metric on $\mathbb{H}$ and $\mathbb{H}/SL(2,\mathbb{Z})$. In order to better understand the moduli space $\mathcal{M}_{1,1}$, we'll make the geometry of the quotient $\mathbb{H}/SL(2,\mathbb{Z})$ (called \emph{modular curve} in the literature) more clear by presenting a \emph{fundamental domain} of the $SL(2,\mathbb{Z})$-action on $\mathbb{H}$. It is a classical result (see Proposition 3.9.14 of~\cite{Hu}) that the region $F_0:=\{z\in\mathbb{H}: -1/2\leq \Re z\leq 1/2 \textrm{ and } |z|\geq 1\}$ is a fundamental domain of this action, in the sense that every $SL(2,\mathbb{Z})$-orbit intersects $F_0$, but it can intersect the \emph{interior} $\textrm{int}(F_0)$ of $F_0$ \emph{at most} once. In the specific case at hand, $SL(2,\mathbb{Z})$ acts on the boundary $\partial F_0$ of $F_0$, that is, on the closed set  $\partial F_0=\{|z|\geq 1 \textrm{ and } \Re z=\pm 1/2\}\cup \{|z|=1 \textrm{ and } |\Re z|\leq 1/2\}$ by sending
\begin{itemize}
\item $\{|z|\geq 1 \textrm{ and } \Re z=-1/2\}$ to $\{|z|\geq 1 \textrm{ and } \Re z=1/2\}$ through the translation $z\mapsto z+1$ or equivalently the parabolic matrix $T=\left(\begin{array}{cc}1 & 1 \\ 0 & 1\end{array}\right)$, and
\item $\{-1/2\leq \Re z\leq 0 \textrm{ and } |z|=1\}$ to $\{0\leq \Re z\leq 1/2 \textrm{ and } |z|=1\}$ through the ``inversion'' $z\mapsto -1/z$ or equivalently the elliptic (rotation) matrix $J=\left(\begin{array}{cc} 0 & -1 \\ 1 & 0\end{array}\right)$.
\end{itemize}
See the figure below for an illustration of the fundamental domain $F_0$:
\begin{figure}[h!]
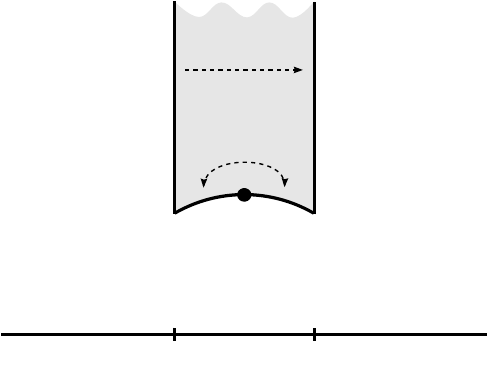
\end{figure}

This explicit description of the genus $1$ case allows us to clarify the role of the several objects introduced above. From the dynamical point of view, it is more interesting to consider the Teichm\"uller flow on moduli spaces than on Teichm\"uller spaces: indeed, the Teichm\"uller flow on Teichm\"uller space is somewhat boring (for instance, it is \emph{not} recurrent), while it is very interesting on moduli spaces (for instance, in the genus $1$ case [i.e., the geodesic flow on the modular curve], it exhibits all nice features [recurrence, exponential mixing, ...] of hyperbolic systems [these properties are usually derived from the connection with continued fractions]). However, from the point of view of analytic structures, Teichm\"uller spaces are better than moduli spaces because Teichm\"uller spaces are complex-analytic manifolds while moduli spaces are \emph{orbifolds}\footnote{In general, the mapping class group doesn't act properly discontinuously on Teichm\"uller space because some Riemann surfaces are ``more symmetric'' (i.e., they have larger automorphisms group) than others. In fact, we already saw this in the case of genus $1$: the modular curve $\mathbb{H}/SL(2,\mathbb{Z})$ isn't smooth near the points $w=i$ and $w=e^{\pi i/3}$ because the (square and hexagonal) torii corresponding to these points have larger automorphisms groups when compared with a typical torus $\mathbb{C}/\Lambda(w)$.}. In any case, it is natural to consider both spaces from the topological point of view since Teichm\"uller spaces are simply connected and thus they are isomorphic to the universal covers of moduli spaces. Finally, closing this section, we note that our discusssion above also shows that, in the genus $1$ case, the mapping class group $\Gamma_1$ is $SL(2,\mathbb{Z})$.


\section{Some structures on the set of Abelian differentials}\label{s.orbifold}

We denote by $\mathcal{L}_g$ the set of Abelian differentials on a Riemann surface of genus $g\geq 1$, or more precisely, the set of pairs $(M, \omega)$ where $M$ denotes a \textrm{Riemann surface structure} on a compact topological surface $S$ of genus $g$ and $\omega$ is a holomorphic $1$-form on $S$. In this notation, the \emph{Teichm\"uller space of Abelian differentials} is the quotient $\mathcal{TH}_g:=\mathcal{L}_g/\textrm{Diff}^+_0(S)$ and the \emph{moduli space of Abelian differentials} is the quotient $\mathcal{H}_g:=\mathcal{L}_g/\Gamma_g$. Here $\textrm{Diff}^+_0(S)$ and $\Gamma_g:=\textrm{Diff}^+(S)/\textrm{Diff}^+_0(S)$ (resp. the set of diffeomorphisms isotopic to the identity and the mapping class group) act on the set of Riemann surface structure in the usual manner, that is, by precomposition with coordinate maps,  while they act on Abelian differentials by pull-back.

In order to equip $\mathcal{L}_g$, $\mathcal{TH}_g$ and $\mathcal{H}_g$ with nice structures, we need a more ``concrete'' presentation of Abelian differentials. In the next subsection, we will see that the notion of a translation structure provides such a description of Abelian differentials.

\subsection{Abelian differentials and translation structures}\label{ss.translation-surfaces} Let
$(M,\omega)$ be the pair  of a Riemann surface $M$ and a holomorphic differential $\omega$ on $M$.
Let $\Sigma \subset M$ denote the set of singularities of $\omega$ on $M$, that is, the finite 
set $\{p\in M \vert \omega(p)=0\}$. Given any point $p\in M-\Sigma$, let's select $U_p$ a small path-connected neighborhood of $p$ such that $U_p\cap\Sigma=\emptyset$.
In this setting, the ``period'' map $\phi_p:U_p\to \mathbb{C}$, $\phi_p(x):=\int_p^x\omega$, obtained by integration along a path inside $U_p$ connecting $p$
and $x$, is well-defined: indeed, this follows from the fact that any Abelian differential is a closed $1$-form (because it is holomorphic), so that the integral $\int_p^z\omega$ doesn't depend on the choice of the path inside $U_p$ connecting $p$ and $x$. Furthermore, since $p\in M-\Sigma$ (so that $\omega(p)\neq 0$), we see that, by reducing $U_p$ if necessary, this ``period'' map is a biholomorphism.

In other words, the collection of all such ``period'' maps $\phi_p$ provides an atlas $\{(U_p, \phi_p)\}$ of $M-\Sigma$ compatible with the Riemann surface structure. Also, by definition,
the push-forward of the Abelian differential $\omega$ by any such $\phi_p$ is precisely the canonical Abelian differential $(\phi_p)_*(\omega)=dz$ on the open subset $\phi_p(U_p)$ of
the complex plane $\mathbb{C}$. Moreover, the ``local'' equality $\int_p^x\omega = \int_q^x \omega + \int_p^q\omega$ implies that all changes of coordinates
have the form $\phi_q\circ\phi_p^{-1}(z)=z+c$ where $c=\int_q^p\omega\in\mathbb{C}$ is a constant (since it doesn't depend on $z$). Furthermore, since
$\omega$ has finite order at its zeroes, it is easy to deduce from Riemann's theorem on removal singularities that this atlas of $M-\Sigma$ can be extended to
$M$ in such a way that the push-forward of $\omega$ by a local chart around a zero $p\in\Sigma$ of order $k$ is the holomorphic form $z^k dz$.

In the literature, a maximal family of compatible atlases whose change of coordinates maps are given by translations $z\mapsto z+c$ of the complex plane, outside a finite set of points $\Sigma\subset S$, is called a \emph{translation surface structure} on $S$. In this language, our previous discussion simply
says that any non-trivial Abelian differential $\omega$ on a compact Riemann surface $M$ gives rise  
to a translation surface structure on $M$ such that $\omega$ is locally the
pull-back of the canonical holomorphic form $dz$ on $\mathbb{C}$. On the other hand, it is clear that every translation surface structure on a topological surface $S$ determines a
Riemann surface $M$ (since translations are a very particular case of biholomorphic maps) and an Abelian differential $\omega$ on $M$ (by locally pulling back the Abelian differential $dz$ on $\mathbb{C}$ under the coordinate maps of any translation structure atlas: this pull-back is a well-defined Abelian differential on $M$ since by definition the coordinate changes are translations  and $dz$ is translation-invariant).

In summary, we have just seen the proof of the following proposition:

\begin{proposition}
The set $\mathcal{L}_g$ of all Abelian differentials on (compact) Riemann surfaces of genus $g\geq 1$ is canonically identified to the set of all translation structures on a (compact) topological surface $S$ of genus $g\geq 1$.
\end{proposition}

\begin{example} In Riemann surfaces courses, a complex torus is quite often presented through a translation surface structure: indeed, by giving a lattice $\Lambda=\mathbb{Z}w_1\oplus\mathbb{Z}w_2\subset\mathbb{C}$, we are saying that the complex torus $\mathbb{C}/\Lambda$ equipped with the (non-vanishing) Abelian differential $dz$ is canonically identified with the translation surface structure represented in the picture below (it truly represents a translation structure since we're gluing opposite parallel sides of the parallelogram determined by $w_1$ and $w_2$ through the translations $z\mapsto z+w_1$ and $z\mapsto z+w_2$).
\end{example}

\begin{figure}[htb!]
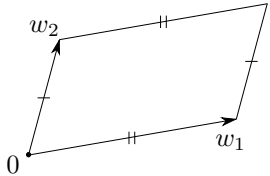
\caption{Complex torus as a translation surface}\label{f.torus}
\end{figure}

\begin{example}\label{e.DCBA} Let us consider the polygon $P$ of Figure~\ref{f.DCBA} below.
\begin{figure}[htb!]
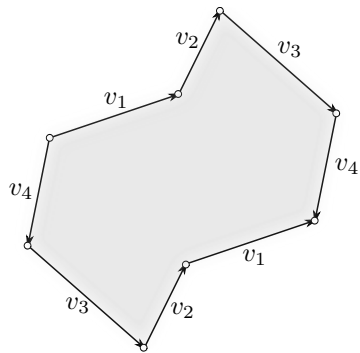
\caption{A genus 2 translation surface}\label{f.DCBA}
\end{figure}

In this picture, we are gluing parallel opposite sides $v_j$, $j=1,\dots,4$, of $P$, so that this is again a valid presentation of a translation surface structure. Let's denote by $(M,\omega)$ the corresponding Riemann surface and Abelian differential. Observe that, by following the sides identifications as indicated in this figure, we see that the vertices of $P$ are all identified to a single point $p$. Moreover, we see that $p$ is a special point when compared with any point of $P-\{p\}$ because, by turning around $p$, we note that the ``total angle''  is $6\pi$, while the total angle around any point of $P-\{p\}$ is $2\pi$, that is, a neighborhood of $p$ inside $M$ looks like  ``3 copies'' of the \emph{flat} complex plane stitched together, while a neighborhood of any other point $q\neq p$ resembles to a single copy of the flat complex plane. In other words, a natural local coordinate around $p$ is $\zeta=z^3$, so that $\omega=d\zeta=d(z^3)=3z^2dz$, i.e., the Abelian differential $\omega$ has a unique zero of order $2$ at $p$. From this, we can infer that $M$ is a compact Riemann surface of genus $2$: indeed,  by Riemann-Hurwitz theorem, the sum of orders of zeroes of an Abelian differential equals $2g-2$ (where $g$ is the genus); in the present case, this means $2 = 2g-2$, i.e., $g=2$; alternatively, one can apply the Poincar\'e-Hopf index theorem to the vector field given by the vertical direction on $M-\{p\}$ (this is well-defined because these points correspond to regular points of the polygon $P$) and vanishing at $p$ (where a choice of ``vertical direction'' doesn't make sense since we have multiple copies of the plane stitched together).
\end{example}

\begin{example}[Rational billiards]\label{ex.rat-polyg} Let $P$ be a rational polygon, that is, a polygon whose angles are all rational multiples of $\pi$. Consider the billiard on $P$: the trajectory of a point in $P$ in a certain direction consists of a straight line until we hit the boundary $\partial P$ of the polygon; at this moment, we apply the usual reflection laws (saying that the angle between the outgoing ray and $\partial P$ is the same as the angle between the incoming ray and $\partial P$) to prolongate the trajectory. See the figure below for an illustration of such an trajectory. 
\begin{figure}[htb!]\label{f.curso11}
\includegraphics[scale=1]{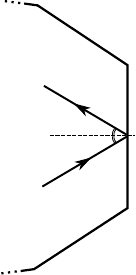}
\end{figure}

In the literature, the study of general billiards (where $P$ is not necessarily a polygon) is a classical subject with physical origins (e.g., mechanics and thermodynamics of Lorenz gases). In the particular case of billiards in rational polygons, an unfolding construction (due to R. Fox and R. Keshner \cite{FK}, and A. Katok and A. Zemlyakov \cite{KaZe}) allows to think of billiards on rational polygons as translation flows on translation surfaces. Roughly speaking, the idea is that each time the trajectory hits the boundary $\partial P$, instead of reflecting the trajectory, we reflect the table itself so that the trajectory remains a straight line: 
\begin{figure}[htb!]\label{f.curso12}
\includegraphics[scale=0.7]{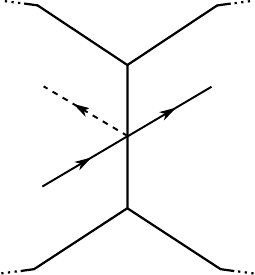}
\end{figure}

The group $G$ generated by the reflections with respect to the edges of $P$ is finite when $P$ is a rational polygon, so that the natural surface $X$ obtained by this unfolding procedure is compact. Furthermore, the surface $X$ comes equipped with a natural translation structure, and the billiard dynamics on $P$ becomes the translation (straight line) flow on $X$. In the picture below we have drawn the translation surface (Swiss cross) obtained by unfolding a L-shaped polygon, 
\begin{center}

\begingroup
  \makeatletter
  \providecommand\color[2][]{%
    \errmessage{(Inkscape) Color is used for the text in Inkscape, but the package 'color.sty' is not loaded}
    \renewcommand\color[2][]{}%
  }
  \providecommand\transparent[1]{%
    \errmessage{(Inkscape) Transparency is used (non-zero) for the text in Inkscape, but the package 'transparent.sty' is not loaded}
    \renewcommand\transparent[1]{}%
  }
  \providecommand\rotatebox[2]{#2}
  \ifx\svgwidth\undefined
    \setlength{\unitlength}{300pt}
  \else
    \setlength{\unitlength}{\svgwidth}
  \fi
  \global\let\svgwidth\undefined
  \makeatother
  \begin{picture}(1,0.53568502)%
    \put(0,0){\includegraphics[width=\unitlength]{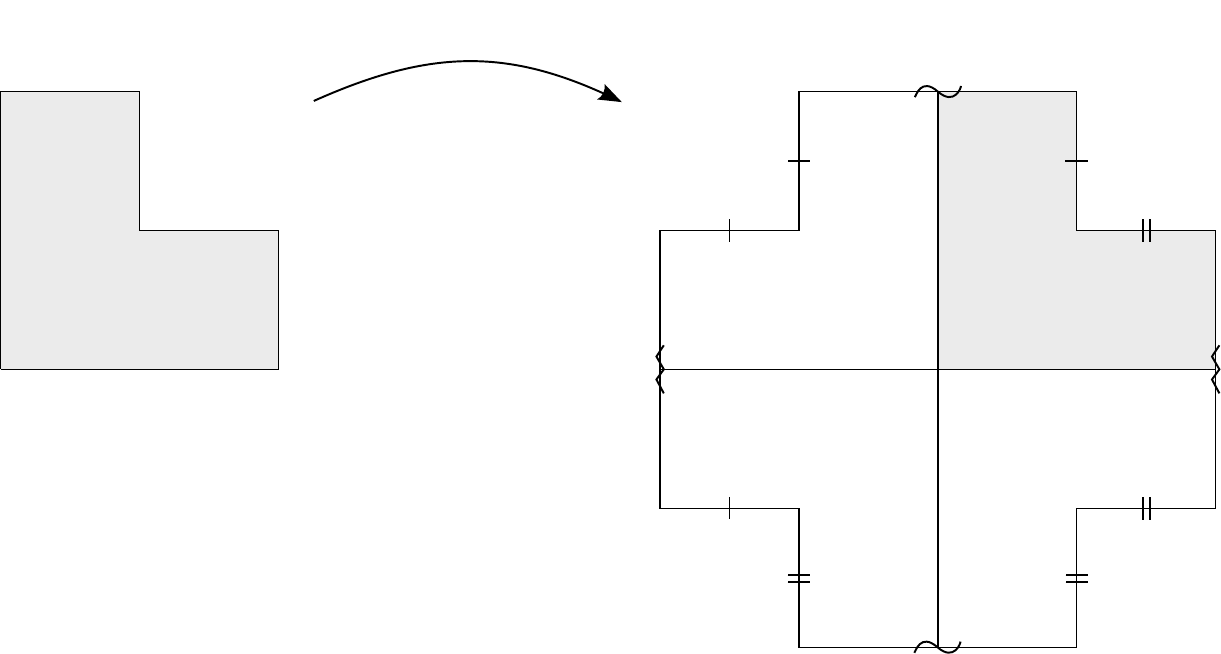}}%
    \put(0.32317559,0.50291327){\it\color[rgb]{0,0,0}\makebox(0,0)[lb]{\smash{unfolding}}}%
  \end{picture}%
\endgroup

\end{center}
and in the picture below we have drawn the translation surface (regular octagon) obtained by unfolding a triangle with angles $\pi/8$, $\pi/2$ and $3\pi/8$. 
\begin{center}

\begingroup
  \makeatletter
  \providecommand\color[2][]{%
    \errmessage{(Inkscape) Color is used for the text in Inkscape, but the package 'color.sty' is not loaded}
    \renewcommand\color[2][]{}%
  }
  \providecommand\transparent[1]{%
    \errmessage{(Inkscape) Transparency is used (non-zero) for the text in Inkscape, but the package 'transparent.sty' is not loaded}
    \renewcommand\transparent[1]{}%
  }
  \providecommand\rotatebox[2]{#2}
  \ifx\svgwidth\undefined
    \setlength{\unitlength}{326.04377441pt}
  \else
    \setlength{\unitlength}{\svgwidth}
  \fi
  \global\let\svgwidth\undefined
  \makeatother
  \begin{picture}(1,0.59937552)%
    \put(0,0){\includegraphics[width=\unitlength]{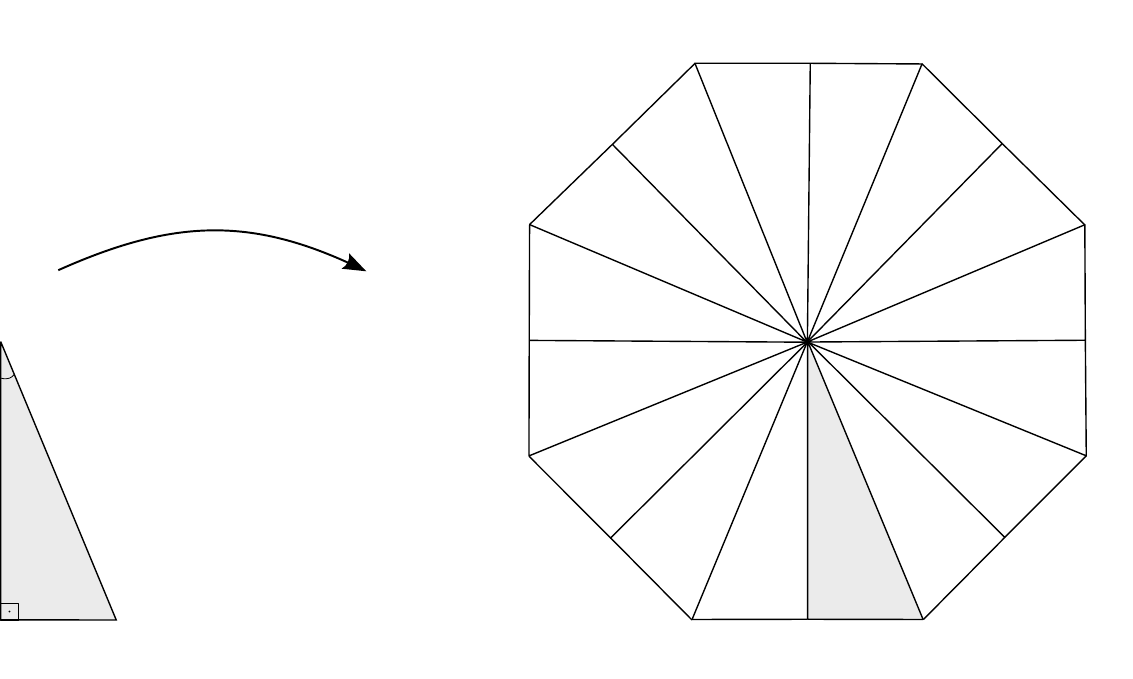}}%
    \put(0.12121468,0.41323041){\it\color[rgb]{0,0,0}\makebox(0,0)[lb]{\smash{unfolding}}}%
    \put(0.42291438,0.28742307){\it\color[rgb]{0,0,0}\makebox(0,0)[lb]{\smash{A}}}%
    \put(0.97253383,0.28742307){\it\color[rgb]{0,0,0}\makebox(0,0)[lb]{\smash{A}}}%
    \put(0.49897778,0.48616939){\it\color[rgb]{0,0,0}\makebox(0,0)[lb]{\smash{B}}}%
    \put(0.69649727,0.56959377){\it\color[rgb]{0,0,0}\makebox(0,0)[lb]{\smash{C}}}%
    \put(0.89647042,0.48616939){\it\color[rgb]{0,0,0}\makebox(0,0)[lb]{\smash{D}}}%
    \put(0.88910945,0.08376943){\it\color[rgb]{0,0,0}\makebox(0,0)[lb]{\smash{B}}}%
    \put(0.69649727,0.00034505){\it\color[rgb]{0,0,0}\makebox(0,0)[lb]{\smash{C}}}%
    \put(0.49407047,0.08131577){\it\color[rgb]{0,0,0}\makebox(0,0)[lb]{\smash{D}}}%
    \put(0.00495543,0.2060851){\it\color[rgb]{0,0,0}\makebox(0,0)[lb]{\smash{$\frac{\pi}{8}$}}}%
  \end{picture}%
\endgroup

\end{center}

In general, a rational polygon $P$ with $k$ edges and angles $\pi m_i/n_i$, $i=1,\dots, N$ has a group of reflections $G$ of order $2N$ and, by unfolding $P$, we obtain a translation surface $X$ of genus $g$ where 
$$2-2g=N(2-k+\sum(1/n_i))$$
In particular, it is possible to show that the only polygons which give rise to a genus $1$ translation  surface by the unfolding procedure are the following: a square, an equilateral triangle, a triangle with angles $\pi/3$, $\pi/2$, $\pi/6$, and a triangle with angles $\pi/4$, $\pi/2$, $\pi/4$ (see the figure below). 
\begin{figure}[htb!]
\includegraphics[scale=1]{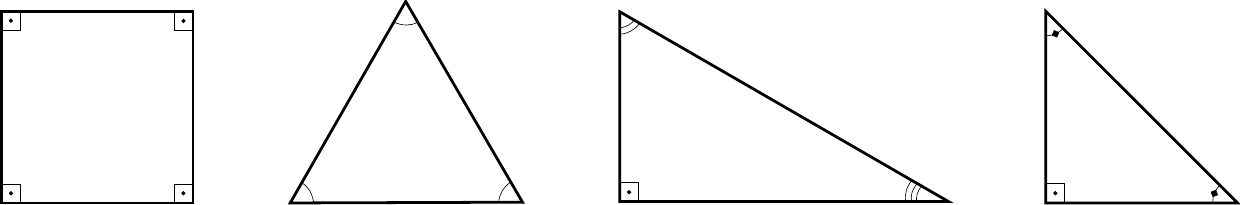}
\end{figure}

For more informations about translation surfaces coming from billiards on rational polygons, see this survey of H. Masur and S. Tabachnikov \cite{MT}.
\end{example}

\begin{example}[Suspensions of interval exchange transformations]\label{ex.suspensions}
An interval exchange transformation (\emph{i.e.t.} for short)  on $d\geq 2$ intervals is a map $T:D_T\to D_{T^{-1}}$ where $D_T, D_{T^{-1}}\subset I$ are subsets of an open bounded interval $I$ such that $\#(I-D_T)=\#(I-D_{T^{-1}})=d-1<\infty$ and the restriction of $T$ to each connected component of $I-D_T$ is a translation onto some connected component of 
$D_{T^{-1}}$. For concrete examples, see Figure~\ref{f.intercambiosI} below. 

\begin{figure}[htb!]
\includegraphics[scale=0.6]{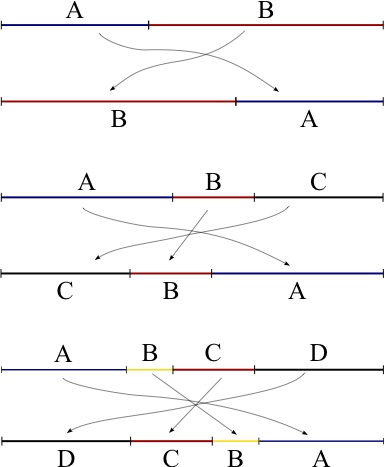}
\caption{Three examples of interval exchange transformations.}\label{f.intercambiosI}
\end{figure}

Usually, we obtain an i.e.t. as a return map of a translation flow on a translation surface. Conversely, given an i.e.t. $T$, it is possible to ``suspend'' it (in several ways) to construct translation flows on translation surfaces such that $T$ is the first return map to an appropriate transversal to the translation flow. For instance, the figure below illustrates a suspension construction due to H. Masur \cite{M82} applied to the third i.e.t. of Figure~\ref{f.intercambiosI}.
\begin{center}

\begingroup
  \makeatletter
  \providecommand\color[2][]{%
    \errmessage{(Inkscape) Color is used for the text in Inkscape, but the package 'color.sty' is not loaded}
    \renewcommand\color[2][]{}%
  }
  \providecommand\transparent[1]{%
    \errmessage{(Inkscape) Transparency is used (non-zero) for the text in Inkscape, but the package 'transparent.sty' is not loaded}
    \renewcommand\transparent[1]{}%
  }
  \providecommand\rotatebox[2]{#2}
  \ifx\svgwidth\undefined
    \setlength{\unitlength}{215.63464355pt}
  \else
    \setlength{\unitlength}{\svgwidth}
  \fi
  \global\let\svgwidth\undefined
  \makeatother
  \begin{picture}(1,0.64461123)%
    \put(0,0){\includegraphics[width=\unitlength]{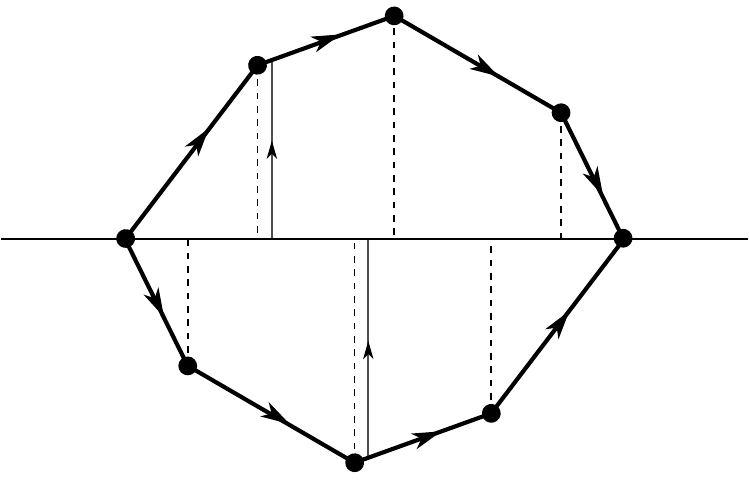}}%
\put(0.18569259,0.45522211){\color[rgb]{0,0,0}\makebox(0,0)[lb]{\smash{$\zeta_A$}}}%
    \put(0.37217966,0.61137359){\color[rgb]{0,0,0}\makebox(0,0)[lb]{\smash{$\zeta_B$}}}%
    \put(0.64413377,0.58066218){\color[rgb]{0,0,0}\makebox(0,0)[lb]{\smash{$\zeta_C$}}}%
    \put(0.80454864,0.41780208){\color[rgb]{0,0,0}\makebox(0,0)[lb]{\smash{$\zeta_D$}}}%
    \put(0.73745021,0.15135193){\color[rgb]{0,0,0}\makebox(0,0)[lb]{\smash{$\zeta_A$}}}%
    \put(0.57186757,0.00934741){\color[rgb]{0,0,0}\makebox(0,0)[lb]{\smash{$\zeta_B$}}}%
    \put(0.31491059,0.04110414){\color[rgb]{0,0,0}\makebox(0,0)[lb]{\smash{$\zeta_C$}}}%
    \put(0.1522245,0.19238293){\color[rgb]{0,0,0}\makebox(0,0)[lb]{\smash{$\zeta_D$}}}%
    \put(0.23248115,0.3349369){\scriptsize\color[rgb]{0,0,0}\makebox(0,0)[lb]{\smash{$A$}}}%
    \put(0.41725147,0.3349369){\scriptsize\color[rgb]{0,0,0}\makebox(0,0)[lb]{\smash{$B$}}}%
    \put(0.62232807,0.33333405){\scriptsize\color[rgb]{0,0,0}\makebox(0,0)[lb]{\smash{$C$}}}%
    \put(0.76926647,0.33413547){\scriptsize\color[rgb]{0,0,0}\makebox(0,0)[lb]{\smash{$D$}}}%
    \put(0.72055805,0.2797466){\scriptsize\color[rgb]{0,0,0}\makebox(0,0)[lb]{\smash{$A$}}}%
    \put(0.54241939,0.2797466){\scriptsize\color[rgb]{0,0,0}\makebox(0,0)[lb]{\smash{$B$}}}%
    \put(0.336766,0.28134945){\scriptsize\color[rgb]{0,0,0}\makebox(0,0)[lb]{\smash{$C$}}}%
    \put(0.19445754,0.28054803){\scriptsize\color[rgb]{0,0,0}\makebox(0,0)[lb]{\smash{$D$}}}%
  \end{picture}%
\endgroup

\end{center}
Here, the idea is that: 
\begin{itemize}
\item the vectors $\zeta_1, \dots, \zeta_d$,  have the form $\zeta_i=\lambda_i+\sqrt{-1}\tau_i \in\mathbb{C}\simeq\mathbb{R}^2$ where $\lambda_i$ are the lengths of the intervals permuted by the map $T$; 
\item we then organize the vectors $\zeta_1, \dots, \zeta_d$ on the plane $\mathbb{R}^2$ in order to get a polygon $P$ so that by going upstairs we meet the vectors $\zeta_i$ in the usual order (i.e., $\zeta_1$, $\zeta_2$, etc.) while by going downstairs we meet the vectors $\zeta_i$ in the order determined by $T$, i.e., by following the combinatorial receipt (say a permutation $\pi$ of $d$ elements) used by $T$ to permute intervals; 
\item gluing by translations the pairs of sides labeled by vectors $\zeta_i$, we obtain a translation surface whose vertical flow has the i.e.t. $T$ as first return map to the horizontal axis $\mathbb{R}\times 
\{0\}$ (e.g., in the picture we have drawn a trajectory of the vertical flow starting at the interval $B$ on the ``top part'' of $\mathbb{R}\times \{0\}$ and coming back at the interval $B$ on the bottom part of $\mathbb{R}\times \{0\}$); 
\item finally, the suspension data $\tau_1, \dots, \tau_d$  can be chosen ``arbitrarily'' as long as the planar figure $P$ is not degenerate: formally, one imposes the condition 
$$\sum\limits_{j<i}\tau_j>0\quad \text{\rm and} \quad \sum\limits_{j<i }\tau_{\pi(j)}<0\,,  \qquad
\text{for all }i \in \{1, \dots,d\}\,.$$ 
\end{itemize}

Of course, there is no unique way of suspending i.e.t.'s to get translation surfaces: for instance, in 
this survey \cite{Y} of J.-C. Yoccoz, one can find a detailed description of an alternative suspension procedure due to W. Veech (and nowadays called  Veech's zippered rectangles construction).  
\end{example}

\begin{example}[Square-tiled surfaces]\label{ex.L-origami}Consider a finite collection of unit squares on the plane such that the leftmost side of each square is glued (by translation) to the rightmost side of another (maybe the same) square, and the bottom side of each square is glued (by translation) to the top side of another (maybe the same) square. Here, we assume that, after performing the identifications, the resulting surface is connected. Again, since our identifications are given by translations, this procedure gives at the end of the day a translation surface structure, that is, a Riemann surface equipped with an Abelian differential (equal to $dz$ on each square). For obvious reasons, these surfaces are called \emph{square-tiled surfaces} and/or \emph{origamis} in the literature. For the sake of concreteness, we have drawn in Figure~\ref{f.L-shaped} below a \emph{L-shaped square-tiled surface} obtained from $3$ unit squares identified as in the picture (i.e., pairs of sides with the same marks are glued together by a translation). 
\begin{figure}[htb!]\label{f.L-billiard}
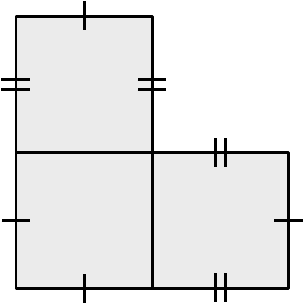
\caption{A L-shaped square-tiled surface.}\label{f.L-shaped}
\end{figure}

By following the same arguments used in the previous example, the reader can easily verify that this L-shaped square-tiled surface with 3 squares corresponds to an Abelian differential with a single zero of order $2$ on a Riemann surface of genus $2$.
\end{example}

\begin{remark} So far we have produced examples of translation surfaces/Abelian differentials from identifications by translation of pairs of parallel sides of a \emph{finite} collection of polygons. The curious reader may ask whether all translation surface structures can be recovered by this procedure. In fact, it is possible to prove that any translation surface admits a triangulation such that the zeros of the Abelian differential appear only in the vertices of the triangles (and the sides of the triangles are saddle connections in the sense that they connect zeroes of the Abelian differential), so that the translation surface can be recovered from this finite collection of triangles. However, if we are ``ambitious'' and try to represent translation surfaces by side identifications of a \emph{single} polygon (like in Example~\ref{e.DCBA}) instead of using a finite collection of polygons, then we'll fail: indeed, there are examples where the saddles connections are badly placed so that one polygon never suffices. However, it is possible to prove (with the help of \emph{Veech's zippered rectangle construction}) that all translation surfaces outside a countable union of codimension $2$ real-analytic suborbifolds can be represented by means of a single polygon whose sides are conveniently identified. See~\cite{Y} for further details.
\end{remark}

In addition to its intrinsic beauty, a great advantage of talking about translation structures instead of Abelian differentials is the fact that several additional structures come for free due to the translation invariance of the corresponding structures on the complex plane $\mathbb{C}$:
\begin{itemize}
\item \emph{a flat metric}: since the usual (flat) Euclidean metric $dx^2+dy^2$ on the complex plane $\mathbb{C}$ is translation-invariant, its pullback by the local charts provided by the translation structure gives a well-defined flat metric on $M-\Sigma$;
\item \emph{a canonical choice of a vertical vector-field}: as we implicitly mentioned by the end of Example~\ref{e.DCBA},  the vertical vertical vector field $\partial/\partial y$ on $\mathbb{C}$ can be pulled back to $M-\Sigma$ in a coherent way to define a canonical choice of north direction;
\item \emph{a pair of transverse measured foliations}: the pullback to $M-\Sigma$ of the horizontal 
$\{x=constant\}$ and vertical $\{y=constant\}$ foliations of the complex plane $\mathbb{C}$ are well-defined transverse foliations $\mathcal{F}_h$ and $\mathcal{F}_v$, which are \emph{measured} in the sense on Thurston: the leaves of these foliations come with canonical transverse measures $|dy|$ and $|dx|$ respectively.
\end{itemize}

\begin{remark} It is important to observe that the flat metric introduced above is a \emph{singular} metric: indeed, although it is a smooth Riemannian metric on $M-\Sigma$, it degenerates when we approach a zero $p\in\Sigma$ of the Abelian differential. Of course, we know from Gauss-Bonnet theorem that no compact surface of genus $g\geq 2$ admits a completely flat metric, so that, in some sense, if we wish to have a flat metric in a large portion of the surface, we're obliged to ``concentrate'' the curvature at some places. From this point of view, the fact that the flat metrics obtained from translation structures are degenerate at a finite number of points reflects the fact that the ``optimal'' way to produce an ``almost completely'' flat model of our genus $g\geq 2$ surface is by concentrating all the (negative) curvature at a finite number of points.
\end{remark}

Once we know that Abelian differentials and translation structures are essentially the same object, we can put some structures on $\mathcal{L}_g$, $\mathcal{TH}_g$ and 
$\mathcal{H}_g$.

\subsection{Stratification} Given a non-trivial Abelian differential $\omega$ on a Riemann surface $M$ of genus $g\geq 1$, we can form a list $\kappa=(k_1,\dots,k_{\sigma})$ by collecting the multiplicities of the (finite set of) zeroes of $\omega$ on $M$. Observe that any such list $\kappa=(k_1,\dots,k_{\sigma})$ verifies the constraint $\sum\limits_{l=1}^{\sigma} k_l=2g-2$ in view of Poincar\'e-Hopf theorem (or alternatively Gauss-Bonnet theorem). Given an unordered list $\kappa=(k_1,\dots,k_{\sigma})$ with $\sum\limits_{l=1}^{\sigma}k_l=2g-2$, the set $\mathcal{L}(\kappa)$ of Abelian differentials whose list of multiplicities of its zeroes coincide with $\kappa$ is called a \emph{stratum} of $\mathcal{L}_g$. Since the actions of $\textrm{Diff}^+_0(M)$ and $\Gamma_g$ respect the multiplicities of zeroes, the quotients $\mathcal{TH}(\kappa):=\mathcal{L}(\kappa)/\textrm{Diff}^+_0(M)$ and $\mathcal{H}(\kappa):=\mathcal{L}(\kappa)/\Gamma_g$ are well-defined. By obvious reasons, $\mathcal{TH}(\kappa)$ and $\mathcal{H}(\kappa)$ are called \emph{strata} of $\mathcal{TH}_g$ and $\mathcal{H}_g$ (resp.). Notice that, by definition,
$$\mathcal{L}_g=\bigsqcup\limits_{\kappa=(k_1,\dots,k_{\sigma}), k_1+\dots+k_{\sigma}=2g-2}\mathcal{L}(\kappa), \quad\quad \mathcal{TH}_g=\bigsqcup\limits_{\kappa=(k_1,\dots,k_{\sigma}), k_1+\dots+k_{\sigma}=2g-2}\mathcal{TH}(\kappa),$$
and
$$\mathcal{H}_g=\bigsqcup\limits_{\kappa=(k_1,\dots,k_{\sigma}), k_1+\dots+k_{\sigma}=2g-2}\mathcal{H}(\kappa).$$
In other words, the sets $\mathcal{L}_g$, $\mathcal{TH}_g$ and $\mathcal{H}_g$ are naturally ``decomposed'' into the subsets (strata) $\mathcal{L}(\kappa)$, $\mathcal{TH}(\kappa)$ and $\mathcal{H}(\kappa)$. However, at this stage, we can't promote this decomposition into disjoint subsets to a \emph{stratification} because, in the literature, a stratification of a set $X$ is a decomposition $X=\bigcup\limits_{i\in I} X_i$ where $I$ is a finite set of indices and the \emph{strata} $X_i$ are disjoint \emph{manifolds/orbifolds} of \emph{distinct} dimensions. Thus, one can't call stratification our decomposition of $\mathcal{L}_g$, $\mathcal{TH}_g$ and $\mathcal{H}_g$ until we put nice manifold/orbifold structures (of different dimension) on the corresponding strata. The introduction of nice complex-analytic manifold/orbifold structures on $\mathcal{TH}(\kappa)$ and $\mathcal{H}(\kappa)$ are the topic of our next subsection.

\begin{remark} The curious reader may ask at this point whether the strata $\mathcal{H}(\kappa)$ are \emph{non-empty}. In fact, this is a natural question because, while the condition $\sum\limits_{s=1}^{\sigma}k_s=2g-2$ is a necessary condition (in view of Poincar\'e-Hopf theorem say), it is not completely obvious that this condition is also \emph{sufficient}. In any case, it is possible to show that the strata $\mathcal{H}(\kappa)$ of \emph{Abelian differentials} are non-empty whenever $\sum\limits_{s=1}^{\sigma}k_s=2g-2$. For comparison, we note that \emph{exact} analog result for strata of \emph{non-orientable quadratic differentials} is \emph{false}. Indeed, if we denote by $\mathcal{Q}(d_1,\dots,d_m,-1^p)$ the stratum of non-orientable quadratic differentials with $m$ zeroes of orders $d_1,\dots,d_m\geq 1$ and $p$ simple poles, the Poincar\'e-Hopf theorem says that a necessary condition is $\sum\limits_{l=1}^m d_l -p = 4g-4$, and a theorem of H. Masur and J. Smillie says that this condition is \emph{almost} sufficient: except for the empty strata $\mathcal{Q}(0), \mathcal{Q}(1,-1), \mathcal{Q}(1,3), \mathcal{Q}(4)$ in genera $1$ and $2$, any stratum $\mathcal{Q}(d_1,\dots,d_m,-1^p)$ such that $d_1, \dots ,d_m$ and $p$ verify the above necessary condition is non-empty. See~\cite{MS} for more details.
\end{remark}

\subsection{Period map and local coordinates}\label{ss.periods}

Let $\mathcal{TH}(\kappa)$ be a stratum, say $\kappa=(k_1,\dots,k_{\sigma})$ with  $\sum\limits_{l=1}^{\sigma} k_l=2g-2$. Given an Abelian differential $\omega\in\mathcal{TH}(\kappa)$, we denote by $\Sigma(\omega)$ the set of zeroes of $\omega$. It is possible to prove that every $\omega_0\in\mathcal{TH}(\kappa)$ has  an open\footnote{Here we're considering the natural topology on strata $\mathcal{TH}(\kappa)$ induced by the \emph{developing map}. More precisely, given $\omega\in\mathcal{L}(\kappa)$, fix $p_1\in\Sigma(\omega)$, an universal cover $p:\widetilde{M}\to M$ and a point $P_1\in\widetilde{M}$ over $p_1$. By integration of $p^*\omega$ from $P_1$ to a point $Q\in\widetilde{M}$, we get, by definition, a developing map $D_{\omega}:(\widetilde{M}, P_1)\to (\mathbb{C},0)$ which determines \emph{completely} the translation surface $(M,\omega)$. In this way, the injective map 
$\omega\mapsto D_{\omega}$ allows us to view $\mathcal{L}(\kappa)$ as a subset of the space $C^0(\widetilde{M},\mathbb{C})$ of complex-valued continuous functions of $\widetilde{M}$. By equipping $C^0(\widetilde{M},\mathbb{C})$ with the \emph{compact-open topology}, we get natural topologies for $\mathcal{L}(\kappa)$ and $\mathcal{TH}(\kappa)$.} neighborhood $U_0\subset\mathcal{TH}(\kappa)$ such that, after identifying, for all $\omega\in U_0$, the cohomology $H^1(M,\Sigma(\omega);\mathbb{C})$ with the \emph{fixed} complex vector space $H^1(M,\Sigma(\omega_0);\mathbb{C})$ via the \emph{Gauss-Manin connection} (i.e., through identification of the integer lattices $H^1(M,\Sigma(\omega);\mathbb{Z}\oplus i\mathbb{Z})$ and $H^1(M,\Sigma(\omega_0);\mathbb{Z}\oplus i\mathbb{Z})$), the \emph{period map} $\Theta: U_0\to H^1(M,\Sigma(\omega_0);\mathbb{C})$, that is, the map
defined by the formula
$$\Theta(\omega):=\left(\gamma\to\int_{\gamma}\omega\right)\in\textrm{Hom}(H_1(M,\Sigma(\omega),\mathbb{Z});\mathbb{C})\simeq H^1(M,\Sigma(\omega);\mathbb{C})\simeq H^1(M,\Sigma(\omega_0);\mathbb{C}) \,,$$
is a local homeomorphism. A sketch of proof of this fact (along the lines given in this article \cite{K73} of A. Katok) goes as follows. We need to prove that two closed complex-valued $1$-forms $\eta_0$ and $\eta_1$ with transverse real and imaginary parts and the same relative periods, i.e., $\Theta(\eta_0)=\Theta(\eta_1)$, are isotopic as far as they are close enough to each other. Up to isotopies, it is not restrictive to assume that $\eta_0$ and $\eta_1$ have the same singularity set, that is, that 
$\Sigma(\eta_0)= \Sigma(\eta_1)$.

The idea to construct such an isotopy is to apply a variant of the so-called \emph{Moser's homotopy trick}. More precisely, one considers the $1$-parameter family
$\eta_t=(1-t)\eta_0+t\eta_1$, for $t\in [0,1]$, and one tries to find the desired isotopy $\{\phi_t\}$ by solving the equation
$$\phi_t^*(\eta_t)=\eta_0\,, \qquad \text{ \rm for all } t\in [0,1].$$
In this direction, let's see what are the properties satisfied by a solution $\{\phi_t\}$. By taking the derivative, we find
$$\frac{d}{dt}\phi_t^*\eta_t =0$$
and by assuming that $\{\phi_t\}$ is the flow of a (non-autonomous) vector field $X_t$, we get 
$$0=\frac{d}{dt}(\phi_t^*\eta_t) = \phi_t^*(\dot{\eta_t} + \mathcal{L}_{X_t}\eta_t)
=\phi_t^*(\eta_1-\eta_0 + \mathcal{L}_{X_t}\eta_t)$$
where $\mathcal{L}_{X_t}$ denotes the Lie derivative in the direction of the vector field $X_t$
on $M$, for each $t\in [0,1]$. 

By hypothesis, $\Theta(\eta_0)=\Theta(\eta_1)$. In particular, $\eta_0$ and $\eta_1$ have the same \emph{absolute} periods, so that $[\dot{\eta_t}]=[\eta_1-\eta_0]=0$ in 
$H^1(M,\mathbb{C})$. In other words, we can find a smooth family 
$\{U_t\}$ of complex-valued functions with $dU_t=\dot{\eta_t}$, hence by the previous equation 
we derive the identity
$$0 = dU_t + \mathcal{L}_{X_t}\eta_t$$
By Cartan's magic formula, $\mathcal{L}_{X_t}\eta_t = i_{X_t}(d\eta_t) + d(i_{X_t}\eta_t)$. Since 
$\eta_t$ is closed, $d\eta_t=0$, so that $\mathcal{L}_{X_t}\eta_t=d(i_{X_t}\eta_t)$ so that
by the previous equation, we get 
$$0 = d(U_t+i_{X_t}\eta_t)\,.$$
At this point, we see how one can hope to solve the original equation $\phi_t^*\eta_t=\eta_0$: firstly, we fix a smooth family $\{U_t\}$ with $dU_t=\dot{\eta_t}=\eta_1-\eta_0$ (unique up to additive constant), secondly, we define a non-autonomous vector field $X_t$ such that $i_{X_t}\eta_t=-U_t$, so that 
$U_t + i_{X_t}\eta_t=0$ and a \emph{fortiori} $d(U_t+i_{X_t}\eta_t)=0$, and finally we let $\{\phi_t\}$ be the isotopy associated to $X_t$. Of course, one must check that these definitions are well-posed: for instance, since the forms $\eta_t$ have singularities (zeroes) at the finite set $\Sigma=\Sigma(\eta_0)=\Sigma(\eta_1)$, we need to know that one can take $U_t=0$ at $\Sigma$, and this is possible because $U_t(p_j)-U_t(p_i)=\int_{p_i}^{p_j}\dot{\eta_t}=\int_{p_i}^{p_j}(\eta_1-\eta_0)$, and $\eta_0$ and 
$\eta_1$ have the same \emph{relative} periods. Finally, a solution $X_t$ of the equation 
$U_t + i_{X_t}\eta_t=0$ on $M\setminus\Sigma$ exists since by assumption the forms $\eta_t$ 
are close to $\eta_0$, hence have transverse  real and imaginary parts.  

We leave the verification of the details of the definition of $\{\phi_t\}$ as an exercise to the reader (whose solution is presented in \cite{K73} for the real case).

For an alternative proof of the fact that the period maps are local coordinates, based on \emph{Veech's zippered rectangles} construction, see e.g. J.-C. Yoccoz's survey~\cite{Y}.

\begin{remark} Concerning the possibility of using \emph{Veech's zippered rectangles} to show that the period maps are local homeomorphisms, we vaguely mentioned this particular strategy because the zippered rectangles construction  is a \emph{fundamental} tool: indeed, besides its usefulness in
endowing $\mathcal{TH}(\kappa)$ with nice geometric structures, it can be applied to understand the \emph{connected components} of $\mathcal{H}(\kappa)$, to derive volume estimates for the canonical measure $\mu_{\kappa}$ on $\mathcal{H}(\kappa)$,  to study the dynamics of the Teichm\"uller flow on 
$\mathcal{H}(\kappa)$ through combinatorial methods (Markov partitions), etc. In other words, Veech's zippered rectangles construction is a powerful tool to reduce the study of the \emph{global} Teichm\"uller geometry and  dynamics on $\mathcal{H}(\kappa)$ to combinatorial questions. Furthermore, it allows to connect the Teichm\"uller dynamics on $\mathcal{H}(\kappa)$ to that of interval exchange transformations and translation flows on surfaces. However, taking into account the usual limitations of space and time, and the fact that this tool is largely discussed in Yoccoz's survey~\cite{Y}, we will \emph{not} give here more details on Veech's zippered rectangles and applications.
\end{remark}

Recall that $H^1(M,\Sigma(\omega_0),\mathbb{C})$ is a complex vector space isomorphic to $\mathbb{C}^{2g+\sigma-1}$. An explicit way to see this fact is by taking a basis $\{\alpha_j,\beta_j\}_{j=1}^{g}$  of the \emph{absolute} homology group $H_1(M,\mathbb{Z})$, and \emph{relative} cycles $\gamma_1, \dots, \gamma_{\sigma-1}$ joining an \emph{arbitrarily} chosen point $p_0\in\Sigma(\omega_0)$ to the other points $p_1,\dots,p_{\sigma-1}\in\Sigma(\omega_0)$. Then, the map
$$\omega\in H^1(M,\Sigma(\omega_0),\mathbb{C})\mapsto \left(\int_{\alpha_1}\omega,\int_{\beta_1}\omega,\dots,\int_{\alpha_g}\omega,\int_{\beta_g}\omega, \int_{\gamma_1}\omega,\dots,\int_{\gamma_l}\omega\right)\in\mathbb{C}^{2g+\sigma-1}$$ gives the desired isomorphism. Moreover, by composing the period maps with such isomorphisms, we see that all changes of coordinates are given by \emph{affine} maps of $\mathbb{C}^{2g+\sigma-1}$. In particular, we can also pull back the Lebesgue measure on $\mathbb{C}^{2g+\sigma-1}$ to get a natural measure $\lambda_{\kappa}$ on $\mathcal{TH}(\kappa)$: indeed, the measure $\lambda_{\kappa}$ is \emph{well-defined modulo normalization} because the changes of coordinates are affine \emph{volume preserving} maps; to normalize the measure in every coordinate chart, we ask that the integral lattice $H^1(M,\Sigma(\omega_0);\mathbb{Z}\oplus i\mathbb{Z})$ has covolume $1$, i.e., that the quotient torus $H^1(M,\Sigma(\omega_0);\mathbb{C})/H^1(M,\Sigma(\omega_0);\mathbb{Z}\oplus i\mathbb{Z})$ has volume $1$.

In summary, by using the period maps as local coordinates, $\mathcal{TH}(\kappa)$ is endowed with a structure of \emph{affine complex manifold} of complex dimension $2g+\sigma-1$ and a \emph{natural (Lebesgue) measure} $\lambda_{\kappa}$. Also, it is not hard to check that the action of the modular group $\Gamma_g:=\textrm{Diff}^+(M)/\textrm{Diff}_0^+(M)$ is compatible with the affine complex structure on $\mathcal{TH}(\kappa)$, so that, by passing to the quotient, one gets that the stratum 
$\mathcal{H}(\kappa)$ of the moduli space of Abelian differentials has the structure of an affine complex \emph{orbifold} and carries a natural Lebesgue measure $\mu_{\kappa}$.

\begin{remark} Note that, as we emphasized above, after passing to the quotient, $\mathcal{H}(\kappa)$ is an affine complex orbifold \emph{at best}. In fact, we can't expect $\mathcal{H}(\kappa)$ to be a manifold since, as we saw in Subsection~\ref{ss.1.torii}, even in the simplest example (genus $1$), 
the space $\mathcal{H}(0)$ is an orbifold but not a manifold. In particular, the statement that the 
period maps are local homeomorphisms holds for the Teichm\"uller space of Abelian differentials 
$\mathcal{TH}(\kappa)$ (not for its quotient $\mathcal{H}(\kappa)$).
\end{remark}

The following example shows a concrete way to geometrically interpret the role of period maps as local coordinates of the strata $\mathcal{TH}(\kappa)$.


\begin{example}\label{e.zorich-periods}
Let $Q$ be a polygon and $(M,\omega_0)$ a translation surface as in Example~\ref{e.DCBA}. As we have seen in that example, $M$ is a genus $2$ Riemann surface and 
$\omega_0$ has an unique zero (of order $2$) at the point $p\in M$ coming from the vertices of $Q$ (because the total angle around $p$ is $6\pi$), that is, 
$\omega_0\in\mathcal{TH}(2)$. It can also be checked that the four closed loops $\alpha_1,\alpha_2,\alpha_3,\alpha_4$ of $M$, obtained by projecting to $M$ the four sides $v_1,v_2,v_3,v_4$ of $Q$, 
are a basis of the absolute homology group $H_1(M,\mathbb{Z})$. It follows that, in this case, the period map in a small neighborhood $U_0$ of 
$\omega_0$ is
$$\omega\in U_0\subset\mathcal{TH}(2)\mapsto \left(\int_{\alpha_1}\omega,\int_{\alpha_2}\omega,\int_{\alpha_3}\omega,\int_{\alpha_4}\omega\right)\in V_0\subset\mathbb{C}^4,$$
where $V_0$ is a small neighborhood of $(v_1,v_2,v_3,v_4)$. Consequently, we see that any Abelian differential $\omega\in\mathcal{H}(2)$ sufficiently close to $\omega_0$ can be obtained geometrically by small arbitrary perturbations (indicated by dashed red lines in the figure below) of the sides $v_1,v_2,v_3,v_4$ of our initial polygon $Q$ (indicated by blue lines in the figure below). 
\begin{center}

\begingroup
  \makeatletter
  \providecommand\color[2][]{%
    \errmessage{(Inkscape) Color is used for the text in Inkscape, but the package 'color.sty' is not loaded}
    \renewcommand\color[2][]{}%
  }
  \providecommand\transparent[1]{%
    \errmessage{(Inkscape) Transparency is used (non-zero) for the text in Inkscape, but the package 'transparent.sty' is not loaded}
    \renewcommand\transparent[1]{}%
  }
  \providecommand\rotatebox[2]{#2}
  \ifx\svgwidth\undefined
    \setlength{\unitlength}{215.63464355pt}
  \else
    \setlength{\unitlength}{\svgwidth}
  \fi
  \global\let\svgwidth\undefined
  \makeatother
  \begin{picture}(1,0.62101225)%
    \put(0,0){\includegraphics[width=\unitlength]{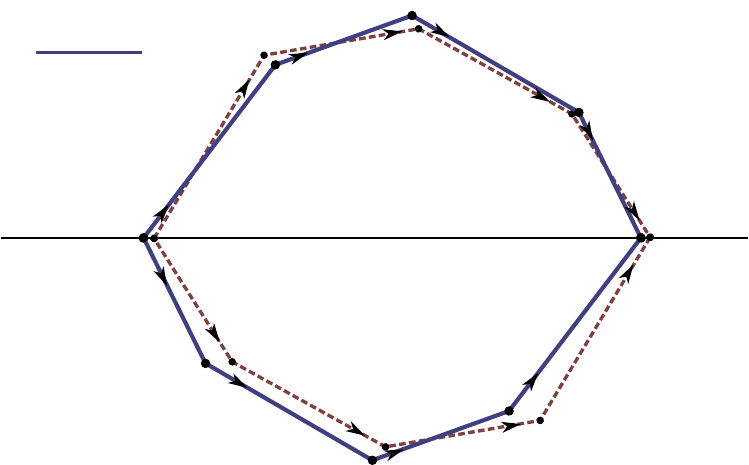}}%
    \put(0.11069762,0.57588393){\color[rgb]{0,0,0}\makebox(0,0)[lb]{\smash{$Q$}}}%
    \put(0.16867205,0.35895616){\color[rgb]{0,0,0}\makebox(0,0)[lb]{\smash{$v_1$}}}%
    \put(0.36629517,0.57588393){\color[rgb]{0,0,0}\makebox(0,0)[lb]{\smash{$v_2$}}}%
    \put(0.58806682,0.58853406){\color[rgb]{0,0,0}\makebox(0,0)[lb]{\smash{$v_3$}}}%
    \put(0.80226125,0.43197776){\color[rgb]{0,0,0}\makebox(0,0)[lb]{\smash{$v_4$}}}%
  \end{picture}%
\endgroup

\end{center}
\end{example}

\begin{remark} The introduction of nice affine complex structures in the case of \emph{non-orientable} (integrable) quadratic differentials is slightly different from the case of Abelian differentials. Given a non-orientable quadratic differential $q\in\mathcal{Q}(\kappa)$ on a Riemann surface $M$ of genus $g\geq 0$, there is a \emph{canonical} (orienting) double-cover $\pi_{\kappa}:\widehat{M}\to M$ such that $\pi_{\kappa}^*(q)=\widehat{\omega}^2$, where $\widehat{\omega}$ is an Abelian differential on $\widehat{M}$. Here, $\widehat{M}$ is a connected Riemann surface because $q$ is non-orientable (otherwise it would be two disjoint copies of $M$). Also, by writing $\kappa=(o_1,\dots,o_{\tau}, e_1,\dots, e_{\nu})$ where $o_j$ are odd integers and $e_j$ are even integers, one can check that $\pi_{\kappa}$ is ramified over singularities of odd orders $o_j$ while $\pi_{\kappa}$ is regular (i.e., it has two pre-images) over singularities of even orders $e_j$. It follows that
$$\widehat{\omega}\in\mathcal{H}\left(o_1+1,\dots,o_{\tau}+1,\frac{e_1}{2},\frac{e_1}{2},\dots,\frac{e_{\nu}}{2},\frac{e_{\nu}}{2}\right).$$
In particular, Riemann-Hurwitz formula implies that the genus $\widehat{g}$ of $\widehat{M}$ is $\widehat{g} = 2g-1+\tau/2$.

Denote by $\sigma:\widehat{M}\to\widehat{M}$ the non-trivial involution such that $\pi_\kappa\circ\sigma=\pi_{\kappa}$ (i.e., $\sigma$ interchanges the sheets of the double-covering $\pi_{\kappa}$). Observe that $\sigma$ induces an involution $\sigma_*$ on the cohomology group $H^1(\widehat{M},\widehat{\Sigma_\kappa},\mathbb{C})$, where $\widehat{\Sigma_\kappa}:=\pi_{\kappa}^{-1}(\Sigma_\kappa)-\pi_\kappa^{-1}(\{p_1,\dots,p_{\eta}\})$, $\Sigma_{\kappa}$ is the set of singularities of $q$ and $p_1,\dots,p_{\eta}$ are the (simple) poles of $q$. Since $\sigma_*$ is an involution, we have the decomposition
$$H^1(\widehat{M},\widehat{\Sigma_\kappa},\mathbb{C})=H_+^1(\widehat{M},\widehat{\Sigma_\kappa},\mathbb{C})\oplus H_-^1(\widehat{M},\widehat{\Sigma_\kappa},\mathbb{C})$$
of the relative cohomology as a direct sum of the subspace $H_+^1(\widehat{M},\widehat{\Sigma_\kappa},\mathbb{C})$ of $\sigma_*$-invariant cycles and the subspace $H_-^1(\widehat{M},\widehat{\Sigma_\kappa},\mathbb{C})$ of $\sigma_*$-anti-invariant cycles. Observe that the Abelian differential $\widehat{\omega}$ is $\sigma_*$-\emph{anti-invariant}: indeed, since $\widehat{\omega}^2$ is $\pi_\kappa(q)$, we have that $\sigma_*(\widehat{\omega})=\pm\widehat{\omega}$; if $\widehat{\omega}$ were $\sigma_*$-invariant, it would follow that $\pi_{\kappa}^*(\widehat{\omega})^2=q$, i.e., $q$ would be the square of an Abelian differential, and hence an orientable quadratic differential (a contradiction). From this, it is not hard to believe that one can prove that a small neighborhood of $q$ can be identified with a small neighborhood of the $\sigma_*$-anti-invariant Abelian differential $\widehat{\omega}$ inside the \emph{anti-invariant subspace} $H_-^1(\widehat{M},\widehat{\Sigma_\kappa},\mathbb{C})$. Again, the changes of coordinates are locally affine, so that $\mathcal{Q}(\kappa)$ also comes equipped with a affine complex structure and a natural Lebesgue measure.
\end{remark}

At this stage, the terminology \emph{stratification} used in the previous subsection is completely justified by now and we pass to a brief discussion of the topology of our strata.

\subsection{Connectedness of strata} At first sight, it might be tempting to say that the strata $\mathcal{H}(\kappa)$ are always connected, i.e., once we fix the list $\kappa$ of orders of zeroes, one might conjecture that there are no further obstructions to deform an arbitrary Abelian differential $\omega_0\in\mathcal{H}(\kappa)$ into another arbitrary Abelian differential $\omega_1\in\mathcal{H}(\kappa)$. However, W. Veech~\cite{V90} discovered that the stratum $\mathcal{H}(4)$ has two connected components. In fact, W. Veech distinguished the connected components of $\mathcal{H}(4)$ by means of a combinatorial invariants called \emph{extended Rauzy classes}\footnote{Extended Rauzy classes
are slightly modified versions of \emph{Rauzy classes}, equivalence classes (composed of pair of permutations with $d$ symbols with $d=2g+s-1$, where $g$ is the genus and $s$ is the number of zeroes) introduced by G. Rauzy~\cite{R} in his study of \emph{interval exchange transformations.}}. Roughly speaking, Veech showed that there is a \emph{bijective} correspondence between connected components of strata and extended Rauzy classes, and, using this fact, he concluded that $\mathcal{H}(4)$ has two connected components because he checked (by hand) that there are precisely two extended Rauzy classes associated to this stratum. By a similar strategy, P.~Arnoux proved that the stratum $\mathcal{H}(6)$ has three connected components.
In principle, one could think of pursuing the strategy of describing extended Rauzy classes to determine the connected components of strata, but this is a hard combinatorial problem: for instance, when trying to compute extended Rauzy classes associated to connected components of strata of Abelian differentials of genus $g$, one should perform several combinatorial operations with pairs of permutations on an alphabet of $d\geq 2g$ letters.\footnote{Just to give an idea of how fast the size of Rauzy classes grows, let's mention that the cardinality of the \emph{largest} Rauzy classes in genera $2$, $3$, $4$ and $5$ are respectively $15$, $2177$, $617401$ and $300296573$. Can you guess the next largest cardinality (for genus $6$)? For more informations on how these numbers can be computed we refer to  V.~Delecroix's work \cite{De}.} Nevertheless, M.~Kontsevich and A.~Zorich~\cite{KZ} managed to classify \emph{completely} the connected components of strata of Abelian differentials with the aid of some invariants of algebraic-geometric nature: technically speaking, there are exactly \emph{three types} of connected components of strata -- \emph{hyperelliptic}, \emph{even spin} and \emph{odd spin}. The Kontsevich--Zorich classification can be summarized as follows:
\begin{theorem}[M. Kontsevich and A. Zorich]Fix $g\geq 4$.
\begin{itemize}
\item the minimal stratum $\mathcal{H}(2g-2)$ has $3$ connected components;
\item $\mathcal{H}(2l,2l)$, $l\geq 2$, has $3$ connected components;
\item any $\mathcal{H}(2l_1,\dots,2l_n)\neq \mathcal{H}(2l,2l)$, $l_i\geq 1$, has $2$ connected components;
\item $\mathcal{H}(2l-1,2l-1)$, $l\geq 2$, has $2$ connected components;
\item all other strata of Abelian differentials of genus $g$ are connected.
\end{itemize}
\end{theorem}
The classification in genus $g=2$ and $g=3$ are slightly different:
\begin{theorem}[M. Kontsevich and A. Zorich] In genus $2$, the strata $\mathcal{H}(2)$ and $\mathcal{H}(1,1)$ are connected. In genus $3$, the strata $\mathcal{H}(4)$ and $\mathcal{H}(2,2)$ both have two connected components, while the other strata are connected.
\end{theorem}
For more details, we strongly recommend the original article by M. Kontsevich and A. Zorich.

\begin{remark}The classification of connected components of strata of \emph{non-orientable} quadratic differentials was obtained by E.~Lanneau~\cite{L}. For genus $g\geq 3$, each of the four strata $\mathcal{Q}(9,-1)$, $\mathcal{Q}(6,3,-1)$, $\mathcal{Q}(3,3,3,-1)$, $\mathcal{Q}(12)$ have exactly two connected components, each of the following strata
$\mathcal{Q}(2j-1,2j-1,2k-1,2k-1)$ ($j,k\geq 0$, $j+k=g$), $\mathcal{Q}(2j-1,2j-1,4k+2)$ ($j,k\geq 0$, $j+k=g-1$), $\mathcal{Q}(4j+2,4k+2)$ ($j,k\geq 0$, $j+k=g-2$) also have exactly two connected components, and all other strata are connected. For genus $0\leq g\leq 2$, we have that any strata in genus $0$ and $1$ are connected, and, in genus $2$, the strata $\mathcal{Q}(6,-1^2)$, $\mathcal{Q}(3,3,-1^2)$ have exactly two connected components each, while the other strata are connected.
\end{remark}


\section{Dynamics on the moduli space of Abelian differentials}\label{s.general-dynamics}

Let $(M,\omega)$ be a compact Riemann surface $M$ of genus $g\geq 1$ equipped with a non-trivial Abelian differential (that is, a holomorphic 1-form)
$\omega\not\equiv 0$. In the sequel, we denote by $\Sigma$ the (finite) set of zeroes of $\omega$
on $M$.

\subsection{$GL^+(2,\mathbb{R})$-action on $\mathcal{H}_g$} The canonical identification between Abelian differentials and translation structures makes transparent the existence of a natural action of $GL^+(2,\mathbb{R})$ on the set $\mathcal{L}_g$ of all Abelian differentials. Indeed, given an Abelian differential $\omega$, let's denote by $\{\phi_{\alpha}(\omega)\}_{\alpha\in I}$ the maximal atlas on $M-\Sigma$ giving the translation structure corresponding to $\omega$ (so that $\phi_\alpha(\omega)^*dz=\omega$ for every index $\alpha$). Here, the local charts $\phi_{\alpha}(\omega)$ map some open set of $M-\Sigma$ to $\mathbb{C}$.

Since any matrix $A\in GL^+(2,\mathbb{R})$ acts on $\mathbb{C}=\mathbb{R}\oplus i \mathbb{R}$, we can post-compose the local charts $\phi_{\alpha}(\omega)$ with $A$, so that we obtain a new atlas $\{A\circ\phi_{\alpha}(\omega)\}$ on $M-\Sigma$. Observe that the changes of coordinates of this new atlas are also given by translations, as a quick computation reveals:
$$(A\circ\phi_{\beta}(\omega))\circ(A\circ\phi_{\alpha})^{-1}(z)=A\circ(\phi_{\beta}(\omega)\circ\phi_{\alpha}(\omega)^{-1})\circ A^{-1}(z) = A(A^{-1}(z)+c)=z+A(c).$$
In other words, $\{A\circ\phi_{\alpha}(\omega)\}$ is a new translation structure. The action
of $GL^+(2,\mathbb{R})$ on $\mathcal{L}_g$ is then defined as follows: for all $A\in 
GL^+(2,\mathbb{R})$ and for all $\omega \in \mathcal{L}_g$, the Abelian differential $A\cdot \omega\in
 \mathcal{L}_g$ is the unique Abelian differential corresponding to the translation structure $\{A\circ\phi_{\alpha}(\omega)\}$.  Since the complex structure of the plane is not preserved by the action of any 
$GL^+(2,\mathbb{R})$ matrix which is not a rotation, the (unique) Riemann surface structures related to 
$\omega$ and $A\cdot \omega$ are distinct (not biholomorphic), unless $A$ belongs to the subgroup of rotations $SO(2,\mathbb{R})\subset  GL^+(2,\mathbb{R})$.

By definition, it is immediate to see this $GL^+(2,\mathbb{R})$-action on Abelian differentials which are given by sides identifications of collections of polygons (as in the previous examples): in fact, given $A\in GL^+(2,\mathbb{R})$ and an Abelian differential $\omega$ related to a finite collection of polygons $\mathcal{P}$ with parallel sides glued by translations, the Abelian differential $A\cdot\omega$ corresponds, by definition, to the finite collection of polygons $A\cdot\mathcal{P}$ obtained by letting $A$ act on the polygons forming $\mathcal{P}$ (as a subset of the plane) and keeping the identifications of parallel sides by translations. Notice that the linear action of $A$ on the plane evidently respects the notion of parallelism, so that this procedure is well-defined. For the sake of concreteness, we drew below some illustrative examples (see Figures 
\ref{f.gt-L-billiard},~\ref{f.T-L-billiard} and~\ref{f.J-L-billiard}) of the actions of the matrices $g_t=\left(\begin{array}{cc}e^t & 0 \\ 0 & e^{-t}\end{array}\right)$, $T=\left(\begin{array}{cc}1 & 1 \\ 0 & 1\end{array}\right)$ and $J=\left(\begin{array}{cc}0 & -1 \\ 1 & 0\end{array}\right)$ on a L-shaped square-tiled surface.

\begin{figure}[h!]
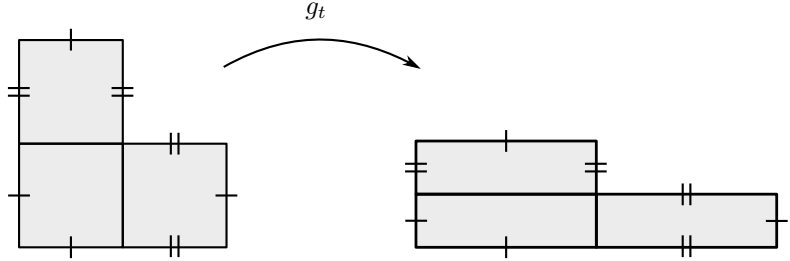
\caption{Action of $g_t$}\label{f.gt-L-billiard}
\end{figure}

\begin{figure}[h!]
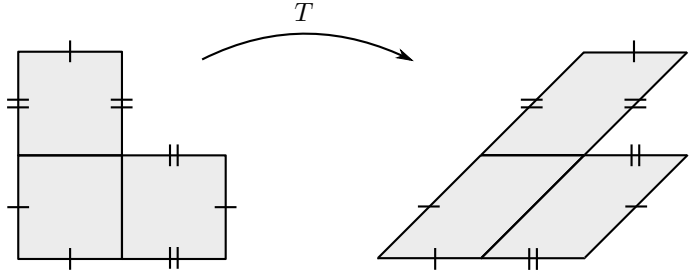
\caption{Action of $T$}\label{f.T-L-billiard}
\end{figure}

\begin{figure}[h!]
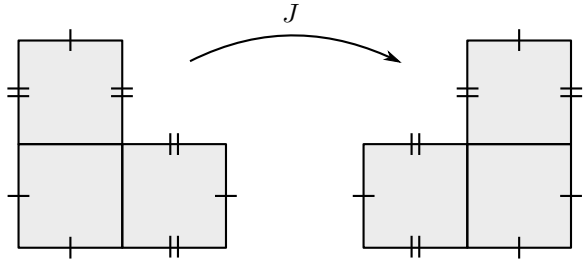
\caption{Action of $J$}\label{f.J-L-billiard}
\end{figure}

Concerning the structures on $\mathcal{H}_g$ introduced in the previous section, we note that this 
$GL^+(2,\mathbb{R})$-action preserves each of the strata $\mathcal{H}(\kappa)$ ($\kappa=(k_1,\dots,k_{\sigma}$), $\sum\limits_{l=1}^{\sigma}k_l=2g-2$) and the action of the subgroup $SL(2, \mathbb R)
\subset GL^+(2,\mathbb{R})$ also preserves the natural (Lebesgue) measures $\lambda_{\kappa}$ on them, since the underlying affine structures of strata and the volume forms are respected. Of course, this observation opens up the possibility of studying this action via ergodic-theoretical methods. However, it turns out that the strata $\mathcal{H}(\kappa)$ are `too large': for instance, they are ruled in the sense that the complex lines $\mathbb{C}\cdot\omega$ foliate them. As a result, it is possible to prove that each $\lambda_{\kappa}$ has infinite mass and the action is never ergodic since the total area
$A(\omega)$ of a translation surface $(M,\omega)$ is preserved under the $SL(2, \mathbb R)$ action. This is not surprising since geodesic flows are never ergodic on the full tangent or cotangent bundle, since the norm of tangent vectors (closely related to the Hamiltonian) is invariant under the flow. Anyway, as in the case of geodesic flows, this difficulty can be easily bypassed by normalizing the total area of Abelian differentials. In this way, by ``killing'' the obvious obstruction given by scaling, we can hope that the restriction of the measure $\lambda_\kappa$ will have finite total mass and be ergodic with 
respect to the $SL(2, \mathbb R)$ action. The details of this procedure are the content of the next subsection.  

\subsection{$SL(2,\mathbb{R})$-action on $\mathcal{H}_g^{(1)}$ and Teichm\"uller geodesic flow}\label{ss.MV-finiteness} We denote by $\mathcal{H}_g^{(1)}$ the set of Abelian differentials $\omega$ on a  Riemann surface $M$ of genus $g\geq 1$ whose induced area form $dA(\omega)$ on $M$ has total area $A(\omega):=\int_M dA(\omega)=1$. At first sight, one is tempted to say that $\mathcal{H}_g^{(1)}$ is some sort of ``unit sphere'' of $\mathcal{H}_g$. However, since the area form $A(\omega)$ of an arbitrary Abelian differential $\omega$ can be expressed as 
$$A(\omega)=\frac{i}{2}\int_M\omega\wedge\bar{\omega} = \frac{i}{2}\sum\limits_{j=1}^g (A_j \bar{B_j}-\bar{A_j}B_j)$$
where $A_j,B_j$ form a canonical basis of absolute periods of $\omega$, i.e., $A_j=\int_{\alpha_j}\omega,B_j=\int_{\beta_j}\omega$ and $\{\alpha_j,\beta_j\}_{j=1}^g$ is a symplectic basis of $H_1(M,\mathbb{R})$ (with respect to the intersection form), we see that $\mathcal{H}_g^{(1)}$ resembles more a ``unit hyperboloid''. 

Again, we can stratify $\mathcal{H}_g^{(1)}$ by considering $\mathcal{H}^{(1)}(\kappa):=\mathcal{H}(\kappa)\cap\mathcal{H}_g^{(1)}$ and, from the definition of the $GL^+(2,\mathbb{R})$-action on the plane $\mathbb{C}\simeq\mathbb{R}^2$, we see that $\mathcal{H}_g^{(1)}$ and its strata $\mathcal{H}^{(1)}(\kappa)$ come equipped with a natural $SL(2,\mathbb{R})$-action. Moreover, by disintegrating the natural Lebesgue measure on $\lambda_{\kappa}$ with respect to the level sets of the total area function $A:\mathcal{H}_g\to\mathbb{R}^+$, $\omega\mapsto A(\omega)$, we can write 
$$d\lambda_{\kappa}=dA\cdot d\lambda^{(1)}_{\kappa}$$ 
where $\lambda_{\kappa}^{(1)}$ is a natural ``Lebesgue'' measure on $\mathcal{H}^{(1)}(\kappa)$. We encourage the reader to compare this with the analogous procedure to get the Lebesgue measure on the unit sphere $S^{n-1}$ by disintegration of the Lebesgue measure of the Euclidean space $\mathbb{R}^n$.

Of course, from the ``naturality'' of the construction, it follows that $\lambda^{(1)}_{\kappa}$ is a $SL(2,\mathbb{R})$-invariant measure on $\mathcal{H}(\kappa)$. The following fundamental result was proved by H. Masur~\cite{M82} and W. Veech~\cite{V82}:
\begin{theorem}[H. Masur/W. Veech]\label{t.MVp1} The total volume (mass) of $\lambda_{\kappa}^{(1)}$ is finite.
\end{theorem}

\begin{remark}The computation of the actual values of these volumes took essentially 20 years to be performed and it is due to A. Eskin and A. Okounkov~\cite{EO}. We will make some comments on this later (in Section~\ref{s.finalremarks}). 
\end{remark}

For a presentation (from scratch) of the proof of Theorem~\ref{t.MVp1} based on Veech's zippered rectangle construction, see e.g. J.-C. Yoccoz survey~\cite{Y}. 

For the sake of reader's convenience, we present here the following intuitive argument of H. Masur \cite{Masursurvey} explaining why $\lambda_{\kappa}^{(1)}$ has finite mass. In the genus $1$ case (of torii), we have seen in Subsection~\ref{ss.1.torii} that the moduli space $\mathcal{H}_1^{(1)}\simeq SL(2,\mathbb{R})/SL(2,\mathbb{Z})$ and the Masur--Veech measure $\lambda_{\{0\}}^{(1)}$ comes from the Haar measure $\lambda_{SL(2,\mathbb{R})}$ of $SL(2,\mathbb{R})$. In this situation, the fact $\lambda_{\{0\}}^{(1)}$ has finite mass is a reformulation of the fact that $SL(2,\mathbb{Z})$ is a \emph{lattice} of $SL(2,\mathbb{R})$. More concretely, $\mathcal{H}_1^{(1)}$ correspond to pairs of vectors $(v_1,v_2)\in\mathbb{R}^2\times \mathbb{R}^2$ in a fundamental domain for the action of the lattice $SL(2,\mathbb{Z})$ on $\mathbb{R}^2 \times \mathbb{R}^2$ with $|v_1\wedge v_2|=1$, so that, by direct calculation\footnote{Actually, since $\mathcal{H}_1^{(1)}$ is the unit cotangent bundle of 
$\mathbb{H}/SL(2,\mathbb{Z})$ and the fundamental domain $\mathcal{F}_0$ in Subsection~\ref{ss.1.torii} has hyperbolic area $\pi/12$ (for the hyperbolic metric of curvature $-4$ compatible with our time normalization of the geodesic flow), one has $\lambda_{\{0\}}^{(1)}(\mathcal{H}_1^{(1)})=\pi^2/6$. However, we will not insist on this explicit computation because it is not easy to generalize it for moduli spaces of higher genera Abelian differentials.}  
$$\lambda_{\{0\}}^{(1)}(\mathcal{H}_1^{(1)})\leq\textrm{Leb}(\{(v_1,v_2)\in\mathbb{R}^2\times\mathbb{R}^2:  |v_1 \cdot v_2| \leq  |v_1\wedge v_2|\leq 1 \} )<\infty \,.$$

For the discussion of the general case, we will need the notion of a \emph{maximal cylinder} of a translation surface. In simple terms, given a closed regular geodesic $\gamma$ in a translation surface $(M,\omega)$, we can form a maximal cylinder $C$ by collecting all closed geodesics of $(M,\omega)$ parallel to $\gamma$ not meeting any zero of $\omega$. In particular, the boundary of $C$ contains zeroes of $\omega$. Given a maximal cylinder $C$, we denote by $w(C)$ its \emph{width} (i.e., the length of its waist curve $\gamma$) and by $h(C)$ its \emph{height} (i.e., the minimum distance across $C$). For example, in the figure below we illustrate two closed geodesics $\gamma_1$ and $\gamma_2$ (in the horizontal direction) and the two corresponding maximal cylinders $C_1$ and $C_2$ of the L-shaped square-tiled surface of Example~\ref{ex.L-origami}. In this picture, we see that $C_1$ has width $2$, $C_2$ has width $1$, and both $C_1$ and $C_2$ have height $1$.
\begin{center}

\begingroup
  \makeatletter
  \providecommand\color[2][]{%
    \errmessage{(Inkscape) Color is used for the text in Inkscape, but the package 'color.sty' is not loaded}
    \renewcommand\color[2][]{}%
  }
  \providecommand\transparent[1]{%
    \errmessage{(Inkscape) Transparency is used (non-zero) for the text in Inkscape, but the package 'transparent.sty' is not loaded}
    \renewcommand\transparent[1]{}%
  }
  \providecommand\rotatebox[2]{#2}
  \ifx\svgwidth\undefined
    \setlength{\unitlength}{170pt}
  \else
    \setlength{\unitlength}{\svgwidth}
  \fi
  \global\let\svgwidth\undefined
  \makeatother
  \begin{picture}(1,0.99815359)%
    \put(0,0){\includegraphics[width=\unitlength]{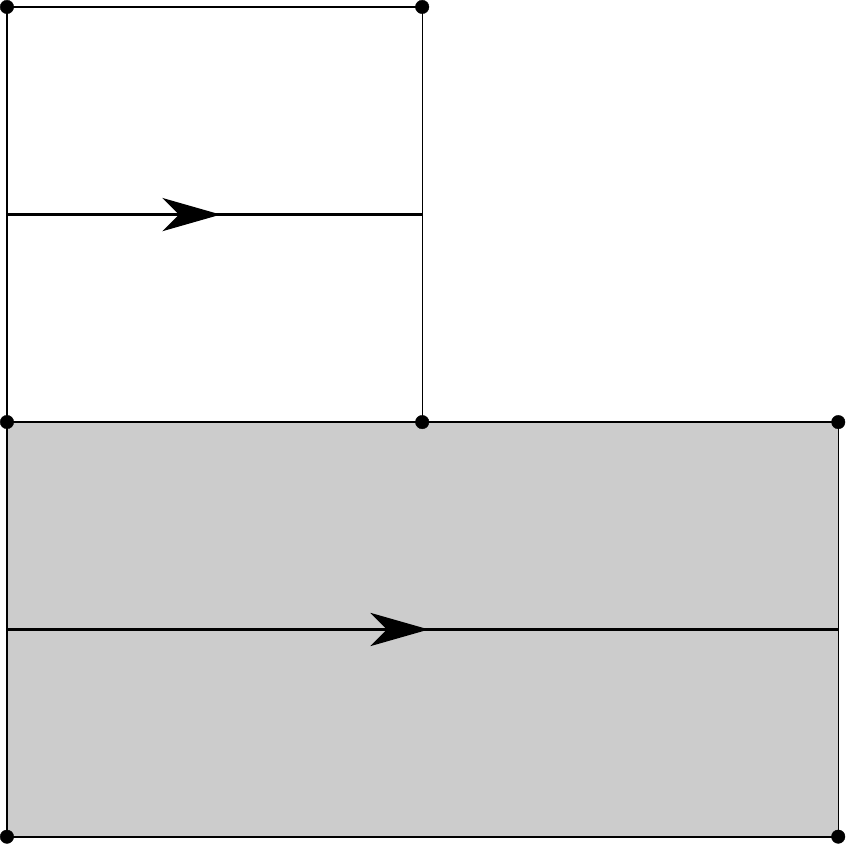}}%
    \put(0.42320816,0.15934056){\color[rgb]{0,0,0}\makebox(0,0)[lb]{\smash{$\gamma_1$}}}%
    \put(0.1795385,0.65598974){\color[rgb]{0,0,0}\makebox(0,0)[lb]{\smash{$\gamma_2$}}}%
  \end{picture}%
\endgroup

\end{center}

Continuing with the argument for the finiteness of the mass of $\lambda_{\kappa}^{(1)}$, one shows that, for each $g\geq 2$, there exists an universal constant $C(g)$ such that any translation surface $(M,\omega)$ of genus $g$ and diameter $\textrm{diam}(M,\omega)\geq C(g)$ has a maximal cylinder of height $h\sim \textrm{diam}(M,\omega)$ in some direction (see the figure below) 
\begin{center}

\begingroup
  \makeatletter
  \providecommand\color[2][]{%
    \errmessage{(Inkscape) Color is used for the text in Inkscape, but the package 'color.sty' is not loaded}
    \renewcommand\color[2][]{}%
  }
  \providecommand\transparent[1]{%
    \errmessage{(Inkscape) Transparency is used (non-zero) for the text in Inkscape, but the package 'transparent.sty' is not loaded}
    \renewcommand\transparent[1]{}%
  }
  \providecommand\rotatebox[2]{#2}
  \ifx\svgwidth\undefined
    \setlength{\unitlength}{180pt}
  \else
    \setlength{\unitlength}{\svgwidth}
  \fi
  \global\let\svgwidth\undefined
  \makeatother
  \begin{picture}(1,1.1612457)%
    \put(0,0){\includegraphics[width=\unitlength]{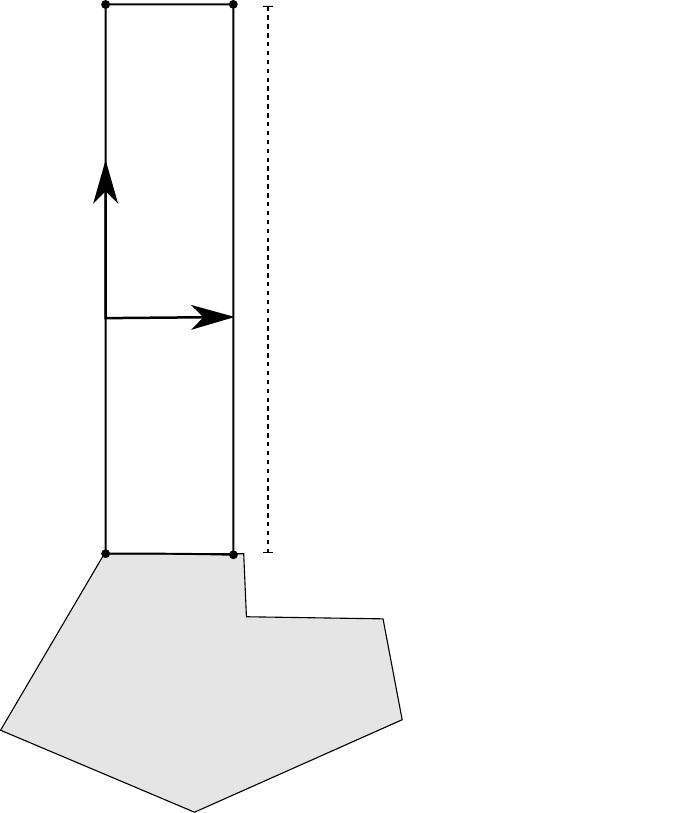}}%
    \put(0.08722082,0.90273617){\color[rgb]{0,0,0}\makebox(0,0)[lb]{\smash{$v_2$}}}%
    \put(0.22873105,0.74361281){\color[rgb]{0,0,0}\makebox(0,0)[lb]{\smash{$v_1$}}}%
    \put(0.39296376,0.75101599){\color[rgb]{0,0,0}\makebox(0,0)[lb]{\smash{$h\sim diam(M,\omega)$}}}%
  \end{picture}%
\endgroup

\end{center}
Next, we recall that $\lambda_{\kappa}^{(1)}$ was defined via the relative periods. In particular, the set of translation surfaces $(M,\omega)\in\mathcal{H}^{(1)}(\kappa)$ of genus $g\geq 2$ with diameter $\leq C(g)$ has finite $\lambda_{\kappa}^{(1)}$-measure. Hence, it remains to estimate the $\lambda_{\kappa}^{(1)}$-measure of the set of translation surfaces $(M,\omega)$ with diameter $\geq C(g)$. Here, we recall that there is a maximal cylinder 
$C\subset (M,\omega)$ with height $h\sim \textrm{diam}(M,\omega)$. Since $(M,\omega)$ has area one, this forces the \emph{width} $w$ of $C$, i.e., the length of a closed geodesic (waist curve) $\gamma$ in $C$ to be small. By taking $\gamma$ and a curve $\rho$ across $C$ as a part of the basis of the relative homology of $(M,\omega)$, we get two vectors  $u=\int_{\rho}\omega$ and $v=\int_{\gamma}\omega$ with $|u\wedge v|\leq 1$ and $v$ small. Since we are interested in the volume of the
moduli space, by the action of Dehn twists around the waist curve $\gamma$ of the cylinder $C$
we can assume that $ | u\cdot v | \leq | u\wedge v |$. In other words, we can think of the \emph{cusp} of 
$\mathcal{H}^{(1)}(\kappa)$ corresponding to translation surfaces $(M,\omega)$ with $v$ small as a subset of the set  
$$
\{(v,u)\in\mathbb{R}^2\times\mathbb{R}^2: | u\cdot v | \leq | u \wedge v |\leq 1\}.
$$ 
This ends the sketch of proof of Theorem~\ref{t.MVp1} because the $\lambda_{\kappa}^{(1)}$-measure 
of \emph{cusps} is then bounded by 
$$Leb(\{(v,u)\in\mathbb{R}^2\times\mathbb{R}^2: | u\cdot v | \leq  
|u\wedge v|\leq 1\})<\infty.$$

In what follows, given any connected component $\mathcal{C}$ of some stratum $\mathcal{H}^{(1)}(\kappa)$, we call the $SL(2,\mathbb{R})$-invariant probability measure $\mu_{\mathcal{C}}$ obtained from the normalization of the restriction to $\mathcal{C}$ of the finite measure $\lambda_{\kappa}^{(1)}$ the \emph{Masur--Veech measure} of $\mathcal{C}$.

In this language, the global picture is the following: we dispose of a $SL(2,\mathbb{R})$-action on connected components $\mathcal{C}$ of strata $\mathcal{H}^{(1)}(\kappa)$ of the moduli space of Abelian differentials with unit area and a naturally invariant probability measure $\mu_{\mathcal{C}}$ (the 
Masur--Veech measure).

Of course, it is tempting to start the study of the statistics of $SL(2,\mathbb{R})$-orbits of this action with respect to Masur--Veech measure, but we'll momentarily refrain from doing so (instead we postpone to the next section this discussion) because this is the appropriate place to introduce the so-called \emph{Teichm\"uller (geodesic) flow}.

The Teichm\"uller flow $(g_t)_{t\in \mathbb R}$ on $\mathcal{H}_g^{(1)}$ is simply the action of the diagonal subgroup $g_t:=\left(\begin{array}{cc}e^{t} & 0 \\ 0 & e^{-t}\end{array}\right)$ of $SL(2,\mathbb{R})$. The discussions we had so far imply that $g_t$ is the geodesic flow of the Teichm\"uller metric (introduced in Section~\ref{s.intro}). Indeed, from Teichm\"uller's theorem (see Theorem~\ref{t.Teichmuller}), it follows that the path $\{(M_t,\omega_t):t\in\mathbb{R}\}$, where $\omega_t=g_t(\omega_0)$ and $M_t$ is the underlying Riemann surface structure such that $\omega_t$ is holomorphic, is a geodesic of Teichm\"uller metric $d$, and $d((M_0,\omega_0),(M_t,\omega_t))=t$ for all $t\in\mathbb{R}$ (i.e., $t$ is the arc-length parameter). 

In Figure~\ref{f.gt-L-billiard} above, we have drawn the action of Teichm\"uller geodesic flow $(g_t)_{t\in \mathbb R}$ on an Abelian differential $\omega\in\mathcal{H}(2)$ associated to a L-shaped square-tiled surface derived from $3$ squares. At a first glance, if the reader forgot the discussion at the end of Section~\ref{s.intro}, he/she will find (again) the dynamics of $(g_t)_{t\in \mathbb R}$ very uninteresting: the initial L-shaped square-tiled surface, no matter how it is rotated in the plane, gets indefinitely squeezed in the vertical direction and stretched in the horizontal direction, so that we don't have any hope of finding a surface whose shape is somehow ``close'' to the initial shape (that is, $(g_t)_{t\in \mathbb R}$ doesn't seem to have any interesting dynamical feature such as recurrence). However, as we already mentioned in by the end of Section~\ref{s.intro} (in the genus $1$ case), while this is true in Teichm\"uller spaces $\mathcal{TH}(\kappa)$, it is rarely true in moduli spaces $\mathcal{H}(\kappa)$: in fact, while in Teichm\"uller spaces we can only identify ``points'' by diffeomorphisms \emph{isotopic to the identity}, in the case of moduli spaces one can profit from the (orientation-preserving) diffeomorphisms \emph{not isotopic to identity}  to \emph{eventually} bring deformed shapes close to the initial one. In other words, the very fact that we deal with the modular group $\Gamma_g=\textrm{Diff}^+(M)/\textrm{Diff}_0^+(M)$ (i.e., diffeomorphisms not necessarily isotopic to identity) in the case of moduli spaces allows to \emph{change names} to homology classes 
of the surfaces as we wish, that is, geometrically we can \emph{cut} our surface along any separating
systems of closed loops, then \emph{glue} back the pieces \emph{by translation} (!) in some other
way, and, by definition, the resulting surface will represent the \emph{same} point in moduli space 
as our initial surface. Below, we included a picture illustrating this:    

\begin{figure}[h!]
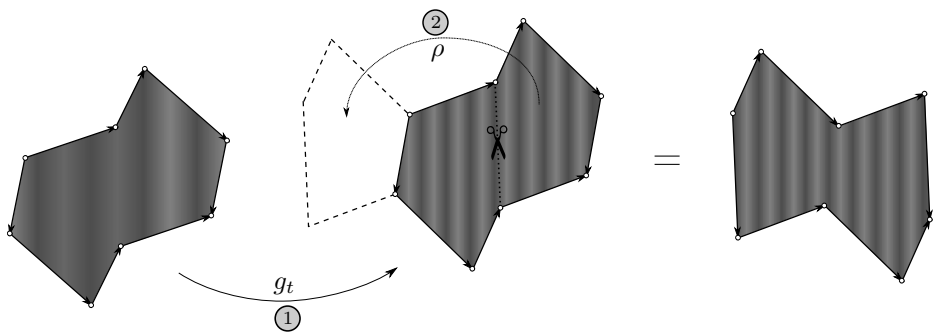
\caption{Actions of $(g_t)_{t\in \mathbb R}$ and an element $\rho$ of $\Gamma_g$}\label{f.zorich-gt}
\end{figure}

To further analyze the dynamics of Teichm\"uller flow $(g_t)_{t\in \mathbb R}$ (and/or of the $SL(2,\mathbb{R})$-action) on $\mathcal{H}_g^{(1)}$, it is surely important to know its derivative $Dg_t$. In the next subsection, we will follow M.~Kontsevich and A.~Zorich to show that the dynamically relevant information about $Dg_t$ is captured by the so-called \emph{Kontsevich--Zorich cocycle}. 

\subsection{Teichm\"uller flow and Kontsevich--Zorich cocycle on the Hodge bundle over $\mathcal{H}_g^{(1)}$}\label{ss.DgtKZ} We start with the following trivial bundle over Teichm\"uller space of Abelian differentials $\mathcal{TH}_g^{(1)}$:
$$\widehat{H^1_g}:=\mathcal{TH}_g^{(1)}\times H^1(M,\mathbb{R})$$ 
and the \emph{trivial (dynamical) cocycle} over Teichm\"uller flow $(g_t)_{t\in \mathbb R}$:
$$\widehat{G_t^{KZ}}:\widehat{H_g^1}\to\widehat{H_g^1}, \quad \widehat{G_t^{KZ}}(\omega,c)=(g_t(\omega),c).$$
Of course, there is not much to say here: we act through Teichm\"uller flow in the first component and we're not acting (or acting trivially if you wish) in the second component  of a trivial bundle.

Now, we observe that the modular group $\Gamma_g$ acts on \emph{both} components of $\widehat{H_g^1}$ by pull-back, and, as we have already seen, the action of Teichm\"uller flow $(g_t)_{t\in \mathbb R}$ commutes with the action of $\Gamma_g$ (since $g_t$ acts by post-composition on the local charts of a translation structure while $\Gamma_g$ acts by pre-composition on them). Therefore, it makes sense to take the quotients 
$$H_g^1:=(\mathcal{TH}_g^{(1)}\times H^1(M,\mathbb{R}))/\Gamma_g$$
and $G_t^{KZ}:=\widehat{G_t^{KZ}}/\Gamma_g$. In the literature, $H_g^1$ is called the (real) \emph{Hodge bundle} over the moduli space of Abelian differentials $\mathcal{H}^{(1)}_g=
\mathcal{TH}^{(1)}_g/\Gamma_g$ and $G_t^{KZ}$ is called the \emph{Kontsevich--Zorich cocycle}\footnote{In fact, this \emph{doesn't} lead to a linear cocycle in the \emph{usual} sense of Dynamical Systems because the Hodge bundle is an \emph{orbifold} bundle. Indeed, $G_t^{KZ}$ is well-defined along Teichm\"uller orbits of translation surfaces \emph{without} non-trivial automorphisms, but there is an \emph{ambiguity} when the translation surface $(M,\omega)$ has a non-trivial group of automorphisms $\textrm{Aut}(M,\omega)$. In simple terms, this ambiguity comes from the fact that the fiber of the Hodge bundle over $(M,\omega)$ is the quotient of $H^1(M,\mathbb{R})$ by the group $\textrm{Aut}(M,\omega)$, so that, if $\textrm{Aut}(M,\omega)\neq\{id\}$, the KZ cocycle induces only linear maps on $H^1(M,\mathbb{R})$ \emph{modulo} the cohomological action of $\textrm{Aut}(M,\mathbb{R})$. Notice that $\textrm{Aut}(M,\omega)=\{id\}$ for almost every $(M,\omega)$ (with respect to Masur--Veech measures), so that this ambiguity problem doesn't concern generic orbits. In any event, as far as Lyapunov exponents are concerned, this ambiguity is not hard to solve. By Hurwitz theorem, $\#\textrm{Aut}(M,\omega)\leq 84(g-1)<\infty$, so that one can get rid of the ambiguity by taking adequate \emph{finite} covers of the KZ cocycle (e.g., by \emph{marking} horizontal separatrices of the translation surfaces). See, e.g., \cite{MYZ} for more details and comments on this.} (KZ cocycle for short) over the Teichm\"uller flow $(g_t)_{t\in \mathbb R}$ on $\mathcal{H}_g^{(1)}$.

We begin by pointing out that the Kontsevich--Zorich cocycle $G_t^{KZ}$ (unlike its ``parent'' $\widehat{G_t^{KZ}}$) is very far from being trivial. Indeed, since we identify $(\omega,c)$ with $(\rho^*(\omega),\rho^*(c))$ for any $\rho\in\Gamma_g$ to construct the Hodge bundle and $G_t^{KZ}$, it follows that the fibers of $\widehat{H_g^1}$ over $\omega$ and $\rho^*(\omega)$ are identified in a \emph{non-trivial} way if the (standard cohomological) action of $\rho$ on $H^1(M,\mathbb{R})$ is non-trivial. 

Alternatively, suppose we fix a fundamental domain $\mathcal{D}$ of $\Gamma_g$ on $\mathcal{TH}_g$ (e.g., through Veech's zippered rectangle construction) and let's say we start with some point $\omega$ at the boundary of $\mathcal{D}$, a cohomology class $c\in H^1(M,\mathbb{R})$ and assume that the Teichm\"uller geodesic through $\omega$ points towards the interior of $\mathcal{D}$. Now, we run Teichm\"uller flow for some (long) time $t_0$ until we hit again the boundary of $\mathcal{D}$ and our geodesic is pointing towards the exterior of $\mathcal D$ . At this stage, since $\mathcal D$ is a fundamental domain, from the definition of Kontsevich--Zorich cocycle, we have the ``option'' to apply an element $\rho$ of the modular group $\Gamma_g$ so that Teichm\"uller flow through $\rho^*(g_{t_0}\omega)$ points towards the interior of $\mathcal{D}$ ``\emph{at the cost}'' of replacing the cohomology class $c$ by $\rho^*(c)$. In this way, we see that $G_{t_0}^{KZ}(\omega,c)=(\rho^*(\omega),\rho^*(c))$ 
is non-trivial in general. 

Below we illustrate (see Figures~\ref{f.domfund1} and~\ref{f.domfund2}) this ``fundamental domain''-based discussion in both genus $g=1$ and $g\geq 2$ cases (the picture in the higher-genus case being idealized, of course, since the moduli space is higher-dimensional). 

\begin{figure}[htb!]
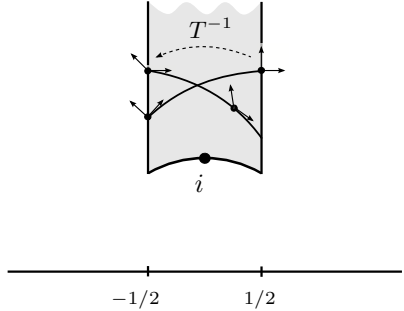
\caption{Kontsevich--Zorich cocycle on the moduli space of torii.}\label{f.domfund1}
\end{figure}

Here we are projecting the picture from the unit cotangent bundle $\mathcal{H}^{(1)}(0)=SL(2,\mathbb{R})/SL(2,\mathbb{Z})$ to the moduli space of torii $\mathcal{M}_{1,1}=\mathbb{H}/SL(2,\mathbb{Z})$, so that the evolution of the Abelian differentials $g_t(\omega)$ are pictured by the tangent vectors to the hyperbolic geodesics, while the evolution of cohomology classes is pictured by the transversal vectors to these geodesics.

\begin{figure}[h!]
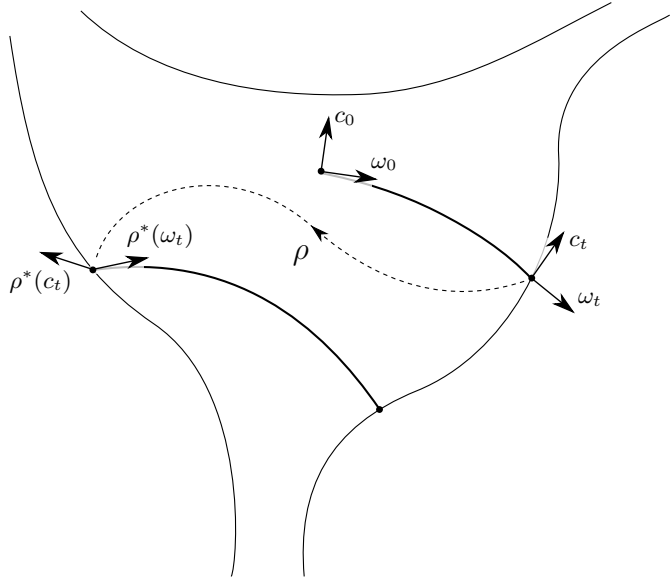
\caption{Kontsevich--Zorich cocycle on the Hodge bundle $H_g^1$, $g\geq 2$.}\label{f.domfund2}
\end{figure}

Next, we observe that the $G_t^{KZ}$ is a \emph{symplectic cocycle} because the action by pull-back of the elements of $\Gamma_g$ on $H^1(M,\mathbb{R})$ preserves the \emph{intersection form} $(c,c')=\int_M c\wedge c'$, a symplectic form on the $2g$-dimensional real vector space $H^1(M,\mathbb{R})$. This has the following consequence for the Ergodic Theory of the KZ cocycle. Given any ergodic Teichm\"uller flow invariant probabilty $\mu$ on $\mathcal H^{(1)}_g$, we know from \emph{Oseledets theorem} that there are real numbers (\emph{Lyapunov exponents}) $\lambda_1^{\mu}>\dots>\lambda_{k}^{\mu}$ and a Teichm\"uller flow equivariant decomposition $H^1(M,\mathbb{R})=E^1(\omega)\oplus\dots\oplus E^{k}(\omega)$ at $\mu$-almost every point $\omega$ such that $E^i(\omega)$ depends measurably on $\omega$ and 
$$\lim\limits_{t\to\pm\infty}\frac{1}{t}\log(\|G_t^{KZ}(\omega,v)\|/\|v\|)=\lambda^{\mu}_i$$
for every $v\in E^i(\omega)-\{0\}$ and any choice\footnote{We will see in Remark~\ref{r.log-int-KZ} below that one can choose the so-called \emph{Hodge norm} here.} of $\|.\|$ such that $\int\log^+\|G_{\pm1}^{KZ}\|d\mu<\infty$. If we allow ourselves to repeat each $\lambda^{\mu}_i$ accordingly to its multiplicity $\textrm{dim}\,E^i(\omega)$, we get a list of $2g$ Lyapunov exponents
$$\lambda^{\mu}_1\geq\dots\geq\lambda_{2g}^{\mu}\,.$$
Such a list is commonly called \emph{Lyapunov spectrum} (of the KZ cocycle with respect to the probability
measure $\mu$ on $\mathcal H^{(1)}_g$). The fact that the KZ cocycle is symplectic means that the Lyapunov spectrum is always \emph{symmetric with respect to the origin}:
$$\lambda_1^{\mu}\geq\dots\geq\lambda_g^{\mu}\geq 0 \geq-\lambda_g^{\mu}\geq\dots-\lambda_1^{\mu}$$
that is, $\lambda_{2g-i}^{\mu} = -\lambda_{i+1}^{\mu}$ for every $i=0,\dots,g-1$. Roughly speaking, this symmetry corresponds to the fact that whenever $\theta$ appears as an eigenvalue of a symplectic matrix $A$, $\theta^{-1}$ is also an eigenvalue of $A$ (so that, by taking logarithms, we ``see'' that the appearance of a Lyapunov exponent $\lambda$ forces the appearance of a Lyapunov exponent $-\lambda$). Thus, it suffices to study the non-negative Lyapunov exponents of the KZ cocycle to determine its entire Lyapunov spectrum.

Also, in the specific case of the KZ cocycle, it is not hard to conclude that $\pm1$ belong to the Lyapunov spectrum of any ergodic probability measure $\mu$. Indeed, by the definition, the family of \emph{symplectic} planes $E(\omega)=\mathbb{R}\cdot[\textrm{Re}(\omega)]\oplus \mathbb{R}\cdot[\textrm{Im}(\omega)]\subset H^1(M,\mathbb{R})$ generated by the cohomology classes of the real and imaginary parts of $\omega$ are Teichm\"uller flow (and even $SL(2,\mathbb{R})$) equivariant. Also, the action of Teichm\"uller flow \emph{restricted} to these planes is, by definition, isomorphic to the action of the matrices $g_t=\left(\begin{array}{cc}e^t & 0 \\ 0 & e^{-t}\end{array}\right)$ on the usual plane $\mathbb{R}^2$ if we identify $[\textrm{Re}(\omega)]$ with the canonical vector $e_1=(1,0)\in\mathbb{R}^2$ and $[\textrm{Im}(\omega)]$ with the canonical vector $e_2=(0,1)\in\mathbb{R}^2$. Actually, the same is true for the entire $SL(2,\mathbb{R})$-action restricted to these planes (where we replace $g_t$ by the corresponding matrices). Since the Lyapunov exponents of the $g_t$ action on $\mathbb{R}^2$ are $\pm1$, we get that $\pm1$ belong to the Lyapunov spectrum of the KZ cocycle. 

Actually, it is possible to prove that   $\lambda_1^{\mu}=1$ (i.e., $1$ is always the top exponent), and $\lambda_1^{\mu}=1>\lambda_2^{\mu}$, i.e., the top exponent has always multiplicity 1, or, in other words, the Lyapunov exponent $\lambda_1^{\mu}$ is always \emph{simple}. However, since this requires some machinery (\emph{variational formulas} for the \emph{Hodge norm} on the Hodge bundle), we postpone this discussion to Subsection~\ref{ss.1stvariation} below, and we close this subsection by relating the KZ cocycle with the tangent cocycle $Dg_t$  of the Teichm\"uller flow $(g_t)_{t\in \mathbb R}$ (a relation which is one of the main motivations for introducing the KZ cocycle).

By writing $H^1(M,\Sigma,\mathbb{C})=\mathbb{C}\otimes H^1(M,\Sigma,\mathbb{R})=\mathbb{R}^2\otimes H^1(M,\Sigma,\mathbb{R})$ and considering the action of $Dg_t$ on each factor of this tensor product, one can check (from the fact that local coordinates of connected components of strata $\mathcal{H}^{(1)}(\kappa)$ are given by period maps) that $Dg_t$ acts  through the usual action of the matrices $g_t=\textrm{diag}(e^t,e^{-t})$ on the first factor $\mathbb{R}^2$ and it acts through the natural generalization $\widetilde{G}_t^{KZ}$ of the KZ cocycle $G_t^{KZ}$ on the second factor $H^1(M,\Sigma,\mathbb{R})$! In particular, the Lyapunov exponents of $Dg_t$ have the form $\pm1+\lambda$ where $\lambda$ are Lyapunov exponents of $\widetilde{G}_t^{KZ}$.

Now, we observe that the relative part doesn't contribute with interesting exponents of $\widetilde{G}_t^{KZ}$, so that it suffices to understand the absolute part. More precisely, we claim that  for any return time $t\in \R$ of the Teichm\"uller orbit $g_t(\omega)$  of an Abelian differential $\omega$ on a Riemann surface $M$ to a fixed compact subset of the moduli space, the natural action 
of $\widetilde{G}_t^{KZ}$ on the quotient\footnote{This argument would be easier to perform if one disposes of equivariant relative parts (i.e., equivariant supplements of $H^1(M,\mathbb{R})$ in $H^1(M,\Sigma,\mathbb{R})$). However, as we will see in Remark~\ref{r.supplement} below, this is not true in general.} $H^1(M,\Sigma,\mathbb{R})/H^1(M,\mathbb{R})$ (``relative part'' of dimension $\sigma-1$) is through \emph{bounded} linear transformations, so that this part contributes with $\sigma-1$ zero exponents  (where $\sigma=\#\Sigma$) to the Lyapunov spectrum of $\widetilde{G}_t^{KZ}$. Indeed, this claim is true because there are \emph{no} long relative cycles in $H_1(M,\Sigma,\mathbb{R})/H_1(M,\mathbb{R})$: for instance, in the figure below we see that any attempt to produce a long relative cycle $c_1\in H_1(M,\Sigma,\mathbb{R})$ between two points $p_1, p_2\in\Sigma$ by applying $\widetilde{G}_t^{KZ}$, say $c_1=\widetilde{G}_t^{KZ}(\omega,c_0)$ for a large $t>0$, can be ``undone'' by taking a \emph{bounded} cycle $c_2$ between $p_1$ and $p_2$ (this is always possible if the translation surface
$(M_t, \omega_t)= g_t(M, \omega)$ has bounded diameter); in this way, $c_1$ differs from $c_2$ by an \emph{absolute} cycle $\gamma=c_1-c_2\in H_1(M,\mathbb{R})$, i.e., $c_1$ and $c_2$ represent the \emph{same} element of $H_1(M,\Sigma,\mathbb{R})/H_1(M,\mathbb{R})$, hence $\widetilde{G}_t^{KZ}$ acts on $H_1(M,\Sigma,\mathbb{R})/H_1(M,\mathbb{R})$ via a bounded linear transformation. 
\begin{center}

\begingroup
  \makeatletter
  \providecommand\color[2][]{%
    \errmessage{(Inkscape) Color is used for the text in Inkscape, but the package 'color.sty' is not loaded}
    \renewcommand\color[2][]{}%
  }
  \providecommand\transparent[1]{%
    \errmessage{(Inkscape) Transparency is used (non-zero) for the text in Inkscape, but the package 'transparent.sty' is not loaded}
    \renewcommand\transparent[1]{}%
  }
  \providecommand\rotatebox[2]{#2}
  \ifx\svgwidth\undefined
    \setlength{\unitlength}{267.57231445pt}
  \else
    \setlength{\unitlength}{\svgwidth}
  \fi
  \global\let\svgwidth\undefined
  \makeatother
  \begin{picture}(1,0.54279952)%
    \put(0,0){\includegraphics[width=\unitlength]{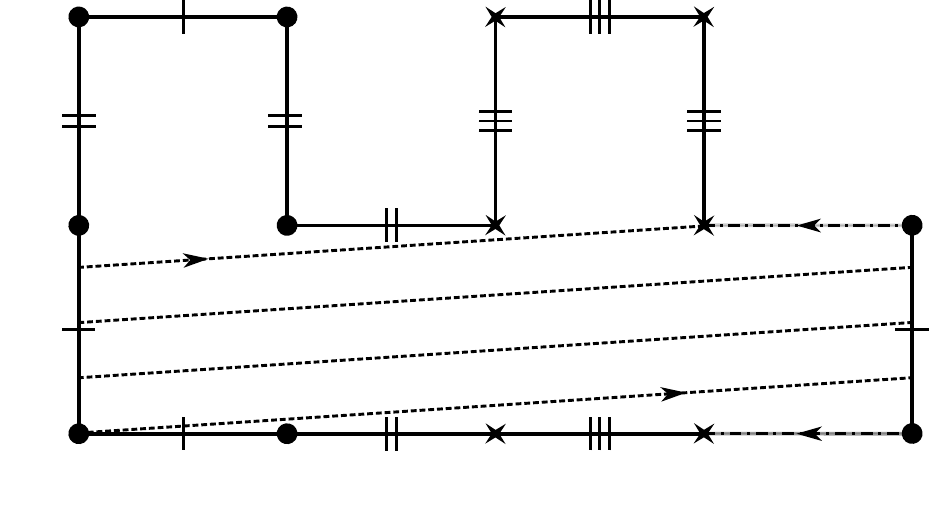}}%
    \put(0.69569702,0.14025403){\color[rgb]{0,0,0}\makebox(0,0)[lb]{\smash{$c_1$}}}%
    \put(0.85680523,0.32336229){\color[rgb]{0,0,0}\makebox(0,0)[lb]{\smash{$c_2$}}}%
    \put(0.70295425,0.32094317){\color[rgb]{0,0,0}\makebox(0,0)[lb]{\smash{$p_2$}}}%
    \put(0.03294313,0.04129949){\color[rgb]{0,0,0}\makebox(0,0)[lb]{\smash{$p_1$}}}%
  \end{picture}%
\endgroup

\end{center}
In other words, the interesting Lyapunov exponents of $\widetilde{G}_t^{KZ}$ come from the absolute part $H^1(M,\mathbb{R})$, i.e., it is the KZ cocycle $G_t^{KZ}:=\widetilde{G}_t^{KZ}|_{H^1(M,\mathbb{R})}$ who describes the most exciting Lyapunov exponents. Equivalently, the Lyapunov spectrum of 
$\widetilde{G}_t^{KZ}$ consists of $\sigma-1$ zero exponents and of the $2g$ Lyapunov exponents 
$\pm\lambda_1^{\mu}, \dots,\pm\lambda_g^{\mu}$ of the KZ cocycle.

Thus, in summary, the Lyapunov spectrum of Teichm\"uller flow $(g_t)_{t\in \mathbb R}$ with respect to an ergodic probability measure $\mu$ supported on a stratum $\mathcal{H}^{(1)}(\kappa)$ (with 
$\kappa=(k_1,\dots,k_{\sigma})$) has the form 
\begin{eqnarray*}
2\geq 1+\lambda_2^{\mu}\geq\dots\geq 1+\lambda_g^{\mu}\geq \overbrace{1=\dots=1}^{\sigma-1}\geq 
1-\lambda_g^{\mu}\geq\dots\geq 1-\lambda_2^{\mu}\geq 0 \\ 
\geq -1+\lambda_2^{\mu}\geq\dots\geq-1+\lambda_g^{\mu}\geq \overbrace{-1=\dots=-1}^{\sigma-1}\geq -1-\lambda_g^{\mu}\geq\dots\geq-1-\lambda_2^{\mu}\geq-2
\end{eqnarray*}
where $1\geq\lambda_2^{\mu}\geq\dots \geq \lambda_g^{\mu}$ are the non-negative exponents of the KZ cocycle $G_t^{KZ}$ with respect to the  probability measure $\mu$ on $\mathcal{H}^{(1)}(\kappa)$.

\begin{remark}\label{r.supplement} Concerning the computation of the exponents $\pm1$ in the relative part of cohomology (i.e., before passing to absolute cohomology), our task would be easier if $H^1(M,\mathbb{C})$ admitted a equivariant supplement inside $H^1(M,\Sigma,\mathbb{C})$. However, it is possible to construct examples to show that this doesn't happen in general. See Appendix B of~\cite{MY} for more details.  
\end{remark}

Therefore, the KZ cocycle captures the ``main part'' of the derivative cocycle $Dg_t$, so that, since we're interested in the Ergodic Theory of Teichm\"uller flow, we will spend sometime in the next sections to analyze KZ cocycle (without much reference to $Dg_t$). 

\subsection{Hodge norm on the Hodge bundle over $\mathcal{H}_g^{(1)}$}\label{ss.Hodgenorm} By definition, the task of studying Lyapunov exponents consists precisely in understanding the growth of the norm of vectors. Of course, the particular choice of norm doesn't affect the values of Lyapunov exponents (essentially because two norms on a finite-dimensional vector space are equivalent), but for the sake of our discussion it will be convenient to work with the so-called \emph{Hodge norm}. 

Let $M$ be a Riemann surface. The \emph{Hodge (intersection) form} $(.,.)$ on $H^1(M,\mathbb{C})$ is defined as
$$(\alpha,\beta):=\frac{i}{2}\int_M \alpha\wedge\overline{\beta} \,, \qquad \textrm{for each } 
\alpha, \beta\in H^1(M,\mathbb{C})\,.$$

The Hodge form is positive-definite on the space $H^{1,0}(M)$ of holomorphic $1$-forms on $M$, and negative-definite on the space $H^{0,1}(M)$ of 
anti-holomorphic $1$-forms on $M$. For instance, given a holomorphic $1$-form $\alpha\neq0$, we can locally write $\alpha(z)=f(z)dz$, so that 
$$\alpha(z)\wedge\overline{\alpha(z)}=|f(z)|^2 dz\wedge\overline{dz} = -2i|f(z)|^2 dx\wedge dy.$$ 
Since $dx\wedge dy$ is an area form on $M$ and $|f(z)|^2\geq 0$, we get that $(\alpha,\alpha)>0$. 

In particular, since $H^1(M,\mathbb{C})=H^{1,0}(M)\oplus H^{0,1}(M)$, and the complex subspaces 
$H^{1,0}(M)$ and $H^{0,1}(M)$ are $g$-dimensional, the Hodge form is a Hermitian form of signature $(g,g)$ on $H^1(M,\mathbb{C})$.

The Hodge form is equivariant with respect to the natural action of the mapping-class group 
$\Gamma_g$ on $\mathcal H^{(1)}_g$. In particular, it induces a Hermitian form (also called Hodge form and denoted by $(.,.)$) on the \emph{complex} Hodge bundle $H_g^1(\mathbb{C})=
(\mathcal{TH}_g^{(1)}\times H^1(M,\mathbb{C}))/\Gamma_g$ over $\mathcal{H}_g^{(1)}$. 

The so-called \emph{Hodge representation theorem} says that any \emph{real} cohomology class $c\in H^1(M,\mathbb{R})$ is the real part of an \emph{unique} holomorphic $1$-form $h(c)\in H^{1,0}(M)$, i.e., $c=[\Re h(c)]$. In particular, one can use the Hodge form $(.,.)$ to induce an inner product on $H^1(M,\mathbb{R})$ via the formula:
$$(c_1,c_2):=(h(c_1), h(c_2))\,,  \qquad \textrm{for each }  c_1, c_2\in H^1(M,\mathbb{R})\,.$$
Again, this induces an inner product $(.,.)$ and a norm $\|.\|$ on the \emph{real} Hodge bundle 
$$H_g^1 =H_g^1(\mathbb{R}):=(\mathcal{TH}_g^{(1)}\times H^1(M,\mathbb{R}))/\Gamma_g\,.$$  In the literature, $(.,.)$ is the \emph{Hodge inner product} and $\|.\|$ is the \emph{Hodge norm} on the real Hodge bundle.
Observe that, in general, the subspaces $H^{1,0}$ and $H^{0,1}$ are \emph{not} equivariant with respect to the natural complex extension of the KZ cocycle to the complex Hodge bundle $H_g^1(\mathbb{C})$, and this is one of the reasons why the Hodge norm $\|.\|$ is \emph{not} preserved by the KZ cocycle in general. In the next subsection, we will study \emph{first variation formulas} for the Hodge norm along the KZ cocycle and their applications to the Teichm\"uller flow. 

\subsection{First variation of Hodge norm and hyperbolic features of Teichm\"uller flow}\label{ss.1stvariation} Let $c\in H^1(M,\mathbb{R})$ be a vector in the fiber of the real Hodge bundle over 
$\omega\in\mathcal{H}_g^{(1)}$. Denote by $\alpha_0$ the holomorphic $1$-forms with $c=\Re\alpha_0$. For all $t\in \R$, let $(M_t, \omega_t)= g_t(M, \omega)$ denote the orbit of the Abelian
differential $\omega$ under the Teichm\"uller flow. By the Hodge representation theorem, there
exists a holomorphic $1$-form  $\alpha_t$ on $M_t$ such that $c=[\Re\alpha_t]$ where $\alpha_t$ is a holomorphic $1$-form with respect to the new Riemann surface structure associated to 
$\omega_t$. 

Of course, by definition, KZ cocycle acts by parallel transport on the Hodge bundle, so that the cohomology classes $c$ are not ``changing''. However, since the representatives 
$\alpha_t$ we use to ``measure'' the ``size'' (Hodge norm) of $c$ are changing, it is an interesting (and natural) problem to know how fast the Hodge norm changes along KZ cocycle, or, equivalently, to compute the \emph{first variation} of the Hodge norm along the Teichm\"uller flow:
$$\frac{d}{dt}\|G_t^{KZ}(\omega,c)\|_{\omega_t}^2:=\frac{d}{dt}\|c\|_{\omega_t}:=\frac{d}{dt}(\alpha_t,\alpha_t) \,, $$
where  $\|.\|_{\omega_t}$ denotes  the Hodge norm with respect to the Riemann surface structure 
$M_t$ induced by $\omega_t$ (that is, the Hodge norm on the fiber $H^1(M_t, \mathbb R)$ of the
real Hodge bundle at $\omega_t \in \mathcal H^{(1)}_g)$. 

In this subsection we will calculate this quantity by following the original article \cite{F02}. By definition of the Teichm\"uller flow, $\omega_t= := e^t \Re\omega +ie^{-t} \Im \omega$, so that 
$$\dot{\omega_t}=\frac{d}{dt}\omega_t = e^t [\Re\omega] - i e^{-t} [\Im\omega]=\overline{\omega_t}\,.$$

Next, we write $0=c-c=[\Re\alpha_t]-[\Re\alpha_0]$, so that we find smooth family $u_t$ with $du_t=\Re\alpha_t-\Re\alpha_0$. By writing $\alpha_t=f_t\omega_t$, and by taking derivatives, we have 
$$d\dot{u_t}=\Re (\dot{f_t}\omega_t+f_t\dot{\omega_t}) = \Re (\dot{f_t}\omega_t+f_t\overline{\omega_t})
= \Re[ ( \dot{f_t} + \overline{f_t})\omega_t ] \,.$$
We note that, since $\overline{\alpha_t}= \overline{ f_t \omega_t}$, as an anti-holomorphic differential,
is  a closed $1$-form, from the above formula we can derive the following identity:
$$
0= \int_M d\dot{u_t} \wedge \overline{ f_t \omega_t} = \int_M ( \dot{f_t} + \overline{f_t})\omega_t \wedge 
 \overline{ f_t \omega_t} = \int_M  \dot{f_t}  \overline{ f_t } \omega_t \wedge 
 \overline{\omega_t} + \int_M  \overline{f_t}^2 \omega_t \wedge 
 \overline{\omega_t} \,.
$$
Finally, by the above identity we can  compute as follows:
\begin{eqnarray*}
\frac{d}{dt}(\alpha_t,\alpha_t) &=& \frac{i}{2} \frac{d}{dt} \int_M f_t\overline{f_t}\, \omega_t\wedge\overline{\omega_t} = \frac{i}{2} \frac{d}{dt} \int_M f_t\overline{f_t} \,\omega \wedge \overline{\omega} \\ 
&=&  2\Re \frac{i}{2}  \int_M \dot{f_t} \overline{f_t} \, \omega \wedge \overline{\omega}  = 
2\Re \frac{i}{2}  \int_M \dot{f_t} \overline{f_t} \, \omega_t\wedge \overline{\omega_t} \\ &=&
-  2\Re \frac{i}{2}\int_M  \overline{f_t}^2 \omega_t \wedge 
 \overline{\omega_t} = -  2\Re \frac{i}{2}\int_M  f_t^2 \omega_t \wedge 
 \overline{\omega_t}= -2\Re \frac{i}{2}\int_M (\frac{\alpha_t}{\omega_t})^2\, \omega_t\wedge \overline{\omega_t}
\end{eqnarray*}

In summary, we proved the following formula (originally from Lemma 2.1' of \cite{F02}). Let
$$
B_{\omega}(\alpha,\beta):=\frac{i}{2} \int_M \frac{ \alpha \beta}{\omega} \frac{\overline{\omega}}
{\omega} \,\qquad \textrm{for all } \alpha, \beta \in H^{1,0}(M)\,.
$$

\begin{theorem}[G. Forni]\label{t.1stvariation} Let $\omega$ be an Abelian differential on a Riemann surface $M$ and let $(M_t, \omega_t)=g_t(M,\omega)$ denote its Teichm\"uller orbit. Let 
$c\in H^1(M,\mathbb{R})$ and denote by $\alpha_t$ the unique holomorphic  $1$-form on $M_t$ with 
$c=[\Re\alpha_t]$. Then, 
$$\frac{d}{dt}\|G_t^{KZ}(\omega,c)\|^2= \frac{d}{dt}\|c\|_{\omega_t}^2 = -2\Re B_{\omega_t}(\alpha_t,\alpha_t)\,.$$
\end{theorem} 

In order to simplify the notation, we set $B_{\omega}^{\mathbb{R}}(c,c):=B_{\omega}(\alpha,\alpha)$ where $\alpha$ is the unique holomorphic $1$-form on $M$ with $c=[\Re\alpha]$. Observe that 
$B_{\omega}^{\mathbb{R}}$ is a \emph{complex-valued} bilinear form.

\begin{corollary}\label{c.1stvariation} One has 
$$\frac{d}{dt}\log\|c\|_{\omega_t} = -\frac{\Re B^{\mathbb{R}}_{\omega_t}(c,c)}{\|c\|_{\omega_t}^2}$$
In particular, 
$$\frac{d}{dt}\log\|c\|_{\omega_t} \leq 1\,.$$ 
\end{corollary}

\begin{proof} The first statement of this corollary follows from the main formula in Theorem~\ref{t.1stvariation}, while the second statement follows from an application of Cauchy-Schwarz inequality:
\begin{equation*}
|B_{\omega}(\alpha,\beta)|\leq \frac{i}{2} \int_M |\frac{\alpha}{\omega}| |\frac{\beta}{\omega}|\,\omega\wedge\overline{\omega}\leq 
\left(\frac{i}{2} \int_M |\frac{\alpha}{\omega}|^2\, \omega\wedge\overline{\omega}\right)^{1/2}
\left(\frac{i}{2}\int_M |\frac{\beta}{\omega}|^2\, \omega\wedge\overline{\omega}\right)^{1/2} = 
\|\alpha\|_{\omega}\|\beta\|_{\omega}
\end{equation*}
\end{proof}

\begin{remark}\label{r.log-int-KZ} Corollary~\ref{c.1stvariation} implies that the KZ cocycle is $\log$-bounded with respect to the Hodge norm, that is, 
$\log\|G_{\pm1}^{KZ}(\omega,c)\|_{g_{\pm1}(\omega)}\leq 1$ for all $c\in H^1(M,\mathbb{R})$ with $\|c\|_{\omega}=1$. Hence, given any  measure $\mu$ on $\mathcal{H}_g^{(1)}$ of finite total mass, we have that 
$$\int_{\mathcal{H}_g^{(1)}} \log^+\|G_{\pm1}^{KZ}(\omega)\|d\mu<\infty$$
\end{remark}

\begin{corollary}\label{c.spectralgapKZ} Let $\mu$ be any $g_t$-invariant ergodic probability measure on $\mathcal{H}_g^{(1)}$. Then, $$\lambda_2^{\mu}<1=\lambda_1^{\mu}\,.$$
\end{corollary}

\begin{proof} By Corollary~\ref{c.1stvariation}, we have that $\lambda_1^{\mu}\leq 1$. Moreover, since the Teichm\"uller flow $g_t(\omega):=e^t\Re\omega+ i e^{-t}\Im\omega=\omega_t$, we have that the $G_t^{KZ}$-invariant $2$-plane $H^1_{st}(M,\mathbb{R}):=\mathbb{R}\cdot[\Re\omega]\oplus \mathbb{R}\cdot[\Im\omega]$ contributes with Lyapunov exponents $\pm1$. In particular, $\lambda_1^{\mu}=1$. 

Now, we note that $H^1_{(0)}(M,\mathbb{R}):=\{c\in H^1(M,\mathbb{R}):c\wedge [\omega]=0\}$ is $G_t^{KZ}$-invariant because the KZ cocycle is symplectic with respect to the intersection form on $H^1(M,\mathbb{R})$ and $H^1_{(0)}(M,\mathbb{R})$ is the symplectic orthogonal of the (symplectic) $2$-plane $H^1_{st}(M,\mathbb{R})$. Therefore, 
$\lambda_2^{\mu}$ is the largest Lyapunov exponent of the restriction of the KZ cocycle to $H^1_{(0)}(M,\mathbb{R})$. 

In order to estimate $\lambda_2^{\mu}$, we observe that, for any $c\in H^1_{(0)}(M,\mathbb{R})-\{0\}$, 
$$\frac{d}{dt}\log\|c\|_{\omega_t}= -\frac{\Re B^{\mathbb{R}}_{\omega_t}(c,c)}{\|c\|_{\omega_t}^2}\leq \Lambda^+(\omega_t):= 
\max\left\{\frac{|B^{\mathbb{R}}_{\omega_t}(h,h)|}{\|h\|_{\omega_t}^2}: h\in H^1_{(0)}(M_t,\mathbb{R})-\{0\}\right\}$$
by Corollary~\ref{c.1stvariation}. Hence, by integration, 
$$\frac{1}{T}(\log\|c\|_{g_T(\omega)}-\log\|c\|_{\omega})\leq\frac{1}{T}\int_0^T\Lambda^+(g_t(\omega))dt$$
By Oseledets theorem and Birkhoff's theorem, for $\mu$-almost every $\omega\in\mathcal{H}_g^{(1)}$, we obtain that  
$$\lambda_2^{\mu}=\lim\limits_{T\to\infty}\frac{1}{T}\log\|c\|_{g_T(\omega)}\leq\lim\limits_{T\to\infty}\frac{1}{T}\int_0^T\Lambda^+(g_t(\omega))dt = 
\int_{\mathcal{H}_g^{(1)}} \Lambda^+(\omega) d\mu(\omega)$$ 

This reduces the task of proving that $\lambda_2^{\mu}<1$ to show that $\Lambda^+(\omega)<1$ for \emph{every} $\omega\in\mathcal{H}_g^{(1)}$. Here, we proceed by contradiction. Assume that there
exists an Abelian differential $\omega \in \mathcal H^{(1)}_g$ such that $\Lambda^+(\omega)=1$. By definition, this means that there exists $h\in H^1_{(0)}(M,\mathbb{R})-\{0\}$ such that
$$|B_{\omega}^{\mathbb{R}}(h,h)|=\|h\|_{\omega}^2\,.$$
In other words, by looking at the proof of Corollary~\ref{c.1stvariation}, we have equality in an estimate derived from Cauchy-Schwarz inequality. Let  $\alpha_h\neq 0$ the holomorphic $1$-form on $M$ such that $h=[\Re\alpha_h]$.  It follows that  the functions $\alpha_h/\omega$ and 
$\overline{\alpha_h}/\overline{\omega}$ differ by a multiplicative constant $a\in\mathbb{C}$, i.e., 
$$\frac{\overline{\alpha_h}}{\overline{\omega}} =a \,\frac{\alpha_h}{\omega}\,. $$
Since $\alpha_h/\omega$ is a meromorphic function and, \emph{a fortiori}, $\overline{\alpha_h}/\overline{\omega}$ is an anti-meromorphic function, this is only possible when $\alpha_h/\omega$ is a \emph{constant} function, that is, $\alpha_h\in\mathbb{C}\cdot\omega-\{0\}$. In particular, $h\wedge 
[\omega] \neq 0$, in contradiction with the assumption that $h\in H^1_{(0)}(M,\mathbb{R})$. 
\end{proof}

The above ``spectral gap'' result has important consequences: on one hand, it implies that
the Teichm\"uller flow is non-uniformly hyperbolic and ergodic with respect to the Masur--Veech
measures and since it is a diagonal flow for an action of $SL(2, \mathbb R)$, it also implies by 
a theorem of Moore that it is mixing; on the other hand, it implies unique ergodicity for almost all
translations flows (hence, for almost all interval exchange transformations) with respect to any
probability Teichm\"uller invariant measures on the moduli space of Abelian differentials. 
The ergodicity of the Teichm\"uller flow was originally proved by H. Masur ~\cite{M82} in the particular case of the so-called principal stratum. Later W. Veech ~\cite{V86} generalized the result and proved that the flow is non-uniformly hyperbolic with respect to a wide class of measures, including the
Masur--Veech measures. The unique ergodicity almost all translation flows a interval exchange
trasnformations was originally proved by H. Masur ~\cite{M82} and W.Veech \cite{V82} independently. 
The proofs presented below are different from the original proofs since they are based on the spectral gap result for the Hodge norm which was discussed above. Our approach also yields that
all  ergodic invariant probability measures for the Teichm\"uller flow are non-uniformly hyperbolic
and are supported on the set of Abelian differentials with uniquely-ergodic horizontal and vertical
foliations. 


\section{Ergodic theory of Teichm\"uller flow with respect to Masur--Veech measures}\label{s.Teich-MV}
We will describe below a proof of ergodicity and mixing of the Teichm\"uller flow with respect
to a class of probability measures which includes the Masur--Veech measures on connected components of strata of Abelian differentials.  We will then outline the theorems of Avila, Gou\"ezel, Yoccoz~\cite{AGY} and Avila, Gou\"ezel~\cite{AG}  on the exponential mixing of the Teichm\"uller flow. For a complete proof of the ergodicity of the Teichm\"uller flow with respect to the Masur--Veech measures, based on the Rauzy--Veech induction, see (again) J.-C. Yoccoz survey~\cite{Y}.

\subsection{Ergodicity of the Teichm\"uller flow}

The Teichm\"uller flow is a locally hyperbolic flow on a non-compact space. For such flows
ergodicity cannot be proved in general. However, in the specific case of Teichm\"uller flow, one 
can determine \emph{explicitly} the invariant manifolds: since $g_t$ acts on $\omega\in \mathcal 
H^{(1)}_g$ by multiplying $[\textrm{Re}(\omega)]$ by $e^t$ and $\textrm{Im}(\omega)$ by $e^{-t}$, 
we infer that the stable manifold $W^s(\omega)$ and the unstable manifold $W^u(\omega)$
are given locally by the equations
$$W^s(\omega_0)=\{\omega\in\mathcal{C}:[\textrm{Re}(\omega)]=[\textrm{Re}(\omega_0)]\} \quad\quad \textrm{and} \quad\quad W^u(\omega_0)=\{\omega\in\mathcal{C}:[\textrm{Im}(\omega)]=[\textrm{Im}(\omega_0)]\}\,.$$
The stable and unstable manifolds form smooth, globally defined invariant foliations, the stable
foliation $W^s$ and the unstable foliation $W ^u$, defined by affine equations  in period coordinates.

We will prove ergodicity for probability measures having connected topological support and 
a local product structure in the following sense. 

\begin{definition}
\label{def:loc_prod_struct}
A Borel probability measures $\mu$ on the moduli space of Abelian differentials has a \emph{local producture structure} if every Abelian differential $\omega \in \mathcal H^{(1)}_g$  has an open neighbourhood  $\mathcal U_\omega$ such that the restriction $\mu\vert \mathcal U_\omega$ is the product of a measure on $W^s(\omega) \cap \mathcal U_\omega$, of a measure on $W^u(\omega)\cap \mathcal U_\omega$ and of the Lebesgue measure $dt$ on $\{g_t(\omega) \vert t\in \mathbb R\} \cap \mathcal U_\omega$.
\end{definition}

\begin{theorem}\label{t.ErgT} The Teichm\"uller geodesic flow $(g_t)_{t\in \mathbb R}$ is ergodic, in fact mixing, and non-uniformly hyperbolic with respect to any invariant probability measures $\mu$ on the moduli space of Abelian differentials which has a local product structure and a connected topological support.
\end{theorem}

Let $\mathcal{C}$ be any connected component of a stratum $\mathcal{H}^{(1)}(\kappa)$ of Abelian differentials with unit area. By the above theorem we immediately derive the following:

\begin{corollary}[H.~Masur, W.~Veech]\label{t.MVp2} The Teichm\"uller geodesic flow $(g_t)_{t\in \mathbb R}$ is ergodic, in fact mixing, with respect to the Masur--Veech measure $\mu_{\mathcal{C}}$ on ${\mathcal{C}}$. 
\end{corollary}

The proof of ergodicity in Theorem~\ref{t.ErgT} is based on the Hopf argument.  In fact, it follows from the simplicity of the top exponent $1=\lambda_1^{\mu}>\lambda_2^{\mu}$ of the Kontsevich--Zorich cocycle with respect to any ergodic invariant probability measure $\mu$ on $\mathcal H^{(1)}_g$ and from the existence of smooth global invariant manifolds for the Teichm\"uller flow. Indeed, as we already know, the simplicity of the top exponent $\lambda_1^\mu=1$ implies that, except for the zero Lyapunov exponent coming from the flow direction, the Teichm\"uller flow $(g_t)_{t\in \mathbb R}$ has no other zero exponents (since $1-\lambda_2^{\mu}>0$ is the smallest non-negative exponent). In other words, the Teichm\"uller flow is \emph{non-uniformly hyperbolic} in the sense of the Pesin theory \emph{with respect to any ergodic probability} $\mu$ on  $\mathcal H^{(1)}_g$.
By the Oseledec's theorem and by ergodic decomposition, we can derive that $\mu$-almost all orbits, with respect to any probability invariant measure $\mu$ on $\mathcal H^{(1)}_g$, are hyperbolic (in the sense that all Lyapunov exponents are well-defined and non-zero except for the exponent of a tangent 
vector in the flow direction). 

This indicates that \emph{Hopf's argument} may apply in our context.  Recall that Hopf's argument starts by observing that ergodic averages are constant along stable and unstable manifolds: more precisely, given a point $x$ such that the ergodic average 
$$\overline{\varphi}(x):=\lim\limits_{t\to+\infty}\frac{1}{t}\int_0^t\varphi(g_s(x))ds$$
exists for a (uniformly) continuous observable $\varphi:\mathcal{C}\to\mathbb{R}$, then the ergodic averages 
$$\overline{\varphi}(y):=\lim\limits_{t\to+\infty}\frac{1}{t}\int_0^t\varphi(g_s(y))ds$$
exists and $\overline{\varphi}(y)=\overline{\varphi}(x)$ for any $y$ in the stable manifold $W^s(x)$ of $x$. Actually, since $y\in W^s(x)$, we have $\lim\limits_{s\to+\infty}d(g_s(y),g_s(x))=0$, so that, by the uniform continuity of $\varphi$, the desired claim follows. Of course, a similar result for ergodic averages along unstable manifolds holds if we replace $t\to+\infty$ by $t\to-\infty$ in the definition of $\overline{\varphi}$. Now, the fact that we consider ``future'' ($t\to+\infty$) ergodic averages along stable manifolds and ``past'' ($t\to-\infty$) ergodic averages along unstable manifolds is not a major problem since Birkhoff's ergodic theorem ensures that these two ``types'' of ergodic averages coincide at 
$\mu$-almost every point. 

In particular, since the ergodicity of $\mu$ is equivalent to the fact that $\overline{\varphi}$ is 
$\mu$- almost everywhere constant, if one could \emph{access} almost any point $y$ starting from 
almost any point $x$ by a zig-zag path along stable and unstable manifolds like in Figure~\ref{f.acessib} below, we would be in good shape (here, we're skipping some details because Hopf's argument needs that the intersection points appearing in Figure~\ref{f.acessib} to satisfy Birkhoff's ergodic theorem; in general, this is issue is strongly related to the so-called \emph{absolute continuity property} of the stable and unstable manifolds, but this is not a problem in our context since by assumption the probability
invariant measure $\mu$ has a local product structure with respect to the invariant foliations). 

\begin{figure}[htb!]
\includegraphics[scale=0.5]{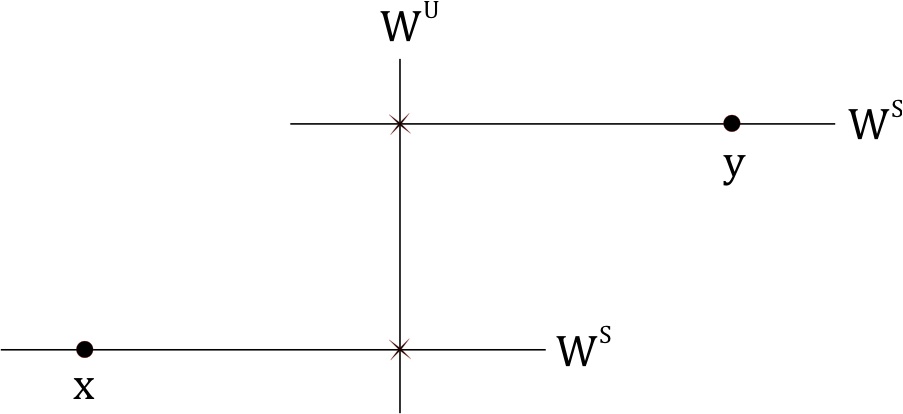}
\caption{A point $y$ which is accessible from $x$ by stable and unstable manifolds.}\label{f.acessib}
\end{figure}

However, it is a general fact that Pesin theory of non-uniformly hyperbolic systems only provides the existence of \emph{short} stable and unstable manifolds. Even worse, the function associating to a typical point the size of its stable/unstable manifolds is only \emph{measurable}. In particular, the nice scenario described above may not happen in general (and actually the best Hopf's argument [alone] can do is to ensure the presence of at most a countable number of open ergodic components). The
situation for the Teichm\"uller flow is much better, since as we have remarked above, the invariant
foliations are globally defined and smooth. In particular, we see that all invariant manifolds are ``large'' subsets corresponding to affine subspaces in period coordinates. Therefore, the potential problem pointed out in the previous paragraph doesn't exist, and one can proceed with Hopf's argument to prove that ergodic components are both {\it open and closed} in the topological support of the measure. 
As we have indicated, they are open by general Pesin's theory, and they are closed since the size of stable and unstable manifolds is bounded below on every compact set. The latter property implies that {\it any} point in the support of the measure has an open neighborhood contained (up to sets of measure zero) in a single ergodic component.  The argument is therefore complete since by assumption the topological support of the measure is connected.

Concerning the second part of the statement of this theorem, we should say that the mixing property of the Teichm\"uller flow follows from its ergodicity and the mere fact that it is the action  of the diagonal
subgroup for an action of the group $\textrm{SL}(2,\mathbb{R})$: indeed, while ergodicity alone doesn't imply mixing in general (e.g., irrational rotations of the circle are ergodic, but not mixing), in the special
case of the diagonal sub-action of an $\textrm{SL}(2,\mathbb{R})$-action, ergodicity implies mixing
by the structure of \emph{unitary representations} of $\textrm{SL}(2,\mathbb{R})$. We discuss this together with the \emph{exponential mixing} property of Teichm\"uller flow in the next subsection.

\subsection{Exponential mixing and spectral gap of $SL(2,\mathbb{R})$ representations}\label{s.AGY} Generally speaking, we say that a flow $(\phi_t)_{t\in\mathbb{R}}$ on a space $X$ is \emph{mixing} with respect to an invariant probability measure $\mu$ when the \emph{correlation function} $C_t(f,g):=\int_X (f\circ\phi_t \cdot g)d\mu - \int_X f d\mu\cdot\int_X g d\mu$ satisfies 
$$\lim\limits_{t\to\infty}|C_t(f,g)|=0$$
for every $f,g\in L^2(X,\mu)$. Of course, the \emph{mixing} property always implies \emph{ergodicity} of $(\phi_t)_{t\in\mathbb{R}}$ but the converse is not always true. However, as we're going to see in a moment, when the flow $(\phi_t)_{t\in\mathbb{R}}$ is part of a larger $SL(2,\mathbb{R})$ action, it is possible to show that ergodicity implies mixing. 

More precisely, suppose that we have a $SL(2,\mathbb{R})$ action on a space $X$ preserving a probability measure $\mu$, and let $(\phi_t)_{t\in\mathbb{R}}$ be the flow on $X$ corresponding to the action of the \emph{diagonal subgroup} $\textrm{diag}(e^t,e^{-t})$ of $SL(2,\mathbb{R})$. In this setting, one has:
\begin{proposition}\label{p.mixing}Assume that $(\phi_t)_{t\in\mathbb{R}}$ is $\mu$-ergodic. Then, $(\phi_t)_{t\in\mathbb{R}}$ is $\mu$-mixing.
\end{proposition}
Of course, the Teichm\"uller flow $(g_t)_{t\in \mathbb R}$ on a connected component $\mathcal{C}$ of a stratum of the moduli space of Abelian differentials equipped with its natural (Masur--Veech) probability measure $\mu_{\mathcal{C}}$ is a prototype example of flow verifying the assumptions of the previous proposition. 

As we pointed out above, the proof of this result uses knowledge of the representation theory of 
$SL(2,\mathbb{R})$. We strongly recommend reading Livio Flaminio's notes in this volume for a nice discussion of this subject. For the sake of convenience, we quickly reviewed some results on this topic in Appendix~\ref{a.Rt} below. In particular, we will borrow the notations from this Appendix. 

We begin by observing that the $SL(2,\mathbb{R})$ action on $(X,\mu)$ induces an unitary representation of $SL(2,\mathbb{R})$ on $\mathcal{H}=L^2_0(X,\mu)$. Here, $L^2_0(X,\mu)$ is the Hilbert space of $L^2$ functions of $(X,\mu)$ with zero mean. In particular, from the semisimplicity of $SL(2,\mathbb{R})$, we can write $\mathcal{H}$ as an integral of irreducible unitary $SL(2,\mathbb{R})$ representations $\mathcal{H}_{\xi}$:
$$\mathcal{H}=\int \mathcal{H}_{\xi}d\lambda(\xi)$$
The fact that $(\phi_t)_{t\in\mathbb{R}}$ is $\mu$-ergodic implies that the $SL(2,\mathbb{R})$ action is $\mu$-ergodic, that is, the trivial representation doesn't appear in the previous integral decomposition. By Bargmann's classification, every nontrivial unitary irreducible representation of $SL(2,\mathbb{R})$ belongs to one of the following three classes (or series): \emph{principal series}, \emph{discrete series} and \emph{complementary series}. See Livio Flaminio's notes in this volume and/or Appendix~\ref{a.Rt} below for more details. 

By M. Ratner's work~\cite{Rt}, we know that, for every $t\geq 1$ and for every $v,w\in\mathcal{H}_{\xi}$ with $\mathcal{H}_{\xi}$ in the \emph{principal or discrete} series, 
$$|C_t(v,w)|:=\left|\int (v\circ\phi_t) \cdot w\, d\mu\right|\leq C\cdot t\cdot e^{-t}\cdot\|v\|_{L^2(X,\mu)}\cdot\|w\|_{L^2(X,\mu)}$$
where $C>0$ is an universal constant. Of course, we're implicitly using the fact that, by hypothesis, $\phi_t$ is exactly the action of the diagonal subgroup $\textrm{diag}(e^t,e^{-t})$ of $SL(2,\mathbb{R})$ on $X$. Also, for every $t\geq 1$ and for every pair $v,w\in\mathcal{H}_{\xi}$ of $C^3$ vectors (see Appendix~\ref{a.Rt} for more details) with $\mathcal{H}_{\xi}$ in the \emph{complementary} series, one can find a parameter $s:=s(\mathcal{H}_{\xi})\in (0,1)$ (related to the eigenvalue $-1/4<\lambda(\mathcal{H}_{\xi})<0$ of the \emph{Casimir operator} of $\mathcal{H}_{\xi}$ by the formula $\lambda(\mathcal{H}_{\xi})= (s^2-1)/4$) and a constant $C_s$ depending only on $s$, such that 
$$|C_t(v,w)|:=\left|\int (v\circ\phi_t) \cdot w\, d\mu\right|\leq C_s\cdot e^{(-1+s)t}\cdot\|v\|_{C^3}\cdot\|w\|_{C^3}\,.$$
In the above estimate  $\|u\|_{C^3}$ is the $C^3$ norm of a $C^3$ vector $u$ along the $SO(2,\mathbb{R})$ direction. In the notation of Appendix~\ref{a.Rt}, 
$$\|u\|_{C^3}:=\|u\|_{L^2(X,\mu)}+\|L_W u\|_{L^2(X,\mu)}+\|L_W^2 u\|_{L^2(X,\mu)}+\|L_W^3 u\|_{L^2(X,\mu)}\,.$$
In addition, the constant $C_s$ can be taken uniform on intervals of the form $s\in[1-s_0,s_0]$ with $1/2<s_0<1$.  By putting these informations together (and by the classical fact that $C^3$ vectors are dense), one obtains that $|C_t(v,w)|\to0$ as $t\to\infty$ (actually, it converges ``exponentially fast'' to zero in the sense explained above) for: 
\begin{itemize}
\item all vectors $v,w\in\mathcal{H}_{\xi}$ when $\mathcal{H}_{\xi}$ belongs to the principal or discrete series;
\item a dense subset (e.g., $C^3$ vectors) of vectors $v,w\in\mathcal{H}_{\xi}$ when $\mathcal{H}_{\xi}$ belongs to the complementary series.
\end{itemize}
As a consequence (by the integral decomposition $\mathcal{H}=\int\mathcal{H}_{\xi}d\lambda(\xi)$), we conclude that $|C_t(v,w)|\to0$ as $t\to\infty$ for a dense subset of vectors $v,w\in \mathcal{H}=L^2_0(X,\mu)$. Finally, an easy approximation argument shows that $|C_t(v,w)|\to0$ as $t\to\infty$ for all $v,w\in L^2_0(X,\mu)$. Hence, $(\phi_t)_{t\in\mathbb{R}}$ is $\mu$-mixing and the proof of Proposition~\ref{p.mixing} is complete. 

Once the Proposition~\ref{p.mixing} is proved, a natural question concerns the ``speed''/``rate'' of convergence of $C_t(f,g)$ to zero (as $t\to\infty$). In a certain sense, this question was already answered during the proof of Proposition~\ref{p.mixing}: using Ratner's results~\cite{Rt}, one can show that $C_t(f,g)$ converges exponentially fast to zero for all $f,g$ in a dense subset of $L_0^2(X,\mu)$(e.g., $f,g$ $C^3$ vectors) if the unitary $SL(2,\mathbb{R})$ representation $\mathcal{H}=L^2_0(X,\mu)$ has \emph{spectral gap}, i.e., there exists $s_0\in(0,1)$ such that, when writing $\mathcal{H}$ as an integral $\mathcal{H}=\int\mathcal{H}_{\xi} d\lambda(\xi)$ of unitary irreducible $SL(2,\mathbb{R})$ representations, no $\mathcal{H}_{\xi}$ in the complementary series has parameter $s=s(\mathcal{H}_{\xi})\in(s_0,1)$. Actually, it is possible to show that the spectral gap property is equivalent to the nonexistence of \emph{almost invariant vectors}: recall that a representation of a Lie group $G$ 
on a Hilbert space $\mathcal{H}$ has almost invariant vectors when, for all  
 $\varepsilon>0$ and for all compact subsets $K \subset G$, there exists an unit vector $v\in\mathcal{H}$ such that $\|gv-v\|<\varepsilon$ for all $g\in K$. It is also possible to prove that the spectral gap condition is not only sufficient
for exponential mixing for smooth vectors but also necessary. In fact, a \emph{reverse
Ratner estimate} (see \cite{AGY}, Appendix B) allows to derive the spectral gap property of an $SL(2,\mathbb R)$ unitary representation from the exponential mixing (with uniform rate) of the diagonal action for all smooth vectors.

In general, it is a hard task to prove the spectral gap property for a given unitary $SL(2,\mathbb{R})$ representation\footnote{This is essentially due to the fact that $SL(2,\mathbb{R})$ doesn't have the so-called Kazhdan's \emph{property T}.}. For the case of the unitary $SL(2,\mathbb{R})$ representation on $L^2_0(\mathcal{C},\mu_{\mathcal{C}})$ obtained from the $SL(2,\mathbb{R})$ action on a connected component $\mathcal{C}$ of a stratum of the moduli space of Abelian differentials \emph{equipped with the natural Masur--Veech measure} $\mu_{\mathcal{C}}$, A. Avila, S. Gou\"ezel and J.-C. Yoccoz proved the following theorem:

\begin{theorem}[A. Avila, S. Gou\"ezel, J.-C. Yoccoz] The Teichm\"uller flow $(g_t)_{t\in \mathbb R}$ on $\mathcal{C}$ is exponentially mixing with respect to $\mu_{\mathcal{C}}$ (in the sense that $C_t(f,g)\to 0$ exponentially as $t\to\infty$ for ``sufficiently smooth'' $f,g$), and the unitary $SL(2,\mathbb{R})$ representation $L^2_0(\mathcal{C},\mu_{\mathcal{C}})$ has a spectral gap.
\end{theorem}

In the proof of this result~\cite{AGY}, the authors proves \emph{firstly} that the Teichm\"uller geodesic flow (i.e., the action of the diagonal subgroup $A=\{a(t):t\in\mathbb{R}\}$ on the moduli space $\mathcal{Q}_g$ of Abelian differentials) is \emph{exponentially mixing} with respect to all Masur--Veech measures (indeed this is the main result of their paper), then  derive the spectral gap property \emph{from} the exponential mixing (and not the other way around!) by the reverse Ratner estimate. Here, the proof of the exponential mixing property with respect to Masur--Veech measures is obtained by delicate (mostly combinatorial) estimates on the so-called \emph{Rauzy--Veech induction}.

More recently, Avila and Gou\"ezel~\cite{AG} developed a more geometrical (and less combinatorial) approach to the exponential mixing of \emph{algebraic} $SL(2,\mathbb{R})$-invariant probability measures. 

Roughly speaking, an algebraic $SL(2,\mathbb{R})$-invariant measure $\mu$ on $\mathcal{C}$ is a measure supported on an \emph{affine suborbifold} $\textrm{supp}(\mu)$ of $\mathcal{C}$ (in the sense that $\textrm{supp}(\mu)$ corresponds, in local period coordinates, to affine subspaces in relative homology) such that $\mu$ is absolutely continuous (with respect to the Lebesgue measure on the affine subspaces corresponding to $\textrm{supp}(\mu)$ in period coordinate charts) and its density is locally constant in period coordinates. The class of algebraic $SL(2,\mathbb{R})$-invariant probability measures contains all ``known'' examples (e.g., Masur--Veech measures $\mu_{\mathcal{C}}$ and the probability measures supported on the $SL(2,\mathbb{R})$-orbits of Veech surfaces\footnote{Cf. Subsection~\ref{ss.Shimura-Teich} below for more details on the construction of these probability measures.} [in particular, square-tiled surfaces]). Actually, an important conjecture in Teichm\"uller dynamics claims that \emph{all} $SL(2,\mathbb{R})$-invariant probability measures are algebraic. If it is true, this conjecture would provide a non-homogenous counterpart to Ratner's theorems~\cite{Rt92} on unipotent actions in homogenous spaces. 

After the celebrated work of K.~Calta~\cite{C} and C.~McMullen~\cite{Mc}, there is a \emph{complete} classification of all $SL(2,\mathbb{R})$-invariant measures in genus $2$ (i.e., $\mathcal{C}=\mathcal{H}(2)$ or $\mathcal{H}(1,1)$). In particular, it follows that such measures are always algebraic (in genus $2$). Furthermore, it was recently announced by A.~Eskin and M.~Mirzakhani \cite{EsMi} that the full conjecture is true. 

In any case, the result obtained by Avila and Gou\"ezel~\cite{AG} is the following:

\begin{theorem}[A. Avila and S. Gou\"ezel] Let $\mu$ be an algebraic $SL(2,\mathbb{R})$-invariant probability measure, and let  $L^2_0(\mathcal{C},\mu)=\int \mathcal{H}_{\xi}d\lambda(\xi)$ be the decomposition of the unitary $SL(2,\mathbb{R})$ representation on $L^2_0(\mathcal{C},\mu)$ into an integral of irreducible components $\mathcal{H}_{\xi}$. Then, for any $\delta>0$, the components $\mathcal{H}_{\xi}$ of the complementary series with parameter $s(\mathcal{H}_{\xi})\in [\delta,1]$ appear for at most \emph{finitely many parameters} (i.e., $\{s\in [\delta,1]: s=s(\mathcal{H}_{\xi}) \textrm{ for some }\xi\}$ is finite) and with \emph{finite multiplicity} (i.e., for each $s\in[\delta,1]$, $\{\xi: s(\mathcal{H}_{\xi})=s\}$ is finite). In particular, the Teichm\"uller geodesic flow $(g_t)_{t\in \mathbb R}$ is exponentially mixing with respect to the probability measure $\mu$ on $\mathcal C$.
\end{theorem}

This completes the discussion of this section on the main ergodic properties of the Teichm\"uller flow with respect to Masur--Veech measures. In the next section, we will discuss finer ergodic properties of the Teichm\"uller flow and of the KZ cocycle with respect to Masur--Veech measures and their applications to the (quantitative) unique ergodicity and the weak mixing properties of i.e.t.'s and translation flows.


\section{Ergodic Theory of the KZ cocycle with respect to Masur--Veech measures}\label{s.KZ-MV}

Again, let $\mathcal{C}$ be a connected component of a stratum $\mathcal{H}^{(1)}(\kappa)$ of Abelian differentials with unit area, and denote by $\mu_{\mathcal{C}}$ the corresponding Masur--Veech probability measure. 

\subsection{Kontsevich--Zorich conjecture (after G.~Forni, and A.~Avila $\&$ M.~Viana)} In the 
mid 1990's, A.~Zorich and M.~Kontsevich~\cite{Z94},~\cite{K} performed several numerical experiments leading them to conjecture that the Kontsevich--Zorich cocycle has \emph{simple} Lyapunov spectrum with respect to any Masur--Veech measure $\mu_{\mathcal{C}}$,  in the sense that the multiplicity of each Lyapunov exponent $\lambda_i^{\mu_{\mathcal{C}}}$, for $i=1,\dots,2g$, is $1$. In other terms the Laypunov
spectrum has the following form:
$$1=\lambda_1^{\mu_{\mathcal{C}}}>\lambda_2^{\mu_{\mathcal{C}}}>\dots>\lambda_g^{\mu_{\mathcal{C}}}>\lambda_{g+1}^{\mu_\mathcal{C}}>\dots>\lambda_{2g}^{\mu_{\mathcal{C}}}=-1$$
As we discussed in the previous section, the Kontsevich--Zorich cocycle $G_t^{KZ}$ is \emph{symplectic}, so that its Lyapunov exponents (with respect to \emph{any} invariant ergodic probability 
measure $\mu$ on $\mathcal H^{(1)}_g$) are symmetric with respect to the origin: $\lambda_{2g-i}^{\mu}=-\lambda_{i+1}^{\mu}$. Also, the top Lyapunov exponent $1=\lambda_1^{\mu}$ is always simple (i.e., $\lambda_1^{\mu}>\lambda_2^{\mu}$). Therefore, the Kontsevich--Zorich conjecture is equivalent to
$$\lambda_2^{\mu_{\mathcal{C}}}>\dots>\lambda_g^{\mu_{\mathcal{C}}}>0$$

In 2002, G.~Forni~\cite{F02} was able to show that $\lambda_g^{\mu_{\mathcal{C}}}>0$ via \emph{second variational formulas} for the Hodge norm and certain formulas for the sum of the Lyapunov exponents of the KZ cocycle (inspired by M. Kontsevich's work). In Subsection~\ref{ss.F02} below, we'll illustrate some of G.~Forni's techniques by sketching a proof of the positivity of the second Lyapunov exponent $\lambda_2^{\mu_{\mathcal{C}}}$ of the KZ cocycle with respect to Masur--Veech measure $\mu_{\mathcal{C}}$. While the fact $\lambda_2^{\mu_{\mathcal{C}}}>0$ is certainly a weaker statement than Forni's theorem $\lambda_g^{\mu_{\mathcal{C}}}>0$, it turns out that it is sufficient for some interesting applications to interval exchange transformations and translation flows. Indeed, using a technical machinery of \emph{probabilistic parameter exclusion} strongly based on the fact that 
$\lambda_2^{\mu_{\mathcal{C}}}>0$, A.~Avila and G.~Forni~\cite{AF} were able to show that almost every i.e.t. (not corresponding to ``rotations'') and almost every translation flow (on genus $g\geq 2$ translation surfaces) are \emph{weakly mixing}.  Here, we say that an i.e.t. corresponds to a rotation if its combinatorial data $\pi:\{1,\dots,d\}\to\{1,\dots,d\}$ has the form $\pi(i)=i+1$ (mod $d$). In this case, one can see that the corresponding i.e.t. can be conjugated to a rotation of the circle, and hence it is never weak-mixing. Observe that, in general, weak-mixing property is the strongest  ergodic property we can expect: indeed, as it was shown by A.~Katok~\cite{Katok}, interval exchange transformations and suspension flows over i.e..t's with a roof function of bounded variation (e.g., translation flows) are never \emph{mixing}. We will come back to this point later in this section.

In 2007, A. Avila and M. Viana~\cite{AV} proved the full Kontsevich--Zorich conjecture by studying a discrete-time analog of Kontsevich--Zorich cocycle over the Rauzy--Veech induction. In a few words, Avila and Viana showed that the symplectic monoid associated to Rauzy--Veech induction is \emph{pinching} (``it contains matrices with simple spectrum'') and \emph{twisting} (``any subspace can be put into generic position by using some matrix of the monoid''), and they used the pinching and twisting properties to ensure simplicity of Lyapunov spectra. In a certain sense, these conditions (pinching and twisting) are analogues (for deterministic chaotic dynamical systems) of the \emph{strong irreducibility} and \emph{proximality} conditions (sometimes derived from a stronger \emph{Zariski density property}) used by Y.~Guivarch and A.~Raugi~\cite{GR}, and I.~Goldsheid and G.~Margulis~\cite{GM} to derive simplicity of Lyapunov exponents for random products of matrices. 

\begin{remark} More recently, G.~Forni extended some techniques of his article \cite{F02} to prove in \cite{F11} a geometric criterion for the non-uniform hyperbolicity of the KZ cocycle (i.e., $\lambda_g^{\mu}>0$) of ``general'' $SL(2,\mathbb{R})$-invariant ergodic probability measures $\mu$ (see Remark~\ref{r.F11} below). As a matter of fact, this general recent criterion strictly includes Masur--Veech measures, but it doesn't allow to derive simplicity of the Lyapunov spectrum in general (see the appendix by C.~Matheus to \cite{F11} for more details). Also, it was recently shown by V.~Delecroix and C.~Matheus~\cite{DeMa} that there is no converse\footnote{I.e., the conditions of Forni's criterion are sufficient but not necessary for non-uniform hyperbolicity} to G.~Forni's criterion. Here, the arguments of  Delecroix and Matheus~\cite{DeMa} are based on a recent criterion for the simplicity of the Lyapunov exponents of the KZ cocycle with respect to $SL(2,\mathbb{R})$-invariant ergodic probability measures supported on the $SL(2,\mathbb{R})$-orbits of square-tiled surfaces due to M.~M\"oller, J.-C.~Yoccoz and C.~Matheus \cite{MMY}. For more comments on this, see Section~\ref{s.finalremarks} below.
\end{remark}

As the reader can imagine, the Kontsevich--Zorich conjecture has applications to the study of \emph{deviations} of ergodic averages along trajectories of  translation flows and interval exchanges transformations. Actually, this was the initial motivation for the introduction of the Kontsevich--Zorich cocycle by A.~Zorich and M.~Kontsevich. 

\subsection{Non-vanishing of the second KZ exponent  with respect to Masur--Veech measures}\label{ss.F02} 
We dedicate this subsection to give a sketch of proof of the following result:
\begin{theorem}\label{t.2nd-exp-MV} Let $\mathcal{C}$ be a connected component of some stratum of 
$\mathcal{H}^{(1)}_g$ and denote by $\mu_{\mathcal{C}}$ the corresponding Masur--Veech measure. Then, $\lambda_2^{\mu_{\mathcal{C}}}>0$.
\end{theorem}

As we already mentioned, this result is part of one of the main results of \cite{F02} showing that 
$\lambda_g^{\mu_{\mathcal{C}}}>0$. However, we'll not discuss the proof of the more general result 
$\lambda_g^{\mu_{\mathcal{C}}}>0$ because 
\begin{itemize}
\item one already finds several of the ideas used to prove that $\lambda_g^{\mu_{\mathcal{C}}}>0$ in the  proof of $\lambda_2^{\mu_{\mathcal{C}}}>0$, and
\item the proof $\lambda_2^{\mu_{\mathcal{C}}}>0$ requires a discussion of the
position the unstable subspaces of the KZ cocycle along certain degenerations of Riemann
surfaces and lengthy additional calculations.
\end{itemize}

In any event, we start the sketch of proof of Theorem~\ref{t.2nd-exp-MV} by recalling (from Subsection~\ref{ss.1stvariation}, Theorem~\ref{t.1stvariation}) that the form $B_{\omega}(\alpha,\beta):=\frac{i}{2}\int\frac{\alpha\beta}{\omega}\overline{\omega}$, $\alpha,\beta\in H^{1,0}(M)$, is relevant in the study of \emph{first variation} of the Hodge norm along a Teichm\"uller deformation in view of the formula:
$$\frac{d}{dt}\|c\|_{\omega_t} =-2\Re B_{\omega_t}(\alpha_t,\alpha_t) = 
-2\Re B^{\mathbb R}_{\omega_t}(c,c)  $$
where $(M_t, \omega_t) := g_t(M, \omega)$  denotes as above the Teichm\"uller orbit of the translation
surface $(M,\omega)$ and $\alpha_t\in H^{1,0}(M_t)$ denotes  the holomorphic representative of $c\in H^1(M_t, \mathbb R)$ on $M_t$, i. e. the unique holomorphic $1$-form on $M_t$ such that 
$c=[\Re\alpha_t] \in H^1(M_t, \mathbb R)$, for all $t\in \mathbb R$.

Also, recall that $H^1(M,\mathbb{R})=H^1_{st}(M,\mathbb{R})\oplus H^1_{(0)}(M,\mathbb{R})$, where $H^1_{st}(M,\mathbb{R})=\mathbb{R}[\Re\omega]\oplus\mathbb{R}[\Im\omega]$ and $H^1_{(0)}(M,\mathbb{R}):=\{c\in H^1(M,\mathbb{R}):c\wedge\omega=0\}$. Moreover, $H^1_{st}(M,\mathbb{R})$ is KZ cocycle invariant and it carries the  Lyapunov exponents $\pm\lambda_1^{\mu}=\pm 1$.  Since the KZ cocycle preserves the symplectic intersection form, the subspace $H^1_{(0)}(M,
\mathbb{R})$ is also KZ cocycle invariant, hence the Lyapunov exponents 
$\lambda_i^{\mu}$, for $2\leq i\leq g$, come from the restriction of the KZ cocycle to the subspace 
$H^1_{(0)}(M,\mathbb{R})$. 

Denoting by $B^{\mathbb{R}}_{\omega}(c_1,c_2)=B_{\omega}(\alpha_1,\alpha_2)$ where $c_i=[\Re\alpha_i]$, for $i=1,2$, the \emph{complex-valued} bilinear form\footnote{Note that it depends real-analytically (in particular continuously) on $\omega\in \mathcal H_g$.} on 
$H^1(M,\mathbb{R})$ induced by $B_{\omega}$ (on $H^{1,0}(M)$) via \emph{Hodge representation theorem} (cf. Subsection~\ref{ss.Hodgenorm}), we obtain the following nice \emph{immediate} consequence of this discussion:

\begin{corollary}\label{c.rankB0exp0} Let $\mu$ be any $g_t$-invariant ergodic probability measure 
on  $\mathcal{H}^{(1)}_g$ and assume that $\textrm{rank}(B^{\mathbb{R}}_{\omega}|_{H^1_{(0)}(M,\mathbb{R})})=0$ for all $\omega\in\textrm{supp}(\mu)$. Then, $\lambda_2^{\mu}=\dots=
\lambda_g^{\mu}=0$.
\end{corollary} 

Geometrically, $B_{\omega}$ is \emph{essentially} the \emph{second fundamental form} (or \emph{Kodaira-Spencer map}) of the holomorphic subbundle $H^{1,0}$ of the complex Hodge bundle $H^1_{\mathbb{C}}$ equipped with the Gauss-Manin connection. Roughly speaking, recall that the second fundamental form $II_{\omega}:H^{1,0}\to H^{0,1}$ is $II_{\omega}(c)=\frac{d}{dt}c_t^{0,1}(0)$, where $c_t^{0,1}$ is the $H^{0,1}$-component of $G_t^{KZ}(c)$. See the figure below. 
\begin{center}

\begingroup
  \makeatletter
  \providecommand\color[2][]{%
    \errmessage{(Inkscape) Color is used for the text in Inkscape, but the package 'color.sty' is not loaded}
    \renewcommand\color[2][]{}%
  }
  \providecommand\transparent[1]{%
    \errmessage{(Inkscape) Transparency is used (non-zero) for the text in Inkscape, but the package 'transparent.sty' is not loaded}
    \renewcommand\transparent[1]{}%
  }
  \providecommand\rotatebox[2]{#2}
  \ifx\svgwidth\undefined
    \setlength{\unitlength}{300pt}
  \else
    \setlength{\unitlength}{\svgwidth}
  \fi
  \global\let\svgwidth\undefined
  \makeatother
  \begin{picture}(1,0.36250366)%
    \put(0,0){\includegraphics[width=\unitlength]{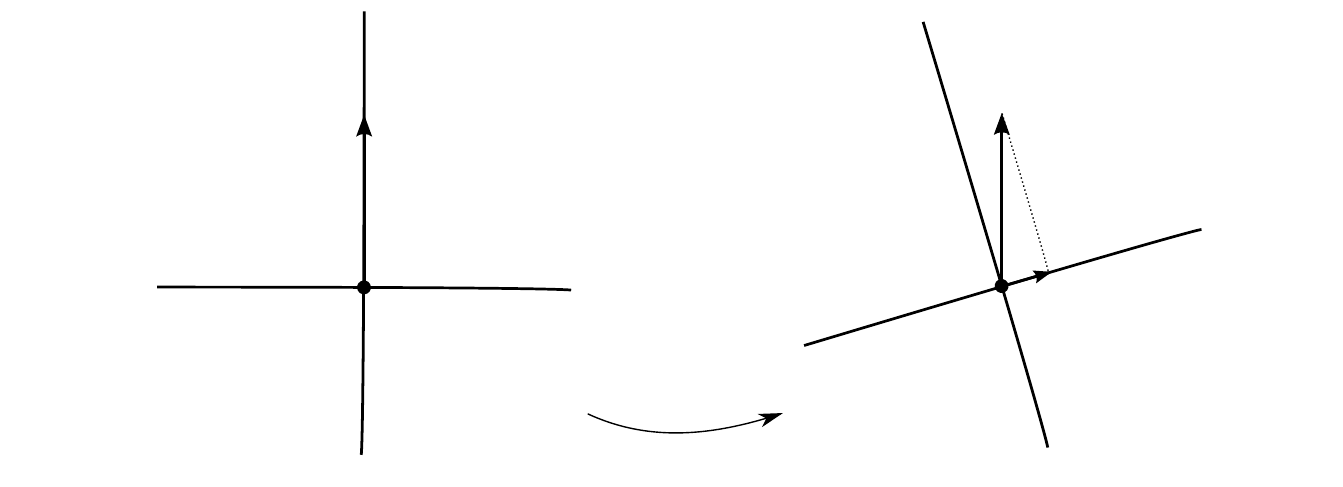}}%
    \put(-0.00158382,0.17350545){\Large\color[rgb]{0,0,0}\makebox(0,0)[lb]{\smash{$H^{1}_\mathbb{C}$}}}%
    \put(0.27740471,0.32661397){\color[rgb]{0,0,0}\makebox(0,0)[lb]{\smash{$H^{1,0}$}}}%
    \put(0.40290862,0.16234708){\color[rgb]{0,0,0}\makebox(0,0)[lb]{\smash{$H^{0,1}$}}}%
    \put(0.22034087,0.105){\color[rgb]{0,0,0}\makebox(0,0)[lb]{\smash{$\omega_0$}}}%
    \put(0.6993814,0.32661397){\color[rgb]{0,0,0}\makebox(0,0)[lb]{\smash{$H^{1,0}$}}}%
    \put(0.87564171,0.19999954){\color[rgb]{0,0,0}\makebox(0,0)[lb]{\smash{$H^{0,1}$}}}%
    \put(0.28320954,0.24144877){\color[rgb]{0,0,0}\makebox(0,0)[lb]{\smash{$c$}}}%
    \put(0.76362168,0.2666974){\color[rgb]{0,0,0}\makebox(0,0)[lb]{\smash{$c$}}}%
    \put(0.70962972,0.105){\color[rgb]{0,0,0}\makebox(0,0)[lb]{\smash{$\omega_t$}}}%
    \put(0.48731024,0.00632114){\color[rgb]{0,0,0}\makebox(0,0)[lb]{\smash{$g_t$}}}%
    \put(0.77082324,0.11256419){\color[rgb]{0,0,0}\makebox(0,0)[lb]{\smash{$c_{t}^{0,1}$}}}%
  \end{picture}%
\endgroup

\end{center}
In this language, it is possible to show that $B_{\omega}(\alpha,\beta)=-(\overline{II_{\omega}(\alpha)},\beta)$ where $(.,.)$ is the Hodge form. See, e.g., \cite{FMZ2} for a detailed presentation of this differential-geometrical interpretation\footnote{A word of \emph{caution}: the second fundamental form $A_{\omega}(c)$ considered in \cite{FMZ2} differs from $II_{\omega(c)}$ by a sign, i.e., $II_{\omega}(c)=-A_{\omega}(c)$!} of $B$. 

Next, by taking  a Hodge-orthonormal basis $\{\omega_1,\dots,\omega_g\}$ of $H^{1,0}$, we have a matrix $B=(B_{jk})_{1\leq j,k\leq g}$, $B_{jk}=\frac{i}{2}\int\frac{\omega_j\omega_k}{\omega}\overline{\omega}$, associated to $B_{\omega}$. Define $H=H_{\omega}:=B\cdot B^*$. The eigenvalues $\Lambda_1(\omega)\geq\dots\geq \Lambda_g(\omega) (\geq 0)$ of $H$ have the form $|\lambda|^2$ where $\lambda$ is an eigenvalue of $B$, i.e., $H$ induces a positive semi-definite form on $H^{1,0}$. As it turns out, $H$ is essentially the \emph{curvature form} of the holomorphic subbundle $H^{1,0}$ of the complex Hodge bundle $H^1_{\mathbb{C}}$ equipped with the Gauss-Manin connection (see \cite{FMZ2} for more details), i.e., the matrix $H$ also a \emph{differential-geometrical} interpretation (similarly to $B$). In particular, this geometrical interpretation \emph{suggests} that $H$ \emph{should} naturally enter into \emph{second variation formulas} for the Hodge norm\footnote{Of course, this should be compared with the fact that $B$ naturally enters into \emph{first variation formulas} for the Hodge norm.} and, \emph{a fortiori}, the eigenvalues of $H$ \emph{should} be relevant in the study of Lyapunov exponents. In fact, as it was proposed by M.~Kontsevich \cite{K} and proved by G.~Forni \cite{F02}, one can relate the eigenvalue of $H$ to Lyapunov exponents of the KZ cocycle via the following formula:

\begin{theorem}[M. Kontsevich, G. Forni]\label{t.F02} Let $\mu$ be an $SL(2,\mathbb{R})$-invariant 
$g_t$-ergodic probability measure on $\mathcal{H}^{(1)}_g$. Then, one has the following formula for the sum of the non-negative Lyapunov exponents of the KZ cocycle with respect to 
$\mu$:
$$\lambda_1^{\mu}+\dots+\lambda_g^{\mu}=\int_{\mathcal{H}^{(1)}_g}(\Lambda_1(\omega)+\dots+\Lambda_g(\omega))\,d\mu(\omega)$$
\end{theorem} 

\begin{remark}\label{r.F02short} Since $B_{\omega}(\omega,\omega):=1$, one can use the argument (Cauchy-Schwarz inequality) of the proof of Corollary~\ref{c.1stvariation} to see that $\Lambda_1(\omega)\equiv 1$ for all $\omega\in \mathcal H^{(1)}_g$. In particular, since $\lambda_1^{\mu}=1$, one can rewrite the formula above as 
$$\lambda_2^{\mu}+\dots+\lambda_g^{\mu}=\int_{\mathcal{H}^{(1)}_g} (\Lambda_2(\omega)+\dots+\Lambda_g(\omega))\,d\mu(\omega)$$
\end{remark}

\begin{remark} Note that there is an \emph{important} difference between the hypotheses of 
Theorem~\ref{t.1stvariation} and those of Theorem~\ref{t.F02} is: in the former $\mu$ is \emph{any} 
$g_t$-invariant probability measure, while in the latter $\mu$ must be $SL(2,\mathbb{R})$-invariant!
\end{remark} 

Before giving a sketch of proof of Theorem~\ref{t.F02}, we observe that from it (and Remark~\ref{r.F02short}) one can immediately deduced the following ``converse'' to Corollary~\ref{c.rankB0exp0}:

\begin{corollary}\label{c.exp0rankB0} Let $\mu$ be a $SL(2,\mathbb{R})$-invariant $g_t$-ergodic probability measure on $\mathcal{H}^{(1)}_g$. Assume that $\lambda_2^{\mu}=\dots=\lambda_g^{\mu}=0$. Then, $\textrm{rank}(B_{\omega}^{\mathbb{R}}|_{H^1_{(0)}(M,\mathbb{R})})=0$ for all $\omega\in\textrm{supp}(\mu)$.
\end{corollary} 

Evidently, this corollary shows how one can prove Theorem~\ref{t.2nd-exp-MV}: since Masur--Veech measures $\mu_{\mathcal{C}}$ are \emph{fully supported}, it suffices to check that $\textrm{rank}(B_{\omega}^{\mathbb{R}}|_{H^1_{(0)}(M,\mathbb{R})})>0$ for \emph{some} $\omega\in\mathcal{C}=\textrm{supp}(\mu_{\mathcal{C}})$. In other words, by \emph{assuming} Theorem~\ref{t.F02}, we just saw that:
\begin{corollary}\label{c.rankB-MV}If $\textrm{rank}(B_{\omega}^{\mathbb{R}}|_{H^1_{(0)}(M,\mathbb{R})})>0$ for some $\omega\in\mathcal{C}$, then $\lambda_2^{\mu_{\mathcal{C}}}>0$.
\end{corollary}
Now, before trying to apply this corollary, let's give an outline of the proof of Theorem~\ref{t.F02}.

\begin{proof}[Sketch of proof of Theorem~\ref{t.F02}] Given $1\leq k\leq g$, let 
$$\Phi_k(\omega, I_k):=2\sum\limits_{i=1}^k H_{\omega}(c_i, c_i) - \sum\limits_{j,m=1}^k |B_{\omega}^{\mathbb{R}}(c_j,c_m)|^2$$
where $I_k$ is a $k$-dimensional \emph{isotropic} subspace of the real Hodge bundle $H^1_{\mathbb{R}}$ and $\{c_1,\dots,c_k\}$ is \emph{any}\footnote{Of course, it is implicit here that the expression $2\sum\limits_{i=1}^k H_{\omega}(c_i, c_i) - \sum\limits_{j,m=1}^k |B_{\omega}^{\mathbb{R}}(c_j,c_m)|^2$ doesn't depend on the choice of Hodge-orthonormal basis $\{c_1,\dots,c_k\}$ but only on the isotropic subspace $I_k\subset H^1(M,\mathbb{R})$. We leave this verification as an exercise to the reader.} Hodge-orthonormal basis of $I_k$. 

In the sequel, we will use the following three lemmas (see \cite{F02} or \cite{FMZ2} for proofs and more details).

\begin{lemma}[Lemma 5.2' of \cite{F02}]\label{l.F02a} Let $\{c_1,\dots,c_k,c_{k+1},\dots,c_g\}$ be \emph{any} Hodge-orthonormal completion of $\{c_1,\dots,c_k\}$ into basis of a Lagrangian subspace of $H^1(M,\mathbb{R})$. Then, 
$$\Phi_k(\omega, I_k)=\sum_{i=1}^g\Lambda_i(\omega)-\sum\limits_{j,m=k+1}^g|B^{\mathbb{R}}_{\omega}(c_j,c_m)|^2$$
\end{lemma}
\begin{remark}[M. Kontsevich's fundamental remark]\label{r.Kontsevich} In the extremal case $k=g$, the right-hand side of the previous equality \emph{doesn't} depend on the Lagrangian subspace $I_g$:
$$\Phi_g(\omega, I_g)=\sum_{i=1}^g\Lambda_i(\omega) = \textrm{tr}(H_{\omega})$$
This fundamental observation  lies at the heart of the main formula of Theorem~\ref{t.F02}.
\end{remark}

It is not hard to see that the notion of Hodge norm $\|.\|_{\omega}$ on vectors $c\in H^1(M,\mathbb{R})$ can be extended to any polyvector $c_1\wedge\dots\wedge c_k$ coming from a (Hodge-orthonormal) system  $\{c_1,\dots,c_k\}$ of an isotropic subspace $I_k$. In fact, the Hodge norm of such a polyvector,
denoted by $\|c_1\wedge\dots\wedge c_k\|_{\omega}$ by a slight abuse of notation,
is defined as  the square root of the $k$-volume (with respect to the Hodge metric) of the $k$-dimensional interval spanned by the system  $\{c_1,\dots,c_k\}$ . 

Note that the Hodge norm $\|.\|_{\omega}$ depends only on the complex structure, so that $\|.\|_{\omega}=\|.\|_{\omega'}$ whenever $\omega'=\textrm{constant}\cdot\omega$. In particular, it makes sense to consider the Hodge norm as a function over the \emph{Teichm\"uller disk} $SO(2,\mathbb{R})\backslash SL(2,\mathbb{R})\cdot\omega$. For subsequent use, we denote by $\Delta_{hyp}$ the hyperbolic (leafwise) Laplacian on $SO(2,\mathbb{R})\backslash SL(2,\mathbb{R})\cdot\omega$ (here, we're taking advantage of the fact that $SO(2,\mathbb{R})\backslash SL(2,\mathbb{R})$ is isomorphic to Poincar\'e's hyperbolic disk $D$).

\begin{lemma}[Lemma 5.2 of \cite{F02}]\label{l.F02b} One has $\Delta_{hyp}\log\|c_1\wedge\dots\wedge c_k\|_{\omega} = 2\Phi_k(\omega, I_k)$.
\end{lemma} 

Finally, in order to connect the previous two lemmas with Oseledets theorem (and Lyapunov exponents), one needs the following fact about hyperbolic geometry:

\begin{lemma}[Lemma 3.1 of \cite{F02}]\label{l.hyperb-geom}Let $L:D\to\mathbb{R}$ be a smooth function. Then, 
$$\frac{1}{2\pi}\frac{\partial}{\partial t}\int_0^{2\pi} L(t,\theta)\,d\theta = \frac{1}{2}\tanh(t)\frac{1}{\textrm{area}(D_t)}\int_{D_t}\Lambda\, d\textrm{area}_P$$
where $\Lambda:=\Delta_{hyp} L$, $(t,\theta)$ are polar coordinates on Poincar\'e's disk, $D_t$ is the disk of radius $t$ centered at the origin $0\in D$ and $\textrm{area}_P$ is Poincar\'e's area form on $D$.
\end{lemma}

Next, the idea to derive Theorem~\ref{t.F02} from the previous three lemmas is the following. Denote by $R_\theta=\left(\begin{array}{cc}\cos\theta & -\sin\theta \\ \sin\theta & \cos\theta\end{array}\right)$, and, given $\omega\in\mathcal{H}^{(1)}_g$, for $SO(2,\mathbb{R})\backslash SL(2,\mathbb{R})\cdot\omega\ni h=g_tR_{\theta}\omega=(t,\theta)$, let $L(h):=\|c_1\wedge\dots\wedge c_k\|_h$. In plain terms, $L$ is measuring how the (Hodge norm) size of the polyvector $c_1\wedge\dots\wedge c_k$ changes along the Teichm\"uller disk of $\omega$. In particular, as we're going to see in a moment, it is not surprising that $L$ has ``something to do'' with Lyapunov exponents.

By Lemma~\ref{l.hyperb-geom}, one has 
$$\frac{1}{2\pi}\frac{\partial}{\partial t}\int_0^{2\pi} L(t,\theta)\,d\theta = \frac{1}{2}\tanh(t)\frac{1}{\textrm{area}(D_t)}\int_{D_t}\Delta_{hyp}L(t,\theta)\, d\textrm{area}_P$$
Then, by integrating with respect to the $t$-variable in the interval $[0,T]$ and by using Lemma~\ref{l.F02b} for the computation of $\Delta_{hyp}L$, one deduces 
$$\frac{1}{2\pi}\frac{1}{T}\int_0^{2\pi} (L(T,\theta)-L(0,\theta))\,d\theta = \frac{1}{T}\int_0^T\frac{\tanh(t)}{\textrm{area}(D_t)}\int_{D_t}\Phi_k(t,\theta)\, d\textrm{area}_P\,dt$$
At this point, by taking an average with respect to $\mu$ on $\mathcal H^{(1)}_g$ and using the $SL(2,\mathbb{R})$-invariance of $\mu$ to get rid of the integration with respect to $\theta$, we deduce that 
$$\frac{1}{T}\int_{\mathcal{H}^{(1)}_g} (L(g_T(\omega))-L(\omega))\,d\mu(\omega) = \frac{1}{T}\int_{\mathcal{H}^{(1)}_g}\int_0^T\frac{\tanh(t)}{\textrm{area}(D_t)}\int_{D_t}\Phi_k(g_tR_{\theta}\omega, I_k)\, d\textrm{area}_P\,dt\, d\mu(\omega)$$  
Now, we observe that: 
\begin{itemize}
\item by the Oseledets theorem, for a ``generic'' isotropic subspace $I_k$ and $\mu$-almost every $\omega\in \mathcal H^{(1)}_g$, one has that  $\frac{1}{T}L(g_T(\omega))$ converges\footnote{Recall that, by definition, the function $t\mapsto L(g_t(\omega))$ is measuring the growth (in Hodge norm) of the polyvector $c_1\wedge\dots\wedge c_k$ along the Teichm\"uller orbit $g_t(\omega)$.} to $\lambda_1^{\mu}+\dots+\lambda_k^{\mu}$ as $T\to \infty$, and
\item by Remark~\ref{r.Kontsevich} $\Phi_g(\omega, I_g)=\Phi_g(\omega)=\Lambda_1(\omega)+\dots+\Lambda_g(\omega)$ is \emph{independent} on $I_g$. 
\end{itemize}
So, for $k=g$, this discussion\footnote{Combined with an application of Lebesgue dominated convergence theorem and the fact that $\tanh(t)/\textrm{area}(D_t)\to 1$ as $t\to\infty$. See \cite{F02} and \cite{FMZ2} for more details.} allows to show that 
$$\lambda_1^{\mu}+\dots+\lambda_g^{\mu}=\int_{\mathcal{H}^{(1)}}(\Lambda_1(\omega)+\dots+\Lambda_g(\omega))\,d\mu(\omega)$$
This completes the sketch of proof of Theorem~\ref{t.F02}. 
\end{proof}

\begin{remark} Essentially the same argument above allows to derive formulas for \emph{partial sums} of Lyapunov exponents. More precisely, given  any 
$SL(2,\mathbb{R})$-invariant $g_t$-ergodic probability measure $\mu$ on $\mathcal H^{(1)}_g$ with $\lambda_k^{\mu}>\lambda_{k+1}^{\mu}$ (for some $1\leq k\leq g-1$), one has
$$\lambda_1^{\mu}+\dots+\lambda_k^{\mu}=\int_{\mathcal{C}} \Phi_k(\omega,E^+_k(\omega))\,d\mu(\omega)$$
where $E^+_k(\omega)$ is the Oseledets subspace associated to the $k$ top Lyapunov exponents. 

In general, this formula is harder to use than Theorem~\ref{t.F02} because the right-hand side of the former implicitly assumes some \emph{a priori} control of $E^+_k(\omega)$ while, by M. Kontsevich
fundamental remark, the right-hand side of the latter is independent of the Lagrangian subspace.  
\end{remark} 

Having obtained Theorem~\ref{t.F02}, we're ready to use Corollary~\ref{c.exp0rankB0} to \emph{reduce} the proof of Theorem~\ref{t.2nd-exp-MV} to the following theorem:

\begin{theorem}\label{t.bdry-MV} In any connected component $\mathcal{C}$ of a stratum of $\mathcal{H}^{(1)}_g$ one can find some Abelian differential $\omega\in\mathcal{C}$ with $\textrm{rank}(B_{\omega}^{\mathbb{R}}|_{H^1_{(0)}(M,\mathbb{R})})=2g-2$.
\end{theorem} 

Roughly speaking, the basic idea (somehow \emph{recurrent} in Teichm\"uller dynamics) to prove this result is to look for $\omega$ \emph{near} the \emph{boundary} of $\mathcal{H}_g$ after passing to an appropriate \emph{compactification}. More precisely, one shows that, by considering the so-called \emph{Deligne-Mumford compactification} $\overline{\mathcal{H}_g}=\mathcal{H}_g\cap\partial{\mathcal H}_g$, there exists an open set $U\subset\mathcal{C}$ near \emph{some} boundary point $\omega_{\infty}\in\partial{\mathcal H}_g$ such that $\textrm{rank}(B_{\omega}^{\mathbb{R}}|_{H^1_{(0)}(M,\mathbb{R})})=2g-2$ for any $\omega\in U$ ``simply'' because the ``same'' is true for $\omega_{\infty}$.  

A complete formalization of this idea is out of the scope of these notes as it would lead us to a serious discussion of Deligne-Mumford compactification, some variational formulas of J. Fay and A. Yamada, etc. Instead, we offer below a very rough sketch of proof of Theorem~\ref{t.bdry-MV} based on some ``intuitive'' properties of Deligne-Mumford compactification (while providing adequate references for the omitted details). 

The first step towards finding the boundary point $\omega_{\infty}$ is to start with the notion of Abelian differentials with \emph{periodic Lagrangian horizontal foliation}: 

\begin{definition} Let $\omega$ be an Abelian differential on a Riemann surface $M$. We say that the horizontal foliation $\mathcal{F}_{hor}(\omega):=\{\Im\omega=\textrm{constant}\}$ is \emph{periodic} whenever all regular leaves of $\mathcal{F}_{hor}(\omega)$ are closed, i.e., the translation surface 
$(M, \omega)$ can be completely decomposed as union of a finite family  $\{C_i\}$ of flat cylinders with
waist curves given by a family of closed regular geodesics $\{\gamma_i\}$ in the horizontal direction.  
The \emph{homological dimension} of an Abelian differential $\omega$ on $M$ with periodic horizontal foliation is the dimension of the (isotropic) subspace of $H_1(M,\mathbb{R})$ generated by the homology classes of the waist curves $\{\gamma_i\}$ of the horizontal maximal cylinders $\{C_i\}$ 
in the above decomposition.
We say that $\omega$ has \emph{periodic Lagrangian horizontal foliation} whenever it has a periodic horizontal foliation and its homological dimension is maximal (i.e., $g$).
\end{definition}

In general, it is not hard to find Abelian differentials with periodic horizontal foliation: for instance, any \emph{square-tiled surface} (see Example~\ref{ex.L-origami}) verifies this property and the class of square-tiled surfaces\footnote{As square-tiled surfaces $(M,\omega)$ are characterized by the rationality of their periods (i.e., $\int_{\gamma}\omega\in\mathbb{Q}\oplus i\mathbb{Q}$ for any $\gamma\in H_1(M,\Sigma,\mathbb{Z})$). See \cite{GJ} for more details.} is \emph{dense} in any connected
component $\mathcal{C}$ of any stratum $\mathcal H^{(1)}(\kappa)$. 

\smallskip
Next, we claim that: 
\begin{lemma}\label{l.4-4F02} Any connected
component $\mathcal{C}$ of any stratum $\mathcal H^{(1)}(\kappa)$ contains Abelian differentials with periodic \emph{Lagrangian} horizontal foliation.
\end{lemma} 
\begin{proof} Of course, the lemma follows once we can show that given $\omega\in\mathcal{C}$ with homological dimension $k<g$, one can produce an Abelian differential $\widetilde{\omega}$ with homological dimension $k+1$. In this direction, given such an $\omega$, we can select a closed curve $\gamma$ disjoint from (i.e., zero algebraic intersection with) the waist curves $\gamma_i$ of horizontal maximal cylinders $C_i$ of $(M,\omega)$ and $\gamma\neq 0$ in $H_1(M,\mathbb{R})$.

Then, let's denote by $[df]\in H^1(M,\mathbb{Z})$ the Poincar\'e dual of $\gamma$ given by taking a small tubular neighborhoods $V\subset U$ of $\gamma$ and taking a \emph{smooth} function $f$ on $M-\gamma$ such that 
$$f(p)=\left\{\begin{array}{cc}1 & \textrm{for } x\in V^-\,, \\ 0 & \textrm{for } x\in M-U^-\,,\end{array}\right.$$ 
where $U^{\pm}$ (resp. $V^{\pm}$) is the connected component of $U-\gamma$ (resp. $V-\gamma$) to the right/left of $\gamma$ with respect to its orientation of $\gamma$ (see the figure below)
\begin{center}

\begingroup
  \makeatletter
  \providecommand\color[2][]{%
    \errmessage{(Inkscape) Color is used for the text in Inkscape, but the package 'color.sty' is not loaded}
    \renewcommand\color[2][]{}%
  }
  \providecommand\transparent[1]{%
    \errmessage{(Inkscape) Transparency is used (non-zero) for the text in Inkscape, but the package 'transparent.sty' is not loaded}
    \renewcommand\transparent[1]{}%
  }
  \providecommand\rotatebox[2]{#2}
  \ifx\svgwidth\undefined
    \setlength{\unitlength}{250pt}
  \else
    \setlength{\unitlength}{\svgwidth}
  \fi
  \global\let\svgwidth\undefined
  \makeatother
  \begin{picture}(1,0.75505426)%
    \put(0,0){\includegraphics[width=\unitlength]{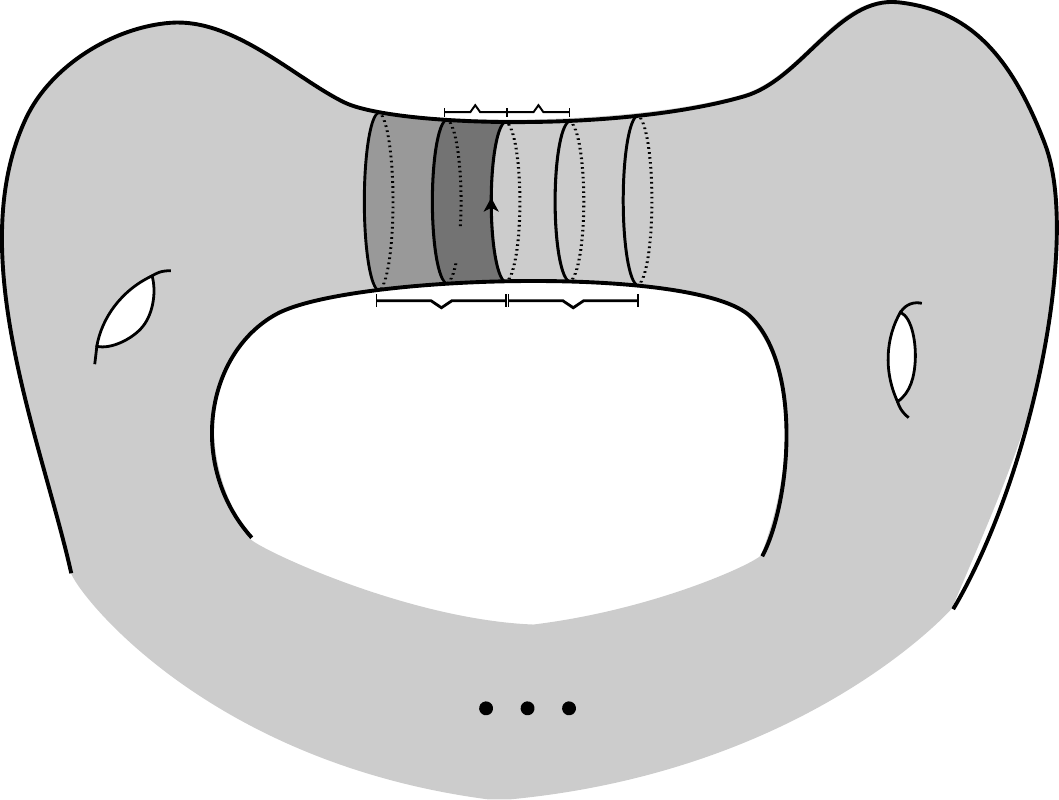}}%
    \put(0.09905215,0.62341384){\color[rgb]{0,0,0}\makebox(0,0)[lb]{\smash{$f=0$}}}%
    \put(0.78550152,0.62341384){\color[rgb]{0,0,0}\makebox(0,0)[lb]{\smash{$f=0$}}}%
    \put(0.42520853,0.67441196){\small\color[rgb]{0,0,0}\makebox(0,0)[lb]{\smash{$V_-$}}}%
    \put(0.48885826,0.6741834){\small\color[rgb]{0,0,0}\makebox(0,0)[lb]{\smash{$V_+$}}}%
    \put(0.51718584,0.41882169){\small\color[rgb]{0,0,0}\makebox(0,0)[lb]{\smash{$U_+$}}}%
    \put(0.3982415,0.41882169){\small\color[rgb]{0,0,0}\makebox(0,0)[lb]{\smash{$U_-$}}}%
    \put(0.43074373,0.56281751){\color[rgb]{0,0,0}\makebox(0,0)[lb]{\smash{$\gamma$}}}%
    \put(0.41805574,0.51756065){\tiny\color[rgb]{0,0,0}\makebox(0,0)[lb]{\smash{$f$}}}%
    \put(0.42805574,0.51756065){\tiny\color[rgb]{0,0,0}\makebox(0,0)[lb]{\smash{$=$}}}%
    \put(0.44805574,0.51756065){\tiny\color[rgb]{0,0,0}\makebox(0,0)[lb]{\smash{$1$}}}%
  \end{picture}%
\endgroup

\end{center}
and let  
$$[df]:=\left\{\begin{array}{cc}df & \textrm{on } U-\gamma\,, \\ 0 & \textrm{on }  (M-U)\cup\gamma\,.\end{array}\right.$$

In this setting, since the waist curves $\gamma_i$ of the maximal cylinders of $C_i$ of $\omega$ generate a $k$-dimensional isotropic subspace $I_k\subset H_1(M,\mathbb{R})$ (as $\omega$ has homological dimension $k$) and $\gamma$ is disjoint from $\gamma_i$'s, it is possible to check (see the proof of Lemma 4.4 of \cite{F02}) that the Abelian differential $\widetilde{\omega}=\omega+r[df]$ has homological dimension at least $k+1$ whenever $r\in\mathbb{Q}-\{0\}$.  In fact,
since $r\in \mathbb Q$ and $[df] \in H^1(M, \mathbb Z)$ the horizontal foliation of $\widetilde{\omega}$
is periodic. We can choose the tubular neighborhood $U$ of $\gamma$ disjoint from all the waist
curves $\gamma_i$ of $\omega$, so that $\widetilde{\omega}$ has maximal cylinders $\widetilde{C}_i$ of waist curves $\gamma_i$. If $\widetilde{\omega}$ has homological dimension $k$ (the same
as $\omega$) it follows that the flat surface $(M, \widetilde{\omega})$ decomposes as a union
of the cylinders $\widetilde{C}_i$. Now, since $k<g$ there exists an oriented closed curve 
$\gamma'$ which does not intersect any of the waist curves $\gamma_i$, but which has algebraic intersection number $1$ with $\gamma$. Thus, on one hand
$$
\int_{\gamma'}  \widetilde{\omega} =  \int_{\gamma'}  \omega  =0\,,
$$
since $\gamma'$ does not intersect any of the waist curves $\gamma_i$ of the differentials 
$\widetilde{\omega}$ and $\omega$; on the other hand 
$$
\int_{\gamma'} \widetilde{\omega} = \int_{\gamma'}  \omega+   r \int_{\gamma'} df = r\not =0\,,
$$
since the algberaic intersection $\gamma \cap \gamma'=1$ and $[df] \in 
H^1(M,\mathbb Z)$ is Poincar\'e dual to $[\gamma] \in H_1(M,\mathbb Z)$.

This completes the proof of the lemma.
\end{proof}

Now, let's fix an Abelian differential $\omega\in\mathcal{C}$ with periodic \emph{Lagrangian} horizontal foliation and let's try to use $\omega$ to reach after some continuous deformation the vicinity of some nice boundary point $\omega_{\infty}$ on the Deligne-Mumford compactification of $\mathcal{C}$ (whatever this means...). Intuitively, we note that horizontal maximal cylinders  $C_i$ of $\omega$ and their waist curves $\gamma_i$ looks like this 
\begin{center}

\begingroup
  \makeatletter
  \providecommand\color[2][]{%
    \errmessage{(Inkscape) Color is used for the text in Inkscape, but the package 'color.sty' is not loaded}
    \renewcommand\color[2][]{}%
  }
  \providecommand\transparent[1]{%
    \errmessage{(Inkscape) Transparency is used (non-zero) for the text in Inkscape, but the package 'transparent.sty' is not loaded}
    \renewcommand\transparent[1]{}%
  }
  \providecommand\rotatebox[2]{#2}
  \ifx\svgwidth\undefined
    \setlength{\unitlength}{150pt}
  \else
    \setlength{\unitlength}{\svgwidth}
  \fi
  \global\let\svgwidth\undefined
  \makeatother
  \begin{picture}(1,0.99652056)%
    \put(0,0){\includegraphics[width=\unitlength]{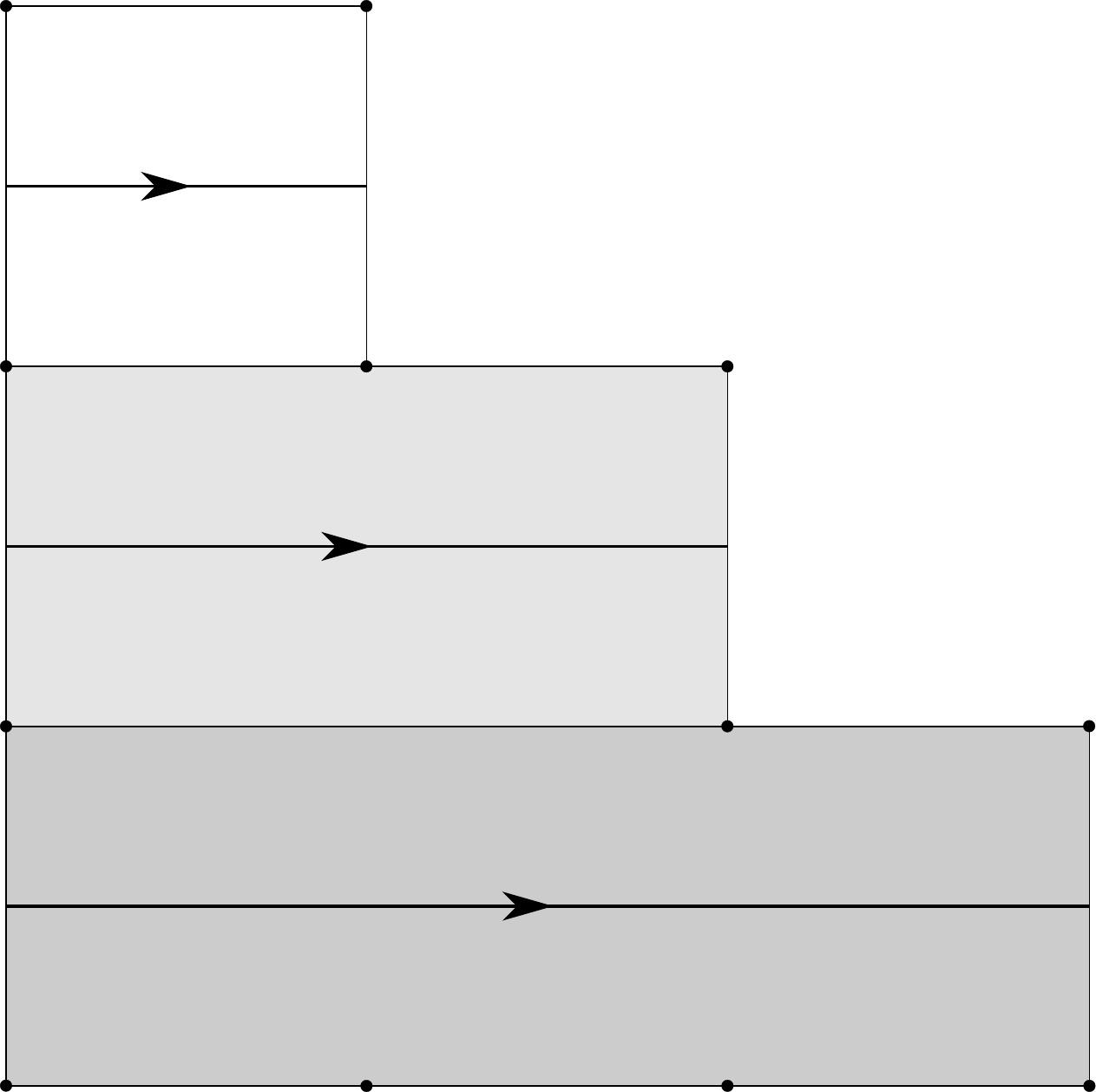}}%
    \put(0.26677429,0.42481809){\color[rgb]{0,0,0}\makebox(0,0)[lb]{\smash{$\gamma_2$}}}%
    \put(0.10195542,0.75735211){\color[rgb]{0,0,0}\makebox(0,0)[lb]{\smash{$\gamma_3$}}}%
    \put(0.43199865,0.09632188){\color[rgb]{0,0,0}\makebox(0,0)[lb]{\smash{$\gamma_1$}}}%
  \end{picture}%
\endgroup

\end{center}
In particular, by applying Teichm\"uller flow $g_t=\textrm{diag}(e^{t},e^{-t})$ and letting $t\to-\infty$, we start to \emph{pinching} off the waist curves $\gamma_i$. As it was observed by H.~Masur (see Section 4 of \cite{F02} and references therein), by an appropriate \emph{scaling} process on $\omega_t=g_t(\omega)$, one can makes sense of a limiting object 
$\omega_{\infty}$ in the Deligne-Mumford compactification of $\mathcal{H}_g$ which looks like this:
\begin{center}

\begingroup
  \makeatletter
  \providecommand\color[2][]{%
    \errmessage{(Inkscape) Color is used for the text in Inkscape, but the package 'color.sty' is not loaded}
    \renewcommand\color[2][]{}%
  }
  \providecommand\transparent[1]{%
    \errmessage{(Inkscape) Transparency is used (non-zero) for the text in Inkscape, but the package 'transparent.sty' is not loaded}
    \renewcommand\transparent[1]{}%
  }
  \providecommand\rotatebox[2]{#2}
  \ifx\svgwidth\undefined
    \setlength{\unitlength}{300pt}
  \else
    \setlength{\unitlength}{\svgwidth}
  \fi
  \global\let\svgwidth\undefined
  \makeatother
  \begin{picture}(1,0.35859542)%
    \put(0,0){\includegraphics[width=\unitlength]{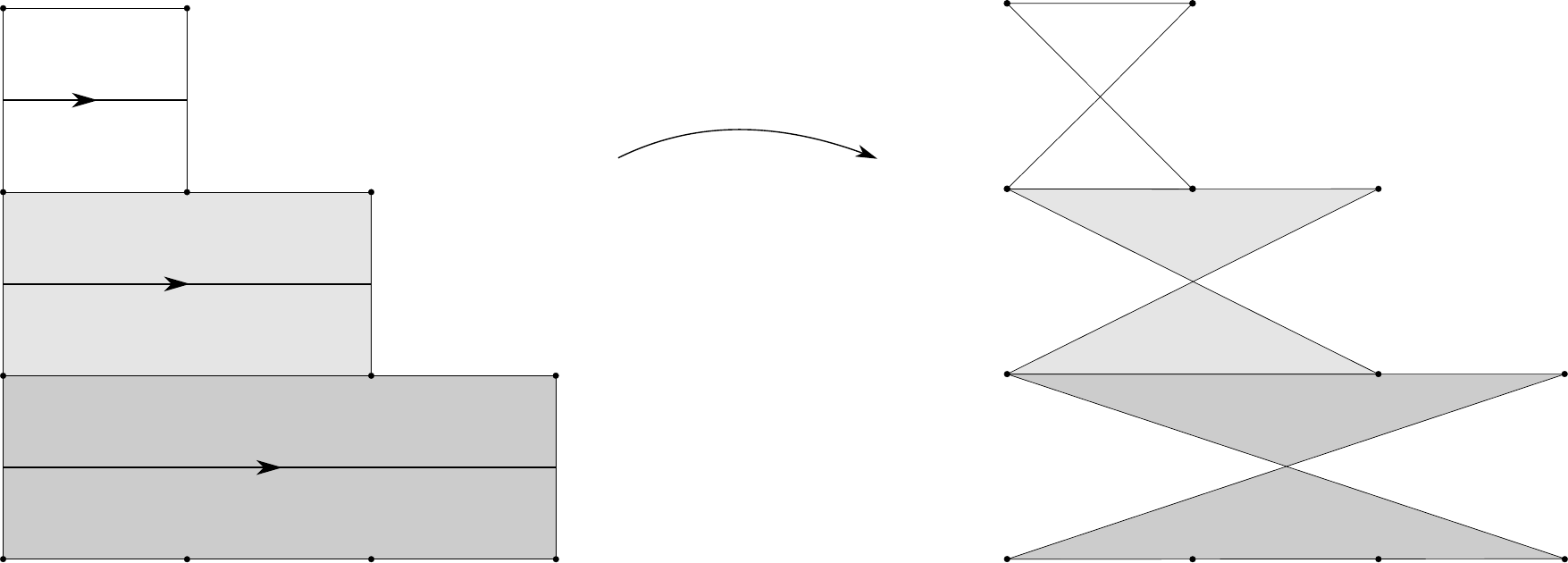}}%
    \put(0.10445528,0.15508074){\small\color[rgb]{0,0,0}\makebox(0,0)[lb]{\smash{$\gamma_2$}}}%
    \put(0.04569937,0.27362504){\small\color[rgb]{0,0,0}\makebox(0,0)[lb]{\smash{$\gamma_3$}}}%
    \put(0.16335574,0.03797587){\small\color[rgb]{0,0,0}\makebox(0,0)[lb]{\smash{$\gamma_1$}}}%
    \put(0.40159418,0.29792156){\color[rgb]{0,0,0}\makebox(0,0)[lb]{\smash{$g_t$}}}%
    \put(0.44159418,0.29792156){\scriptsize\color[rgb]{0,0,0}\makebox(0,0)[lb]{\smash{$(t\to -\infty)$}}}%
  \end{picture}%
\endgroup

\end{center}
Roughly speaking, this picture means that $\omega_{\infty}$ lives on a \emph{stable curve} $M_{\infty}$, i.e., a Riemann surface with \emph{nodes} at the \emph{punctures} $p_i$ obtained by pinching the loops $\gamma_i$'s, and it is a meromorphic Abelian differential with \emph{simple} poles with \emph{non-zero residues} at the punctures and the same multiplicities of  zeroes as  $\omega$ on $M$.
If $\omega$ has homological dimension $g$, it is possible to check that $\omega_{\infty}$ lives on a  union of $n\geq 1$ \emph{spheres} with $2(g+n-1)$ paired punctures and has opposite non-zero residues at (at least) $2g$ of them. In this situation, asymptotic formulas of J.~Fay and A.~Yamada for Abelian differentials near the boundary of the Deligne-Mumford compactification allow to prove that as 
$\omega_t$ approaches $\omega_{\infty}$, one has 
$$B^{\mathbb R}_{\omega_t}(c_i^t,c_j^t)\to-\delta_{ij}$$
whenever $\{c_1^t,\dots,c_g^t\}$ is a Hodge-orthonormal basis of the Poincar\'e dual of the ($g$-dimensional) subspace of $H_1(M,\mathbb{R})$ generated by the waist curves $\gamma_i$'s of the cylinder decomposition of  $(M,\omega)$. In other words, up to orthogonal matrices, the matrix of the form $B_{\omega_t}$ approaches $-\textrm{Id}_{g\times g}$ as $t\to-\infty$. Hence, $\textrm{rank}(B_{\omega_t}^{\mathbb{R}}):=2\cdot\textrm{rank}(B_{\omega_t})=2g$ as $t\to-\infty$, and, \emph{a fortiori}, the rank of $B_{\omega}^{\mathbb{R}}|_{H^1_{(0)}(M,\mathbb{R})}$ is $2g-2$ as $t\to-\infty$. Thus, this completes the sketch of the proof of Theorem~\ref{t.bdry-MV}.

\begin{remark} Actually, the fact that $B^{\mathbb R}_{\omega_t}$ ``approaches'' 
$-\textrm{Id}_{g\times g}$ implies that 
$$\sup\limits_{\omega\in\mathcal{C}}\Lambda_i(\omega)=1\,, \quad \textrm{ for all } 1\leq i\leq g.$$

\end{remark}

In a nutshell, the previous discussion around Theorem~\ref{t.bdry-MV} can be summarized as follows: firstly, we searched (in $\mathcal{C}$) some $\omega$ with periodic Lagrangian horizontal foliation; then, by using the Teichm\"uller flow orbit $\omega_t=g_t(\omega)$ of $\omega$ and by letting $t\to-\infty$, we \emph{spotted} an open region 
$U$ of 
$\mathcal{C}$ (near a certain ``boundary'' point $\omega_{\infty}$) where the form $B$ becomes an ``almost'' diagonal matrix with \emph{non-vanishing} diagonal terms, so that the rank of $B$ is maximal. Here, we remark that the idea of \emph{spotting} $U$ near the boundary of $\mathcal{C}$ is inspired
by the fact that $B$ is related (by a fundamental  variational formula of Rauch) to the derivative of the 
so-called \emph{period matrix} $\Pi$, and it is well-known since the work of J.~Fay (and A.~Yamada) that the period matrix $\Pi$ (and therefore $B$) has nice asymptotic (Taylor) expansions  (in terms
of pinching parameters) \emph{near} points at the \emph{boundary} of the moduli space 
$\mathcal{H}_g$. The following picture summarizes the discussion of this paragraph:
\begin{center}

\begingroup
  \makeatletter
  \providecommand\color[2][]{%
    \errmessage{(Inkscape) Color is used for the text in Inkscape, but the package 'color.sty' is not loaded}
    \renewcommand\color[2][]{}%
  }
  \providecommand\transparent[1]{%
    \errmessage{(Inkscape) Transparency is used (non-zero) for the text in Inkscape, but the package 'transparent.sty' is not loaded}
    \renewcommand\transparent[1]{}%
  }
  \providecommand\rotatebox[2]{#2}
  \ifx\svgwidth\undefined
    \setlength{\unitlength}{250pt}
  \else
    \setlength{\unitlength}{\svgwidth}
  \fi
  \global\let\svgwidth\undefined
  \makeatother
  \begin{picture}(1,0.86763386)%
    \put(0,0){\includegraphics[width=\unitlength]{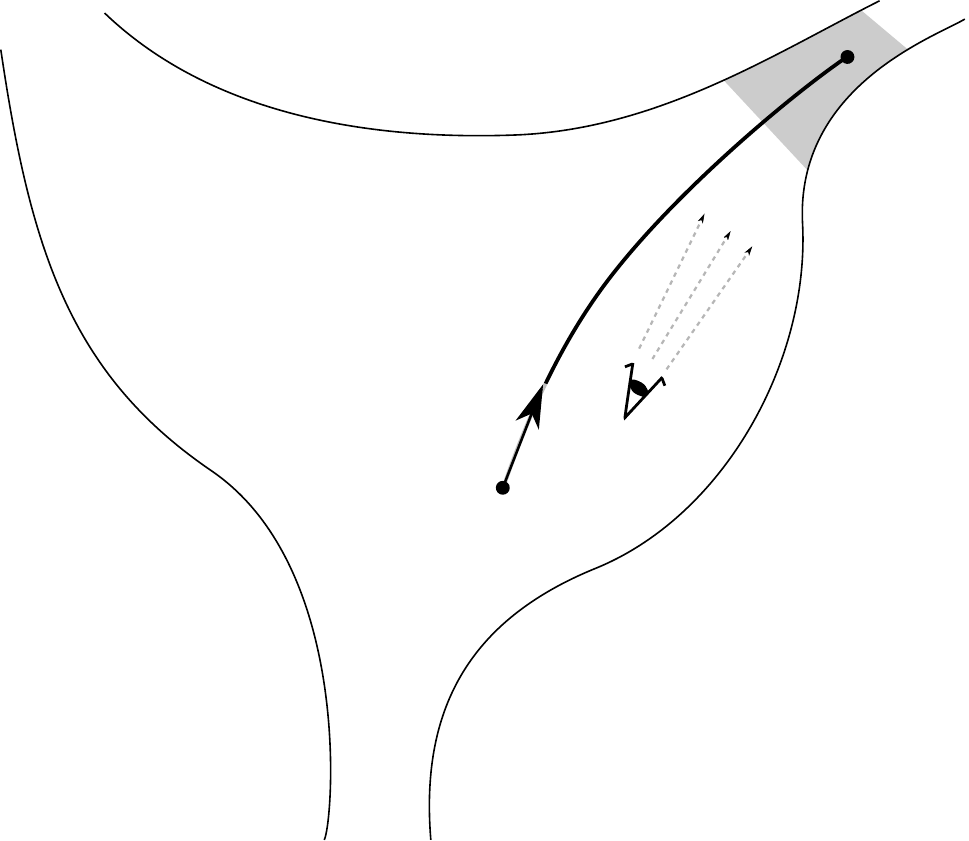}}%
    \put(0.48485221,0.31779284){\color[rgb]{0,0,0}\makebox(0,0)[lb]{\smash{$\omega$}}}%
    \put(0.7473639,0.82081263){\color[rgb]{0,0,0}\makebox(0,0)[lb]{\smash{$U$}}}%
    \put(0.87437886,0.73927831){\color[rgb]{0,0,0}\makebox(0,0)[lb]{\smash{$g_t (\omega)$}}}%
  \end{picture}%
\endgroup

\end{center}

Obviously, the proof of the main result of this subsection (namely Theorem~\ref{t.2nd-exp-MV}) is now complete in view of Theorem~\ref{t.bdry-MV} and Corollary~\ref{c.exp0rankB0}.

\begin{remark}\label{r.F11} These arguments (concerning exclusively \emph{Masur--Veech measures}) were extended by G.~Forni \cite{F11} to give the following far-reaching criterion for the \emph{non-uniform hyperbolicity} of the KZ cocycle with respect to any $SL(2,\mathbb{R})$-invariant $g_t$-ergodic probability measure $\mu$ on $\mathcal H^{(1)}_g$ (satisfying the  \emph{local product structure} condition of Definition~\ref{def:loc_prod_struct}): if one can find  in the topological support of $\mu$ an Abelian differential with periodic Lagrangian horizontal foliation (i.e., there is some $\omega\in\textrm{supp}(\mu)$ with  periodic horizontal foliation and homological dimension $g$), then $\lambda_g^{\mu}>0$.
\end{remark}

\section{Ergodic Theory of Translation Flows and Interval Exchange Transformations}\label{s.ET-IET}
\subsection{Unique ergodicity of interval exchange maps}

One important application of the fact that the Teichm\"uller flow preserves a natural (Masur--Veech) probability measure is the \emph{unique ergodicity} of ``almost every'' \emph{interval exchange transformation} (i.e.t. for short). Recall that an \emph{i.e.t.} is a map $T:D_T\to D_{T^{-1}}$ where $D_T, D_{T^{-1}}\subset I$ are subsets of an open bounded interval $I$ such that $I-D_T$ and $I-D_{T^{-1}}$ are finite sets and the restriction of $T$ to each connected component of $I-D_T$ is a translation onto some connected component of $D_{T^{-1}}$. For a concrete example, see the picture below.

\begin{figure}[htb!]
\includegraphics[scale=0.7]{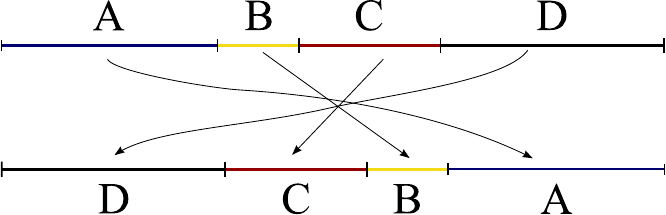}
\end{figure}

It is not hard to see that an i.e.t. $T$ is determined by \emph{metric data}, i.e., lengths of the connected components of $I-D_T$, and \emph{combinatorial data}, i.e., a permutation $\pi$ indicating how the connected components of $I-D_T$ are ``rearranged'' after applying $T$ to them. For instance, in the example of the picture above (where $4$ intervals are exchanged), the combinatorial data is the permutation $\pi:\{1,2,3,4\}\to\{1,2,3,4\}$ with $(\pi(1),\pi(2),\pi(3),\pi(4))=(4,3,2,1)$. 

In particular, it makes sense to talk about ``almost every'' i.e.t.: it means that a certain property holds for almost every choice of metric data with respect to the Lebesgue measure. 

\begin{remark}In the sequel, we will \emph{always} assume that the combinatorial data $\pi$ is \emph{irreducible}, i.e., if $\pi$ is a permutation of $d$ elements $\{1,\dots,d\}$, we require that, for every $k<d$, $\pi(\{1,\dots,k\})\neq\{1,\dots,k\}$. The meaning of this condition is very simple: if $\pi$ is not irreducible, there is $k<d$ such that $\pi(\{1,\dots,k\})=\{1,\dots,k\}$ and hence we can study \emph{any} i.e.t. $T$ with combinatorial data $\pi$ by juxtaposing two i.e.t.'s (one with $k$ intervals and another with $d-k$ intervals).   
\end{remark}

By applying their result (Theorem~\ref{t.MVp1}), H. Masur~\cite{M82} and W. Veech~\cite{V82} deduced that:
\begin{theorem}[H. Masur, W. Veech]\label{t.MViet} Almost every i.e.t. is uniquely ergodic.
\end{theorem}

Philosophically speaking, the derivation of this result from Theorem~\ref{t.MVp1} is part of a long tradition (in Dynamical Systems) of ``plough in parameter space, and harvest in phase space'' (following Adrien Douady's dictum about complex quadratic polynomials and Mandelbrot set [cf. http://en.wikipedia.org/wiki/Mandelbrot$\_\,$set]). In broad terms, the idea is that given a paramter family of dynamical systems and an appropriate renormalization procedure (defined at least for a significant part of the parameter space), one can often infer properties of the dynamical system for ``typical parameters'' by studying the dynamics of the renormalization. 
\begin{figure}[htb!]
\includegraphics[scale=0.5]{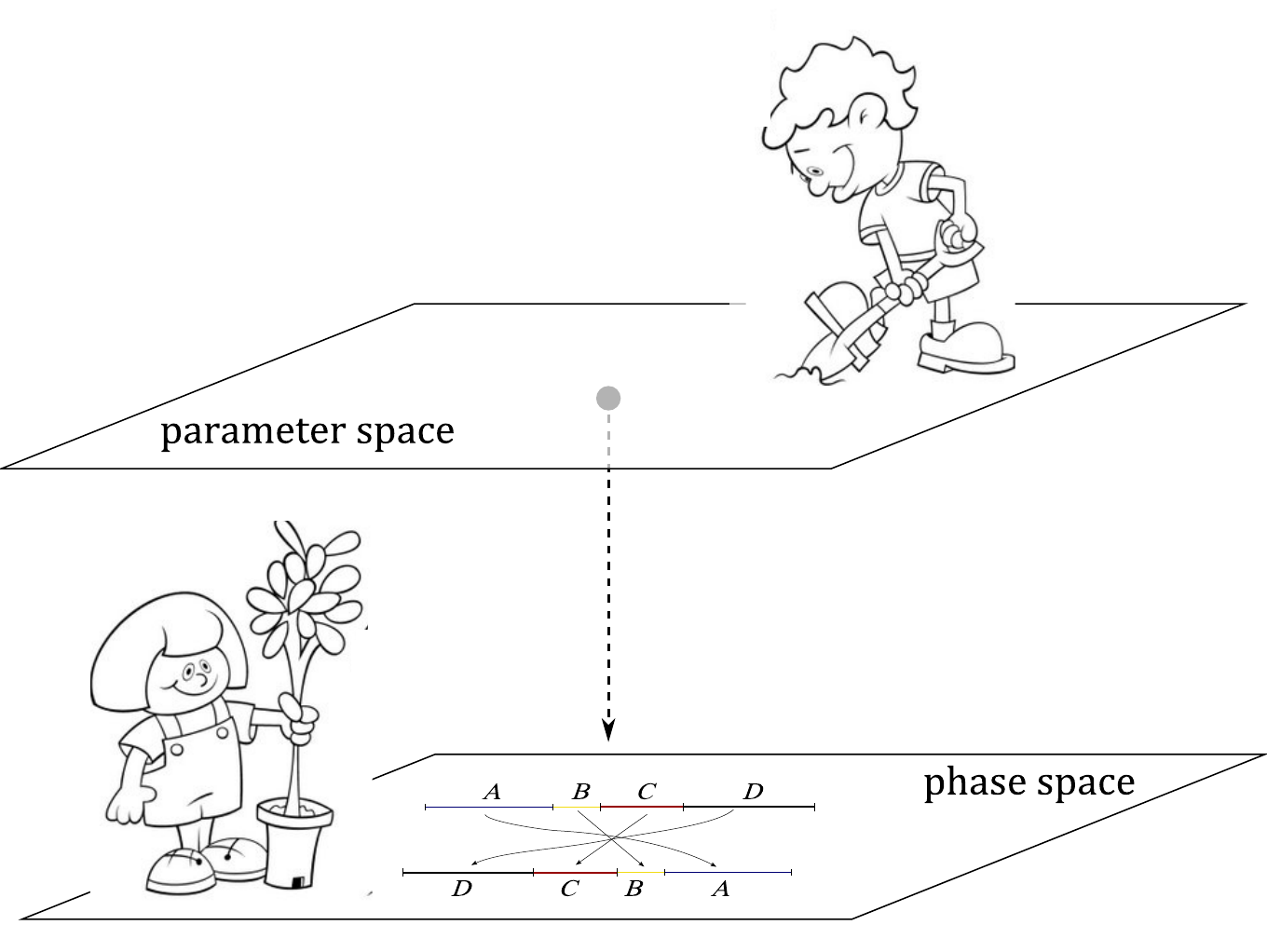}
\end{figure}

For the case at hand, we can describe this idea in a nutshell as follows. An i.e.t. $T$ can be ``\emph{suspended}'' (in several ways) to a translation surface $(M,\omega)$: the two most ``popular'' ways are \emph{Masur's suspension construction} and \emph{Veech's zippered rectangles construction} (cf. Example~\ref{ex.suspensions} above). For example, in Figure~\ref{f.flotvertical} below, we see a genus 2 surface (obtained by glueing the opposites sides of the polygon marked with the same letter $A$, $B$, $C$ or $D$ by appropriate translations) presented as a (Masur's) suspension of an i.e.t. with combinatorial data $(\pi(1),\pi(2),\pi(3),\pi(4))=(4,3,2,1)$. To see that this is the combinatorial data of the i.e.t., it suffices to ``compute'' the return map of vertical translation flow to the special segment in the ``middle'' of the polygon.


By definition, $T$ is the first return time map of the \emph{vertical translation flow} of the Abelian differential $\omega\in \mathcal H^{(1)}_g$ to an appropriate horizontal separatrix associated to some singularity of $\omega$. Here, the vertical translation flow $(\phi_t^{\omega})_{t\in \mathbb R}$ associated to a translation surface $(M,\omega)$ is the flow obtained by following (with unit speed) vertical geodesics of the flat metric corresponding to $\omega$. In particular, since the flat metric has singularities (in general), $(\phi_t^{\omega})_{t\in \mathbb R}$ is defined \emph{almost} everywhere (as vertical trajectories are ``forced'' to stop when they hit singular points [zeroes] of $\omega$)! See Figure~\ref{f.flotvertical} below for an illustration of these objects. There one can see an orbit through a point $q$ hitting a singularity in finite time (and hence stopping there) and an orbit through a point $p$ whose orbit never hits a singularity (and hence it can be prolonged forever). 

\begin{figure}[htb!]
\includegraphics[scale=0.45]{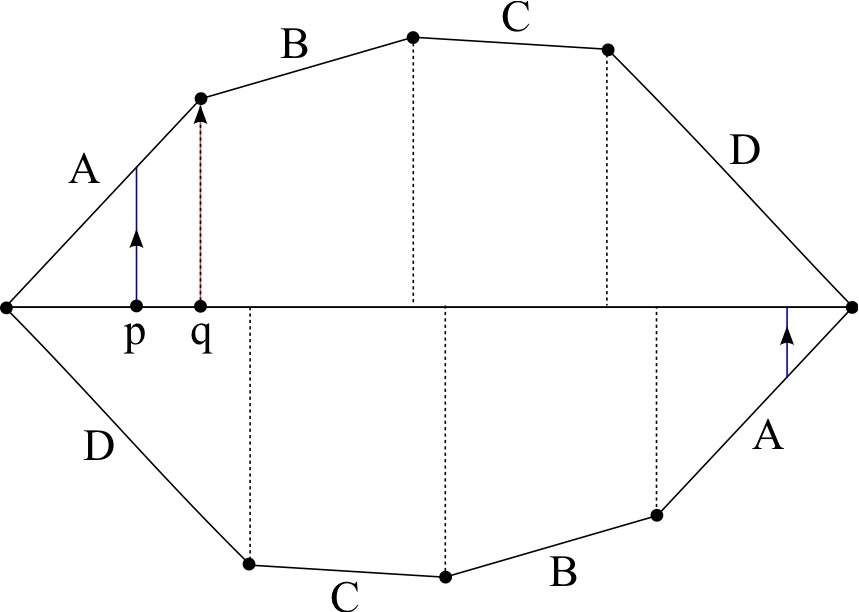}
\caption{Two pieces of orbits of a vertical translation flow: the orbit through $q$ (in red) hits a singularity in finite time (and then stops), while the orbit through $p$ (in blue) winds around the surface without hitting singularities (and thus can be continued indefinitely).}\label{f.flotvertical}
\end{figure}

In particular, we can study orbits of $T$ by looking at orbits of the vertical flow on $(M,\omega)$. Here, the idea is that long orbits of the vertical flow can wrap around a lot on $(M,\omega)$, so that a natural procedure is to use Teichm\"uller flow $g_t=\textrm{diag}(e^t,e^{-t})$ to make the long vertical orbits shorter and shorter (so that they wrap around the surface less and less), thus making it reasonably easier to analyze.  In other words, we use the Teichm\"uller flow to \emph{renormalize} the dynamics of the vertical flow on translation surfaces (and/or i.e.t.'s). Of course, the price we pay is that this procedure \emph{changes} the shape of $(M,\omega)$ (into $(M,g_t(\omega))$). But, if the Teichm\"uller flow $(g_t)_{t\in \mathbb R}$ has nice \emph{recurrence} properties (so that the shape $(M,g_t(\omega))$ is very close to $(M,\omega)$ for appropriate choice of large $t$), one can hope to bypass the difficulty 
imposed by the change of shape.  It is in fact enough to have a bound on the speed of the ``degeneration'' of the shape of $(M,g_t(\omega))$ as gets large  (see the work of Y.~Cheung and A.~Eskin \cite{CE} and R. Trevi\~no \cite{Tre}.) 

In their proof of unique ergodicity of almost every i.e.t., H.~Masur and W.~Veech observed that recurrence for almost all Abelian differentials can be derived from Poincar\'e's  theorem applied to Teichm\"uller flow endowed with any Masur--Veech measure. Of course, for this application of Poincar\'e recurrence theorem, it is crucial to know that all Masur--Veech measures are probability measures (i.e., they have finite mass), a property ensured by Theorem~\ref{t.MVp1}. 

Evidently, this is a very rough sketch of the proof of Theorem~\ref{t.MViet}. For more details, see 
J.-C.~Yoccoz survey~\cite{Y} for a complete proof using the Rauzy--Veech induction.

Note that the same kind of reasoning as above indicates that the unique ergodicity property must also be true for ``almost every'' translation flow in the sense that the vertical translation flow on $\mu_{\mathcal{C}}$ almost every translation surface structure $(M,\omega)\in\mathcal{C}$ is uniquely ergodic. Indeed, the following theorem (again by H. Masur~\cite{M82} and W. Veech~\cite{V82}) says that this is the case:
\begin{theorem}[H. Masur, W. Veech]\label{t.MVtf}Almost every translation flow is uniquely ergodic.
\end{theorem}

\begin{remark} In his original proof, H.~Masur showed the following result: if the vertical translation flow on the translation surface $(M,\omega)$ is minimal but \emph{not} uniquely ergodic, then the trajectory $g_t(\omega)$ of $(M,\omega)$ under the Teichm\"uller flow is \emph{divergent}. Equivalently, he gave a sufficient criterion (nowadays known as \emph{Masur's criterion}) for the unique ergodicity of a minimal vertical translation flows: it suffices to check that the Teichm\"uller trajectory of the underlying translation surface returns infinitely often to any given compact set. Note that the converse is \emph{not true} in general as it was shown by Y.~Cheung and A.~Eskin \cite{CE} (more recently R.~Trevi\~no \cite{Tre} has improved upon the result of  Cheung and Eskin by an approach inspired to the proof of unique ergodicity explained below, based on first variational formulas for the Hodge norm). We also note that, by the Poincar\'e recurrence theorem, the Masur's criterion implies the unique ergodicity of the vertical flow for almost all translation surfaces, \emph{with respect to any Teichm\"uller invariant probability measure}.
\end{remark}

Let $\mu$ be any ergodic Teichm\"uller invariant probability measure on $\mathcal H^{(1)}$.  In the sequel, we will present a sketch of proof of the unique ergodicity of the vertical translation flow 
$(\phi_t^{\omega})_{t\in \mathbb R}$ for $\mu$-almost all Abelian differentials $\omega\in 
\mathcal H^{(1)}_g$.
This result based on the recurrence of Teichm\"uller flow, and on the simplicity of the top exponent $1=\lambda_1^{\mu}>\lambda_2^{\mu}$ (see Corollary~\ref{c.spectralgapKZ} above). We recall that
this ``spectral gap'' result was derived from first variational formulas for the Hodge norm.

We start by assuming that the vertical translation flow $(\phi_t^{\omega})_{t\in \mathbb R}$ of our translation surface $(M,\omega)$ is \emph{minimal}, that is, every orbit defined for all times $t\geq 0$ is dense: this condition is well-known to be related to the absence of saddle connections (see, e.g., J.-C. Yoccoz survey~\cite{Y}), and the latter property has full $\mu$-measure by Poincar\'e recurrence 
theorem. In fact, the forward orbit of any Abelian differential with vertical saddle connections cannot 
be recurrent as the length of any given vertical saddle connection converges to zero under the
Teichm\"uller flow, while for the initial surface the lengths of all vertical saddle connections 
have a (strictly) positive lower bound.

Now, given an ergodic $\phi_t^{\omega}$-invariant probability measure $\nu$ on $M$, let $x\in M$ be a Birkhoff generic point, and let $T\geq 0$. Let $\gamma_T(x)\in H_1(M,\mathbb{Z})$ be the homology class obtained by ``closing'' the piece of (vertical) trajectory $[x,\phi_T^{\omega}(x)]:=\{\phi_t^{\omega}(x):t\in[0,T]\}$ with a bounded (usually \emph{short}) segment connecting $x$ to $\phi_T^{\omega}(x)$. A well-known theorem of Schwartzman~\cite{Sc} says that 
$$\lim\limits_{T\to\infty}\frac{\gamma_T(x)}{T}=\rho(\nu)\in H_1(M,\mathbb{R})-\{0\}\,.$$
In the literature, $\rho(\nu)$ is called \emph{Schwartzman asymptotic cycle}. By Poincar\'e duality, the Poincar\'e dual of $\rho(\nu)$ gives us a class $c(\nu)\in H^1(M,\mathbb{R})-\{0\}$. Geometrically, 
the class $c(\nu)$ is related to the \emph{flux} of $(\phi_t^{\omega})_{t\in \mathbb R}$ through transverse closed curves  with respect to $\nu$ on $M$. More precisely, given  any closed curve $\gamma$ transverse to $(\phi_t^{\omega})_{t\in \mathbb R}$, the flux thorough $\gamma$ is defined as follows:
$$\langle c(\nu),\gamma\rangle:=\lim\limits_{t\to0}\frac{\nu\left(\bigcup\limits_{s\in[0,t]}\phi_s^{\omega}(\gamma)\right)}{t}$$

For the sake of the discussion below, we recall that any $\phi_t^{\omega}$-invariant probability 
measure $\nu$ induces a \emph{transverse measure} $\widehat{\nu}$ on  any segment $\delta$ transverse to $(\phi_t^{\omega})_{t\in \mathbb R}$: indeed, we define $\widehat{\nu}(\delta)$ by the flux through $\delta$, i.e., $\lim\limits_{t\to0}\nu(\bigcup\limits_{s\in [0,t]}\phi_t^{\omega}(\delta))/t$. By construction the measure $\nu$ can be locally written as $\textrm{Leb}\times\widehat{\nu}$ in any ``product'' open set of the form $\bigcup\limits_{s\in [0,t]}\phi_s^{\omega}(\delta)$ not meeting singularities of the Abelian differential  $\omega\in \mathcal H_g$ for any transverse segment $\delta\subset M$.
  
We claim that under the minimality assumption the map $\nu\to c(\nu)$ is \emph{injective} on the cone
of invariant measures (into the cohomology space $H^1(M, \mathbb R)$). Indeed, given two $\phi_t^{\omega}$-invariant probability measures $\nu_1$ and $\nu_2$ with $c(\nu_1)=c(\nu_2)$, we observe that the \emph{transverse measures} $\widehat{\nu}_1$ and $\widehat{\nu}_2$ induced by them on a closed curve $\gamma$ transverse to $(\phi_t^{\omega})_{t\in \mathbb R}$ differ by the derivative of a continuous function $U$ defined on $\gamma$. Indeed, $U$ can be obtained by integration: by fixing an ``origin'' $0\in\gamma$ and an orientation on $\gamma$, we declare $U(0)=0$ and $U(x):=\widehat{\nu}_1(\gamma([0,x]))-\widehat{\nu}_2(\gamma([0,x]))+U(0)$, where $\gamma([0,x])$ is the segment of $\gamma$ going from $0$ to $x$ in the positive sense (with respect to the fixed orientation). Since by minimality the invariant measures $\widehat{\nu}_1$ and $\widehat{\nu}_2$ have no atoms, it follows that $U$ is continuous. Of course, the fact that $U$ is well-defined\footnote{I.e., our definition gives the same value for $U(0)$ when we go around $\gamma$.} is guaranteed by the assumption that $c(\nu_1)=c(\nu_2)$. We then note that $U$ is invariant under the return map induced by $(\phi_t^{\omega})_{t\in \mathbb R}$ on transverse segment $\delta \subset M$, so that, by the minimality of the flow 
$(\phi_t^{\mu})_{t\in \mathbb R}$, hence of its return map, we conclude that the continuous function $U$ must be constant. Therefore, $\widehat{\nu}_1=\widehat{\nu}_2$, i.e., $\nu_1$ and  $\nu_2$ have the same transverse measures. Since $\nu_1$ and $\nu_2$ are equal to the Lebesgue measure along the flow direction, we obtain that $\nu_1=\nu_2$, so that the claim is proved. 

Next, we affirm that $c(\nu)$ (or equivalently $\rho(\nu)$) decays exponentially fast \emph{like} $e^{-t}$ under the KZ cocycle whenever the Teichm\"uller flow orbit $g_t(\omega)$ of $\omega$ is \emph{recurrent}. Indeed, let us fix $t\geq 0$ such that $g_t(\omega)$ is very close to $\omega$, and we consider the action of the KZ cocycle $G_t^{KZ}$ on $\rho(\nu)$. Since, by definition, $\rho(\nu)$ is approximated by $\gamma_T(x)/T$ as $T\to\infty$, we have that 
$$G_t^{KZ}(\rho(\nu)) = \lim\limits_{T\to\infty}\frac{1}{T}G_t^{KZ}(\gamma_T(x)).$$ 
On the other hand, since $g_t$ contracts the vertical direction by a factor of $e^{-t}$ and $\gamma_T(x)$ is essentially a vertical trajectory (except for a \emph{bounded} piece of segment connecting $x$ to $\phi_T^{\omega}(x)$), we get 
$$\|G_t^{KZ}(\rho(\nu))\|_{g_t(\omega)} = \lim\limits_{T\to\infty}\frac{1}{T}\|G_t^{KZ}(\gamma_T(x))\|_{g_t(\omega)} = e^{-t}\lim\limits_{T\to\infty}\frac{1}{T}\|\gamma_T(x)\|_{\omega} = 
e^{-t}\|\rho(\nu)\|_{\omega}\,,$$
where $\|.\|_{\theta}$ denotes the \emph{stable norm} on $H_1(M,\mathbb{R})$ with respect to the
flat metric induced by the Abelian differential $\theta \in \mathcal H_g$ (the stable norm $\|\rho\|_{\theta}$ 
of a homology class $\rho\in H_1(M,\mathbb{R})$  is obtained roughly speaking by taking the infimum of the $\theta$-lengths of all $1$-cycles representing that class). In the previous calculation, we implicitly used the fact that for any fixed $t>0$ the factor of $1/T$ can ``\emph{kill}'' any \emph{bounded} term coming from the ``closing'' procedure used to define $\gamma_T(x)$ in the limit as $T>0$ gets large. 
Therefore, by taking into account that all continuous norms on a finite dimensional bundle are equivalent on compact sets, our statement is proved.
 
Finally, we note that the fact $1=\lambda_1^{\mu}>\lambda_2^{\mu}$ (i.e., simplicity of the top KZ cocycle exponent, see Corollary~\ref{c.spectralgapKZ} above) means that, for $\mu$-almost all
$\omega \in \mathcal H^{(1)}_g$, there is only \emph{one} direction in $H_1(M,\mathbb{R})$ which is contracted like $e^{-t}$! (namely, $\mathbb{R}\cdot[\textrm{Im}(\omega)]$) Therefore, given $\omega\in \mathcal H^{(1)}_g$ with minimal vertical translation flow and recurrent Teichm\"uller flow orbit, any $\phi_t^{\omega}$-invariant ergodic probability measure 
$\nu$ satisfies $c(\nu)\in\mathbb{R}\cdot[\textrm{Im}(\omega)]$. Since $(\phi_t^{\omega})_{t\in 
\mathbb R}$ preserves the Lebesgue measure $Leb$ (the flat area induced by $\omega$), we obtain that any $\phi_t^{\omega}$-invariant ergodic probability measure $\mu$ is a multiple of $Leb$, and, \emph{a fortiori}, $\nu=Leb$. Thus, $(\phi_t^{\omega})_{t\in \mathbb R}$ is uniquely ergodic for such $\omega$'s.  Since we already saw that $\mu$ almost everywhere the vertical translation flow is minimal, we have only to show that $\mu$-almost every $\omega$ is recurrent under Teichm\"uller flow to complete the proof of Theorem~\ref{t.MVtf}, but this is immediate from Poincar\'e's recurrence theorem. 
Since all the Masur--Veech measures are Teichm\"uller invariant measures of \emph{finite mass},
the above argument applies to this case in particular, hence Theorem~\ref{t.MVtf} is proved.

\medskip
As we have seen the above proof of unique ergodicity of the vertical foliation of an Abelian differential 
is based on the recurrence of the Teichm\"uller orbit in moduli space. In fact, such a recurrence is
not necessary and it is enough to have a bound on the speed of escape of the Teichm\"uller orbit
from a base point in the moduli space. Such a distance gives a measure  of how deformed the geometry 
of the surface is. For a translation surface $(M,\omega)$ the deformation of the geometry is well 
estimated ed by the amount of ``pinching'', that is, by the flat length $\delta(\omega)$ of the shortest  homopically non-trivial loop on the surface. In fact, for all $\delta>0$ the set $\{ (M,\omega) \vert \delta(\omega) \geq \delta\}$ is compact in the moduli space. Rodrigo Trevi\~no \cite{Tre} has extended our spectral gap argument for unique ergodicity to escaping Teichm\"uller trajectories:

\begin{theorem} (R.~Trevi\~no \cite{Tre}) Let $(M,\omega)$ be any translation surface. Assume that
$$
\int_0^{+\infty}  \delta^2(g_t\omega) dt = + \infty\,.
$$
Then the vertical flow of $(M,\omega)$ is uniquely ergodic.
\end{theorem} 
 
We recall that by H.~Masur~\cite{MLL} logarithmic law of geodesics, for any translation surface $(M, \omega)$
and for almost all $\theta \in [0, 2\pi]$,
$$
\limsup_{t\to + \infty}   \frac{  -\log \delta(g_t R_\theta \omega)}{\log t} = \frac{1}{2}\,.
$$ 
It is an open question whether the above condition implies unique ergodicity. As communicated to us by Y.~Cheung, known techniques should allow to construct, for any $\epsilon>1/2$, examples of Abelian differentials $\omega\in \mathcal H^{(1)}_g$ with minimal non uniquely ergodic vertical foliation such that
$$
\limsup_{t\to + \infty}   \frac{  -\log \delta(g_t \omega)}{\log t} = \epsilon>1/2\,.
$$ 
Trevi\~no's unique ergodicity theorem nevertheless implies that, whenever there exists $C>0$ such 
that, for all $t>0$,
$$ 
- \log \delta(g_t \omega) \leq \frac{1}{2} \log t + C\,.
$$
then the vertical flow is uniquely ergodic. For any given $\epsilon>0$ it is possible
to consider the set $\mathcal S_\epsilon$ of translation surfaces $(M,\omega)$ such that 
there exists a constant $C>0$ such that, for all $t>0$,
$$
- \log \delta(g_t \omega) \leq \epsilon \log t + C\,.
$$
Cheung and Eskin \cite{CE} proved that there exists $\epsilon >0$ such that all
$(M,\omega)\in \mathcal S_\epsilon$ have uniquely ergodic vertical flow. By Trevi\~no's
theorem,  the largest value of $\epsilon>0$ such that all $(M,\omega)\in \mathcal S_\epsilon$ 
have uniquely ergodic  vertical flow  is at least $1/2$. Conjecturally, by counterexamples
to unique ergodicity which according to Y.~Cheung can be constructed by known techniques, the 
set $\mathcal S_\epsilon$ contains Abelian differential with non-uniquely ergodic vertical foliation as soon as $\epsilon>1/2$. 

\subsection{Deviation of ergodic averages} 
Let us consider a generic vertical translation flow $(\phi_t^{\omega})_{t\in \mathbb R}$ on a translation surface $(M,\omega)$ (so that it is uniquely ergodic) and let us choose a typical point $p$ (so that $\phi_t^{\omega}(p)$ is defined for every time $t>0$), e.g., as in Figure~\ref{f.flotvertical} above. For all $T>0$ large enough, let us denote by $\gamma_T(x)\in H_1(M,\mathbb{R})$ the homology class obtained by ``closing'' the piece of (vertical) trajectory $[x,\phi_T^{\omega}(x)]:=\{\phi_t^{\omega}(x):t\in[0,T]\}$ with a bounded (usually \emph{short}) segment connecting $x$ to $\phi_T^{\omega}$. We recall that Schwartzman theorem~\cite{Sc} says that 
$$\lim\limits_{T\to\infty}\frac{\gamma_T(x)}{T}=c\in H_1(M,\mathbb{R})-\{0\}\,.$$ 

For genus $g=1$ translation surfaces (i.e., flat torii), this is very good and fairly complete result: indeed, it is not hard to see that the \emph{deviation} of $\gamma_T(x)$ from the line $E_1:=\mathbb{R}\cdot c$ spanned by the Schwartzman asymptotic cycle is \emph{bounded}. 

For genus $g=2$ translation surfaces, the global scenario gets richer: by doing numerical experiments, what one sees is that the \emph{deviation} of $\gamma_T(x)$ from the line $E_1$ has amplitude roughly
 $T^{\lambda_2}$ with $\lambda_2<1$ around a certain line. In fact, the deviation of $\gamma_T(x)$ from the Schwartzman asymptotic cycle is not completely random: it occurs along an \emph{isotropic} 
 $2$-dimensional plane $E_2\subset H_1(M,\mathbb{R})$ containing $E_1$. Again, in genus $g=2$, this is a ``complete'' picture in the sense that numerical experiments indicate that the deviation of $\gamma_T(x)$ from $E_2$ is again bounded.

More generally, for arbitrary genus $g\geq 1$, the numerical experiments indicate the existence of an \emph{asymptotic Lagrangian flag}, i.e., a sequence of isotropic subspaces $E_1\subset E_2\subset\dots\subset E_g\subset H_1(M,\mathbb{R})$ with $\textrm{dim}(E_i)=i$ and a \emph{deviation spectrum} $1=\lambda_1>\lambda_2>\dots>\lambda_g>0$ such that 
$$\lim\limits_{T\to\infty}\frac{\log \textrm{dist}(\gamma_T(x),E_i)}{\log T}=\lambda_{i+1}\,, \quad
\text{ \rm for every } i=1,\dots,g-1\,,$$
$$\sup\limits_{T\in[0,\infty)}\textrm{dist}(\gamma_T(x),E_g)<\infty.$$
For instance, the reader can see below two pictures (Figures~\ref{f.zorich-deviation-1} and~\ref{f.zorich-deviation-2}) extracted from A. Zorich's survey~\cite{Z} showing numerical experiments related to the \emph{deviation phenomenon} or \emph{Zorich phenomenon} discussed above for a genus $3$ translation surface. There, we have a slightly different notation for the involved objects: $c_n$ denotes $\gamma_{T_n}(x)$ for a convenient choice of $T_n$, the subspaces $\mathcal{V}_i$ correspond to the subspaces $E_i$, and the numbers $\nu_i$ correspond to the numbers $\lambda_i$.

\begin{figure}[htb!]
\includegraphics[scale=0.5]{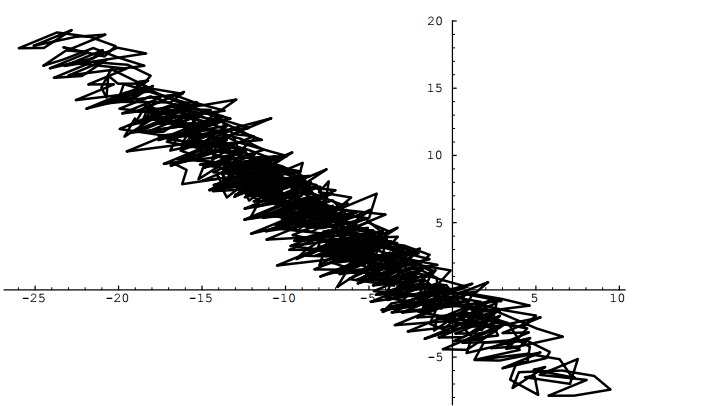}
\caption{Projection of a broken line joining $c_1,c_2,\dots,c_{100000}$ to a plane orthogonal to Schwartzmann cycle (in a genus $3$ case).}\label{f.zorich-deviation-1}
\end{figure}

\begin{figure}[htb!]
\includegraphics[scale=0.5]{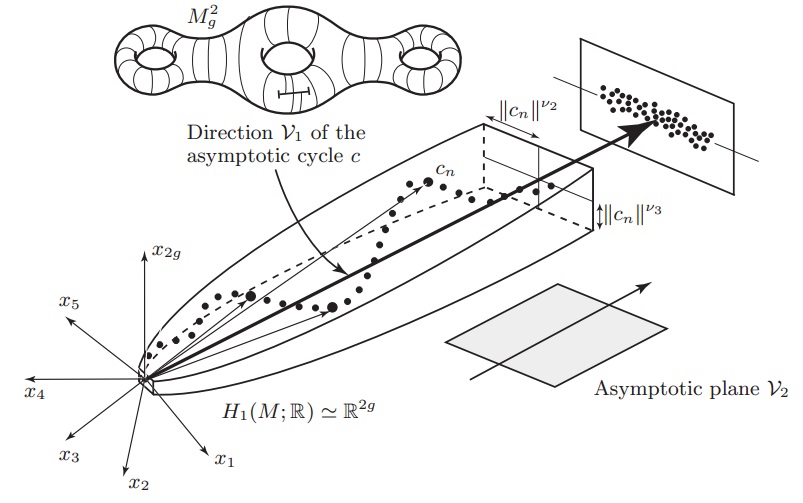}
\caption{Deviation from Schwartzmann asymptotic cycle.}\label{f.zorich-deviation-2}
\end{figure}

This scenario supported by numerical experimental was made rigorous by A. Zorich~\cite{Z94} using the Kontsevich--Zorich cocycle: more precisely, he proved that the previous statement is true with $E_i$ given by the sum of the Oseledets subspaces associated to the first $i$ non-negative exponents of the KZ cocycle, and $\lambda_i$ corresponding to the $i$-th Lyapunov exponent of the KZ cocycle with respect to the appropriate Masur--Veech measure $\mu_{\mathcal{C}}$. Of course, to get the \emph{complete} description of the deviation phenomenon (i.e., the fact that $\textrm{dim}\,E_i=i$, that is, the asymptotic flat $E_1\subset\dots\subset E_g$ is Lagrangian and complete), one needs to know that the Kontsevich--Zorich conjecture is true. So, in this sense, A. Zorich's theorem is a conditional statement depending on the Kontsevich--Zorich conjecture.  We note that the existence of a \emph{partial} Lagrangian flag can already be derived from the the non-uniform hyperbolicity of the KZ cocycle
(proved in \cite{F02} for Masur--Veech measures and in \cite{F11} for more general $SL(2, \mathbb R)$-invariant measures whose support contains Abelian differentials of maximal homological dimension), while the existence of a \emph{complete} Lagrangian flag follows from the simplicity of the Lyapunov spectrum of the KZ cocycle (proved in~\cite{AV} for the Masur--Veech measures and recently in~\cite{MMY} for measures coming from square-tiled surfaces under certain conditions).

Closing this subsection, let us mention that a similar scenario of deviations of ergodic averages for i.e.t.'s is true (as proved by A. Zorich in~\cite{Z94}). The precise statement concerns the deviation from the
mean of Birkhoff sums of linear combinations of characteristic functions of subintervals (see also J.-C. Yoccoz survey~\cite{Y}). In \cite{K} M.~Kontsevich also conjectured that the deviation theorem in fact holds for Birkhoff averages of smooth functions along translation flows. This conjecture was proved
in \cite{F02}. Another proof based on a reduction to the cohomological equation for interval exchange
transformations can be derived from the results of S.~Marmi, P.~Moussa and 
J.-C.~Yoccoz~\cite{MMY05}. Recently, A.~Bufetov~\cite{Bu} has developed the methods of Forni's 
paper~\cite{F02} and proved refined deviation results and limit theorems for Birkhoff averages of Lipschitz functions along  translation flows. All deviation results for translation flows have a counterpart
for interval exchange transformations, although the dictionary is not immediate.

\medskip The above deviation results hold for the vertical flow of almost all translation surfaces
for with respect to Masur--Veech measures. The corresponding full measure set  is extremely
difficult to describe since it consists of points in the moduli space which are generic in the sense
of Birkhoff ergodic theorem and Oseledets theorem. However, it is possible to prove polynomial
bounds on the speed of ergodicity of smooth functions based on first variational formulas for
the Hodge norm and of the quantitative recurrence estimates. J.~Athreya~\cite{At} proved in fact
that, for any given translation surface, the set of directions with long excursion outside compact sets
of a certain type has angular Lebsegue measure exponentially small with respect to the length
of the excursion. This result implies in particular that the Teichm\"uller in almost all directions
spends a strictly positive fraction time inside a compact set, hence by the variational formulas
for the Hodge norm it is possible to prove the following result.

\begin{theorem} (Athreya and Forni \cite{AtF}) For any translation surface $(M, \omega)$
there exists $\alpha \in (0,1]$ and a measurable function  $K:[0, 2\pi]\to \R$ such that
the following holds. The ergodic integrals of any weakly differentiable function $f$ on $M$ 
(square-integrable together with its first derivatives) along the directional flow 
$(\phi^\theta_t)_{t\in \mathbb R}$ satisfy the following bound: for almost all $\theta\in [0, 2\pi]$,
for all $x\in M$ with regular forward trajectory and for all $T>0$, 
$$
\vert \int_0^T f \circ \phi^\theta_{t} (x) dt - \frac{i T}{2} \int_M f \omega \wedge \overline{\omega} \vert  
\leq K(\theta) \Vert f\Vert_1 T^{1-\alpha} \,.
$$
(Here $\Vert \cdot \Vert_1$ denotes the Sobolev norm on the space $W^1(M)$, that is,
the sum of the $L^2$ norm of the function of its first derivatives).
\end{theorem}
It is tempting to conjecture that for every translation surfaces there exists $\beta\in (0, 1)$
such that for almost all $\theta\in [0,2\pi]$, for all  `generic' zero average smooth functions 
(that is, all outside a subspace of codimension $1$) and for all $x\in M$  with regular forward trajectory,
$$
 \limsup_{T\to +\infty}  \frac{1}{ T^{1-\beta}} \vert \int_0^T f \circ \phi^{\theta}_t (x) dt \vert   >0\,,
$$
that is, that  ergodic integrals do grow according to a power law. In particular, they cannot be
bounded by logarithmic functions. This conjecture is probably true for 
virtually all translation surfaces of higher genus, but it has at least $2$ exceptions (one in genus
$3$ and one in genus $4$), as we will see in the next section. 

We conclude this section by a short summary of Avila and Forni~\cite{AF} proof of weak mixing for almost all interval exchange transformations and translation flows.

\subsection{Weak mixing property for i.e.t.'s and translation flows} The plan for this subsection is to \emph{briefly} sketch how the knowledge of the positivity of the second Lyapunov exponent $\lambda_2^{\mu_{\mathcal{C}}}$ of the KZ cocycle with respect to Masur--Veech measures $\mu_{\mathcal{C}}$ on connected components $\mathcal{C}$ of strata was used by A. Avila and G. Forni \cite{AF} to prove the weak mixing property for i.e.t.'s and translation flows. The basic references for the subsection are the original article \cite{AF} and the survey \cite{Fo-ArturBrin}. 

Recall that a dynamical system $T:X\to X$ preserving a probability measure $\mu$ on the measure space $X$ is \emph{weakly mixing} whenever 
$$\lim\limits_{N\to\infty}\frac{1}{N}\sum\limits_{n=0}^{N-1}|\mu(T^{-n}(A)\cap B)-\mu(A)\mu(B)|=0$$
for any measurable subsets $A, B\subset X$. Equivalently, $(T,\mu)$ is weak mixing if given measurable subsets $A, B\subset X$, there exists a subset $E\subset\mathbb{N}$ of density one\footnote{I.e., $\lim\limits_{N\to\infty}\frac{1}{N}\cdot\#(E\cap\{1,\dots,N\})=1$.} such that 
$$\lim\limits_{\substack{n\to\infty \\ n\in E}}\mu(T^{-n}(A)\cap B)=\mu(A)\cdot\mu(B)$$

For the case of i.e.t.'s and translation flows, it is particularly interesting to consider the following \emph{spectral} characterization of weak mixing. 

An i.e.t. $T:D_T\to D_{T^{-1}}$ is weak mixing if for any $t\in\mathbb{R}$ there is \emph{no} non-constant measurable function $f:D_T\to\mathbb{C}$ such that, for every $x\in D_T$,
$$f(T(x))=e^{2\pi i t}f(x)\,.$$

Similarly, a (vertical) translation flow $(\phi_s^{\omega})_{s\in \mathbb R}$ on a translation surface $(M,\omega)$ represented by the suspension of an i.e.t. $T:D_T\to D_{T^{-1}}$, say $D_T=\bigcup\limits_{\alpha\in\mathcal{A}}I_\alpha$ with a (piecewise constant) roof function $h(x)=h_{\alpha}$ for $x\in I_{\alpha}$ (see the picture below)

\begin{figure}[htb!]
\includegraphics[scale=0.5]{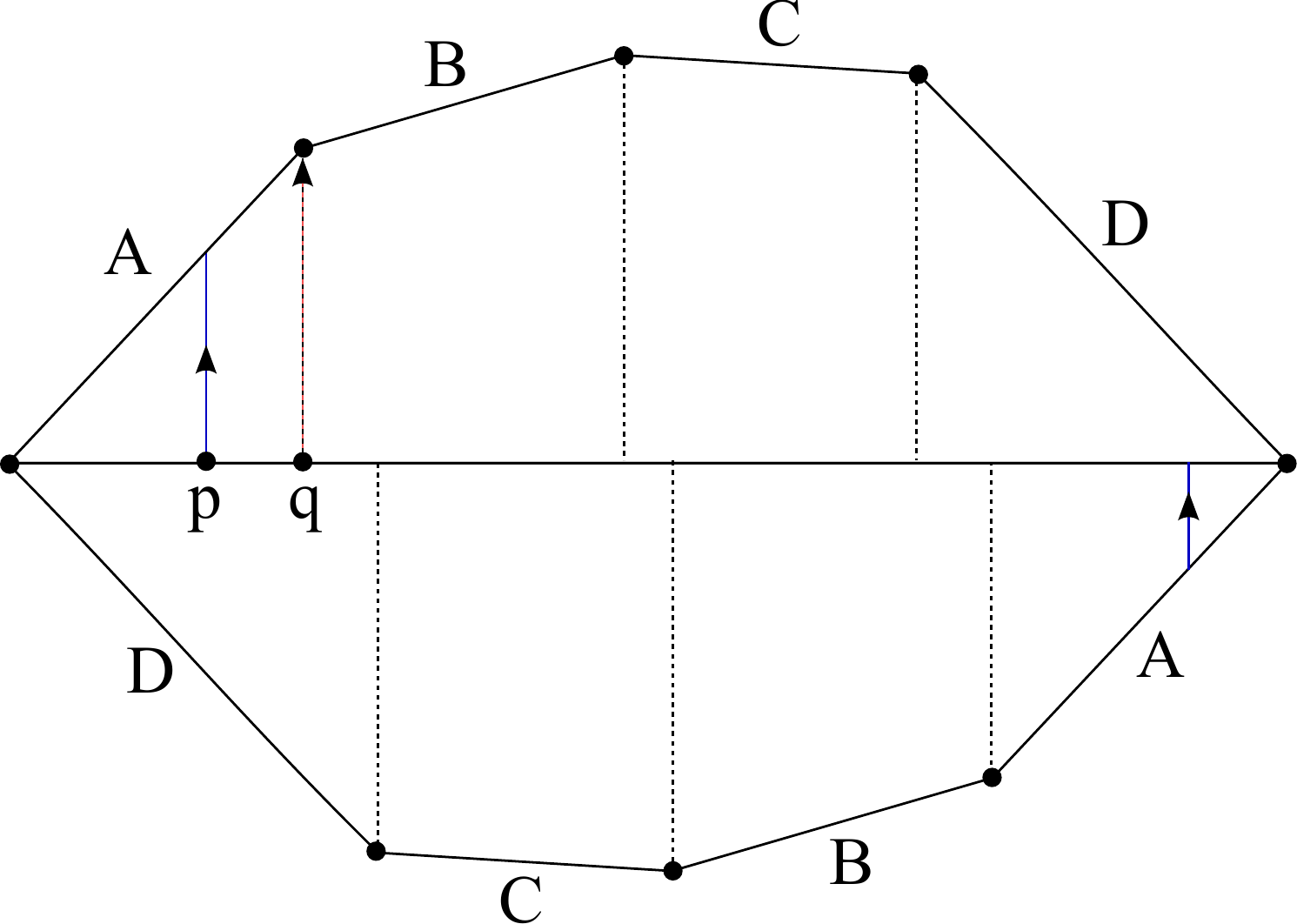}
\end{figure}

\noindent is weak mixing if for any $t\in\mathbb{R}$ there is \emph{no} non-constant measurable function $f:D_T\to\mathbb{C}$ such that, for every $\alpha\in\mathcal{A}$ and  $x\in I_{\alpha}$,
$$f(T(x))=e^{2\pi i t h_{\alpha}}f(x)\,.$$

The roof function vector  $h:=(h_{\alpha})_{\alpha\in\mathcal{A}}\in \mathbb{R}^{\mathcal{A}}$ an be
 identified with an element of the (relative) homology of $(M,\omega)$, that is, 
 $h\in H_1(M,\Sigma,\mathbb{R})$. 

This spectral characterization of weak mixing allowed W.~Veech \cite{V84} to set up a criterion of weak mixing for i.e.t.'s and translation flows: roughly speaking, in the case of translation flows $(\phi_s^{\omega})_{s\in \mathbb R}$, the criterion states that if $(\phi_s^{\omega})_{s\in \mathbb R}$ is \emph{not} weak mixing, that is, the equation 
$$f(T(x))=e^{2\pi i t h_{\alpha}}f(x)$$ 
has a non-constant measurable solution $f$ for some $t\in\mathbb{R}$, then, by considering the times when the Teichm\"uller orbit of the translation surface comes back near itself, i.e., the times $t_n>0$ such that the $g_{t_n}(\omega)$ is close to $\omega$ in  moduli space, the KZ cocycle $G_{t_n}^{KZ}(\omega)$ sends $t\cdot h$ near the integer lattice $\mathbb{Z}^{\mathcal{A}}\simeq H_1(M,\Sigma,\mathbb{Z})$ in (relative) homology:
\begin{equation}\label{e.wm-AF}
\lim\limits_{n\to\infty}\textrm{dist}_{\mathbb{R}^{\mathcal{A}}}(G_{t_n}^{KZ}(\omega)(t\cdot h), \mathbb{Z}^{\mathcal{A}})=0
\end{equation}
Actually, this is a very crude approximation of Veech's criterion: the formal statement depends on the relationship between Teichm\"uller flow/KZ cocycle and the Rauzy--Veech-Zorich algorithm, and we will not try to recall it here. Instead, we will close this subsection with a sketch of Avila and Forni's argument
for translation flows and interval exchange transformations.

The overall strategy of the proof is to prove that $G_t^{KZ}$ ``tends'' to keep ``typical'' lines $t\cdot h\in\mathbb{R}\cdot h\subset H_1(M,\Sigma,\mathbb{R})$ in homology sufficiently ``far away'' from the integral lattice $H_1(M,\Sigma,\mathbb{Z})$ when the second Lyapunov exponent $\lambda_2^{\mu_{\mathcal{C}}}$ (with respect to Masur--Veech measures $\mu_{\mathcal{C}}$) is positive. In other words, the main goal in \cite{AF} is to show that Equation~\ref{e.wm-AF} can be \emph{contradicted} for ``almost every'' i.e.t.'s and translation flows when $\lambda_2^{\mu_{\mathcal{C}}}>0$, so that Veech's criterion implies weak mixing property for ``almost every'' i.e.t.'s and translation flows.

The set of vectors $t\cdot h\in \mathbb{R}^{\mathcal{A}} \approx H_1(M,\Sigma,\mathbb{R})$ 
such that formula~\eqref{e.wm-AF} form the so-called {\it weak stable} space $\mathcal W_\omega
\subset \mathbb{R}^{\mathcal{A}} $ at $\omega \in \mathcal H^{(1)}_g$. It is immediate to verify that the weak stable space contains the Oseledets stable space of the Kontsevich--Zorich cocycle. However it could be larger (and indeed it is in general). Since the cocycle on the quotient bundle with fiber 
$H_1(M,\Sigma,\mathbb{R})/H_1(M,\mathbb{R})$ is isometric, the weak stable space is a subset of $H_1(M,\R)$. It is not too hard to prove that if $\lambda^{\mu_{\mathcal{C}}}_2>0$ then the stable space at almost all Abelian differentials (with respect to a Masur--Veech measure) has Hausdorff dimension at most $2g-2$. The proof of \emph{weak mixing for almost all translation flows}
 then follows, since by a dimension count, for almost all $h\in  H_1(M,\R)$ the line $\R\cdot h$ does 
 not intersect  the weak stable space, hence the corresponding translation flow is weakly mixing. Indeed, since $\lambda_g^{\mu_{\mathcal{C}}}>0$ the weak stable space has Hausdorff dimension $g\geq 1$ almost everywhere (it is a product of a zero dimensional Cantor set with the Oseledets stable 
space of the cocycle). This implies that the set of non weakly mixing translation flows has Hausdorff
dimension at most $g+1$. Note that in the above argument we work with a Oseledets generic Abelian differentials (with respect to the KZ cocycle) with fixed vertical foliation, which corresponds to a fixed
interval exchange transformation,  and we prove weak mixing for almost all choices of the horizontal
foliation, that is, for almost all roof function vectors $h \in \mathbb{R}^{\mathcal{A}}$ . In this way we derive weak mixing for almost all translation flows with respect to any Masur--Veech measure.

The proof of \emph{weak mixing for almost all interval exchange transformations} is much harder
since we 'loose' the parameter $h\in \mathbb{R}^{\mathcal{A}}$. In fact, weak mixing for
interval exchange transformations can be read as a property of the translation flow with {\it fixed}
roof function vector  $h= (1, \dots, 1)$.  A \emph{probabilistic elimination argument} allows to prove that
for \emph{any} fixed $h\in \mathbb{R}^{\mathcal{A}}$, the set of Abelian differentials $\omega \in 
\mathcal H^{(1)}_g$ such that the line $\mathbb R \cdot h$ intersects  the weak stable space 
$\mathcal W_\omega$ outside the integer lattice has measure zero with respect to any Masur--Veech measure. As above, it is possible  to assume that $h\in  H_1(M,\mathbb{R})$ since the KZ cocycle is isometric on the quotient $H_1(M,\Sigma,\mathbb{R})/H_1(M,\mathbb{R})$. Indeed, in the case that 
$(1, \dots, 1) \not \in  H_1(M,\mathbb{R})$, which corresponds to special permutations, W.~Veech \cite{V84} had already proved  weak mixing for almost all length data. The outline of the proof goes as follows. After sufficient iteration of the cocycle, any line $\mathbb R h$ can be assumed to lie with a
 compact cone around the unstable direction of the cocycle.  It is then sufficient to ``eliminate''
 from the weak stable space any given line segment not containing the origin and parallel to a 
 direction in the cone. Note that points of the segment which lie in the weak stable space are such
 that, for any $\delta>0$ and for sufficiently large time $t_n\in \mathbb R$, they lie in a $\delta$-neighborhood of the integer lattice. By the hyperbolicity of the cocycle, in fact, the positivity 
 of the second exponent suffices, any segment lying near the integer lattice will be pushed 
 away under iteration of the cocycle (with high probability). If  the ``excursions''  $t_{n+1} -t_n$ 
 were uniformly bounded, this
 observation would be enough to complete the proof, since after sufficient the distance from
 the integer lattice would be larger than a chosen constant. However, the above assumption
 is very restrictive since it holds only on a set of measure zero of Abelian differentials. In the
 typical situation, our segment may be ``kicked '' from time to time from the vicinity of an lattice
 point to the vicinity of another lattice point, whenever the excursion time $t_{n+1} -t_n$ is
 large enough (how large depends on the size $\delta>0$ of the neighborhood). More precisely,
 the ``parent'' segment will be so stretched by the KZ cocycle that its image will intersects 
 the $\delta$-neighborhood of several other lattice points, thereby generating ``children''
 intervals. Such ``children'' intervals correspond to subintervals of the parent interval which 
 cannot yet be eliminated from the weak stable space. An elementary estimate allows to bound
 the number of ``children'' and their proximity to the integer lattice in terms of the norm of the
 cocycle map $G^{KZ}_{t_{n+1}-t_n}(\omega)$, that is, in terms of the size $t_{n+1}-t_n$ of the excursion.
 
 We have thus described two opposite mechanisms. In the first mechanism, the parent interval 
 will almost surely be eliminated from the weak stable space since it is pushed away from the integer lattice by the positivity of the second exponent. This elimination process will take longer the closer
 the interval is to the integer lattice. In the second mechanism, from time to time the parent interval generates children, that is, new intervals closer to other lattice points, in a controlled manner.  
 A rather technical probabilistic argument,  based in crucial way on the \emph{strong mixing properties} of the appropriate return maps of the Teichm\''uller flow, proves that the first mechanism prevails 
 on the second in the long run with full probability. Very roughly, the survival probability of a 
 given interval is weighted by a small inverse power of its distance from the integer lattice. This 
 weighted survival probability is expected to decrease under the effect of the hyperbolicty 
 of the cocycle on the parent interval and to increase by a controlled amount whenever children 
 are generated.  The technical core of the argument proves that the weighted survival probability converges to zero with time. Heuristically, this follows from the fact that by the ergodic theorem the probability of large excursions is small, hence the event that many children are generated with large weighted survival probability is too rare to
 counteract the decay of the weighted survival probability caused by the hyperbolicity
 of the KZ cocycle. We hope that the above sketch gives at least a vague idea of  Avila--Forni's 
 proof  of weak mixing for interval exchange transformations. For a more detailed 
 outline of the argument the reader can consult the survey paper~\cite{Fo-ArturBrin}.


\section{Veech's question}\label{s.Veech}

As we saw in the previous chapter, after the work of G.~Forni \cite{F02}, and A.~Avila and M.~Viana \cite{AV} on the Kontsevich--Zorich conjecture, the Lyapunov exponents of the Kontsevich--Zorich cocycle with respect to \emph{Masur--Veech measures} $\mu_{\mathcal{C}}$ are ``well-understood'': their multiplicities are equal to $1$ (i.e., they are simple) and $0$ doesn't belong to the Lyapunov spectrum. Moreover, by the Avila and Forni \cite{AF} weak mixing theorem, a part of this result (namely, the positivity of the second top Lyapunov exponent) implies weak mixing for ``typical'' translation flows (and i.e.t.'s) with respect to $\mu_{\mathcal{C}}$. In other words, one can say that we reasonably understand the Ergodic Theory of the Kontsevich--Zorich cocycle with respect to Masur--Veech measures (and its consequences to ``typical'' i.e.t.'s and translation flows).

On the other hand, by working with Masur--Veech measures, one misses all applications of the KZ cocycle to the study of \emph{particular} but \emph{physically interesting} translations flows (such as the ones associated to billiards in \emph{rational polygons}\footnote{Cf. Example~\ref{ex.rat-polyg} of Section~\ref{s.orbifold}}). Indeed, this is so because these particular translation flows are usually associated to Abelian differentials in closed $SL(2,\mathbb{R})$-invariant sets of \emph{zero} Masur--Veech measure (and hence the tools from Ergodic Theory can't be applied directly). 

Therefore, it is natural to ask how much of the discussion of the previous chapter still applies to other (Teichm\"uller and/or $SL(2,\mathbb{R})$) invariant measures. In this direction, after the completion of the work~\cite{F02}, W.~Veech asked whether it was possible to prove the non-vanishing of the
exponents (at least of the second exponent) for \emph{arbitrary} $SL(2,\mathbb{R})$-invariant probability measures on the moduli space of Abelian differentials.

The reader maybe wondering why Veech  did not include \emph{all} Teichm\"uller invariant probability measures in his question. As it turns out, there are at least two good reasons 
to restrict to $SL(2,\mathbb{R})$-invariant measures: 
\begin{itemize}
\item the Teichm\"uller flow is non-uniformly hyperbolic with respect to Masur--Veech measures, and hence it has a \emph{lot} of invariant measures and a complete study of the Lyapunov spectrum of all such measures seems a very hard task. More concretely, W.~Veech \cite{V82} (see also Appendix~\ref{a.pAjc} below) constructed (with the aid of the so-called Rauzy--Veech diagrams) a periodic orbit of the Teichm\"uller flow (i.e., a \emph{pseudo-Anosov element} of the mapping class group), in the stratum $\mathcal{H}(2)$ of  Abelian differentials with a single double zero on surfaces of genus $2$, such that the Teichm\"uller invariant probability supported in this periodic orbit has a \emph{vanishing} second Lyapunov exponent for the KZ cocycle. This is in sharp contrast with the fact that the second Lyapunov exponent of the KZ cocycle w.r.t. the Masur--Veech measure $\mu_{\mathcal{H}(2)}$ of the (connected) stratum $\mathcal{H}(2)$ is non-zero (as it follows from the work of G.~Forni \cite{F02}, and Avila and Viana \cite{AV}) and it shows that the description of Lyapunov spectra of the KZ cocycle with respect to arbitrary Teichm\"uller invariant measures can be non-trivial \emph{even} at the level of periodic orbits (pseudo-Anosov elements).
\item As we mentioned in the previous chapter, in analogy to Ratner's work on unipotent flows, it is conjectured that all $SL(2,\mathbb{R})$-invariant probability in the moduli space of Abelian differentials are  ``algebraic'' (a proof of  this conjecture was announced by A.~Eskin and 
M.~Mirzakhani \cite{EsMi}), and this is actually the case for all known examples of $SL(2,\mathbb{R})$-invariant measures. In particular, $SL(2,\mathbb{R})$-invariant probability measures are much better behaved than general Teichm\"uller invariant measures (which can be supported on fractal-like objects), and thus provide a natural family of measures to consider in the study of the possible Lyapunov spectra of the KZ cocycle. 
\end{itemize}

During this section, we'll describe of two examples answering Veech's question.

\subsection{\emph{Eierlegende Wollmilchsau}} Given $x_1, \dots, x_4\in\overline{\mathbb{C}}$ four distinct points in the 
Riemann sphere $\overline{\mathbb{C}} = \mathbb{C}\cup\{\infty\}$, let's consider the Riemann surface $M_3(x_1, \dots, x_4)$ defined by (the solutions of) the algebraic equation 
$$\{(x,y):y^4 = (x-x_1)\dots(x-x_4)\}\,.$$
This Riemann surface is a \emph{cyclic cover} of the Riemann sphere branched at 4 points in the sense that we have the covering map $p:M_3(x_1,\dots,x_4)\to \overline{\mathbb{C}}$, $p(x,y) = x$, is ramified precisely over $x_1,\dots,x_4$, and the automorphism $T(x,y) = (x,iy)$, $i=\sqrt{-1}$, of $M_3(x_1, \dots, x_4)$ is a generator of the Galois group $\mathbb{Z}/4\mathbb{Z}$  (cyclic group of order $4$) of the covering $p$.
\begin{remark} By Galois theory, a Riemann surface $M$ coming from a normal cover $p:M\to\overline{\mathbb{C}}$ branched at $4$ points $x_1,\dots,x_4\in \overline{\mathbb{C}}$ and with cyclic Galois group (of deck transformations) is given by an algebraic equation of the form 
$y^N=(x-x_1)^{a_1}\dots(x-x_4)^{a_4}$.
\end{remark}
As the reader can check, the normal cover $p:M\to\overline{\mathbb{C}}$ is \emph{not} ramified at 
$\infty$, and, by Riemann-Hurwitz formula, the surface $M_3(x_1, \dots, x_4)$ has genus $3$. 

\begin{remark} Since the group of M\"obius transformations (automorphisms of $\overline{\mathbb{C}}$) acts (sharply) $3$-transitively on $\overline{\mathbb{C}}$, we get that $M_3(x_1,\dots,x_4)$ is isomorphic to $M_3(0,1,\infty,\lambda)$ where 
$$\lambda=\lambda(x_1,\dots,x_4):=\frac{(x_4-x_1)(x_2-x_3)}{(x_4-x_3)(x_2-x_1)}\in \overline{\mathbb{C}}-\{0,1,\infty\}$$
is the \emph{cross-ratio} of $x_1,\dots,x_4$. In other words, the complex structures of the family $M_3(x_1,\dots,x_4)$ are parametrized by a single complex parameter $\lambda\in\overline{\mathbb{C}}-\{0,1,\infty\}$.
\end{remark}

Next, we consider $dx/y^2$ on $M_3(x_1,\dots,x_4)$. 

\begin{lemma}\label{l.EW-stratum}$dx/y^2$ is an Abelian differential with $4$ simple zeroes at $x_1,\dots,x_4$. In particular, $(M_3(x_1,\dots,x_4),dx/y^2)\in 
\mathcal{H}(1,1,1,1)$.
\end{lemma}
\begin{proof}
The lemma follows by studying $dx/y^2$ near the points $x_1,\dots,x_4$ and $\infty$: 
\begin{itemize}
\item near $x_i$, i.e., $x\sim x_i$, the natural coordinate is 
$y$ and one has $y^4\sim (x-x_i)$, so that $y^3dy\sim dx$; hence, near $x_i$, $\frac{dx}{y^2}\sim \frac{y^3dy}{y^2}=ydy$; i.e., $dx/y^2$ has \emph{simple zeroes} at 
$x_i$'s.
\item near $\infty$, the natural coordinate is $\zeta=1/x$ and one has $y^4=x^4(1-x_1/x)\dots(1-x_4/x)\sim\zeta^{-4}$ (i.e., $y\zeta\sim 1$) and $d\zeta\sim dx/x^2=\zeta^2dx$, so that $dx/y^2\sim \zeta^{-2}d\zeta/\zeta^{-2}=d\zeta$, that is, $dx/y^2$ is holomorphic and \emph{non-vanishing} near $\infty$.
\end{itemize} 
Thus, one has that\footnote{Note that this is in agreement with the Riemann-Hurwitz formula $2g-2=\sum k_i$, where $g$ is the genus of $M$ and $k_i$ are the orders of zeroes of an Abelian differential $\omega$ on $M$.} $(M_3(x_1,\dots,x_4),dx/y^2)\in\mathcal{H}(1,1,1,1)$.
\end{proof} 

Now, we define $\omega_{EW}=c(x_1,\dots,x_4)dx/y^2$ on $M_3(x_1,\dots,x_4)$, where $c(x_1,\dots,x_4)\in\mathbb{R}$ is the unique positive real number such that 
the translation surface $(M_3(x_1,\dots,x_4),\omega_{EW})$ has \emph{unit area}. 

\begin{lemma}\label{l.EW-locus} $\mathcal{EW}:=\{(M_3(x_1,\dots,x_4),\omega_{EW}): x_1,\dots, x_4\in\overline{\mathbb{C}} \textrm{ distinct}\}$ is the $SL(2,\mathbb{R})$-orbit of a square-tiled surface\footnote{Cf. Example~\ref{ex.L-origami} of Section~\ref{s.orbifold} for the definition of square-tiled surface}. In particular, $\mathcal{EW}$ is a closed $SL(2,\mathbb{R})$-invariant locus of $\mathcal{H}^{(1)}(1,1,1,1)$.
\end{lemma}

\begin{proof} Let's show that $\mathcal{EW}\subset \mathcal{H}^{(1)}(1,1,1,1)$ is a closed $SL(2,\mathbb{R})$-invariant locus. 

Note that $\omega_{EW}$ is \emph{anti-invariant} with respect to the action $T^*$ (by pull-back) of the automorphism $T(x,y)=(x,iy)$ of $M_3(x_1,\dots,x_4)$, i.e., $T^*(\omega_{EW})=-\omega_{EW}$. In fact, 
$$T^*(\omega_{EW}):=c(x_1,\dots, x_4)dx/(iy)^2=-c(x_1,\dots, x_4)dx/y^2=-\omega_{EW}\,.$$ Therefore, the quadratic differential $q_{EW}:=\omega_{EW}^2$ is $T^*$-invariant. Since $T$ generates the Galois group of deck transformations of the normal cover $p:M\to\overline{\mathbb{C}}$, this means that $q$ projects under $p$ to a quadratic differential $q_0$ on $\overline{\mathbb{C}}$ with $4$ simple poles. 

One can see this directly: since $q_{EW}=\omega_{EW}^2=c(x_1,\dots,x_4)^2dx^2/y^4$, it projects\footnote{Essentially this is the fact that $M_3(x_1,\dots,x_4)$ is given by the equation $y^4=(x-x_1)\dots(x-x_4)$.} to $q_0=q_0(x_1,\dots,x_4):=c(x_1,\dots, x_4)^2dx^2/(x-x_1)\dots(x-x_4)$ under $p(x,y)=x$. 

Therefore, the elements of $\mathcal{EW}$ are obtained by appropriate cyclic covers of elements 
$$(\overline{\mathbb{C}},q_0(x_1,\dots, x_4))\in \mathcal{Q}(-1,-1,-,1,-1)\,.$$ 
Actually, since $\omega_{EW}$ has unit area and $p$ has degree $4$, one has that $q_0$ has area $1/8$, i.e., 
$$(\overline{\mathbb{C}},q_0(x_1,\dots, x_4))\in \mathcal{Q}^{(1/8)}(-1,-1,-,1,-1)\simeq \mathcal{Q}^{(1)}(-1,-1,-1,-1)$$
In other words, the locus $\mathcal{EW}$ is a \emph{copy} of $\mathcal{Q}^{(1)}(-1,-1,-1,-1)$ inside $\mathcal{H}^{(1)}(1,1,1,1)$. 

Since any (ramified) cover is defined by \emph{pre}-composition with charts\footnote{In particular, the translation atlas of $(M_3,\omega_{EW})$ is obtained by pre-composing (half) translation charts of $(\overline{\mathbb{C}}, q_0)$ with the covering map $p:M \to \overline{\mathbb{C}}$.} and $SL(2,\mathbb{R})$ acts by \emph{post}-composition with charts, the operations of taking covers and letting $SL(2,\mathbb{R})$ act \emph{commute}. Hence, it follows that $\mathcal{EW}$ is a closed $SL(2,\mathbb{R})$-invariant locus of $\mathcal{H}^{(1)}(1,1,1,1)$ simply because the elements of $\mathcal{EW}$ are obtained by appropriate cyclic covers of elements of the closed $SL(2,\mathbb{R})$-invariant locus\footnote{This is true because \emph{any} stratum of quadratic/Abelian differentials is closed and $SL(2,\mathbb{R})$-invariant.} $\mathcal{Q}^{(1)}(-1,-1,-1,-1)$. 

Finally, $\mathcal{EW}$ is the $SL(2,\mathbb{R})$-orbit of a square-tiled surface because $h(x,y)=(x,y^2)$ is a covering map from $M_3(x_1,\dots,x_4)$ to the elliptic curve (genus $1$ Riemann surface) $$E(x_1,\dots,x_4)=\{w^2=(x-x_1)\dots(x-x_4)\} \,,$$
such that $h_*(\omega_{EW})=c(x_1,\dots,x_4)dx/w$, and the locus $\{E(x_1,\dots,x_4),c(x_1,\dots,x_4)dx/w\}$ is precisely the moduli space $\mathcal{H}^{(1/4)}(0)$ of genus $1$ Abelian differentials of area $1/4$ (with one marked point) and we know that $\mathcal{H}^{(1/4)}(0)\simeq SL(2,\mathbb{R})/SL(2,\mathbb{Z})$ is (isomorphic to) the $SL(2,\mathbb{R})$-orbit of the square-tiled surface $(\mathbb{C}/(\mathbb{Z}\oplus i\mathbb{Z}),(1/4)dx)$. 
More concretely, by taking $(x_1,x_2,x_3,x_4)=(-1,0,1,\infty)$, by the identities \footnote{See Theorem 1.7 of \cite{BB} for a proof of this identity.}  
$$\int_{-\infty}^{-1}\frac{dx}{\sqrt{x^3-x}}=\int_{-1}^{0}\frac{dx}{\sqrt{x^3-x}}=\int_0^{1}\frac{dx}{\sqrt{x^3-x}}=\int_1^{\infty}\frac{dx}{\sqrt{x^3-x}}=\frac{\Gamma(1/4)^2}{2\sqrt{2\pi}}$$
among integrals representing the four different periods, one has that $(E(-1,0,1,\infty),dx/w)$ is isomorphic (up to isogeny, i.e., scaling factor) to $(\mathbb{C}/(\mathbb{Z}\oplus i\mathbb{Z}),dz)$, so that $(M_3(-1,0,1,\infty),\omega_{EW})$ is a square-tiled surface\footnote{We will see a concrete model of this translation surface in next section.} and its $SL(2,\mathbb{R})$-orbit is \emph{contained} in $\mathcal{EW}$. Since $\mathcal{EW}$ is a closed connected locus of real dimension $3$ (as it is a copy of the stratum $\mathcal{H}^{(1)}(0)\simeq SL(2,\mathbb{R})/SL(2,\mathbb{Z})$ and $\textrm{dim}(SL(2,\mathbb{R}))=3$), one gets that $\mathcal{EW}$ must \emph{coincide} with the $SL(2,\mathbb{R})$-orbit of the square-tiled surface $(M_3(-1,0,1,\infty),\omega_{EW})$.
\end{proof} 

The $SL(2,\mathbb{R})$-orbits of square-tiled surfaces $(M,\omega)$ are ``well-behaved'' objects in moduli spaces of Abelian differentials. For instance, since a square-tiled surface $(M,\omega)$ is naturally a finite cover of the flat torus $(\mathbb{T}^2=\mathbb{C}/(\mathbb{Z}\oplus i\mathbb{Z}), dz)$ ramified only at $0\in\mathbb{T}^2$, one can see that the \emph{stabilizer} $SL(M,\omega)$ of $(M,\omega)$ in $SL(2,\mathbb{R})$ is a finite-index subgroup of $SL(2,\mathbb{Z})$ (when the periods of $(M,\omega)$ generate the lattice $\mathbb{Z}\oplus i\mathbb{Z}$). Since $SL(2,\mathbb{Z})$ is a \emph{lattice} of $SL(2,\mathbb{R})$, we have that $SL(2,\mathbb{R})\cdot (M,\omega)\simeq SL(2,\mathbb{R})/SL(M,\omega)$ supports an \emph{unique} $SL(2,\mathbb{R})$-invariant $g_t$-ergodic probability measure $\mu$. In particular, it makes sense to talk about Lyapunov exponents of square-tiled surfaces $(M,\omega)$: they're the Lyapunov exponents of the KZ cocycle with respect to the unique $SL(2,\mathbb{R})$-invariant measure $\mu$ supported on $SL(2,\mathbb{R})\cdot
(M,\omega)$. By combining this discussion with Lemma~\ref{l.EW-locus}, we have that $\mathcal{EW}$ supports an unique $SL(2,\mathbb{R})$-invariant probability measure $\mu_{\mathcal{EW}}$. 

\begin{theorem}[Forni \cite{F06}]\label{t.F06} The Lyapunov spectrum of the KZ cocycle with respect to $\mu_{\mathcal{EW}}$ is \emph{totally degenerate} in the sense that 
$$\lambda_2^{\mu_{\mathcal{EW}}}=\lambda_3^{\mu_{\mathcal{EW}}}=0\,.$$
\end{theorem} 

\begin{proof} As the reader can check, the set  
$$\left\{\theta_1=\omega_{EW}=c\cdot\frac{dx}{y^2}, \quad \theta_2:=\frac{dx}{y^3}, \quad \theta_3:=\frac{x\,dx}{y^3}\right\}$$
is a basis of the space $H^{1,0}(M_3)$ of holomorphic $1$-forms on $M_3:=M_3(x_1,\dots,x_4)$. Note that it diagonalizes the cohomological action $T^*$ of the automorphism 
$T(x,y)=(x,iy)$: indeed,  
$$T^*(\omega_{EW})=-\omega_{EW}, \quad T^*(\theta_2)=i\theta_2, \quad T^*(\theta_3)=i\theta_3.$$
Let's denote by $\lambda(n)$ the $T^*$-eigenvalue of $\theta_n$, i.e., $T^*(\theta_n)=\lambda(n)\theta_n$.

We can compute the (symmetric, complex-valued) form $B_{\omega_{EW}}$ on $H^{1,0}$ in this basis as follows. Firstly, we recall that 
$B_{\omega_{EW}}(\omega_{EW}, \omega_{EW})=1$. Secondly, by using the automorphism $T$ to perform a change of variables, we get that 
\begin{eqnarray*}
B_{\omega_{EW}}(\theta_n,\theta_m)&:=&\frac{i}{2}\int\frac{\theta_n\theta_m}{\theta_1}\overline{\theta_1} = \frac{i}{2}\int\frac{T^*(\theta_n)T^*(\theta_m)}{T^*(\theta_1)}T^*(\overline{\theta_1}) \\ &=& \lambda(n)\lambda(m)\frac{i}{2}\int\frac{\theta_n\theta_m}{\theta_1}\overline{\theta_1}  =
\lambda(n)\lambda(m) B_{\omega_{EW}}(\theta_n,\theta_m)\,.
\end{eqnarray*}
In particular, $B_{\omega_{EW}}(\theta_n,\theta_m)=0$ whenever $\lambda(n)\lambda(m)\neq 1$. Since $\lambda(n)\lambda(m)\neq 1$ for $(n,m)\neq (1,1)$, we obtain that the matrix of $B$ in the basis $\{\theta_1,\theta_2,\theta_3\}$ is 
$$B_{\omega_{EW}} = \left(\begin{array}{ccc}1&0&0\\0&0&0\\0&0&0\end{array}\right)$$

Therefore, $B_{\omega_{EW}}$ has rank $1$ and, by Theorem~\ref{t.F02}, we conclude that the $SL(2,\mathbb{R})$-invariant $g_t$-ergodic probability measure $\mu_{\mathcal{EW}}$ has Lyapunov exponents $\lambda_2^{\mu_{\mathcal{EW}}}=\lambda_3^{\mu_{\mathcal{EW}}}=0$. 
\end{proof} 

Evidently, this theorem (of G.~Forni in 2006) answers Veech's question in a definitive way: there is 
\emph{no} non-uniform hyperbolicity and/or simplicity statement for the KZ cocycle with respect to 
\emph{general} $SL(2,\mathbb{R})$-invariant $g_t$-ergodic probability measures!  This result
in fact leads to the problem of classifying the Lyapunov structure of the KZ cocycle with respect
to \emph{all} $SL(2,\mathbb{R})$-invariant $g_t$-ergodic probability measures. In this broader 
context the above example seems rather special.

In the literature, the square-tiled surface $(M_3(-1,0,1,\infty),\omega_{EW})$ in the support of 
$\mu_{\mathcal{EW}}$ was named \emph{Eierlegende Wollmilchsau} by F.~Herrlich, M.~M\"oller, 
G.~Schmith\"usen \cite{HeSc} because it has peculiar algebro-geometrical properties (see \cite{HeSc}) in addition to its totally degenerate Lyapunov spectrum. In fact, the German term \emph{Eierlegende Wollmilchsau} literally is ``egg-laying wool-milk-sow'' in English and it means ``a tool for several purposes'' (after ``Wiktionary''). The picture below (found on the internet) illustrates the meaning of this German expression: 
\begin{figure}[htb!]
\includegraphics[scale=0.5]{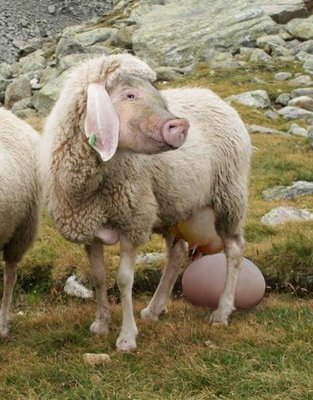}
\end{figure}

\begin{remark}\label{r.isometric-EW} In fact, by examining the above argument, one realizes that it was shown that $B_{\omega_{EW}}^{\mathbb{R}}$ \emph{vanishes} on $H^1_{(0)}(M_3,\mathbb{R})$. In particular, by combining this with Theorem~\ref{t.1stvariation}, one concludes that the KZ cocycle $G_t^{KZ}$ acts by \emph{isometries} (of the Hodge norm) on $H^1_{(0)}(M_3,\mathbb{R})$. Notice that this is \emph{stronger} than simply saying that Lyapunov exponents vanish: indeed, zero Lyapunov exponents \emph{in general} allow to conclude only \emph{subexponential growth (at most)} of norms of vectors, while in this example we observe \emph{no growth at all}! In particular, this illustrates the usefulness of variational formulas (such as Theorem~\ref{t.1stvariation}) for the Hodge norm in the particular context of the KZ cocycle.
\end{remark} 

\subsection{\emph{Ornithorynque}} Let us now construct another $SL(2,\mathbb{R})$-invariant $g_t$-invariant probability measure $\mu_{\mathcal{O}}$ with totally degenerate Lyapunov spectrum in the sense that the KZ cocycle has vanishing second Lyapunov exponent with respect to 
$\mu_{\mathcal{O}}$.

Let us consider the Riemann surface $M_4(x_1,\dots,x_4)$ determined by the \emph{desingularization} of the algebraic equation 
$$y^6=(x-x_1)^3(x-x_2)(x-x_3)(x-x_4)$$
where $x_1,\dots, x_4\in\overline{\mathbb{C}}$ are four distinct points.

The map $p:M_4(x_1,\dots,x_4)\to\overline{\mathbb{C}}$, defined as $p(x,y)=x$, is a covering branched (precisely) at $x_1,\dots,x_4$, and the Galois group of its deck transformations is generated by the automorphism given by the formula $T(x,y) = (x,\varepsilon_6y)$, $\varepsilon_6=\exp(2\pi i/6)$ (and hence it is isomorphic to $\mathbb{Z}/6\mathbb{Z}$). 

In summary, $M_4(x_1,\dots,x_4)$ is (also) a cyclic cover of the Riemann sphere branched at $4$ points. By Riemann-Hurwitz formula applied to $p$, the reader can check that $M_4(x_1, \dots, x_4)$ has genus $4$.

By reasoning similarly to the proofs of Lemmas~\ref{l.EW-stratum} and~\ref{l.EW-locus}, one can show that 

\begin{lemma}\label{l.O-stratum-locus}$(x-x_1)dx/y^3$ is an Abelian differential with $3$ double zeroes at $x_2, x_3, x_4$, i.e., $(M_4(x_1,\dots,x_4),(x-x_1)dx/y^2)\in 
\mathcal{H}(2,2,2)$.

Moreover, by letting $\omega_{O}=c(x_1,\dots,x_4)(x-x_1)dx/y^3$, where $c(x_1,\dots,x_4)\in\mathbb{R}$ is the unique positive real number such that 
the translation surface $(M_4(x_1,\dots,x_4),\omega_{O})$ has \emph{unit area}, one has that  $\mathcal{O}:=\{(M_4(x_1,\dots,x_4),\omega_{O}): x_1,\dots, x_4\in\overline{\mathbb{C}} \textrm{ distinct}\}$ is the $SL(2,\mathbb{R})$-orbit of a square-tiled surface. In particular, $\mathcal{O}$ is a closed $SL(2,\mathbb{R})$-invariant locus of $\mathcal{H}^{(1)}(2,2,2)$.
\end{lemma}

We leave the verification of this lemma as an exercise to the reader. 

\begin{remark} By the classification of connected components of strata of Kontsevich and Zorich~\cite{KZ}, the stratum $\mathcal{H}(2,2,2)$ is \emph{not} connected. Indeed, it has $2$ connected components distinguished by the so-called \emph{parity of the spin structure}. In \cite{MY}, it is shown that $\mathcal{O}$ is contained in the \emph{even spin} connected component of $\mathcal{H}(2,2,2)$ 
(see \cite{KZ} and \cite{MY} for more details). On the other hand, notice that this issue wasn't raised in the case of $\mathcal{EW}$ because the stratum $\mathcal{H}(1,1,1,1)$ is connected.
\end{remark}

By Lemma~\ref{l.O-stratum-locus}, we have an unique $SL(2,\mathbb{R})$-invariant probability measure $\mu_{\mathcal{O}}$ supported on $\mathcal{O}$. 

\begin{theorem}\label{t.FM08} The Lyapunov spectrum of the KZ cocycle with respect to $\mu_{\mathcal{O}}$ is (also) \emph{totally degenerate}, in the sense that 
$$\lambda_2^{\mu_{\mathcal{O}}}=\lambda_3^{\mu_{\mathcal{O}}}
=\lambda_4^{\mu_{\mathcal{O}}}=0\,.$$ 
Furthermore, KZ cocycle acts isometrically (with respect to the Hodge norm) on $H^1_{(0)}(M_4,\mathbb{R})$.
\end{theorem} 

\begin{proof} The argument is similar to the proof of Theorem~\ref{t.F06}. One starts by noticing that the set  
$$\left\{\theta_1=\omega_{O}=c\cdot\frac{(x-x_1)dx}{y^3}, \quad \theta_2:=\frac{(x-x_1)dx}{y^4}, \quad \theta_3:=\frac{(x-x_1)^2dx}{y^5}, \quad \theta_4:=\frac{(x-x_1)^3 dx}{y^5}\right\}$$
is a basis of the space $H^{1,0}(M_4)$ of holomorphic $1$-forms on $M_4=M_3(x_1,\dots,x_4)$ diagonalizing the cohomological action $T^*$ of the automorphism 
$T(x,y)=(x,\varepsilon_6y)$:  
$$T^*(\omega_{O})=-\omega_{O}, \quad T^*(\theta_2)=\varepsilon_6^2\theta_2, \quad T^*(\theta_3)=\varepsilon_6\theta_3, \quad T^*(\theta_4)=\varepsilon_6\theta_4$$

Again, let's denote by $\lambda(n)$ the $T^*$-eigenvalue of $\theta_n$, i.e., $T^*(\theta_n)=\lambda(n)\theta_n$, and let's use the automorphism $T$ to perform a change of variables to compute $B_{\omega_{O}}$:
\begin{eqnarray*}
B_{\omega_{O}}(\theta_n,\theta_m)&:=&\frac{i}{2}\int\frac{\theta_n\theta_m}{\theta_1}\overline{\theta_1} = \frac{i}{2}\int\frac{T^*(\theta_n)T^*(\theta_m)}{T^*(\theta_1)}T^*(\overline{\theta_1}) \\ &=& \lambda(n)\lambda(m)\frac{i}{2}\int\frac{\theta_n\theta_m}{\theta_1}\overline{\theta_1}  =
\lambda(n)\lambda(m) B_{\omega_{O}}(\theta_n,\theta_m)
\end{eqnarray*}

As before, $B_{\omega_{O}}(\theta_n,\theta_m)=0$ whenever $\lambda(n)\lambda(m)\neq 1$, and $\lambda(n)\lambda(m)\neq 1$ for $(n,m)\neq (1,1)$. Since $B_{\omega_O}(\theta_1,\theta_1)=1$, we conclude that the matrix of $B$ in the basis $\{\theta_1,\theta_2,\theta_3,\theta_4\}$ is 
$$B_{\omega_{O}} = \left(\begin{array}{cccc}1&0&0&0\\0&0&0&0\\0&0&0&0\end{array}\right)$$

Thus, $B_{\omega_{O}}$ has rank $1$ and, by Theorem~\ref{t.F02}, one has $\lambda_2^{\mu_{\mathcal{O}}}=\lambda_3^{\mu_{\mathcal{O}}}=\lambda_3^{\mu_{\mathcal{O}}}=0$. Finally, the last statement of the theorem follows from Theorem~\ref{t.1stvariation} because $B_{\omega_{O}}^{\mathbb{R}}$ vanishes on $H^1_{(0)}(M,\mathbb{R})$.
\end{proof} 

In summary, $\mu_{\mathcal{O}}$ is another example answering (negatively) Veech's question. This example was announced in \cite{FM} and it appeared later in \cite{FMZ1}. After a suggestion of Vincent Delecroix and Barak Weiss, the square-tiled surface $(M_4(-1,0,1,\infty),\omega_O)\in\textrm{supp}(\mu_{\mathcal{O}})$ was named \emph{Ornithorynque} (french for Platypus):
\begin{figure}[htb!]
\includegraphics[scale=0.1]{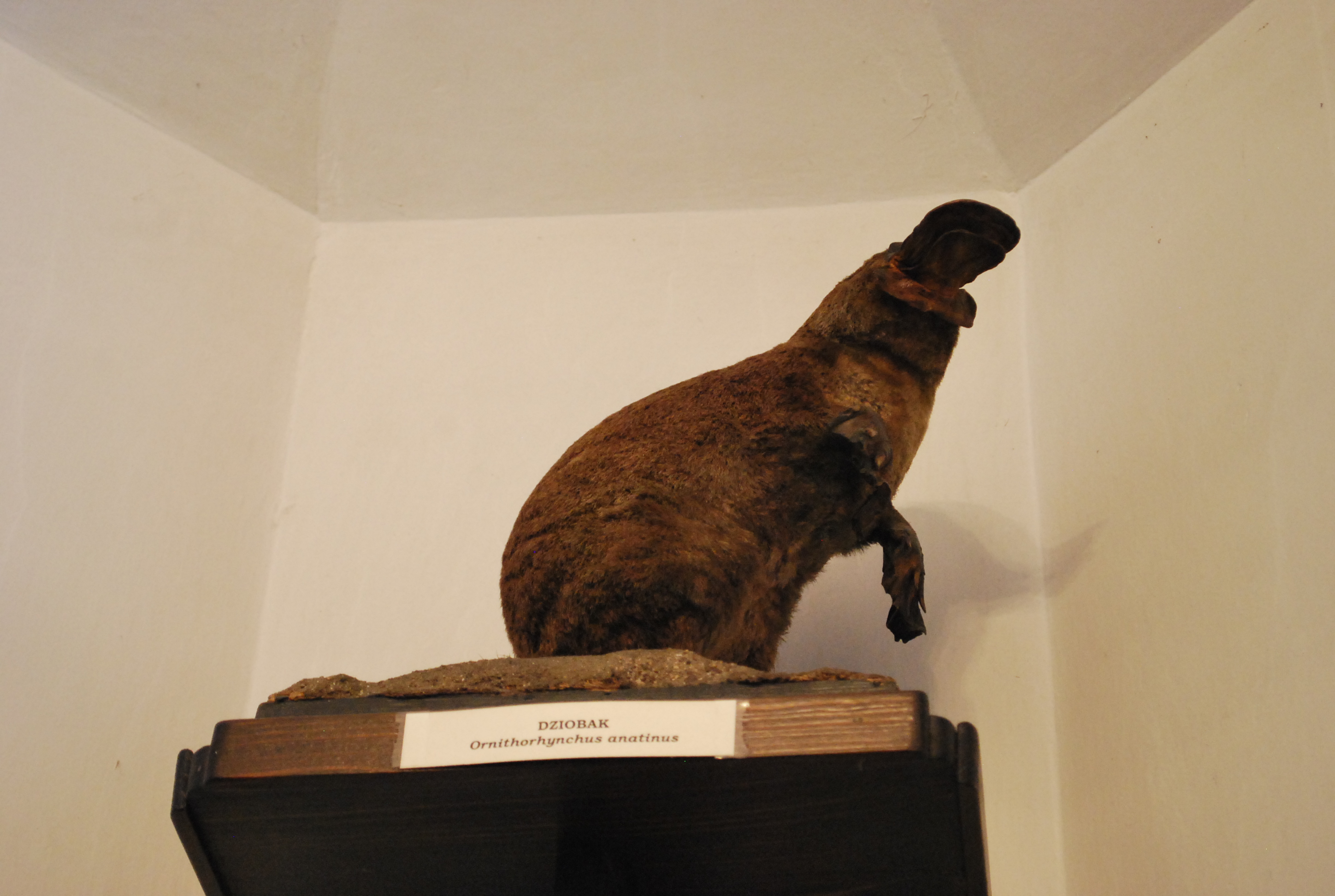}
\end{figure}

In fact, as we'll see later (in Section~\ref{s.FMZI}), \emph{Eierlegende Wollmilchsau} and \emph{Ornithorynque} are members of a class of translation surfaces called \emph{square-tiled cyclic covers}, but let's not insist on this for now. Instead, let's ask: 
\begin{center}
\emph{What about other examples of $SL(2,\mathbb{R})$-invariant $g_t$-ergodic probability measures with totally degenerate Lyapunov spectrum?}
\end{center}

In the rest of this section, we'll see an \emph{almost} complete answer to this question.

\subsection{M. M\"oller's work on Shimura and Teichm\"uller curves}\label{ss.Shimura-Teich} In the sequel, we will need the following notion:

\begin{definition}\label{d.Teich-curve}A \emph{Teichm\"uller curve} in $\mathcal{H}_g$ is a \emph{closed} $SL(2,\mathbb{R})$-orbit. 
\end{definition} 

By a theorem of J.~Smillie \cite{SW}, a $SL(2,\mathbb{R})$-orbit $SL(2,\mathbb{R})\cdot(M,\omega)$ is closed if and only if $(M,\omega)$ is a \emph{Veech surface}: 
\begin{definition}\label{d.Veech-surf} The \emph{Veech group} $SL(M,\omega)$ of a translation surface $(M,\omega)\in\mathcal{H}_g$ is the stabilizer of $(M,\omega)$ with respect to the natural action of $SL(2,\mathbb{R})$ on $\mathcal{H}_g$.

In this language, a \emph{Veech surface} is a translation surface $(M,\omega)$ whose Veech group $SL(M,\omega)$ is a \emph{lattice} in $SL(2,\mathbb{R})$, i.e., $SL(M,\omega)$ has finite covolume in $SL(2,\mathbb{R})$. 
\end{definition}

For example, as we already mentioned, square-tiled surfaces are Veech surfaces because their Veech groups are finite-index subgroups of $SL(2,\mathbb{Z})$. 

The motivation for the terminology \emph{Teichm\"uller curve} comes from the following facts: 
\begin{itemize}
\item let $\pi:\mathcal{H}_g\to\mathcal{M}_g$, $\pi(M,\omega)=M$ be the natural projection from the moduli space of Abelian differentials $\mathcal{H}_g$ to the moduli space of curves $\mathcal{M}_g$; then, the image of closed $SL(2,\mathbb{R})$-orbits under $\pi$ are \emph{complex geodesics} of $\mathcal{M}_g$, i.e., algebraic curves (Riemann surfaces) immersed in $\mathcal{M}_g$ in an \emph{isometric} way with respect to Teichm\"uller metric (cf. Section~\ref{s.intro});
\item \emph{conversely}, all totally geodesic algebraic curves in $\mathcal{M}_g$ are projections of closed $SL(2,\mathbb{R})$-orbits in $\mathcal{H}_g$.
\end{itemize} 

The characterization of closed $SL(2,\mathbb{R})$-orbits via Veech surfaces immediately implies that they support $SL(2,\mathbb{R})$-invariant $g_t$-ergodic probability measures: indeed, the orbit 
$SL(2,\mathbb{R})\cdot(M,\omega)\simeq SL(2,\mathbb{R})/SL(M,\omega)$ carries an unique $SL(2,\mathbb{R})$-invariant $g_t$-ergodic probability whenever $(M,\omega)$ is a Veech surface, since 
$SL(M,\omega)$ is a lattice. In particular, it makes sense to talk about Lyapunov exponents of Veech surfaces: this means simply the Lyapunov exponents of the KZ cocycle with respect to the $SL(2,\mathbb{R})$-invariant probability supported in the corresponding closed $SL(2,\mathbb{R})$-orbit.

In the paper~\cite{Mo}, M.~M\"oller studied the question of classifying \emph{Shimura-Teichm\"uller curves}, i.e., Veech surfaces with totally degenerate Lyapunov spectrum. The name \emph{Shimura-Teichm\"uller curve} is motivated by the fact that Teichm\"uller curves (Veech surfaces) with totally degenerate Lyapunov spectrum have the following algebro-geometrical characterization. Let $\widetilde{\pi}:\mathcal{H}_g\to\mathcal{A}_g$ the natural map obtained by composition of $\pi:\mathcal{H}_g\to\mathcal{M}_g$ with the natural (\emph{Jacobian}) inclusion $\mathcal{M}_g\to\mathcal{A}_g$ of $\mathcal{M}_g$ into the moduli space of \emph{principally polarized Abelian varieties} of dimension $g$. Then, a Teichm\"uller curve $SL(2,\mathbb{R})\cdot(M,\omega)$ has totally degenerate Lyapunov spectrum if and only if $\widetilde{\pi}(SL(2,\mathbb{R})\cdot(M,\omega))$ is \emph{isometric} with respect to the Hodge norm\footnote{This should be compared with Remark~\ref{r.isometric-EW}.}, or equivalently, the family of Jacobians $\widetilde{\pi}(SL(2,\mathbb{R})\cdot(M,\omega))$ has a \emph{fixed part} of (maximal) dimension $g-1$. See the original article~\cite{Mo} for more comments and references.  In this setting, M\"oller~\cite{Mo} proved that:

\begin{theorem}[M.~M\"oller] There are no Shimura--Teichm\"uller curves in genera $g=2$ and $g\geq 6$, while in genera $g=3$ and $g=4$, the only Shimura-Teichm\"uller curves are Eierlegende Wollmilchsau (in genus $g=3$) and the Ornithorynque (in genus $g=4$). 
\end{theorem} 

In other words, this theorem says that \emph{essentially} we know all Shimura-Teichm\"uller curves: there are no other Shimura-Teichm\"uller curves besides Eierlegende Wollmilchsau and Ornithorynque except possibly for some new examples of genus $g=5$. 

Actually, M. M\"oller \cite{Mo} showed that any candidate for Shimura-Teichm\"uller curve in genus $g=5$ must satisfy several constraints (e.g., they must belong to specific strata, etc.). In particular, he \emph{conjectured} that there are no Shimura-Teichm\"uller curves in genus $5$. 

In summary, we have a fairly satisfactory understanding of $SL(2,\mathbb{R})$-invariant probability measures  with totally degenerate Lyapunov exponents \emph{coming from Veech surfaces}. Next, let's consider the analogous question for \emph{more general classes} of $SL(2,\mathbb{R})$-invariant probability measures.

\subsection{Sums of Lyapunov exponents (after A. Eskin, M. Kontsevich $\&$ A. Zorich)}\label{ss.EKZ} In a recent work, A. Eskin, M. Kontsevich and A. Zorich \cite{EKZ} proved a formula (announced 15 years ago) for the sum of Lyapunov exponents of certain $SL(2,\mathbb{R})$-invariant probability measures. In order to state their theorem, we'll need a couple of definitions. 

\subsubsection{Regular affine measures on moduli spaces} A $SL(2,\mathbb{R})$-invariant $g_t$-ergodic probability measure $\mu$ on a stratum 
$\mathcal{H}^{(1)}(\kappa)$ is called \emph{affine} whenever 
\begin{itemize}
\item $\mathbb{R}\cdot\textrm{supp}(\mu)=\{\omega\in\mathcal{H}(\kappa): \frac{1}{\textrm{area}(\omega)}\omega\in\textrm{supp}(\mu)\}$ is an \emph{affine} suborbifold of $\mathcal{H}(\kappa)$ in the sense that it is described by affine subspaces of relative cohomology in local period coordinates (cf. Subsection~\ref{ss.periods} of Section~\ref{s.orbifold} for the definitions);
\item the measure $\nu$ on $\mathcal{H}(\kappa)$ given by $d\nu=da\cdot d\mu$ (where $a(\omega)=\textrm{area}(\omega)$ is the total area function) is equivalent to the Lebesgue measure on the affine suborbifold $\mathbb{R}\cdot\textrm{supp}(\mu)$ (or equivalently, to the Lebesgue measure of the affine subspace representing $\mathbb{R}\cdot\textrm{supp}(\mu)$ in local period coordinates).
\end{itemize} 

It was recently announced\footnote{This result can be thought of a version of \emph{Ratner's theorem} in the \emph{non-homogenous} context of moduli spaces of Abelian differentials.} by A.~Eskin and 
M.~Mirzakhani \cite{EM} that \emph{all} $SL(2,\mathbb{R})$-invariant $g_t$-ergodic probability measures are affine.   

In any event, one of the main goals in \cite{EKZ} was the development of a formula for the sums of Lyapunov exponents of the KZ cocycle with respect to affine measures. 
However, for technical reasons (related to a certain ``integration by parts'' argument), A.~Eskin, M.~Kontsevich and A.~Zorich need a ``regularity'' condition. More precisely, we say that an affine $\mu$ on $\mathcal{H}^{(1)}(\kappa)$ is \emph{regular} if there exists a constant $K>0$ such that 
$$\lim\limits_{\varepsilon\to 0}\frac{\mu(\mathcal{C}_2(K,\varepsilon))}{\varepsilon^2}=0.$$
Here, $\mathcal{C}_2(K,\varepsilon)$ is the set of (unit area) translation surfaces $(M,\omega)\in\mathcal{H}^{(1)}(\kappa)$ possessing two \emph{non-parallel} (flat) maximal cylinders\footnote{See Subsection~\ref{ss.MV-finiteness} of Section~\ref{s.general-dynamics} for the definitions and some comments.} $C_1$, $C_2$ of widths $w(C_i)<\varepsilon$ and heights $h(C_i)>Kw(C_i)$ (i.e., \emph{moduli} $\textrm{mod}(C_i)=h(C_i)/w(C_i)>K$). 

In plain terms, $\mu$ is regular if the probability of seeing \emph{non-parallel} ``very thin and high'' cylinders in translation surfaces in the support of $\mu$ is ``very small''. 

As a matter of fact, all \emph{known} examples of affine measures are regular and it is conjectured in \cite{EKZ} that any affine measure is regular.

\subsubsection{Siegel--Veech constants} The idea of Eskin--Kontsevich--Zorich formula is to express the sum of Lyapunov exponents of a regular affine measures $\mu$ in terms of its \emph{Siegel--Veech constant}, a geometrical quantity that we discuss below. 

\begin{definition} Let $(M,\omega)$ be a translation surface. Given $L>0$, we define 
$$N_{\textrm{area}}(\omega,L)=\sum\limits_{\substack{C \textrm{ maximal horizontal cylinder } \\ \textrm{ of width }w(C)<L}}\frac{\textrm{area}(C)}{\textrm{area}(\omega)}$$
\end{definition}

Informally, $N_{\textrm{area}}(\omega, L)$ \emph{counts} the fraction of the area of the translation surface $(M,\omega)$ occupied by maximal horizontal cylinders of width bounded by $L>0$. 

Of course, the quantity $N_{\textrm{area}}(\omega, L)$ depends  \emph{heavily} on the geometry of $(M,\omega)$ and the real number $L>0$. However, W.~Veech and Ya.~Vorobets discovered that given any $SL(2,\mathbb{R})$-invariant $g_t$-ergodic probability measure $\mu$, the quantity 
$$c(\mu)=\frac{\pi}{3L^2}\int N_{\textrm{area}}(\omega, L) d\mu(\omega)$$
\emph{doesn't} depend on $L>0$. In the literature, $c(\mu)$ is called the \emph{Siegel--Veech constant} of $\mu$. 

\begin{remark} Our choice of normalization of the quantity $\frac{1}{L^2}\int N_{\textrm{area}}(\omega, L) d\mu(\omega)$ leading to the Siegel--Veech constant here is \emph{not} the same of \cite{EKZ}. Indeed, what \cite{EKZ} call Siegel--Veech constant is $3c(\mu)/\pi^2$ in our notation. Of course, there is no conceptual different between these normalizations, but we prefer to take a different convention from \cite{EKZ} because $c(\mu)$ appears more ``naturally'' in the statement of Eskin--Kontsevich--Zorich formula.
\end{remark} 

\begin{remark}\label{r.SV-positive} It is not hard to see from the definition that Siegel--Veech constants $c(\mu)$ are always \emph{positive}, i.e., $c(\mu)>0$ for any $SL(2,\mathbb{R})$-invariant $g_t$-ergodic $\mu$.
\end{remark}

The Siegel--Veech constants of \emph{Masur--Veech measures} were computed by A.~Eskin, H.~Masur and A.~Zorich \cite{EMZ} and they are intimately related to volumes\footnote{Cf. Subsection~\ref{ss.MV-finiteness} of Section~\ref{s.general-dynamics} and Appendix~\ref{s.Eskin-Okounkov}.} 
$\lambda^{(1)}_{\kappa}(\mathcal{H}^{(1)}(\kappa))$ of strata which were calculated by A.~Eskin and 
A.~Okounkov~\cite{EO}. However, we'll not discuss further this interesting topic in these notes because it would lead us too far from the study of the Lyapunov exponents of the KZ cocycle. Instead, we'll conclude our considerations on Siegel--Veech constants by showing the following result (of A.~Eskin, M.~Kontsevich and A.~Zorich \cite{EKZ}) allowing to compute Siegel--Veech constants of measures coming from \emph{square-tiled (Veech) surfaces}. 

Any square-tiled surface $S_0=(M_0,\omega_0)$  comes from a finite covering $(M_0,\omega_0)\to\mathbb{T}^2=(\mathbb{C}/(\mathbb{Z}\oplus i\mathbb{Z}), dz)$ branched only at $0\in \mathbb{T}^2$. Since $SL(2,\mathbb{Z})$ is the stabilizer of $\mathbb{T}^2$ in $SL(2,\mathbb{R})$ (when the periods of $(M_0,\omega_0)$ generate the lattice $\mathbb{Z}\oplus i\mathbb{Z}$), the $SL(2,\mathbb{Z})$-orbit of $(M_0, \omega_0)$ give all square-tiled surfaces in the $SL(2,\mathbb{R})$-orbit of 
$(M_0,\omega_0)$. Moreover, since the Veech group $SL(M_0,\omega_0)$ is a finite-index subgroup of $SL(2,\mathbb{Z})$, one has 
$$SL(2,\mathbb{Z})\cdot (M_0,\omega_0)=\{S_0,S_1,\dots,S_{k-1}\},$$ 
where $k=[SL(2,\mathbb{Z}):SL(M_0,\omega_0)]=\#SL(2,\mathbb{Z})\cdot (M_0,\omega_0)$.

In this context, for each $S_j\in SL(2,\mathbb{Z})\cdot (M_0,\omega_0)$, we write $S_j=\bigcup C_{ij}$ where $C_{ij}$ are the maximal horizontal cylinders of $S_j$, and we denote the width and height of $C_{ij}$ by $w_{ij}$ and $h_{ij}$. 

\begin{theorem}[Theorem 4 of \cite{EKZ}]\label{t.SV-origami} The Siegel--Veech constant of the $SL(2,\mathbb{R})$-invariant $g_t$-ergodic probability supported on the $SL(2,\mathbb{R})$-orbit of the square-tiled surface $(M_0,\omega_0)$ is 
$$\frac{1}{\# SL(2,\mathbb{Z})\cdot (M_0,\omega_0)}\sum\limits_{S_j\in SL(2,\mathbb{Z})\cdot (M_0,\omega_0)}\sum\limits_{S_j=\bigcup C_{ij}}\frac{h_{ij}}{w_{ij}}$$
\end{theorem}

For example, the picture below illustrates the computation of the $SL(2,\mathbb{Z})$-orbit of a $L$-shaped square-tiled surface $(M_0,\omega_0)$ with $3$ squares (shown in the middle of the picture):

\begin{figure}[htb!]
\includegraphics[scale=0.3]{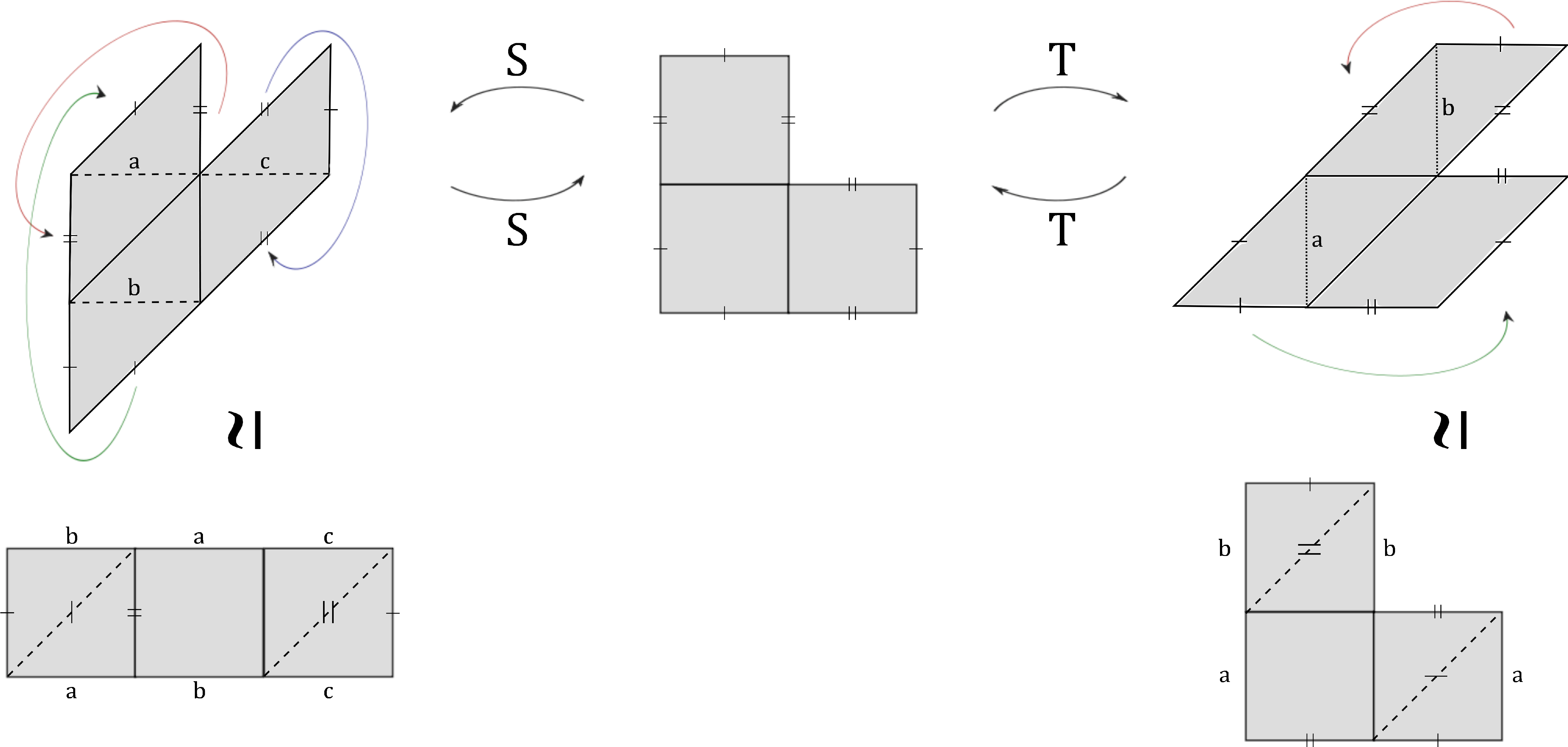}
\end{figure}

Here, we're using the fact that the group $SL(2,\mathbb{Z})$ is generated by the matrices $S=\left(\begin{array}{cc}1&0\\1&1\end{array}\right)$ and 
$T=\left(\begin{array}{cc}1&1\\0&1\end{array}\right)$, so that $SL(2,\mathbb{Z})$-orbits of square-tiled surfaces can be determined by successive applications of $S$ and $T$. 

From the picture we infer that $\#SL(2,\mathbb{Z})\cdot(M_0,\omega_0)=3$ and
\begin{itemize}
\item $(M_0,\omega_0)=C_{10}\cup C_{20}$, where $C_{10}, C_{20}$ are horizontal maximal cylinders with $h(C_{10})=h(C_{20})=1$ and $w(C_{10})=1$, $w(C_{20})=2$;
\item $(M_1,\omega_1):=S\cdot(M,\omega_0)=C_{11}$, where $C_{11}$ is a horizontal cylinder of heigth $1$ and width $3$;
\item $(M_2,\omega_2):=T\cdot(M_0,\omega_0) = C_{12}\cup C_{22}$, where  $C_{12}, C_{22}$ are horizontal maximal cylinders with $h(C_{12})=h(C_{22})=1$ and $w(C_{12})=1$, $w(C_{22})=2$
\end{itemize} 

By plugging this into Theorem~\ref{t.SV-origami}, we get that the Siegel--Veech constant of the $SL(2,\mathbb{R})$-invariant probability supported on $SL(2,\mathbb{R})\cdot(M_0,\omega_0)$ is 
$$\frac{1}{3}\left\{\left(\frac{1}{1}+\frac{1}{2}\right)+\frac{1}{3}+\left(\frac{1}{1}+\frac{1}{2}\right)\right\} = \frac{10}{9}$$

\subsubsection{Statement of Eskin--Kontsevich--Zorich formula and some of its consequences} At this point, we dispose of all elements to state the Eskin--Kontsevich--Zorich formula:
\begin{theorem}[Eskin--Kontsevich--Zorich~\cite{EKZ}]\label{t.EKZ-formula} Let $\mu$ be a regular affine probability measure supported on a stratum $\mathcal{H}^{(1)}(k_1,\dots,k_s)$ of the moduli space of Abelian differentials of genus $g\geq 2$. Then, the sum of the top $g$ Lyapunov exponents of the KZ cocycle with respect to $\mu$ is
$$\lambda_1^{\mu}+\dots+\lambda_g^{\mu}=\frac{1}{12}\sum\limits_{j=1}^s\frac{k_j(k_j+2)}{(k_j+1)}+c(\mu)$$
\end{theorem}

The proof of this fundamental theorem is long\footnote{The article \cite{EKZ} has 106 pages!} and sophisticated, and hence a complete discussion is out of the scope of these notes. Instead, we offer only a \emph{very rough idea} on how the argument goes on. Firstly, one uses the formula (Theorem~\ref{t.F02}) of M.~Kontsevich (and G.~Forni) for sums of Lyapunov exponents to think of $\lambda_1^{\mu}+\dots+\lambda_g^{\mu}$ as a certain integral over the stratum $\mathcal{H}^{(1)}(k_1,\dots, k_s)$. Then, by studying the integral, one can apply an \emph{integration by parts} (``Stokes'') argument to express it as a \emph{main term} and a \emph{boundary term}. At this point, the so-called \emph{Riemann-Roch-Hirzebruch-Grothendieck theorem} allows to compute the \emph{main term} and the outcome (depending \emph{only} on the stratum) is precisely 
$$\frac{1}{12}\sum\limits_{j=1}^s\frac{k_j(k_j+2)}{(k_j+1)}$$
If the (strata of) moduli spaces of Abelian differentials were \emph{compact}, there would be \emph{no} contribution from the \emph{boundary term} and the deduction of the formula would be complete. Of course, $\mathcal{H}^{(1)}(k_1,\dots,k_s)$ is \emph{never} compact, and the contribution of the boundary term is \emph{not} negligible. Here, the study of the geometry of translation surfaces near the boundary of the moduli spaces (and the \emph{regularity} assumption on the probability measure $\mu$ in the statement) plays a crucial role in the proof that the boundary term is given by the Siegel--Veech constant $c(\mu)$ and this completes this crude sketch of the arguments in \cite{EKZ}.

Coming back to the question of studying $SL(2,\mathbb{R})$-invariant measures with totally degenerate Lyapunov spectrum, let's now apply Eskin--Kontsevich--Zorich formula to \emph{rule out} the existence of \emph{regular affine measures} with totally degenerate spectrum in \emph{high} genus:

\begin{proposition}[Corollary 5 of \cite{EKZ}]\label{p.EKZ-corollary} Let $\mu$ be a regular affine probability measure on a stratum $\mathcal{H}^{(1)}(k_1,\dots,k_s)$ of Abelian differentials of genus $g\geq 7$. Then, 
$$\lambda_2^{\mu}>0$$
(and, actually, $\lambda_{[(g-1)g/(6g-3)]+1}^{\mu}>0$).
\end{proposition}

\begin{proof} Since $\lambda_1^{\mu}=1$, it suffices to show that the right-hand side of Eskin--Kontsevich--Zorich formula is $>1$ to get that $\lambda_2^{\mu}>0$, and this follows from the computation
$$\lambda_1^{\mu}+\dots+\lambda_g^{\mu}=\frac{1}{12}\sum\limits_{j=1}^s\frac{k_j(k_j+2)}{(k_j+1)}+c(\mu)\geq \frac{1}{12}\sum\limits_{j=1}^s\frac{k_j(k_j+2)}{(k_j+1)}>\frac{1}{12}\sum\limits_{j=1}^s k_j=\frac{2g-2}{12}\geq 1$$
based on the non-negativity of the Siegel--Veech constant $c(\mu)$ and the assumption $g\geq 7$.
\end{proof}

At this stage, we can summarize this section as follows. The Eierlegende Wollmilchsau and Ornithorynque are two examples answering Veech's question negatively because their  Lyapunov spectra are totally degenerate, but these examples are rare among regular affine measures. 

We close this section with the following remarks.

\begin{remark}\label{r.EKZ-g2} A.~Eskin, M.~Kontsevich and A.~Zorich also proved in \cite{EKZ} a \emph{version} of the their formula for \emph{quadratic differentials}, and derived Siegel--Veech constants of $SL(2,\mathbb{R})$-invariant $g_t$-ergodic probability measures $\mu$ supported in the \emph{hyperelliptic connected components} $\mathcal{H}_{hyp}(2g-2)$ and $\mathcal{H}_{hyp}(g-1,g-1)$ of the strata $\mathcal{H}(2g-2)$ and $\mathcal{H}(g-1,g-1)$. The outcome of these computations is that  Siegel--Veech constants of \emph{any} such $SL(2,\mathbb{R})$-invariant $g_t$-ergodic $\mu$ is given by the formula
$$c(\mu)=\left\{\begin{array}{cc}\frac{g(2g+1)}{3(2g-1)} & \textrm {if }\textrm{supp}(\mu)\subset \mathcal{H}_{hyp}(2g-2) \\ \frac{2g^2+3g+1}{6g} & \textrm {if }\textrm{supp}(\mu)\subset \mathcal{H}_{hyp}(2g-2)\end{array}\right.$$ 
and, hence (by Theorem~\ref{t.EKZ-formula}), in this case, the sum of Lyapunov exponents is 
$$\lambda_1^{\mu}+\dots+\lambda_g^{\mu}=\left\{\begin{array}{cc}\frac{g^2}{(2g-1)} & \textrm {if }\textrm{supp}(\mu)\subset \mathcal{H}_{hyp}(2g-2) \\ \frac{g+1}{2} & \textrm {if }\textrm{supp}(\mu)\subset \mathcal{H}_{hyp}(g-1,g-1)\end{array}\right.$$ 
In particular, since the only two strata $\mathcal{H}(2)$ and $\mathcal{H}(1,1)$ in genus $2$ are hyperelliptic connected components, one has that, for any  $SL(2,\mathbb{R})$-invariant $g_t$-ergodic $\mu$, 
$$\lambda_2^{\mu}=\left\{\begin{array}{cc}1/3 & \textrm {if }\textrm{supp}(\mu)\subset \mathcal{H}(2) \\ 
1/2 & \textrm {if }\textrm{supp}(\mu)\subset \mathcal{H}(1,1)\end{array}\right.$$
because $\lambda_1^{\mu}=1$. These identities  were conjectured by M.~Kontsevich and A.~Zorich
(also on the basis of numerical experiments), and were later proved by M.~Bainbridge~\cite{Ba} a few years before the article \cite{EKZ} became available.
\end{remark}

\begin{remark} \label{remark:Aulicino} In a very recent work, D.~Aulicino \cite{Au} further studied the problem of classifying $SL(2,\mathbb{R})$-invariant measures with totally degenerate spectrum from the point of view of the \emph{Teichm\"uller disks}\footnote{I.e., $SL(2,\mathbb{R})$-orbits. Note that \emph{Teichm\"uller disks} are more general objects than \emph{Teichm\"uller curves} because we \emph{don't} require $SL(2,\mathbb{R})$-orbits to be closed in the definition of the former.} contained in the \emph{rank-one locus}. More precisely, following \cite{F02} and \cite{F06}, we define the \emph{rank}-$k$ \emph{locus} of the moduli space $\mathcal{H}_g$ of Abelian differentials of genus $g$ to be the subvariety  $\mathcal{D}_g(k):=\{\omega\in\mathcal{H}_g: \textrm{rank}(B_{\omega}) \leq k\}$. Note that  $\mathcal{D}_g(1)\subset\dots\subset\mathcal{D}_g(g-1)$. In the literature, the locus $\mathcal{D}_g(g-1)$ is sometimes called \emph{determinant locus} (because $\mathcal{D}_g(g-1)=\{\omega\in\mathcal{H}_g: \det B_{\omega}=0\}$). Observe that these loci are naturally related to the study of Lyapunov exponents of the KZ cocycle: for instance, it is proved in \cite{F02} that $\lambda_2^\mu>0$ (and $\lambda^\mu_{[(g+1)/2]}>0$ for $g\geq 3$) for any $SL(2,\mathbb{R})$-invariant probability measure $\mu$ with $\textrm{supp}(\mu)\not \subset\mathcal{D}_g(g-1)$; also, by Theorem~\ref{t.F02}, any $SL(2,\mathbb{R})$-invariant probability measure $\mu$ with $\textrm{supp}(\mu)\subset\mathcal{D}_g(1)$ has totally degenerate spectrum. 

In his work \cite{Au}, D.~Aulicino showed that there are \emph{no} Teichm\"uller \emph{disks} $SL(2,\mathbb{R})\cdot(M,\omega)$ contained in $\mathcal{D}_g(1)$ for $g=2$ or $g\geq 13$, that the Eierlegende Wollmilchsau and the Ornithorynque are the only Teichm\"uller \emph{disks} contained 
respectively in $\mathcal{D}_3(1)$ and $\mathcal{D}_4(1)$, and, furthermore, if there are no 
Teichm\"uller \emph{curves} contained in $\mathcal{D}_5(1)$, then there are no Teichm\"uller \emph{disks} contained in $\mathcal{D}_g(1)$ for $g\geq 5$. It is worth to point out that Teichm\"uller disks are more general objects than regular affine measures, so that Proposition~\ref{p.EKZ-corollary} doesn't allow to recover immediately the Aulicino's results. However, part of his results follow from Eskin--Kontsevich--Zorich formula and from Eskin--Mirzakhani's work under the conjectural assumption that any affine probability measure is also regular. In fact, by Eskin--Mirzakhani's work \cite{EM} the closure of any Teichm\"uller disk is the support of an affine measure and under the regularity assumption it follows from Eskin--Kontsevich--Zorich formula that the support of such a measure cannot be contained in $\mathcal{D}_g(1)$ for $g\geq 7$. Aulicino's proof is however significantly simpler and more
direct than the alternative approach outlined above. Aulicino's other results (in genus $3$ and $4$
and his conditional result in genera $g=5$ and $g=6$ cannot be derived from Proposition~\ref{p.EKZ-corollary}. In fact, their derivation from the Eskin--Konntsevich--Zorich formula would require general (and effective) lower bounds on the Siegel--Veech constants. 

\end{remark}

\begin{remark} In the next section, we will show that 
\begin{itemize}
\item the Eierlegende Wollmilchsau $(M_3(-1,0,1,\infty),\omega_{EW})\in\mathcal{H}(1,1,1,1)$ can be decomposed into two maximal horizontal cylinders $C_1$, $C_2$ of heights $h(C_1)=h(C_2)=1$ and widths $w(C_1)=w(C_2)=4$, and its Veech group is $SL(2,\mathbb{Z})$;
\item the Ornithorynque $(M_4(-1,0,1,\infty),\omega_O)\in\mathcal{H}(2,2,2)$ can be decomposed into two maximal horizontal cylinders $C_1$, $C_2$ of heights $h(C_1)=h(C_2)=1$ and widths 
$w(C_1)=w(C_2)=6$, and its Veech group is $SL(2,\mathbb{Z})$.
\end{itemize}
See Remark~\ref{r.EW-O-geometry-Veech} below for more details. By plugging these facts into Theorem~\ref{t.SV-origami}, one can compute the Siegel--Veech constants of the measures $\mu_{\mathcal{EW}}$ and $\mu_{\mathcal{O}}$, and then, by Theorem~\ref{t.EKZ-formula}, one can calculate the sum of their Lyapunov exponents. By doing so, one finds:
$$\lambda_1^{\mu_{\mathcal{EW}}}+\lambda_2^{\mu_{\mathcal{EW}}}+\lambda_3^{\mu_{\mathcal{EW}}} = \frac{1}{12}\cdot4\cdot \frac{1\cdot 3}{2} + \frac{1}{1}\cdot \left(\frac{1}{4}+\frac{1}{4}\right)=1$$
and 
$$\lambda_1^{\mathcal{O}}+\lambda_2^{\mathcal{O}}+\lambda_3^{\mathcal{O}}+\lambda_4^{\mathcal{O}} = \frac{1}{12}\cdot 3\cdot \frac{2\cdot 4}{3} + \frac{1}{1}\cdot \left(\frac{1}{6}+\frac{1}{6}\right)=1$$
Since $\lambda_1^{\mu}=1$ for any $g_t$-invariant ergodic $\mu$ on the moduli space of Abelian differentials, one concludes that $\lambda_2^{\mu_{\mathcal{EW}}}=\lambda_2^{\mu_{\mathcal{EW}}}=0$ and $\lambda_2^{\mu_{\mathcal{O}}}=\lambda_3^{\mu_{\mathcal{O}}}=\lambda_4^{\mu_{\mathcal{O}}}=0$, a fact that we already knew from Theorems~\ref{t.F06} and~\ref{t.FM08}.
\end{remark}

\begin{remark}\label{r.EKZ-Sage} Whenever the $SL(2,\mathbb R)$-invariant, $g_t$-ergodic probability measure $\mu$ comes from a square-tiled surface $(M_0,\omega_0)$, the formula in Theorem~\ref{t.SV-origami} for the Siegel--Veech constant $c(\mu)$, combined with Theorem~\ref{t.EKZ-formula}, suggests that one can write down \emph{computer programs} to calculate the sum of Lyapunov exponents. Indeed, as we'll see in Appendix~\ref{s.Eskin-Okounkov}, $(M_0,\omega_0)$ \emph{is determined}\footnote{A pair of permutations $h,v\in S_N$ gives rise to a square-tiled surface with $N$ squares by taking $N$ unit squares $Q_i$, $i=1,\dots, N$, and by gluing (by translations) the rightmost vertical side of $Q_i$ to the leftmost vertical side of $Q_{h(i)}$ and the topmost horizontal side of $Q_i$ to the bottommost horizontal side of $Q_{v(i)}$. Of course, by \emph{renumbering} the squares of a given square-tiled surface we may end up with different pairs of permutations, so that a square-tiled surface determines a pair $h,v\in S_N$ modulo \emph{simultaneous conjugation}, i.e., modulo the equivalence relation $(h',v')\sim(h,v)$ if and only if $h'=\phi h \phi^{-1}$ \emph{and} $v'=\phi v\phi^{-1}$ for some $\phi\in S_N$. Finally, it is possible to check that the action of the matrices $S=\left(\begin{array}{cc}1 & 0 \\ 1 & 1\end{array}\right)$ and $T=\left(\begin{array}{cc}1 & 1 \\ 0 & 1\end{array}\right)$ on square-tiled surfaces corresponds to the action $S(h,v)=(hv^{-1},v)$ and $T(h,v)=(h,vh^{-1})$ on pairs of permutations. Cf. Appendix~\ref{s.Eskin-Okounkov} for more comments on this construction.} 
by a pair of permutations $h,v\in S_N$ modulo simultaneous conjugations, the heights and widths of its horizontal cylinders are determined by the cycles of the permutation $h$ and, under this correspondence, the matrices $S=\left(\begin{array}{cc}1 & 0 \\ 1 & 1\end{array}\right)$ and $T=\left(\begin{array}{cc}1 & 1 \\ 0 & 1\end{array}\right)$ act on pairs of permutations as $S(h,v)=(hv^{-1},v)$ and $T(h,v)=(h,vh^{-1})$. So, by recalling that $S$ and $T$ generate $SL(2,\mathbb{Z})$, we can use the right-hand side of the formula in Theorem~\ref{t.SV-origami} to convert the computation of the Siegel--Veech constant of the $SL(2,\mathbb{R})$-invariant probability measure coming from a (Veech) square-tiled surface $(M_0,\omega_0)$ into a \emph{combinatorial} calculation with pairs of permutations that an adequate computer program can perform. 
In fact, such computer programs for \emph{Mathematica} and \emph{SAGE} were written by, e.g., 
A.~Zorich and V.~Delecroix, and it is likely that they will be publicly available soon. 
\end{remark}


\section{Explicit computation of Kontsevich--Zorich cocycle over two totally degenerate examples}\label{s.MY-JMD}

We saw in the previous section that examples of $SL(2,\mathbb{R})$-invariant $g_t$-ergodic probability measures with totally degenerate spectrum are \emph{rare} and it is likely that there are only two of them coming from two \emph{square-tiled surfaces}, namely, the Eierlegende Wollmilchsau and the Ornithorynque. In this section, we will investigate more closely the Kontsevich--Zorich cocycle over the $SL(2,\mathbb{R})$-orbit of these special examples of square-tiled surfaces.   

\subsection{Affine diffeomorphisms, automorphisms and Veech groups} Let $(M,\omega)$ be a \emph{translation surface}. We denote by $\textrm{Aff}(M,\omega)$ its group of \emph{affine diffeomorphisms}, i.e., the group of (orientation-preserving) homeomorphisms $f$ of $M$ preserving the set $\Sigma$ of zeroes of $\omega$ that are affine in the translation charts (local primitives of $\omega$) of $M-\Sigma$. In translation charts, the linear part (differential) of $f\in\textrm{Aff}(M,\omega)$ is a well-defined matrix $Df\in SL(2,\mathbb{R})$. One obtains in this way a homomorphism:
$$D:\textrm{Aff}(M,\omega)\to SL(2,\mathbb{R})$$ 
The kernel of $D$ is, by definition, the group $\textrm{Aut}(M,\omega)$ of \emph{automorphisms} of $(M,\omega)$, the image of $D$ is, by definition, the \emph{Veech group} $SL(M,\omega)$ of $(M,\omega)$, and it is possible to show that 
$$1\to\textrm{Aut}(M,\omega)\to\textrm{Aff}(M,\omega)\to SL(M,\omega)\to 1$$
is an exact sequence. 

\begin{remark} In fact, we introduced the Veech group in Definition~\ref{d.Veech-surf} above as the stabilizer of $(M,\omega)$ with respect to the action of $SL(2,\mathbb{R})$ on 
$\mathcal{H}_g$. As it turns out, it is possible to show that these definitions coincide. See the survey \cite{HuSc} of P.~Hubert and T.~Schmidt for more details.
\end{remark}

It is possible to show that, in genus $g\geq 2$, the affine group $\textrm{Aff}(M,\omega)$ injects in the modular group $\Gamma_g$, and the stabilizer of the $SL(2,\mathbb{R})$-orbit of $(M,\omega)$ in $\Gamma_g$ is \emph{precisely} $\textrm{Aff}(M,\omega)$ (see \cite{V89}). In particular, since the KZ cocycle is the quotient of the trivial cocycle over Teichm\"uller flow
$$\textrm{diag}(e^t,e^{-t})\times \textrm{id}: \mathcal{TH}_g\times H_1(M,\mathbb{R})\to \mathcal{TH}_g\times H_1(M,\mathbb{R})$$
by the natural action of $\Gamma_g$, we conclude that the KZ cocycle on $SL(2,\mathbb{R})\cdot (M,\omega)$ is the quotient of the trivial cocycle 
$$\textrm{diag}(e^t,e^{-t})\times \textrm{id}: SL(2,\mathbb{R})\cdot (M,\omega)\times H_1(M,\mathbb{R})\to SL(2,\mathbb{R})\cdot (M,\omega)\times H_1(M,\mathbb{R})$$ 
by the natural action of the affine group $\textrm{Aff}(M,\omega)$.  

\begin{remark} Our original definition of the KZ cocycle used \emph{cohomology} groups $H^1(M,\mathbb{R})$ in the fibers instead of \emph{homology} groups 
$H_1(M,\mathbb{R})$. However, since cohomology groups are \emph{dual} to homology groups, there is \emph{no} harm in replacing cohomology by homology in our considerations (as far as, e.g., Lyapunov exponents are concerned).
\end{remark}

\subsection{Affine diffeomorphisms of square-tiled surfaces and Kontsevich--Zorich cocycle}\label{ss.Aff-KZ} Let $(M,\omega)$ be a square-tiled surface, i.e., a finite covering $p:M\to\mathbb{T}^2$ of the torus $\mathbb{T}^2=\mathbb{C}/(\mathbb{Z}\oplus i\mathbb{Z})$ ramified only at $0\in\mathbb{T}^2$ with $\omega=p^*(dz)$. 

We define the subspace $H_1^{(0)}(M,\mathbb{Q})$ to be the kernel of the map $p_\ast: H_1(M, \mathbb Q) \to H_1(\mathbb{T}^2,\mathbb{Q})$ and the subspace $H_1^{st}(M,\mathbb{Q}):=p_\ast^{-1}(H_1(\mathbb{T}^2,\mathbb{Q}))$. It is not hard to show that 
$$H_1(M,\mathbb{Q}) = H_1^{(0)}(M,\mathbb{Q})\oplus H_1^{st}(M,\mathbb{Q})\,.$$

The \emph{standard} subspace $H_1^{st}(M,\mathbb{Q})$ and $H_1^{(0)}(M,\mathbb{Q})$ can be alternatively described as follows. Let $\textrm{Sq}(M,\omega)$ be the set of squares constituting $(M,\omega)$, i.e., $\textrm{Sq}(M,\omega)$ is the set of connected components of $p^{-1}((0,1)^2)$ (where $(0,1)^2\subset\mathbb{R}^2$ is the open unit square inside $\mathbb{T}^2\simeq\mathbb{R}^2/\mathbb{Z}^2$). Put $\Sigma'=p^{-1}(\{0\})$, so that $\Sigma'$ contains the set $\Sigma$ of zeroes of $\omega$, and, for each square $\alpha\in\textrm{Sq}(M,\omega)$, let $\sigma_\alpha\in H_1(M,\Sigma',\mathbb{Z})$ be the cycle going from the bottom left corner of $\alpha$ to the bottom right corner of $\alpha$, and $\zeta_\alpha\in H_1(M,\Sigma',\mathbb{Z})$ be the cycle going from the bottom left corner of $\alpha$ to the top left corner of $\alpha$, as illustrated in the following picture:  

\begin{center}

\begingroup
  \makeatletter
  \providecommand\color[2][]{%
    \errmessage{(Inkscape) Color is used for the text in Inkscape, but the package 'color.sty' is not loaded}
    \renewcommand\color[2][]{}%
  }
  \providecommand\transparent[1]{%
    \errmessage{(Inkscape) Transparency is used (non-zero) for the text in Inkscape, but the package 'transparent.sty' is not loaded}
    \renewcommand\transparent[1]{}%
  }
  \providecommand\rotatebox[2]{#2}
  \ifx\svgwidth\undefined
    \setlength{\unitlength}{120pt}
  \else
    \setlength{\unitlength}{\svgwidth}
  \fi
  \global\let\svgwidth\undefined
  \makeatother
  \begin{picture}(1,0.79662529)%
    \put(0,0){\includegraphics[width=\unitlength]{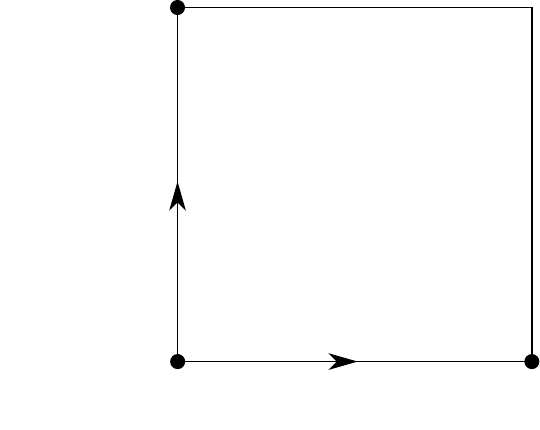}}%
    \put(0.6045947,0.42426361){\color[rgb]{0,0,0}\makebox(0,0)[lb]{\smash{$i$}}}%
    \put(0.22,0.42908678){\color[rgb]{0,0,0}\makebox(0,0)[lb]{\smash{$\zeta_i$}}}%
    \put(0.62,0.04){\color[rgb]{0,0,0}\makebox(0,0)[lb]{\smash{$\sigma_i$}}}%
  \end{picture}%
\endgroup

\end{center}

In this notation, we can form the \emph{absolute cycles} 
$$\sigma=\sum\limits_{\alpha\in\textrm{Sq}(M,\omega)}\sigma_\alpha \quad \textrm{and} \quad \zeta=\sum\limits_{\alpha\in\textrm{Sq}(M,\omega)}\zeta_\alpha,$$ 
and one can show that $H_1^{st}(M,\mathbb{Q})=\mathbb{Q}\sigma\oplus\mathbb{Q}\zeta$, and 
$H_1^{(0)}(M,\mathbb{Q})$ is the symplectic orthogonal (with respect to the intersection form) to 
$H_1^{st}(M,\mathbb{Q})$. 

Note that, from this description, we immediately see several important facts. Firstly,  $H_1^{st}$ is a $2$-dimensional (symplectic) subspace and $H_1^{(0)}$ is a $2g-2$-dimensional (symplectic) subspace. Secondly, the homological action of $\textrm{Aff}(M,\omega)$ on $H_1^{st}$ occurs \emph{via} the standard action of the Veech group $D(\textrm{Aff}(M,\omega)):=SL(M,\omega)\subset SL(2,\mathbb{R})$ on the plane $\mathbb{R}^2\simeq\mathbb{R}\sigma\oplus\mathbb{R}\zeta$. Finally, the homological action of $\textrm{Aff}(M,\omega)$ on $H_1(M,\mathbb{Q})$ preserves the decomposition 
$$H_1(M,\mathbb{Q})=H_1^{(0)}(M,\mathbb{Q})\oplus H_1^{st}(M,\mathbb{Q})\,,$$
since this action is symplectic with respect to the intersection form on homology. 

Let us consider now the restriction of the KZ cocycle to the \emph{closed} orbit of the square-tiled surface $(M,\omega)$ and denote by $\mu$ the $SL(2,\mathbb{R})$-invariant probability supported on $SL(2,\mathbb{R})\cdot (M,\omega)$. By combining the discussion of the previous paragraph with the fact that the KZ cocycle acts on $SL(2,\mathbb{R})\cdot (M,\omega)$ through $\textrm{Aff}(M,\omega)$, we see that (in the present context) the \emph{tautological} Lyapunov exponents $\lambda_1^{\mu}=1$ and $\lambda_{2g}^{\mu}=-1$ of the KZ cocycle (with respect to $\mu$) come from the restriction of the KZ cocycle to the $2$-dimensional symplectic subspace $H_1^{st}$. Therefore, the interesting part $\lambda_2^{\mu}\geq\dots\geq\lambda_g^{\mu}$ of the Lyapunov spectrum of the KZ cocycle with respect to $\mu$ comes from its restriction to the symplectic subspace $H_1^{(0)}$. 

In particular, we ``reduced'' the study of the KZ cocycle over the Eierlegende Wollmilchsau and Ornithorynque to the computation of the homological action of their affine diffeomorphisms on $H_1^{(0)}$. Evidently, it is convenient to get \emph{concrete} models of these square-tiled surfaces because they allow to write down \emph{explicit} bases of $H_1^{(0)}$, so that the action of affine diffeomorphisms can be encoded by concrete \emph{matrices}. 

Keeping this in mind, we start now the description of concrete models of the Eierlegende Wollmilchsau and Ornithorynque surfaces.

\subsection{Combinatorics of square-tiled cyclic covers}\label{ss.square-tiled-cyclic} The Eierlegende Wollmilchsau and the Ornithorynque are naturally included into the following family. Let $N$, 
$0< a_1,\dots,a_4<N$ be non-negative integers such that 
\begin{itemize}
\item[(a)] $\gcd(N, a_1,\dots,a_4)=1$;
\item[(b)] $a_1+\dots+a_4$ is a multiple of $N$, i.e., $a_1+\dots+a_4\in\{N, 2N, 3N\}$; 
\item [(c)] $N$ is even and $a_1,\dots,a_4$ are odd.
\end{itemize}
By varying these parameters, one can form a family $M_N(a_1,\dots,a_4)$ of pairs (Riemann surface, Abelian differential) by taking $x_1,\dots,x_4\in\overline{\mathbb{C}}$ distinct and considering the algebraic equations
$$y^N=(x-x_1)^{a_1}\dots(x-x_4)^{a_4}$$
equipped with the Abelian differential $\omega=(x-x_1)^{b_1}\dots(x-x_4)^{b_4}dx/y^{N/2}$, where $2b_j:=a_j-1$. Here, (a) ensures that the Riemann surfaces are \emph{connected}, (b)  that they are cyclic covers of $\overline{\mathbb{C}}$ branched at $x_1,\dots,x_4$ but \emph{not} branched at 
$\infty$, and (c) that $\omega$ is a \emph{well-defined, holomorphic} $1$-form. 

In this notation, Eierlegende Wollmilchsau is $M_4(1,1,1,1)$ and Ornithorynque is 
$M_6(3,1,1,1)$.

The translation surfaces  $M_N(a_1,\dots,a_4)$ are called \emph{square-tiled cyclic covers}. In the sequel, our discussion will follow closely \cite{FMZ1}.
Note that the square of the Abelian differential $\omega$ is the pull-back of the quadratic differential $q_0=\frac{(dx)^2}{(x-x_1)\dots(x-x_4)}$ on $\overline{\mathbb{C}}$ under the natural projection $p(x,y)=x$. By choosing $(x_1,\dots,x_4)=(-1,0,1,\infty)$, we have\footnote{This is better appreciated by noticing that the square of the Abelian differential $dx/z$ on the elliptic curve (torus) $z^2=(x-x_1)\dots(x-x_4)$ is the pull-back of $q_0$ under $h(x,z)=x$, and $(E(-1,0,1,\infty),dx/z)$ is the flat torus $\mathbb{R}2/\mathbb{Z}^2$ (up to scaling) because the periods $\int_{-\infty}^{-1}\frac{dx}{\sqrt{x^3-x}}=\int_{-1}^0\frac{dx}{\sqrt{x^3-x}}=\int_{0}^1\frac{dx}{\sqrt{x^3-x}}=\int_{1}^{\infty}\frac{dx}{\sqrt{x^3-x}}(=\frac{\Gamma(1/4)^2}{2\sqrt{2\pi}}$).} that the flat structure associated to $q_0$ is given by two flat unit squares glued by their boundaries:
\begin{center}
\begin{figure}[hbt!]
   %
  %
%
\setlength{\unitlength}{100pt}
\begin{picture}(1,0.99815359)
\put(0,0){\includegraphics[scale=0.5]{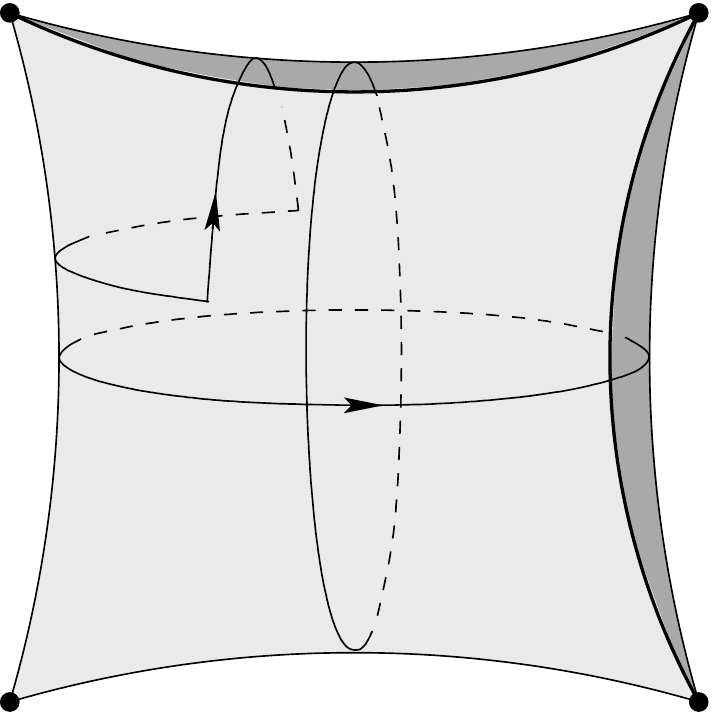}}
\put(-0.08,1){$z_1$}
\put(-0.08,-0.02){$z_2$}
\put(1.03,-0.02){$z_3$}
\put(1.03,1){$z_4$}
\put(0.04,0.9){\small{$A$}}
\put(0.04,0.05){\small{$B$}}
\put(0.9,0.04){\small{$C$}}
\put(0.88,0.9){\small{$D$}}
\put(0.2,0.65){\small{$\sigma_1$}}
\put(0.65,0.4){\small{$\tau_h$}}
\end{picture}
   %
%
   %
\end{figure}
\end{center} 

Now, we can use this concrete description of the flat structure of $q_0$ to obtain a concrete model for $(M_N(a_1,\dots,a_4), \omega)$ as follows. We have $2$ squares tilling the flat model of $q_0$, a white and a black. Since the covering $p:M_N(a_1,\dots,a_4)\to\overline{\mathbb{C}}$, $p(x,y)=x$, has degree $N$, we have that  $(M_N(a_1,\dots,a_4),\omega)$ has $2N$ squares naturally colored white or black. 

Let's take arbitrarily a white square $S_0$ of $M_N(a_1,\dots,a_4)$ and let's number it $0$. Then, we consider the black square $S_1$ adjacent to $0$ via the side $[CD]$ (see the figure above) and we number it $1$. Next, we consider the action of  the automorphism of $M_N(a_1,\dots,a_4)$ generating the Galois group of the covering $p$, that is, the map $T(x,y)=(x,\varepsilon y)$, $\varepsilon=\exp(2\pi ik/N)$, and  we number $2k$ the white square $S_{2k}=T^k(S_0)$ and $2k+1$ the black square $S_{2k+1}=T^k(S_1)$. Here, we take $k$ \emph{modulo} $N$ (so that one may always take  $0\leq k<N$). 

The endpoint of the lift of the path $\tau_h$ (see the figure above) to $M_N(a_1,\dots,a_4)$ is obtained from the starting point of the lifted path by applying $T^{a_1+a_4}=T^{-a_2-a_3}$ (here item (b) above was used). In this way, by moving two squares to the right, we go from the square number $j$ to the square number $j+2(a_1+a_3) (\textrm{mod }N)$. In particular, by successively applying $T^{a_1+a_3}$ we can construct all \emph{horizontal cylinders} of the translation surface 
$(M_N(a_1,\dots,a_4),\omega)$. 

Similarly, we can deduce neighbors in the vertical direction by using small (positively oriented) paths $\sigma_i$ encircling $z_i$ (see the picture above for $\sigma_1$). Indeed, since the extremal points of the lift of $\sigma_i$ to $M_N(a_1,\dots,a_4)$ differ by $T^{a_i}$, by going around a corner (in the counterclockwise sense) of the square numbered 
$j$, we end up in the square numbered $j+2a_i (\textrm{mod }N)$. 

These ``local moves'' obtained by lifting $\tau_h$ and $\sigma_1, \sigma_1^{-1}, \sigma_2,\dots$ are described in a nutshell in the following picture:

\begin{figure}[htb]
  %
  %
\begin{picture}(0,0)(100,0)

\begin{picture}(0,0)(-100,0)
\put(-163,-165.5){\includegraphics[scale=1.2]{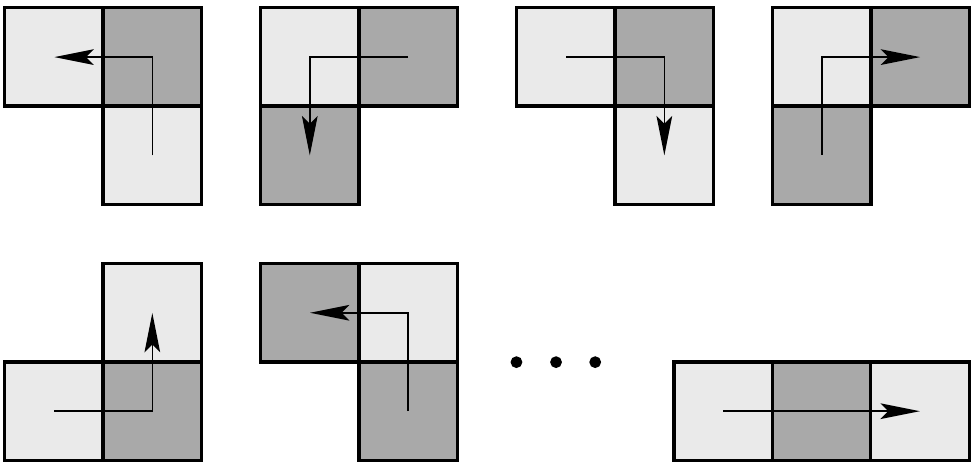}}%
\end{picture}

\begin{picture}(0,0)(0,0)
\put(-55,-63){$\sigma_1$}
\put(-38,-48){\tiny\textit A}
\put(-18,-66){\scriptsize $j$}
\put(-62,-20){\scriptsize $j\!+\!2a_1$}
\end{picture}
\begin{picture}(0,0)(-100,0)
\put(-32,-63){$\sigma_1$}
\put(-45,-48){\tiny\textit A}
\put(-76,-66){\scriptsize $j\!+\!2a_1$}
\put(-28,-25){\scriptsize $j$}
\end{picture}
\begin{picture}(0,0)(-171,0)
\put(-57,-63){$\sigma^{-1}_1$}
\put(-37.5,-48){\tiny\textit A}
\put(-27,-66){\scriptsize $j\!-\!2a_1$}
\put(-54,-25){\scriptsize $j$}
\end{picture}
\begin{picture}(0,0)(-271,0)
\put(-32,-63){$\sigma^{-1}_1$}
\put(-45,-48){\tiny\textit A}
\put(-65,-66){\scriptsize $j$}
\put(-40,-20){\scriptsize $j\!-\!2a_1$}
\end{picture}
\begin{picture}(0,0)(0,88)
\put(-67,-25){$\sigma_2$}
\put(-51,-40.5){\tiny\textit B}
\put(-66,-60){\scriptsize $j$}
\put(-40,-23){\scriptsize $j\!+\!2a_2$}
\end{picture}
\begin{picture}(0,0)(-100,89)
\put(-82,-60){$\sigma_2$}
\put(-65,-47){\tiny\textit B}
\put(-44,-66){\scriptsize $j$}
\put(-89,-20){\scriptsize $j\!+\!2a_2$}
\end{picture}
\begin{picture}(0,0)(-200,88)
\put(-7,-25){$\tau_h$}
\put(-55,-40){\tiny\textit A}
\put(-20,-40){\tiny\textit D}
\put(14,-40){\tiny\textit A}
\put(48,-40){\tiny\textit D}
\put(-56,-82){\tiny\textit B}
\put(-22,-82){\tiny\textit C}
\put(13,-82){\tiny\textit B}
\put(47,-82){\tiny\textit C}
\put(-42,-60){\scriptsize $j$}
\put(17,-53){\scriptsize $j\!\!+\!2(\!a_1\!+$}
\put(27,-66){\scriptsize $+a_4)$}
\end{picture}
\end{picture}
\vspace{180bp}
%
   %
\end{figure}

Of course, the picture above makes clear that we can extract concrete square-tiled models for $M_N(a_1,\dots,a_4)$. 
For example, it is an instructive exercise for the reader to apply this procedure to the Eierlegende Wollmilchsau $M_4(1,1,1,1)$ and the Ornithorynque $M_6(3,1,1,1)$, and check that the pictures one gets for them are the following:

\begin{figure}[htb!]
\includegraphics{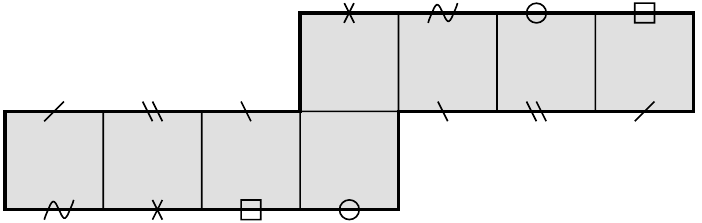}
\caption{Eierlegende Wollmilchsau $M_4(1,1,1,1)$}\label{f.eier}
\end{figure}

\begin{figure}[htb!]
  %
  %
\begin{picture}(0,0)(100,0)
\begin{picture}(0,0)(0,0.5)
\put(8,-312){\includegraphics[scale=0.7]{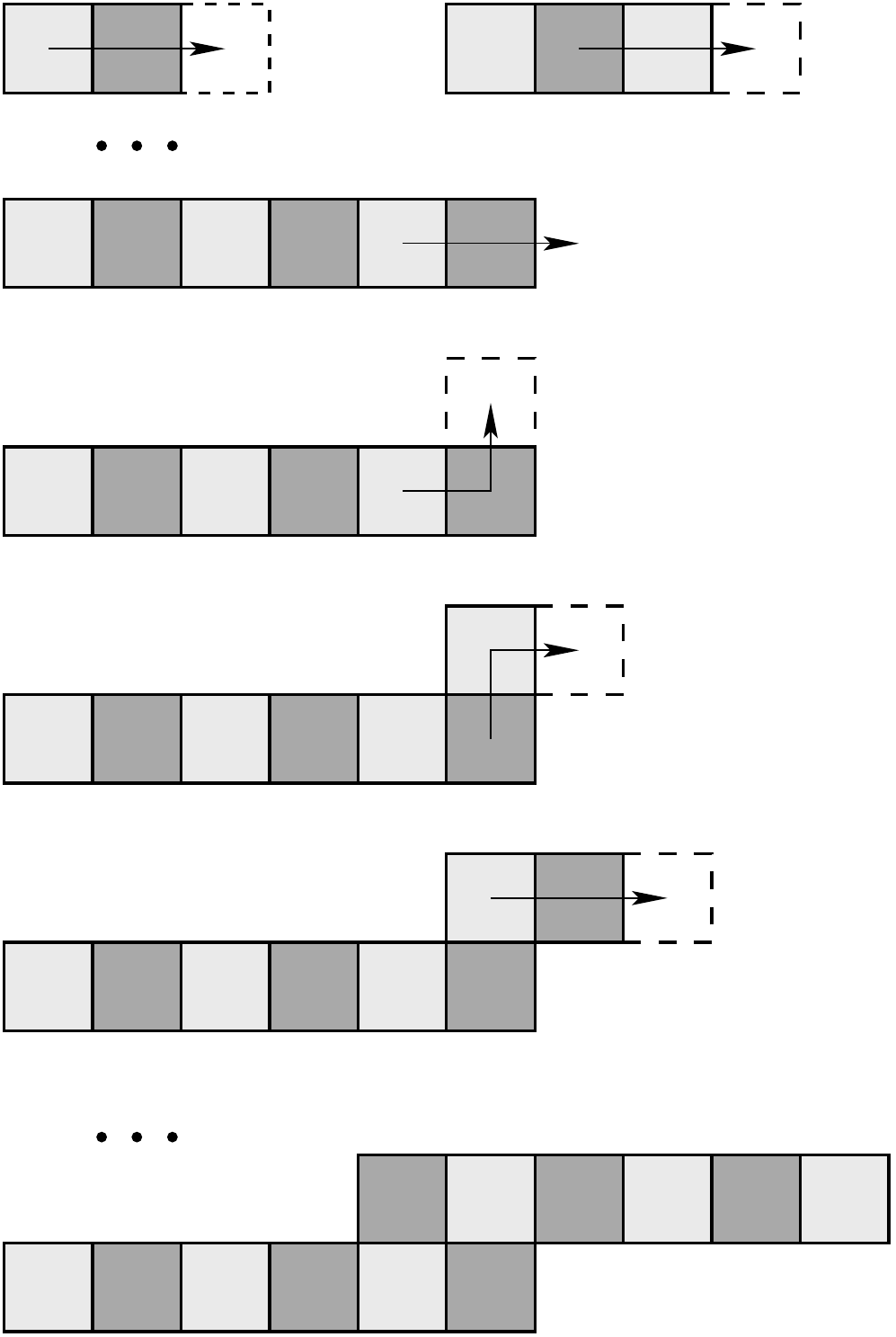}}%
\end{picture}
\begin{picture}(0,0)(0,0.5)
\put(-25,-23){$\tau_h$}
\put(10,-25){0}
\put(30,-25){1}
\put(50,-25){8}
\put(110,-25){0}
\put(130,-25){1}
\put(150,-25){8}
\put(170,-25){9}
\put(1,-10){\tiny\textit A}
\put(21,-10){\tiny\textit D}
\put(41,-10){\tiny\textit A}
\put(1,-37){\tiny\textit B}
\put(21,-37){\tiny\textit C}
\put(41,-37){\tiny\textit B}
\put(101,-10){\tiny\textit A}
\put(121,-10){\tiny\textit D}
\put(141,-10){\tiny\textit A}
\put(161,-10){\tiny\textit D}
\put(101,-37){\tiny\textit B}
\put(121,-37){\tiny\textit C}
\put(141,-37){\tiny\textit B}
\put(161,-37){\tiny\textit C}
\end{picture}
\begin{picture}(0,0)(2.5,44.5) 
\put(-25,-23){$\tau_h$}
\put(10,-25){0}
\put(30,-25){1}
\put(50,-25){8}
\put(70,-25){9}
\put(90,-25){4}
\put(110,-25){5}
\put(130,-25){0}
\put(150,-25){Closed up a cylinder}
\put(1,-10){\tiny\textit A}
\put(21,-10){\tiny\textit D}
\put(41,-10){\tiny\textit A}
\put(61,-10){\tiny\textit D}
\put(81,-10){\tiny\textit A}
\put(101,-10){\tiny\textit D}
\put(121,-10){\tiny\textit A}
\put(1,-37){\tiny\textit B}
\put(21,-37){\tiny\textit C}
\put(41,-37){\tiny\textit B}
\put(61,-37){\tiny\textit C}
\put(81,-37){\tiny\textit B}
\put(101,-37){\tiny\textit C}
\put(121,-37){\tiny\textit B}
\end{picture}
\begin{picture}(0,0)(5,100)
\put(-25,-23){$\sigma_4$}
\put(10,-25){0}
\put(30,-25){1}
\put(50,-25){8}
\put(70,-25){9}
\put(90,-25){4}
\put(110,-25){5}
\put(106,-5){10}
\put(1,-10){\tiny\textit A}
\put(21,-10){\tiny\textit D}
\put(41,-10){\tiny\textit A}
\put(61,-10){\tiny\textit D}
\put(81,-10){\tiny\textit A}
\put(101,-10){\tiny\textit D}
\put(121,-10){\tiny\textit A}
\put(1,-37){\tiny\textit B}
\put(21,-37){\tiny\textit C}
\put(41,-37){\tiny\textit B}
\put(61,-37){\tiny\textit C}
\put(81,-37){\tiny\textit B}
\put(101,-37){\tiny\textit C}
\put(121,-37){\tiny\textit B}
\end{picture}
\begin{picture}(0,0)(8,156) 
\put(-25,-23){$\sigma^{-1}_1$}
\put(10,-25){0}
\put(30,-25){1}
\put(50,-25){8}
\put(70,-25){9}
\put(90,-25){4}
\put(110,-25){5}
\put(106,-5){10}
\put(130,-5){3}
\put(1,-10){\tiny\textit A}
\put(21,-10){\tiny\textit D}
\put(41,-10){\tiny\textit A}
\put(61,-10){\tiny\textit D}
\put(81,-10){\tiny\textit A}
\put(103,-10){\tiny\textit D}
\put(123,-10){\tiny\textit A}
\put(1,-37){\tiny\textit B}
\put(21,-37){\tiny\textit C}
\put(41,-37){\tiny\textit B}
\put(61,-37){\tiny\textit C}
\put(81,-37){\tiny\textit B}
\put(101,-37){\tiny\textit C}
\put(121,-37){\tiny\textit B}
\end{picture}
\begin{picture}(0,0)(12,211) 
\put(-25,-23){$\tau_h^{-1}$}
\put(10,-25){0}
\put(30,-25){1}
\put(50,-25){8}
\put(70,-25){9}
\put(90,-25){4}
\put(110,-25){5}
\put(106,-5){10}
\put(130,-5){3}
\put(150,-5){2}
\end{picture}
\begin{picture}(0,0)(14,279.5) 
\put(10,-25){0}
\put(30,-25){1}
\put(50,-25){8}
\put(70,-25){9}
\put(90,-25){4}
\put(110,-25){5}
\put(86,-5){11}
\put(106,-5){10}
\put(130,-5){3}
\put(150,-5){2}
\put(170,-5){7}
\put(190,-5){6}
\put(-2,-24){\tiny 5}
\put(11,-10){\tiny 7}
\put(31,-10){\tiny 6}
\put(51,-10){\tiny 3}
\put(71,-10){\tiny 2}
\put(78,-3){\tiny 6}
\put(91,10){\tiny 8}
\put(11,-37){\tiny 3}
\put(28,-37){\tiny 10}
\put(48,-37){\tiny 11}
\put(71,-37){\tiny 6}
\put(91,-37){\tiny 7}
\put(111,-37){\tiny 2}
\put(123,-24){\tiny 0}
\put(111,10){\tiny 1}
\put(131,10){\tiny 0}
\put(151,10){\tiny 5}
\put(171,10){\tiny 4}
\put(191,10){\tiny 9}
\put(131,-17){\tiny 8}
\put(151,-17){\tiny 9}
\put(171,-17){\tiny 0}
\put(191,-17){\tiny 1}
\put(202,-3){\tiny 11}
\end{picture}
\end{picture}
\vspace{320bp} 
\caption{
\label{fig:cartoon}
Cartoon movie construction of $M_6(1,1,1,3)$.}
\end{figure}
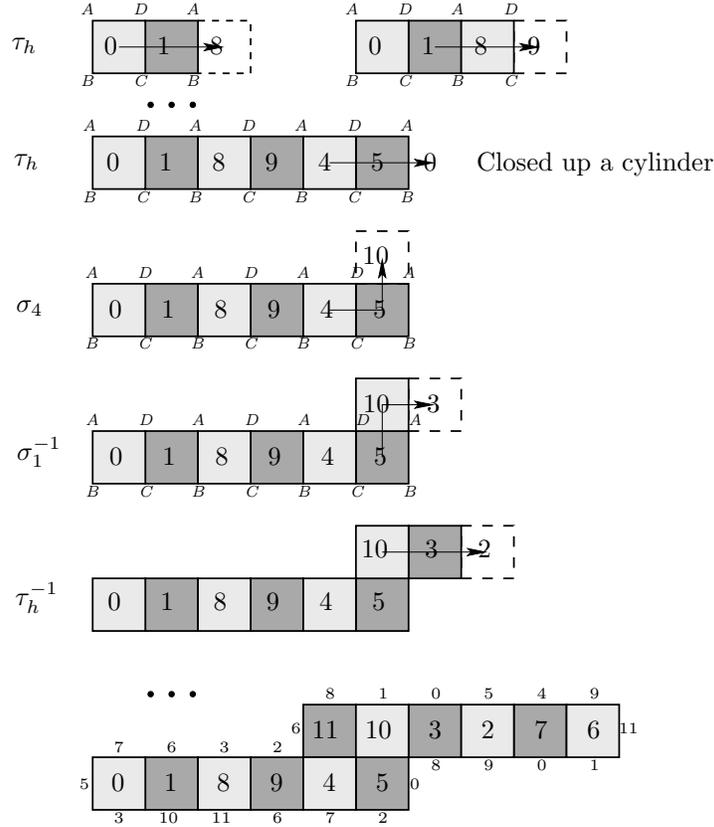

\begin{remark}\label{r.EW-O-geometry-Veech} From Figure~\ref{f.eier}, we see that the Eierlegende Wollmilchsau can be decomposed into two maximal horizontal cylinders, both of height $1$ and width $4$. Similarly, from Figure~\ref{fig:cartoon}, we see that the Ornithorynque can be decomposed into two maximal horizontal cylinders, both of height $1$ and width $6$. Moreover, by applying the matrices $S=\left(\begin{array}{cc}1&0\\1&1\end{array}\right)$ and $T=\left(\begin{array}{cc}1&1\\0&1\end{array}\right)$ to the two figures above, and by using adequate elements of the modular group to cut and paste the resulting objects, the reader can verify that $S$ and $T$ stabilize both the Eierlegende Wollmilchsau and Ornithorynque, i.e., $S$ and $T$ belong to the Veech group of them. Since $S$ and $T$ generate $SL(2,\mathbb{Z})$, this shows that the Veech group of both Eierlegende Wollmilchsau and Ornithorynque is $SL(2,\mathbb{Z})$.
\end{remark}

Once we dispose of these concrete models for Eierlegende Wollmilchsau and Ornithorynque (and more generally square-tiled cyclic covers), it is time use them to produce nice bases of their homology groups.

\subsection{\emph{Eierlegende Wollmilchsau} and the quaternion group} By carefully looking at Figure~\ref{f.eier}, the reader can verify that the Eierlegende Wollmilchsau admits the presentation given
below in Figure~\ref{f.eier-q}.

\begin{figure}
\includegraphics[scale=0.3]{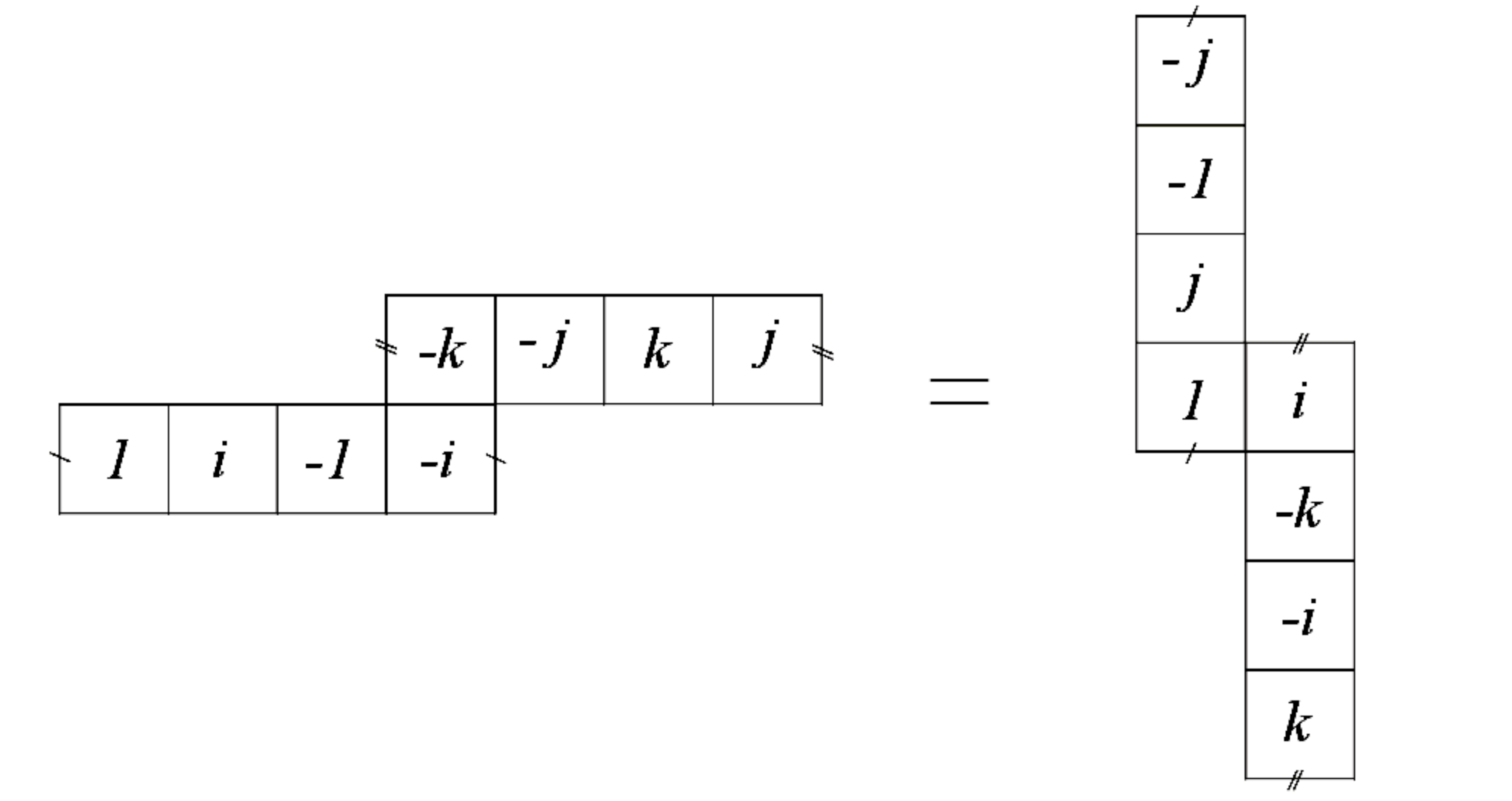}
\caption{Labelling squares of Eierlegende Wollmilchsau via the quaternion group.}\label{f.eier-q}
\end{figure}

In such a presentation the squares of the Eierlegende Wollmilchsau are labelled via the elements of the \emph{quaternion group}\footnote{The multiplication rules are $i^2=j^2=k^2=-1$ and $ij=k$.} $Q=\{\pm1,\pm i,\pm j,\pm k\}$. A great advantage of this presentation is the fact that one can easily obtain the neighbors of squares by \emph{right} multiplication by $\pm i$ or $\pm j$: indeed, given a square $g\in Q$, its neighbor to the right is the square $g\cdot i$ and its neighbor on top is $g\cdot j$. In this way, we can identify the group of \emph{automorphisms} $\textrm{Aut}(M_{EW},\omega_{EW})$ of the Eierlegende Wollmilchsau $(M_{EW},\omega_{EW})$ with the quaternion group $Q$:  any element 
$h\in Q$  of the quaternion group is associated to the automorphism sending the square $g$ to the square $h\cdot g$ obtained from $g$ by \emph{left} multiplication by $h$.

Also, the zeroes of $\omega_{EW}$ are located at the left bottom corners of the squares. Since the left bottom corners of the squares $g$ and $-g$ are identified together, the set $\Sigma_{EW}$ of zeroes of $\omega$ is naturally identified with the group $\overline{Q}=Q/\{\pm1\}=\{\overline{1},\overline{i}, \overline{j}, \overline{k}\}$. Note that $\overline{Q}$ is isomorphic to Klein's group $\mathbb{Z}/2\mathbb{Z}\times \mathbb{Z}/2\mathbb{Z}$. 

In what follows, we'll follow closely the work~\cite{MY} of C.~Matheus and J-C.~Yoccoz and compute the homological action of the affine group of the Eierlegende Wollmilchsau on  the cycles $\sigma_g, \zeta_g\in H_1(M_{EW},\Sigma_{EW},\omega_{EW})$ introduced in Subsection~\ref{ss.Aff-KZ} 
above, i.e., 
\begin{figure}[htb!]
\includegraphics[scale=0.15]{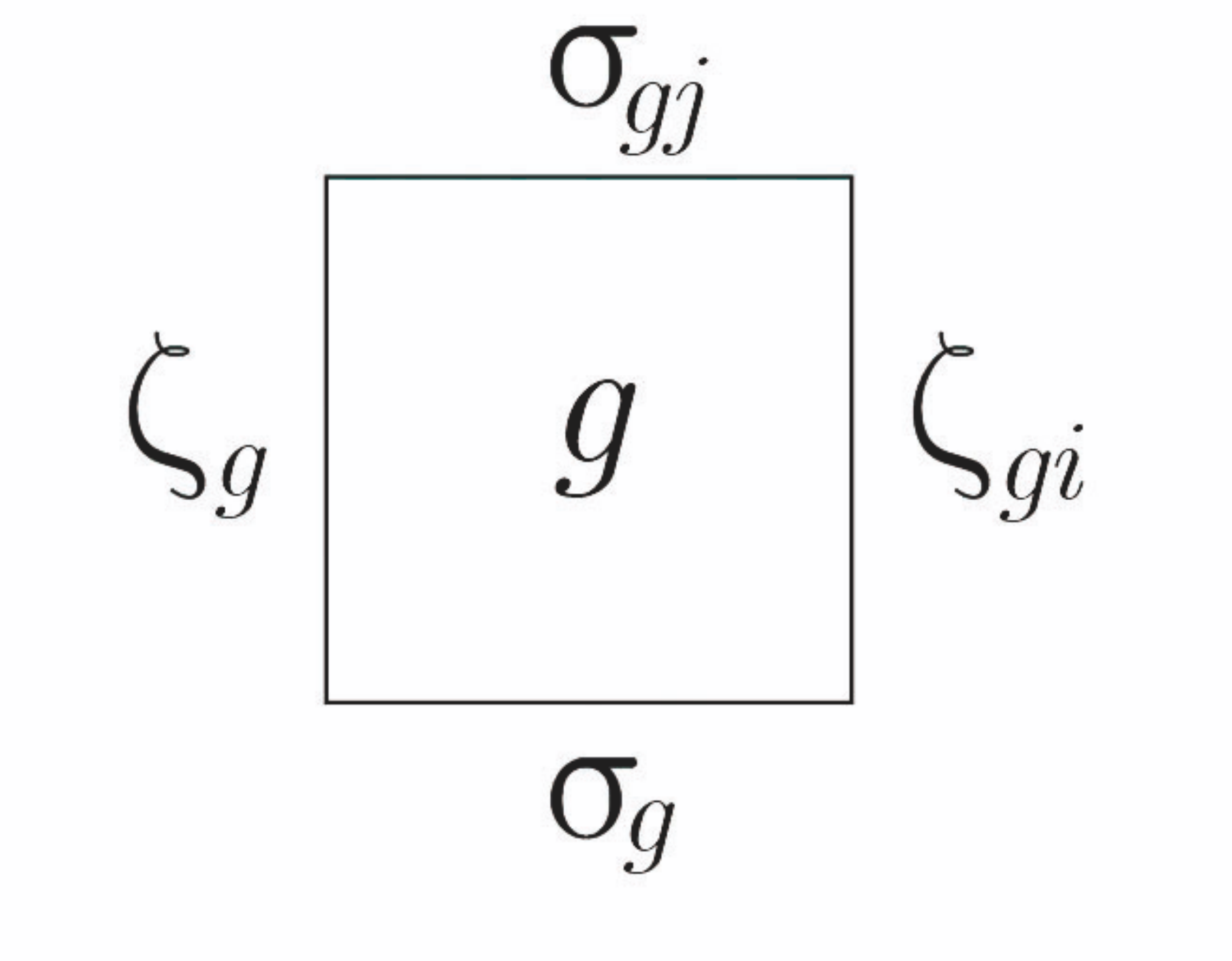}
\end{figure}

\begin{remark} Note that, from this picture, we have that $\sigma_g+\zeta_{gi}-\sigma_{gj}-\zeta_g=0$ in homology. We'll systematically use this relation in the sequel.
\end{remark}

However, before rushing to the study of the \emph{whole} action of the \emph{affine} group, let's first investigate the action of the group of \emph{automorphisms}. The automorphism group of the Eierlegende Wollmilchsau is isomorphic to $Q$. In particular, one can select a nice basis of the homology of the Eierlegende Wollmilchsau using the \emph{representation theory} of $Q$. More precisely, we know $Q$ has $4$ irreducible $1$-dimensional representations $\chi_1, \chi_i, \chi_j, \chi_k$ and $1$ irreducible $2$-dimensional representation $\chi_2$. They can be seen inside the \emph{regular representation} of $Q$ in $\mathbb{Z}(Q)$ via the submodules generated as follows:
$$\chi_1: [1]+[-1]+[i]+[-i]+[j]+[-j]+[k]+[-k]\,,$$
$$\chi_i: [1]+[-1]+[i]+[-i]-[j]-[-j]-[k]-[-k]\,,$$
$$\chi_j: [1]+[-1]-[i]-[-i]+[j]+[-j]-[k]-[-k]\,,$$
$$\chi_k: [1]+[-1]-[i]-[-i]-[j]-[-j]+[k]+[-k]\,,$$
$$\chi_2: [1]-[-1],\, [i]-[-i],\, [j]-[-j],\, [k]-[-k]\,.$$
Thus, we see that the character table of $Q$ is
\begin{center}
{\renewcommand{\arraystretch}{1.5}
\renewcommand{\tabcolsep}{0.2cm}
\begin{tabular}{|c|c|c|c|c|c|}
\hline
 & $1$ & $-1$ & $\pm i$ & $\pm j$ & $\pm k$  \\
\hline
$\chi_1$ & $1$ & $1$ & $1$ & $1$ & $1$ \\
\hline
 $\chi_i$ & $1$ & $1$ & $1$ & $-1$ & $-1$\\
\hline
$\chi_j$ & $1$ & $1$ & $-1$ & $1$ & $-1$\\
\hline
$\chi_k$ & $1$ & $1$ & $-1$ & $-1$ & $1$\\
\hline
tr $\chi_2$ & $2$ & $-2$ & $0$ & $0$ & $0$\\
\hline
\end{tabular}}
\end{center}
\medskip

This motivates the introduction of the following \emph{relative} cycles:
$$w_i=\zeta_{1}+\zeta_{-1}+\zeta_{i}+\zeta_{-i}-\zeta_{j}-\zeta_{-j}-\zeta_{k}-\zeta_{-k}\,,$$
$$w_j=\sigma_{1}+\sigma_{-1}+\sigma_{j}+\sigma_{-j}-\sigma_{i}-\sigma_{-i}-\sigma_{k}-\sigma_{-k}\,,$$
\begin{eqnarray*}
w_k&=&\zeta_{1}+\zeta_{-1}+\zeta_{k}+\zeta_{-k}-\zeta_{i}-\zeta_{-i}-\zeta_{j}-\zeta_{-j} \\ 
&=&\sigma_{1}+\sigma_{-1}+\sigma_{k}+\sigma_{-k}-\sigma_{i}-\sigma_{-i}-\sigma_{j}-\sigma_{-j}\,.
\end{eqnarray*}
\begin{remark}Note that 
$$\zeta_{1}+\zeta_{-1}+\zeta_{j}+\zeta_{-j}+\zeta_{j}-\zeta_{-i}-\zeta_{k}-\zeta_{-k} = \sigma_{1}+\sigma_{-1}+\sigma_{i}+\sigma_{-i}-\sigma_{j}-\sigma_{-j} 
-\sigma_{k}-\sigma_{-k} = 0$$
in homology.
\end{remark}

Indeed, we have that the cycle $w_i$ is relative because its boundary is $\partial w_i=4(\overline{j}+\overline{k}-\overline{i}-\overline{1})$ and the action of an automorphism $g\in Q$ is $g\cdot w_i = \chi_i(g)\cdot w_i$ (see Figure~\ref{f.eier-q}). Also, one has similar formulas for $w_j$ and $w_k$, so that the action of $Q=\textrm{Aut}(M_{EW},\omega_{EW})$ on the subspace  $H_{rel}:=\mathbb{Q}w_i\oplus\mathbb{Q}w_j\oplus\mathbb{Q}k$ is fairly well-understood. 
Observe that $H_{rel}$ is a \emph{relative subspace} in the sense that it is a \emph{complement} of the absolute homology group $H_1(M_{EW},\mathbb{Q})$ in $H_1(M_{EW},\Sigma_{EW},\mathbb{Q})$.

Next, we consider the following \emph{absolute cycles}:
$$\sigma:=\sum\limits_{g\in Q}\sigma_g, \quad \zeta:=\sum\limits_{g\in Q}\zeta_g,$$
$$\widehat{\sigma}_g:=\sigma_{g}-\sigma_{-g}, \quad \widehat{\zeta}_g:=\zeta_{g}-\zeta_{-g},$$
$$\varepsilon_g:=\widehat{\sigma}_g-\widehat{\sigma}_{gj}=\widehat{\zeta}_g-\widehat{\zeta}_{gi}$$

In the notation introduced in Subsection~\ref{ss.Aff-KZ}, one can check that $H_1^{st}(M_{EW},\mathbb{Q})=\mathbb{Q}\sigma\oplus\mathbb{Q}\zeta$ and $H_1^{(0)}(M_{EW},\mathbb{Q})$ is spanned by the cycles $\widehat{\sigma}_g$, $\widehat{\zeta}_g$. Here, we notice that, since the Eierlegende Wollmilchsau has genus $g=3$, i.e., the homology  $H_1(M_{EW},\mathbb{Q})$ has dimension $2g=6$, the subspace $H_1^{st}(M_{EW},\mathbb{Q})$ has dimension $2$ and 
$$H_1(M_{EW},\mathbb{Q}) = H_1^{st}(M_{EW},\mathbb{Q})\oplus H_1^{(0)}(M_{EW},\mathbb{Q})\,,$$
 it follows that $H_1^{(0)}(M_{EW},\mathbb{Q})$ has dimension $4$.

For later use, we observe that $\widehat{\sigma}_{-g}=-\widehat{\sigma}_{g}$, $\widehat{\zeta}_{-g}=-\widehat{\zeta}_{g}$, $\varepsilon_{-g}=-\varepsilon_g$ and 
$$\widehat{\sigma}_g=\frac{1}{2}(\varepsilon_{g}+\varepsilon_{gj}), \quad \widehat{\zeta}_g=\frac{1}{2}(\varepsilon_{g}+\varepsilon_{gi})\,.$$ 

Therefore, one can write $H_1^{(0)}(M_{EW},\mathbb{Q})$ (in several ways) as a sum of two copies of the $2$-dimensional irreducible $Q$-representation $\chi_2$. 

Finally, we observe that, for any $g\in Q\simeq\textrm{Aut}(M_{EW},\omega_{EW})$, 
$$g\cdot \sigma=\sigma, \quad g\cdot\zeta=\zeta$$
and
$$g\cdot\widehat{\sigma}_h=\widehat{\sigma}_{gh}, \quad g\cdot\widehat{\zeta}_h=\widehat{\zeta}_{gh}, \quad g\cdot\varepsilon_h=\varepsilon_{gh}$$ 
See Figure~\ref{f.eier-q}. Hence, the action of the group $Q=\textrm{Aut}(M_{EW},\omega_{EW})$ on the absolute homology $H_1(M_{EW},\mathbb{Q})=H_1^{st}(M_{EW},\mathbb{Q})\oplus H_1^{(0)}(M_{EW},\mathbb{Q})$ is also fairly well-understood. 

In summary, by looking at the representation theory of the (finite) group of automorphisms $Q$ of the Eierlegende Wollmilchsau, we selected a nice generating set of the relative cycles $w_i, w_j, w_j$ and the absolute cycles $\sigma,\zeta, \widehat{\sigma}_g, \widehat{\zeta}_g$ such that the action of $Q$ is easily computed. 

After this first (preparatory) step of studying the homological action of $\textrm{Aut}(M_{EW},\omega_{EW})\simeq Q$, we are ready to investigate the homological action of $\textrm{Aff}(M_{EW},\omega_{EW})$.

\subsection{The action of the affine diffeomorphisms of the \emph{Eierlegende Wollmilchsau}} Let us denote by $\textrm{Aff}_{(1)}(M_{EW},\omega_{EW})$ the subgroup of $\textrm{Aff}(M_{EW},\omega_{EW})$ consisting of affine diffeomorphisms fixing $\overline{1}\in\Sigma_{EW}$ (and, \emph{a fortiori}, each $\overline{g}\in\Sigma_{EW}$). Since $\textrm{Aut}(M_{EW},\omega_{EW})\simeq Q$ acts transitively on $\Sigma_{EW}$, we have that $\textrm{Aff}_{(1)}(M_{EW},\omega_{EW})$ has index $4$ inside $\textrm{Aff}(M_{EW},\omega_{EW})$. 

Since the elements of $\textrm{Aff}(M_{EW},\omega_{EW})$ differ from those $\textrm{Aff}_{(1)}(M_{EW},\omega_{EW})$ by composition with some element in $\textrm{Aut}(M_{EW},\omega_{EW})\simeq Q$ and we completely understand the homological action of $Q$, our task is reduced to compute the action of $\textrm{Aff}_{(1)}(M_{EW},\omega_{EW})$. 

At this point, we introduce the elements $\widetilde{S}\in \textrm{Aff}_{(1)}(M_{EW},\omega_{EW})$, resp. $\widetilde{T}\in \textrm{Aff}_{(1)}(M_{EW},\omega_{EW})$, obtained by the lifting $S=\left(\begin{array}{cc}1&0\\1&1\end{array}\right)\in SL(2,\mathbb{Z})$, resp. $T=\left(\begin{array}{cc}1&1\\0&1\end{array}\right)\in SL(2,\mathbb{Z})$, in such a way that the square $1$ intersects its image under $\widetilde{S}$, resp. $\widetilde{T}$. It can be checked that 
$$(\widetilde{S}\widetilde{T}^{-1}\widetilde{S})^4=-1\,.$$
so that $\widetilde{S}$ and $\widetilde{T}$ generate $\textrm{Aff}_{(1)}(M_{EW},\omega_{EW})$. In other words, it suffices to compute the action of $\widetilde{S}$ and $\widetilde{T}$ to get our hands in the action of $\textrm{Aff}_{(1)}(M_{EW},\omega_{EW})$. 

A direct inspection of Figure~\ref{f.eier-q} reveals that the actions of $\widetilde{S}$ and $\widetilde{T}$ on the cycles $\sigma_g$ and $\zeta_g$ are
\begin{equation*}\label{e.S1-g3}
\widetilde{S}(\zeta_g)=\left\{
\begin{array}{cc}
\zeta_g & \text{if } g\in\{\pm1,\pm j\},\\
\zeta_{jg} & \text{if } g\in\{\pm i,\pm k\},
\end{array} \right., \quad 
\widetilde{S}(\sigma_g)=\left\{
\begin{array}{cc}
\sigma_g + \zeta_{gi} & \text{if } g\in\{\pm1,\pm j\},\\
\sigma_{jg} + \zeta_{gk} & \text{if } g\in\{\pm i,\pm k\},
\end{array} \right.
\end{equation*} 
and
\begin{equation*}\label{e.T1-g3}
\widetilde{T}(\sigma_g)=\left\{
\begin{array}{cc}
\sigma_g & \text{if } g\in\{\pm1,\pm i\},\\
\sigma_{ig} & \text{if } g\in\{\pm j,\pm k\},
\end{array} \right., \quad \widetilde{T}(\zeta_g)=\left\{
\begin{array}{cc}
\zeta_g + \sigma_{gj} & \text{if } g\in\{\pm1,\pm i\},\\
\zeta_{ig} + \sigma_{-gk} & \text{if } g\in\{\pm j,\pm k\}.
\end{array} \right.
\end{equation*}

From these formulas, one deduces that $\widetilde{S}$ and $\widetilde{T}$ act on $H_1^{st}=\mathbb{Q}\sigma\oplus\mathbb{Q}\zeta$ in the standard way (cf. Subsection~\ref{ss.Aff-KZ})
\begin{eqnarray*}\label{e.st-g3}
\widetilde{S}(\sigma) & =\sigma + \zeta, \hspace{2cm} \widetilde{S}(\zeta) & = \zeta,\\
\widetilde{T}(\zeta) & =\sigma + \zeta, \hspace{2cm} \widetilde{T}(\sigma) & = \sigma,\nonumber
\end{eqnarray*}
while they act on the relative part $H_{rel}=\mathbb{Q}w_i\oplus\mathbb{Q}w_j\oplus\mathbb{Q}w_k$ via the symmetry group of a tetrahedron :
\begin{alignat}{4}\label{rel-g3}
\widetilde{S}(w_i) & =w_k, \hspace{15mm} \widetilde{S}(w_j) & = w_j, \hspace{15mm} \widetilde{S}(w_k) & = w_i, \nonumber \\ 
\widetilde{T}(w_i) & =w_i, \hspace{15mm} \widetilde{T}(w_j) & = w_k, \hspace{15mm} \widetilde{T}(w_k) & = w_j. \nonumber
\end{alignat}
Last, but not least, we have that $\widetilde{S}$ and $\widetilde{T}$ act on $H_1^{(0)}$ (that is, the subspace carrying the non-tautological exponents of the KZ cocycle) as:
\begin{equation*}\label{e.S5-g3}
\widetilde{S}(\varepsilon_g)=\left\{
\begin{array}{cc}
\frac 12 ( \varepsilon_{g} + \varepsilon_{gi} + \varepsilon_{gj} + \varepsilon_{gk})& \text{if } g\in\{\pm1,\pm j\},\\
\frac 12 ( \varepsilon_{g} - \varepsilon_{gi} - \varepsilon_{gj} + \varepsilon_{gk}) & \text{if } g\in\{\pm i,\pm k\},
\end{array} \right.
\end{equation*}
and
\begin{equation*}\label{e.T5-g3}
\widetilde{T}(\varepsilon_g)=\left\{
\begin{array}{cc}
\frac 12 ( \varepsilon_{g} + \varepsilon_{gi} + \varepsilon_{gj} - \varepsilon_{gk})& \text{if } g\in\{\pm1,\pm i\},\\
\frac 12 ( \varepsilon_{g} - \varepsilon_{gi} - \varepsilon_{gj} - \varepsilon_{gk}) & \text{if } g\in\{\pm j,\pm k\}.
\end{array} \right.
\end{equation*}
Here, our choice of computing $\widetilde{S}$ and $\widetilde{T}$ in terms of $\varepsilon_g$ was not arbitrary: indeed, a closer inspection of these formulas shows that $\widetilde{S}$ and $\widetilde{T}$ are acting on the $4$-dimensional subspace $H_1^{(0)}$ via the automorphism group of the \emph{root system}
$$R=\{\varepsilon_g+\varepsilon_h:g\neq\pm h\}$$
\emph{of type $D_4$}, that is,  by equipping $H_1^{(0)}$ with the inner product such that 
the set $\{\varepsilon_1,\varepsilon_i, \varepsilon_j, \varepsilon_k\}$ is an orthonormal basis , 
one has that $\widetilde{S}$ and $\widetilde{T}$ act on $H_1^{(0)}$ via the \emph{finite} group $O(R)$ of orthogonal linear transformations of $H_1^{(0)}$ preserving $R$. 
In particular, this discussion shows that $\textrm{Aff}_{(1)}(M_{EW},\omega_{EW})$ acts on $H_1^{(0)}$ via a certain \emph{finite} subgroup of orthogonal $4\times 4$ matrices. 

Actually, one can follow \cite{MY} to develop these calculations (using the knowledge of the structure of the automorphism and Weyl groups of root systems of type $D_4$) to prove that the affine group $\textrm{Aff}(M_{EW},\omega_{EW})$ acts on $H_1^{(0)}$ via a subgroup of order\footnote{The order of $O(R)$ is order $1152$, so that the affine group of the Eierlegende Wollmilchsau acts via a index $12$ subgroup of $O(R)$.} $96$ of orthogonal $4\times 4$ matrices. However, we will not insist on this point. Instead, we take the opportunity to observe that this allows to \emph{re-derive} the \emph{total degeneracy} of the Lyapunov spectrum of the Eierlegende Wollmilchsau. Indeed, since $\textrm{Aff}(M_{EW},\omega_{EW})$ acts on $H_1^{(0)}$ via a finite group of matrices, it preserves some norm (actually we already determined it), that is, $\textrm{Aff}(M_{EW},\omega_{EW})$ acts \emph{isometrically} on $H_1^{(0)}$. Hence, by our discussion in Subsection~\ref{ss.Aff-KZ}, this means that the restriction of the KZ cocycle to $H_1^{(0)}(M_{EW},\omega_{EW})$ acts isometrically over the $SL(2,\mathbb{R})$-orbit of the Eierlegende Wollmilchsau, so that all (non-tautological) Lyapunov exponents of the Eierlegende Wollmilchsau must vanish. 

\begin{remark}Of course, \emph{a priori} the fact that the restriction of the KZ cocycle to $H_1^{(0)}(M_{EW},\omega_{EW})$ acts via a finite group is stronger than simply knowing that it acts isometrically. But, as it turns out, it is possible to show that in the case of square-tiled surfaces, if the KZ cocycle acts isometrically on $H_1^{(0)}$, then it must act through a finite group: indeed, in the square-tiled surface case, the KZ cocycle acts on $H_1^{(0)}$ via the \emph{discrete} group $Sp(2g-2,\mathbb{Z}[1/N])$ for some $N\in\mathbb{N}$; so, if the KZ cocycle also acts on $H_1^{(0)}$ \emph{isometrically}, one conclude that the KZ cocycle acts on $H_1^{(0)}$ via $Sp(2g-2,\mathbb{Z}[1/N])\cap O(2g-2)$, a finite group. See \cite{Mo} for more details. However, we avoided using this fact during this section to convince the reader that the homological action of the affine group is so \emph{concrete} that  one can actually \emph{explicitly determine} (with bare hands) the matrices involved in it (at least if one is sufficiently patient).
\end{remark}

\subsection{The action of the affine diffeomorphisms of the \emph{Ornithorynque}} As the reader can imagine, the calculations of the previous two subsections can be mimicked in the context of the Ornithorynque. Evidently, the required modifications are somewhat straightforward, so we will
present below a mere outline of the computations (referring to the original article \cite{MY} for details). 

We start by considering a ``better'' presentation of the  Ornithorynque $(M_{O},\omega_O)$ (to be compared with the presentation in Figure~\ref{fig:cartoon}):

\begin{figure}[htb!]
\includegraphics[scale=0.3]{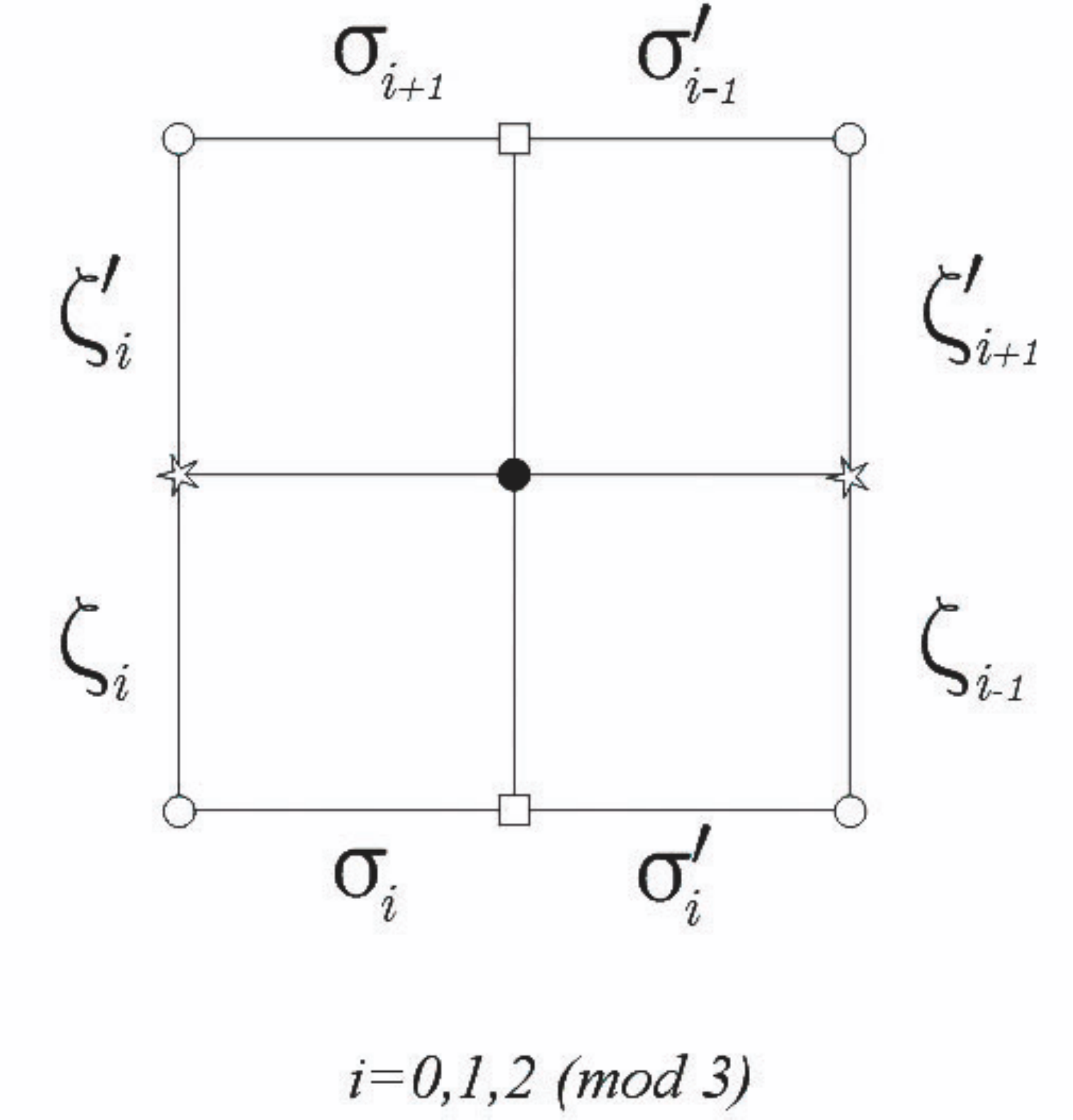}
\end{figure}

The choice of indices $i=0,1,2$ (mod $3$) means that we're considering three copies of the ``basic'' pattern, and we identify sides with the same ``name'' $\sigma_i$, $\sigma_i'$, $\zeta_i$ or $\zeta_i'$ by taking into account that the subindices $i$ are defined modulo $3$. The (three) black dots are \emph{regular} points while the other ``special'' points indicated the (three) double zeroes of the Abelian differential $\omega_O$. Also, it is clear from this picture that the cycles  $\sigma_i$, $\sigma_i'$, $\zeta_i$ and $\zeta_i'$ are \emph{relative}, and they verify the relation 
$$\sigma_i+\sigma_i'+\zeta_{i-1}+\zeta_{i+1}'-\sigma_{i-1}'-\sigma_{i+1}-\zeta_{i}'-\zeta_i=0\,,
\quad \text{for each }\, i\in\mathbb{Z}/3\mathbb{Z}\,.$$

Let's now focus on the action of $\textrm{Aff}(M_O,\omega_O)$ on $H_1^{(0)}(M_O,\omega_O)$ (the interesting part of the homology containing the non-tautological exponents of the KZ cocycle). The natural $\mathbb{Z}/3\mathbb{Z}$ \emph{symmetry} (i.e., group of \emph{automorphisms}) of the figure above motivates the following choice of cycles in $H_1^{(0)}$:
$$a_i:=\sigma_i-\sigma_{i+1}, \quad a_i':=\sigma_i'-\sigma_{i+1}'$$
and
$$b_i:=\zeta_i-\zeta_{i+1}, \quad b_i':=\zeta_i'-\zeta_{i+1}'\,.$$ 

The relation between the cycles  $\sigma_i$, $\sigma_i'$, $\zeta_i$ and $\zeta_i'$ written above implies that 
$$a_i-a_{i-1}'+b_{i-1}-b_i'=0\,.$$

This suggests the introduction of the following cycles in $H_1^{(0)}$:
$$\tau_i:=a_i-a_{i-1}'=b_i'-b_{i-1}, \quad \widehat{\sigma}_i:=a_i+a_{i-1}', \quad \widehat{\zeta}_i=b_i'+b_{i-1}\,.$$

Note that $\sum\limits_{i\in\mathbb{Z}/3\mathbb{Z}}\tau_i=\sum\limits_{i\in\mathbb{Z}/3\mathbb{Z}}\widehat{\sigma}_i=\sum\limits_{i\in\mathbb{Z}/3\mathbb{Z}}\widehat{\zeta}_i=0$. Actually, this is the only relation satisfied by them because it is possible to show that  the set
$$\{\tau_i,\widehat{\sigma}_i,\widehat{\zeta}_i:i\in\mathbb{Z}/3\mathbb{Z}-\{0\}\}$$
is a \emph{basis} of $H_1^{(0)}(M_O,\mathbb{Q})$. Let us point out that this is in agreement with the 
fact that $H_1^{(0)}$ has dimension $6$ ($=2g-2$ for $g=g(M_O)=4$). 
Let's denote by $H_{\tau}$ the 2-dimensional subspace of $H_1^{(0)}$ spanned by $\tau_i$'s and $\widehat{H}$ the 4-dimensional subspace of $H_1^{(0)}$ spanned by $\widehat{\sigma}_i$'s and $\widehat{\zeta}_i$'s, so that 
$$H_1^{(0)}(M_O,\mathbb{Q})=H_{\tau}\oplus\widehat{H}$$ 

This splitting is \emph{natural} in our context because it is preserved by the homological action of the affine group. Indeed, it is easy to check that this splitting is preserved by the group of automorphisms $\mathbb{Z}/3\mathbb{Z}$, so that we can restrict our attention to the subgroup $\textrm{Aff}_{(1)}(M_O,\omega_O)$ of affine diffeomorphisms fixing \emph{each} zero of $\omega_O$. 

Let us denote by $\widetilde{S}\in\textrm{Aff}_{(1)}(M_O,\omega_O)$, resp. $\widetilde{T}\in\textrm{Aff}_{(1)}(M_O,\omega_O)$, the elements with linear parts $S=\left(\begin{array}{cc}1&0\\1&1\end{array}\right)\in SL(2,\mathbb{Z})$, resp. $T=\left(\begin{array}{cc}1&1\\0&1\end{array}\right)\in SL(2,\mathbb{Z})$. Again, one can check that $\widetilde{S}$ and $\widetilde{T}$ generate $\textrm{Aff}_{(1)}(M_O,\omega_O)$, and their action on the subspaces $H_{\tau}$ and $\widehat{H}$ are given by the formulas
$$\widetilde{S}(\tau_{i})=-\tau_{i+1}, \quad \widetilde{S}(\widehat{\sigma}_i)=\widehat{\sigma}_i + \widehat{\zeta}_{i-1}, \quad 
\widetilde{S}(\widehat{\zeta}_i)=\widehat{\zeta}_{i+1}$$
and 
$$\widetilde{T}(\tau_{i})=-\tau_{i-1}, \quad \widetilde{T}(\widehat{\sigma}_i)=\widehat{\sigma}_{i-1}, \quad 
\widetilde{T}(\widehat{\zeta}_i)=\widehat{\zeta}_i+\widehat{\sigma}_{i+1}\,.$$

Therefore, $\textrm{Aff}_{(1)}(M_O,\omega_O)$ preserves the decomposition $H_1^{(0)} = H_{\tau}\oplus \widehat{H}$ and it acts on $H_{\tau}$ via the cyclic group $\mathbb{Z}/3\mathbb{Z}$. 
Finally, a careful inspection of the formulas above shows that the action of $\textrm{Aff}_{(1)}(M_O,\omega_O)$ on $\widehat{H}$ preserves the root system  
$$R=\{\pm\widehat{\sigma}_i, \pm\widehat{\zeta}_i, \pm(\widehat{\sigma}_i+ \widehat{\zeta}_{i-1}), \pm(\widehat{\sigma}_i-\widehat{\zeta}_{i+1})\}$$
of type $D_4$. Hence, $\textrm{Aff}_{(1)}(M_O,\omega_O)$ acts on $\widehat{H}$ via a subgroup of the \emph{finite} group of automorphisms of the root system $R$. Actually, one can perform this computation \emph{jusqu'au but} to check that the (whole) affine group $\textrm{Aff}(M_O,\omega_O)$, or equivalently, the KZ cocycle over the $SL(2,\mathbb{R})$-orbit of Ornithorynque, acts via an explicit subgroup of order $72$ of the group of automorphisms of $R$.


\section{Cyclic covers}\label{s.FMZI}

In last section we studied combinatorial models of the Eierlegende Wollmilchsau and of the Ornithorynque by taking advantage of the fact that they belong to the class of \emph{square-tiled cyclic covers}. Then, we used these combinatorial models to ``put our hands'' on the KZ cocycle over their $SL(2,\mathbb{R})$-orbits via the homological action of the group of affine diffeomorphisms. 

In this section, we'll be ``less concrete but more conceptual'' in order to systematically treat the Lyapunov spectrum of the KZ cocycle over square-tiled cyclic covers in a unified way. From this framework we will derive that all zero Lyapunov exponents \emph{in the class of square-tiled cyclic covers} have a 
 \emph{geometrical} explanation: they are carried by the \emph{annihilator of the second fundamental form} $B^{\mathbb R}_\omega$ on the real Hodge bundle. A striking consequence of this fact is the \emph{continuous} (actually, \emph{real-analytic}) dependence of the \emph{neutral Oseledets subspace} on the base point. However, by the end of this section, we will see that this beautiful scenario is \emph{not} true in general: indeed, we'll construct \emph{other} (not square-tiled) cyclic covers leading to a merely \emph{measurable} neutral Oseledets subspace.

\subsection{Hodge theory and the Lyapunov exponents of square-tiled cyclic covers} Let us consider a square-tiled cyclic cover  
$$M=M_N(a_1,\dots,a_4)=\{y^N=(x-x_1)^{a_1}\dots(x-x_4)^{a_4}\}, \quad \omega:=(x-x_1)^{b_1}\dots(x-x_4)^{b_4}dx/y^{N/2}\,,$$
where $\gcd(N,a_1,\dots,a_4)=1$, $2b_j=a_j-1$, $N$ is even and $0<a_j<N$ are odd. Cf. Subsection~\ref{ss.square-tiled-cyclic} for more details. 

The arguments in the previous two sections readily show that the locus of $(M_N(a_1,\dots,a_4),\omega)$ is the $SL(2,\mathbb{R})$-orbit of a square-tiled surface.\footnote{In addition, the Veech group has a simple dependence on $N, a_1,\dots, a_4$ and it has index $1, 2, 3$ or $6$ inside $SL(2,\mathbb{Z})$. See \cite{FMZ1}.} Thus, it makes sense to discuss the Lyapunov exponents of the KZ cocycle with respect to the unique $SL(2,\mathbb{R})$-invariant probability supported on the locus of $(M_N(a_1,\dots,a_4),\omega)$.

For ``linear algebra reasons'', it is better to work with the \emph{complex} version of the KZ cocycle on the \emph{complex} Hodge bundle: indeed, as we shall see in a moment, it is easier to \emph{diagonalize by blocks} the complex KZ cocycle, while the Lyapunov exponents of latter coincide
with those of the usual (real) KZ cocycle. 
We can diagonalize by blocks the complex KZ cocycle over the locus of  square-tiled cyclic covers
$(M_N(a_1,\dots,a_4),\omega)$  by exploiting the automorphism defined as 
$T(x,y)=(x,\varepsilon y)$, with $\varepsilon=\exp(2\pi i/N)$. More precisely, since 
$T^N=\textrm{Id}$, one can write 
$$H^1(M,\mathbb{C})=\bigoplus\limits_{j=1}^{N-1} H^1(\varepsilon^j)\,,$$
where $H^1(\varepsilon^j)$ denotes the eigenspace of the action $T^*$ of the automorphism $T$ on the complex cohomology, relative to the eigenvalue $\varepsilon^j\in \mathbb C$.
\begin{remark}
Here, the eigenspace $H^1(\varepsilon^0)$ is not present since any element of $H^1(\varepsilon^0)$ is $T^*$-invariant, hence it maps to $H^1(\overline{\mathbb{C}})=\{0\}$ under the projection 
$p:M_N(a_1,\dots, a_4)\to\overline{\mathbb{C}}$,  defined as $p(x,y)=x$.
\end{remark} 

We affirm that these blocks are invariant under the complex KZ cocycle. Indeed, the action $T^\ast$
of the automorphism $T$ on cohomology is locally constant over the locus of $(M_N(a_1,\dots,a_4),\omega)$, hence it \emph{commutes} with the complex KZ cocycle and, \emph{a fortiori}, the eigenspaces of $T^*$ serve to diagonalize by blocks the complex KZ cocycle.  

Next, we recall that the complex KZ cocycle preserves the Hodge form $(\alpha,\beta)=\frac{i}{2}\int\alpha\wedge\overline{\beta}$, a \emph{positive} definite form on $H^{1,0}$ and \emph{negative} definite form on $H^{0,1}$ (cf. Subsection~\ref{ss.Hodgenorm}). Since $H^1(M,\mathbb{C})=H^{1,0}\oplus H^{0,1}$, we have that  $H^1(\varepsilon^j)=H^{1,0}(\varepsilon^j)\oplus H^{0,1}(\varepsilon^j)$, and, thus, the restriction of the complex KZ cocycle to $H^1(\varepsilon^j)$ acts via elements of the group $U(p_j, q_j)$ of complex matrices preserving a non-degenerate (pseudo-)Hermitian form of signature 
$$p_j:=\textrm{dim}_{\mathbb{C}}H^{1,0}(\varepsilon^j), \quad q_j:=\textrm{dim}_{\mathbb{C}}H^{0,1}(\varepsilon^j)\,.$$

In the setting of \emph{square-tiled cyclic covers}, these signatures are \emph{easy} to compute in terms of $N$, $a_1,\dots, a_4$ in view of the next lemma:
\begin{lemma}[I.~Bouw]\label{l.Bouw} Let $[x]$ be the integer part of $x$. One has 
$$p_j=\sum\limits_{n=1}^4\left[\frac{a_n j}{N}\right]-1$$ 
and 
$$q_j=\sum\limits_{n=1}^4\left[\frac{a_n (N-j)}{N}\right]-1$$ 
In particular, $p_j,q_j\in \{0,1,2\}$ and $p_j+q_j\in \{0,1,2\}$.
\end{lemma} 

\begin{proof} A sketch of proof of this result goes as follows. Since $H^{0,1}(\varepsilon^j)=\overline{H^{1,0}(\varepsilon^{N-j})}$, we have that $q_j=p_{N-j}$ and, therefore, it suffices to compute $p_j$. 

For this, it suffices to study whether the meromorphic form 
$$\alpha_j(b_1,b_2,b_3,b_4):=(x-x_1)^{b_1}(x-x_2)^{b_2}(x-x_3)^{b_3}(x-x_4)^{b_4}dx/y^j$$ 
is holomorphic near $x_1,\dots, x_4$ and $\infty$. 

By performing the necessary calculations, one verifies that there exists a choice of $b_1,\dots, b_4$ with $\alpha_j(b_1,\dots,b_4)$ holomorphic if and only if 
$$\sum\limits_{n=1}^4\left[\frac{a_n j}{N}\right]\geq 2$$
Moreover, if this inequality is satisfied then $\alpha_j(b_1,\dots, b_4)$ is holomorphic for $b_n:=\left[\frac{a_n j}{N}\right]$. Furthermore, if $\sum\limits_{n=1}^4\left[\frac{a_n j}{N}\right]=3$, one can also check that $H^{1,0}(\varepsilon^j)$ is spanned by $\alpha_j(b_1,\dots,b_4)$ and $x\cdot\alpha_j(b_1,\dots,b_4)$ for $b_n=:=\left[\frac{a_n j}{N}\right]$. For the details of this computation, see, e.g., \cite{EKZ-JMD}.

Note that this completes the sketch of proof of lemma: indeed, since $0<a_n<N$ and $a_1+\dots+a_4$ is a multiple of $N$, we have that 
$$1\leq\sum\limits_{n=1}^4\left[\frac{a_n j}{N}\right]\leq 3$$
because $a_1+\dots+a_4\in\{N, 2N, 3N\}$. So, our discussion above covers all cases.
\end{proof}

This lemma suggests the following ``clustering'' of the blocks $H^1(\varepsilon^j)$:
$$\mathcal{N}:=\{0<j<N-1:p_j \textrm{ or } q_j =0\}$$
and 
$$\mathcal{P}:=\{0<j<N-1: p_j=q_j=1\}$$
In fact, by Lemma~\ref{l.Bouw}, we have that $\mathcal{N}\cup\mathcal{P}=\{1, \dots, N-1\}$ and $N/2\in\mathcal{P}$, so that for each $j\in \{1, \dots, N-1\}$  the restriction of the complex KZ cocycle to the subbundle $H^1(\varepsilon^j)$ acts via matrices in the following group:
\begin{itemize}
\item $U(p_j,0)$, $0\leq p_j\leq 2$, or $U(0,q_j)$, $0\leq q_j\leq 2$, whenever $j\in\mathcal{N}$;
\item $U(1,1)$, whenever $j\in\mathcal{P}$.
\end{itemize} 
From the point of view of Lyapunov exponents, this ``clustering'' is natural because the groups $U(p_j,0)$ or $U(0,q_j)$ are \emph{compact}, while the group $U(1,1)\simeq SL(2,\mathbb{R})$ is not compact. An immediate consequence of the compactness of the groups $U(p_j,0)$ or $U(0,q_j)$ is the following:
\begin{corollary}\label{c.N} If $j\in\mathcal{N}$, then all Lyapunov exponents of the restriction of the complex KZ cocycle to the subbundle $H^1(\varepsilon^j)$ of the complex Hodge bundle vanish.
\end{corollary} 

Alternatively, this corollary can be derived by examining the second fundament form $B_{\omega}$. More precisely, let $H^1_j(M,\mathbb{R}):=(H^1(\varepsilon^j)\oplus H^1(\varepsilon^{N-j}))\cap H^1(M,\mathbb{R})$. It is possible to show (by adapting the arguments of the proof of Theorem~\ref{t.F02}) that the sum of non-negative Lyapunov exponents of the KZ cocycle in $H^1_j(M,\mathbb{R})$ coincide with the average of the sum of eigenvalues of the restriction of $H_{\omega}=B_{\omega}\cdot B_{\omega}^*$ to $H^{1,0}(\varepsilon^j)\oplus H^{1,0}(\varepsilon^{N-j})$. See, e.g., \cite{FMZ2} for more details. So,  an alternative proof of the above corollary can be derived by showing that $B_{\omega}$ vanishes on $H^{1,0}(\varepsilon^j)\oplus H^{1,0}(\varepsilon^{N-j})$ \emph{when} $j\in \mathcal{N}$. Then, one realizes that this is true since $j\in\mathcal{N}$ implies that $H^{1,0}(\varepsilon^j)\oplus H^{1,0}(\varepsilon^{N-j})=H^{1,0}(\varepsilon^j)$ or $H^{1,0}(\varepsilon^{N-j})$, and the restriction of 
$B_{\omega}$ to $H^{1,0}(\varepsilon^k)$ vanishes for every $k\neq N/2$ because of the following computation\footnote{Here, we used the automorphism $T$ to change variables in the integral defining $B_{\omega}$. Cf. Theorem~\ref{t.F06}.}: for any $\alpha,\beta\in H^{1,0}(\varepsilon^k)$, 
$$B_{\omega}(\alpha,\beta)=B_{T^*(\omega)}(T^*(\alpha),T^*(\beta))=\varepsilon^{2k}B_{\omega}(\alpha,\beta)$$

Actually, one can further play with the form $B_{\omega}$ to show a ``converse'' to this corollary, that is, the Lyapunov exponents of the restriction of the complex KZ cocycle to $H^1(\varepsilon^j)$ are non-zero \emph{whenever} $j\in\mathcal{P}$. In fact, if $j\in\mathcal{P}$, the restriction of the KZ cocycle to $H^1(\varepsilon^j)$ has Lyapunov exponents 
$\pm\lambda_{(j)}$ (because it acts via matrices in $U(1,1)\simeq SL(2,\mathbb{R})$). Moreover, the restriction of the KZ cocycle to $H^1(\varepsilon^j)$ is conjugated to the restriction of the KZ cocycle to $H^1(\varepsilon^{N-j})$, so that $\lambda_{(j)}=\lambda_{(N-j)}$. Therefore, by the discussion of the previous paragraph, we can deduce that 
$\lambda_{(j)}=\lambda_{(N-j)}$ is non-zero by showing that the restriction of $B_{\omega}$ to $H^{1,0}(\varepsilon^j)\oplus H^{1,0}(\varepsilon^{N-j})$ is not degenerate. Here, this latter assertion is true because $H^{1,0}(\varepsilon^j):=\mathbb{C}\cdot\alpha_j:=\mathbb{C}\cdot\alpha_j(b_1(j), \dots, b_4(j))$, $b_k(j):=[a_kj/N]$, for $j\in\mathcal{P}$ (cf. Lemma~\ref{l.Bouw}), so that  
$$B_{\omega}(\alpha_j, \alpha_{N-j})=\int_M\prod\limits_{k=1}^4|x-x_k|^{a_k-1}|dx|^2/|y|^N\neq 0$$ 
Here, we used the following identity: $[a_k j/N]+[a_k(N-j)/N]=a_k-1$ for $j\in\mathcal{P}$.

In other words, we just proved that 
\begin{corollary}\label{c.P} If $j\in\mathcal{P}$, the Lyapunov exponents $\pm\lambda_{(j)}$ of the restriction of the complex KZ cocycle to the subbundle $H^1(\varepsilon^j)$ of the complex Hodge
bundle are non-zero.
\end{corollary}

At this point, we can say (in view of Corollaries~\ref{c.N} and~\ref{c.P}) that the Lyapunov spectrum of the KZ cocycle over square-tiled cyclic covers is \emph{qualitatively} well-known: it can be diagonalized by blocks by restriction to the subbundles $H^1(\varepsilon^j)$ and zero Lyapunov exponents come precisely from blocks $H^1(\varepsilon^j)$ with $j\in\mathcal{N}$. However, in some applications\footnote{For instance, the precise knowledge of Lyapunov exponents of the KZ cocycle for a certain $SL(2,\mathbb{R})$-invariant $g_t$-ergodic probability measure supported in $\mathcal{H}_5$ recently allowed V.~Delecroix, P.~Hubert and S.~Leli\`evre \cite{DHL} to confirm a conjecture of the physicists J.~Hardy and J.~Weber that the so-called \emph{Ehrenfest wind-tree model} of Lorenz gases has \emph{abnormal} diffusion for typical choices of parameters.}, it is important to determine \emph{quantitatively} individual exponents of the KZ cocycle. In the case of cyclic covers, A.~Eskin, M.~Kontsevich and 
A.~Zorich \cite{EKZ-JMD} determined the value of $\lambda_{(j)}$ for $j\in\mathcal{P}$. Roughly speaking, they start the computation 
$\lambda_{(j)}=\lambda_{(N-j)}$ from the fact (already mentioned) that $2\lambda_{(j)}=\lambda_{(j)}+\lambda_{(N-j)}$ coincides with the average of the eigenvaules of $H_{\omega}=B_{\omega}\cdot B_{\omega}^*$ restricted to $H^{1,0}(\varepsilon^j)\oplus H^{1,0}(\varepsilon^{N-j})$. Then, they use the fact that $H^{1,0}(\varepsilon^j)$ has complex dimension $1$ for $j\in\mathcal{P}$ to reduce the calculation of the aforementioned average to the computation of the \emph{orbifold degree} of the \emph{line bundle} $H^{1,0}(\varepsilon^j)$. After this, the calculation of the orbifold degree of $H^{1,0}(\varepsilon^j)$ can be performed \emph{explicitly} by noticing that $H^{1,0}(\varepsilon^j)$ has a \emph{global section} $H^{1,0}(\varepsilon^j)=\mathbb{C}\cdot\alpha_j(b_1(j),\dots, b_4(j))=:\mathbb{C}\cdot\alpha_j$ over the $SL(2,\mathbb{R})$-orbit 
$$(\{y^N=(x-x_1)^{a_1}\dots(x-x_4)^{a_4}\},\omega).$$ 
Since the orbifold degree is expressed as a certain integral depending on $\alpha_j$, a sort of integration by parts argument can be used to rewrite the orbifold degree in terms of the ``\emph{behavior at infinity}'' of $\alpha_j$, that is, the behavior of $\alpha_j$ when $x_i$ \emph{approaches} $x_j$ for some $i\neq j$. This led them to the following result:

\begin{theorem}[A.~Eskin, M.~Kontsevich and A.~Zorich] Let $b_k(j)=[a_k j/N]$. Then, 
$$\lambda_{(j)}=2\cdot\min\{b_1(j), 1-b_1(j), \dots, b_4(j), 1-b_4(j)\}\,, \quad \text{ \rm for any }
j\in\mathcal{P}\,.$$
\end{theorem} 

Coming back to the qualitative analysis of the KZ cocycle (and its Lyapunov spectrum) over square-tiled cyclic covers, we observe that our discussion so far shows that the \emph{neutral Oseledets bundle} $E^c$ (associated to zero Lyapunov exponents) of the KZ cocycle coincides with the \emph{annihilator} $\textrm{Ann}(B^{\mathbb{R}})$ in the case of square-tiled cyclic covers. In other words, the zero Lyapunov exponents of the KZ cocycle have a nice \emph{geometrical} explanation \emph{in the case of square-tiled cyclic covers}: they come precisely from the annihilator of \emph{second fundamental form} $B^{\mathbb{R}}$! In particular, since $B_{\omega}$ depends continuously (real-analytically) on $\omega$, we conclude that the neutral Oseledets bundle $E^c$ of the KZ cocycle is \emph{continuous}\footnote{At the \emph{complex} level, this can be also seen from the fact that, by Corollaries~\ref{c.N} and~\ref{c.P}, $\bigoplus\limits_{j\in\mathcal{N}}H^1(\varepsilon^j)$ is the neutral Oseledets bundle of the complex KZ cocycle, and the subspaces $H^1(\varepsilon^k)$ vary continuously with $\omega$ because they are eigenspaces of the cohomological action of an automorphism.} in the case of square-tiled cyclic covers! Of course, this should be compared with the general statement of Oseledets theorem ensuring only \emph{measurability} of Oseledets subspaces. A nice consequence of the equality $E^c=\textrm{Ann}(B^{\mathbb{R}})$ for square-tiled cyclic covers is the fact that, in this setting, $E^c=\textrm{Ann}(B^{\mathbb{R}})$ is $SL(2,\mathbb{R})$-\emph{invariant}: indeed, this is a mere corollary of the $g_t$-invariance of $E^c$ (coming from Oseledets theorem), the $SO(2,\mathbb{R})$-invariance of $\textrm{Ann}(B^{\mathbb{R}})$ (coming from the definition of $B_\omega$) and the fact that $SL(2,\mathbb{R})$ is generated by the diagonal subgroup $g_t$ and the rotation subgroup $SO(2,\mathbb{R})$. Furthermore, by combining the $SL(2,\mathbb{R})$-invariance of $E^c=\textrm{Ann}(B^{\mathbb{R}})$ for square-tiled cyclic covers with Theorem~\ref{t.1stvariation}, we deduce that the KZ cocycle acts \emph{isometrically} on $E^c=\textrm{Ann}(B^{\mathbb{R}})$ in this setting\footnote{We saw this fact in the particular cases of Eierlegende Wollmilchsau (Remark~\ref{r.isometric-EW}) and Ornithorynque (Theorem~\ref{t.FM08})}. In summary, this discussion of this paragraph proves the following result (from \cite{FMZ2}):
\begin{theorem} The neutral Oseledets bundle $E^c$ of the KZ cocycle over square-tiled cyclic covers
coincides with the annihilator $\textrm{Ann}(B^{\mathbb{R}})$ of the second fundamental form, hence it
 is a continuous $SL(2,\mathbb{R})$-invariant subbundle of the Hodge bundle on which the KZ cocycle acts isometrically.
\end{theorem} 

It is tempting to ask whether this scenario holds for \emph{other} $SL(2,\mathbb{R})$-invariant probability measures, that is, whether the annihilator of the second fundamental form is \emph{always} the explanation of zero Lyapunov exponents of the KZ cocycle with respect to $SL(2,\mathbb{R})$-invariant probability measures. As it turns out, we'll see below an example where the mechanism responsible for the neutral Oseledets bundle is \emph{not} related to the second fundamental form.

\subsection{Other cyclic covers} Let us consider the following family of Riemann surfaces 
$$M=M_{10}(x_1,\dots,x_6)=\{y^6=(x-x_1)\dots(x-x_6)\},$$
where $x_1,\dots,x_6\in\overline{\mathbb{C}}$ are mutually distinct, equipped with the meromorphic differential 
$$\eta=(x-x_1)dx/y^3$$
By Riemann-Hurwitz formula applied to $p:M_{10}(x_1,\dots,x_6)\to\overline{\mathbb{C}}$, $p(x,y)=x$, one can check that $M_{10}(x_1,\dots,x_6)$ has genus 10. Furthermore, by studying $\eta$ near $x_1,\dots,x_6$ (and $\infty$), one can see that $\eta$ has a zero of order $8$ over $x_1$ and $5$ double zeroes over $x_2,\dots, x_6$, i.e., 
$$\eta\in\mathcal{H}(8,2^5):=\mathcal{H}(8,\underbrace{2,\dots,2}_{5})\,.$$ 
We denote by $\mathcal{Z}$ the locus of $\mathcal{H}(8,2^5)$ determined by the family $(M_{10}(x_1,\dots,x_6),\omega)$ where $\omega$ is \emph{the} positive multiple of 
$\eta$ with \emph{unit} area.

\begin{lemma} $\mathcal{Z}$ is a closed $SL(2,\mathbb{R})$-invariant locus of $\mathcal{H}(8,2^5)$ naturally isomorphic to the stratum $\mathcal{H}^{(1)}(2)$ of unit area translation surfaces in $\mathcal{H}(2)$.
\end{lemma}

\begin{proof} Let's consider the covering $h:M_{10}(x_1,\dots,x_6)\to M_2(x_1,\dots,x_6)$, $h(x,y)=y^3$, where 
$$M_2(x_1,\dots,x_6):=\{z^2=(x-x_1)\dots(x-x_6)\}$$
Note that $M_2(x_1,\dots,x_6)$ is a genus 2 Riemann surface (by Riemann-Hurwitz formula). Moreover, 
$$\omega=h^*(\theta)$$
where $\theta$ is \emph{the} positive multiple of $(x-x_1)dx/z\in\mathcal{H}(2)$ with \emph{unit} area. 

Since any Riemann surface of genus 2 is \emph{hyperelliptic}, it is not hard to see that the family 
$$(M_2(x_1,\dots,x_6),\theta)$$
parametrizes the \emph{entire} stratum $\mathcal{H}^{(1)}(2)$ of unit area translation surfaces in $\mathcal{H}(2)$: indeed, any Riemann surface genus $2$ has six \emph{Weierstrass points} (in this case they are the fixed points of the hyperelliptic involution), so that it can be represented by an algebraic equation of the form 
$$z^2=(x-x_1)\dots(x-x_6)$$
Here, the six Weierstrass points are located over $x_1,\dots,x_6$. In other words, any genus $2$ Riemann surface is biholomorphic to some $M_2(x_1,\dots,x_6)$. Also, the zero of an Abelian differential $\theta$ in $M_2(x_1,\dots,x_6)$ must be located at one of the Weierstrass points. Thus, by renumbering the points $x_1,\dots,x_6$ (in order to place the zero over $x_1$), we can write any $\theta\in\mathcal{H}(2)$ on $M_2(x_1,\dots,x_6)$ as a multiple of $(x-x_1)dx/z$. Alternatively, one can show that 
$$\{(M_2(x_1,\dots,x_6),\theta): x_1,\dots,x_6\in\overline{\mathbb{C}} \textrm{ distincts}\}$$
is the entire stratum $\mathcal{H}^{(1)}(2)$ by counting dimensions. More precisely, by using M\"oebius transformations, one can normalize $x_1=0$, $x_2=1$, $x_3=\infty$ and check that the Riemann surface structure of $M_2(x_1,\dots,x_6)$ depends on $3$ \emph{complex} parameters (namely, $x_4,x_5,x_6$ after normalization). Furthermore, the choice of an Abelian differential $\theta$ in $\mathcal{H}(2)$ on $M_2(x_1,\dots,x_6)$ depends on $1$ complex parameter \emph{in general}. However, $\theta$ is normalized to have \emph{unit} area, the \emph{actual} choice of $\theta$ depends on $1$ real parameter. Hence, the locus 
$$\{(M_2(x_1,\dots,x_6),\theta): x_1,\dots,x_6\in\overline{\mathbb{C}} \textrm{ distincts}\}\subset\mathcal{H}^{(1)}(2)$$
has \emph{real} dimension $7$. Now, we observe that the stratum $\mathcal{H}(2)$ has \emph{complex} dimension $4$, i.e., \emph{real} dimension $8$ (cf. Section~\ref{s.orbifold}), so that the locus $\mathcal{H}^{(1)}(2)$ of unit area translation surfaces in $\mathcal{H}(2)$ has \emph{real} dimension $7$, and a connectedness argument can be applied to show
$$\{(M_2(x_1,\dots,x_6),\theta): x_1,\dots,x_6\in\overline{\mathbb{C}} \textrm{ distincts}\}=\mathcal{H}^{(1)}(2)\,.$$

Therefore, we have that the locus $\mathcal{Z}$ can be recovered from $\mathcal{H}^{(1)}(2)=\{(M_2(x_1,\dots,x_6),\theta)\}$ by taking triple covers $p:(M_{10}(x_1,\dots,x_6),\omega)\to(M_2(x_1,\dots,x_6),\theta)$, $p(x,y)=(x,y^3)$. Since $SL(2,\mathbb{R})$ acts on translation surfaces by \emph{post-composition} with translation charts while the covering $p$ is obtained by \emph{pre-composition} with translation charts, we deduce that $\mathcal{Z}$ is a $SL(2,\mathbb{R})$-invariant closed locus from the fact that $\mathcal{H}^{(1)}(2)$ has the same properties.
\end{proof}

Just to get a ``feeling'' on how the flat structure of translation surfaces in $\mathcal{Z}$ look like, we notice the following facts. It is not hard to check that the flat structure associated to $(M_2(x_1,\dots,x_6), \theta))$ is described by the following octagon (whose opposite parallel sides are identified):

\begin{center}
\begin{figure}[htb!]
\includegraphics[scale=0.5]{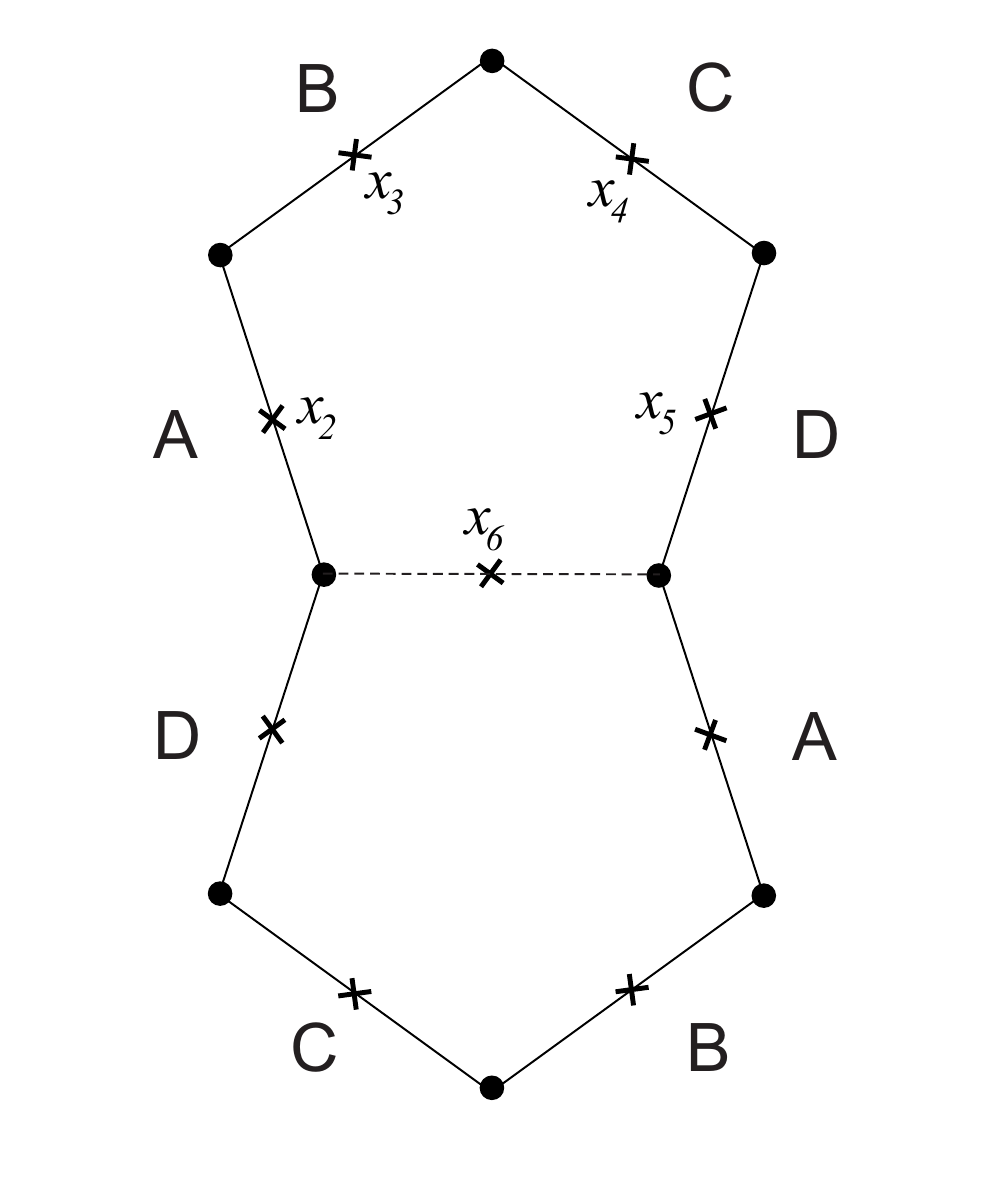}
\end{figure}
\end{center}

Here, the vertices of this octagon are all identified to a single point corresponding to $x_1$. Moreover, since $x_2,\dots,x_6$ are Weierstrass points of $M_2(x_1,\dots,x_6)$, one can organize the picture in such a way that the four points $x_2,\dots, x_5$ are located exactly at the middle points of the sides, and $x_6$ is located at the ``symmetry center'' of the octagon. See the picture above for an indication of the relative positions of $x_1$ (marked by a black dot) and $x_2,\dots, x_6$ (marked by crosses). In this way, we obtain a concrete description of $\mathcal{H}(2)$.

Now, since $\mathcal{Z}$ is defined by Abelian differentials $(M_{10}(x_1,\dots,x_6),\omega)$ given by certain triple (ramified) covers of the Abelian differentials 
$(M_2(x_1,\dots,x_6),\theta))\in\mathcal{H}(2)$, one can check that the flat structure associated to $(M_{10}(x_1,\dots,x_6),\omega)$ is described by the following picture:

\begin{center}
\begin{figure}[htb!]
\includegraphics[scale=0.6]{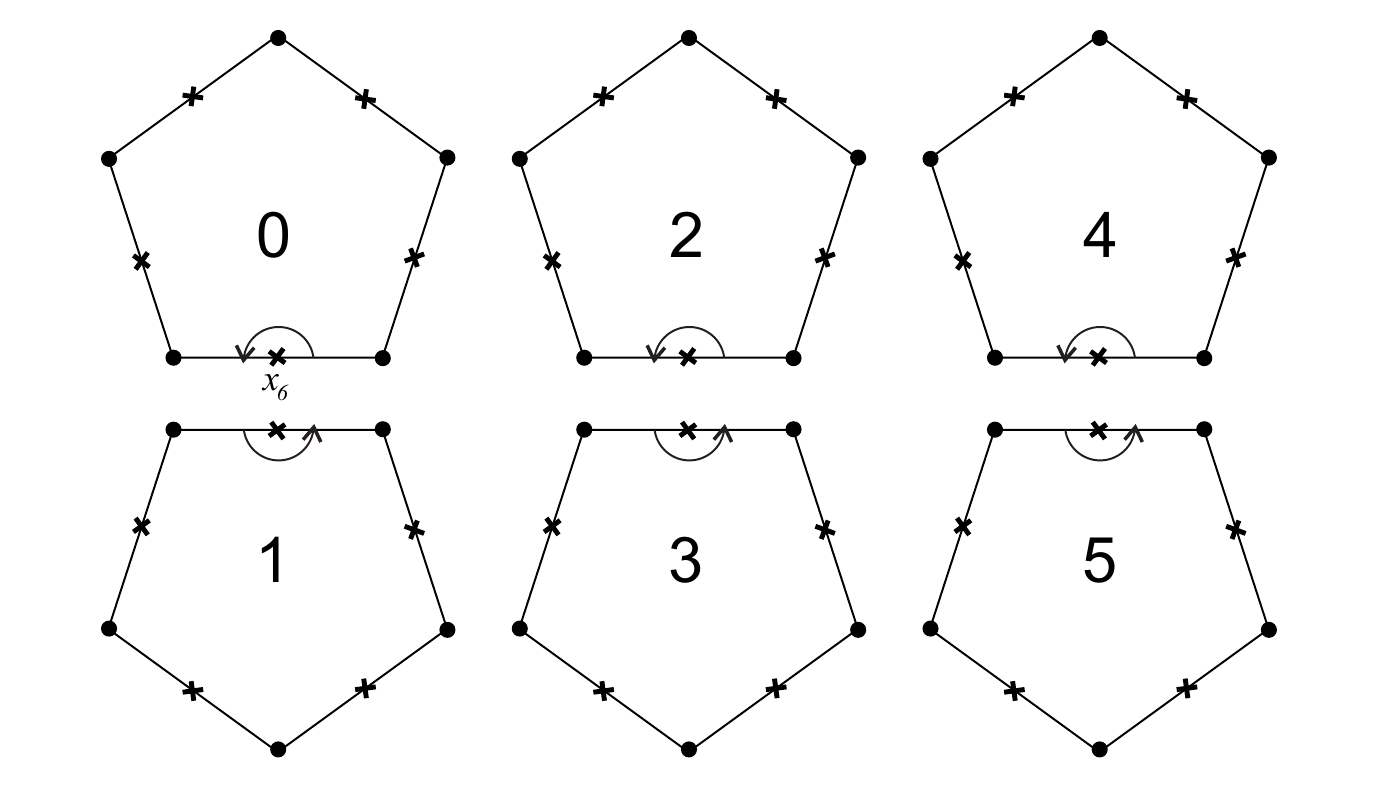}
\end{figure}
\end{center}

Here, we glue the half-sides determined by the vertices (black dots) and the crosses of these five pentagons in a cyclic way, so that every time we positively cross the side of a pentagon indexed by $j$, we move to the corresponding side on the pentagon indexed $j+1$ (mod $5$). For instance, in the figure above we illustrated the effect of going around the singularity point over $x_6$.

The natural isomorphism between the locus $\mathcal{Z}$ and $\mathcal{H}^{(1)}(2)$ enables us to put $SL(2,\mathbb{R})$-invariant measures on 
$\mathcal{Z}$: for example, by pulling back to $\mathcal{Z}$ the Masur--Veech probability measure 
of $\mathcal{H}^{(1)}(2)$ we obtain a fully supported $SL(2,\mathbb{R})$-invariant 
$g_t$-ergodic probability measure on $\mu_{\mathcal{Z}}$. 

In what follows, we will study the Lyapunov spectrum of the KZ cocycle with respect to $\mu_{\mathcal{Z}}$. By the reasons explained in the previous subsection, we consider the \emph{complex} KZ cocycle over $\mathcal{Z}$. Denoting by $T(x,y)=(x,\varepsilon y)$, $\varepsilon=\exp(2\pi i/6)$, the automorphism of order $6$ of $M_{10}(x_1,\dots,x_6)$ generating the Galois group of the covering $p:M_{10}(x_1,\dots,x_6)\to\overline{\mathbb{C}}$, $p(x,y)=x$, we can write 
$$H^1(M,\mathbb{C})=\bigoplus\limits_{j=1}^5H^1(\varepsilon^j)$$
where $H^1(\varepsilon^j)$ is the eigenspace of the eigenvalue $\varepsilon^j$ of the cohomological action $T^*$ of $T$. Again, the fact that $T$ is an automorphism implies that the complex KZ cocycle $G_t^{KZ,\mathbb{C}}$ preserves each $H^1(\varepsilon^j)$, that is, we can use these eigenspaces to diagonalize by blocks the complex KZ cocycle.

A direct computation reveals that 
$$\{(x-x_1)^{k-1}dx/y^j:0<k<j<6\}$$
is a basis of holomorphic differentials of $M_{10}(x_1,\dots,x_6)$. In particular, $\textrm{dim}_{\mathbb{C}} H^{1,0}(\varepsilon^j)=j-1$. 

Since $H^1(\varepsilon^j)=H^{1,0}(\varepsilon^j)\oplus H^{0,1}(\varepsilon^j)$ and $H^{0,1}(\varepsilon^{6-j})=\overline{H^{1,0}(\varepsilon^j)}$, we deduce that the restriction of the complex KZ cocycle to $H^1(\varepsilon^j)$ acts via matrices in the group
$$U(j-1,5-j)\,.$$

Therefore, we obtain that $G_t^{KZ,\mathbb{C}}|_{H^1(\varepsilon^j)}$ has only zero Lyapunov exponents for $j=1$ and $5$ because of the compactness of the groups $U(0,4)$ and $U(4,0)$. 

Next, we observe that one can ``reduce'' the study of $G_t^{KZ,\mathbb{C}}|_{H^1(\varepsilon^3)}$ to the complex KZ cocycle over $\mathcal{H}^{(1)}(2)$: more precisely, the fact that $\varepsilon^3=-1$ shows that $H^1(\varepsilon^3)=h^*(H^1(M_2(x_1,\dots,x_6),\mathbb{C}))$ where $h:M_{10}(x_1,\dots,x_6)\to M_2(x_1,\dots,x_6)$, $h(x,y)=(x,y^3)$, that is, $G_t^{KZ,\mathbb{C}}|_{H^1(\varepsilon^3)}$ is a \emph{copy} of the KZ cocycle over the stratum $\mathcal{H}^{(1)}(2)$. By Remark~\ref{r.EKZ-g2}, this means that we know the Lyapunov exponents of $G_t^{KZ,\mathbb{C}}|_{H^1(\varepsilon^3)}$, namely, they are $\pm1$ and $\pm1/3$. 

Thus, it remains ``only'' to investigate the cocycles $G_t^{KZ,\mathbb{C}}|_{H^1(\varepsilon^2)}$ and $G_t^{KZ,\mathbb{C}}|_{H^1(\varepsilon^4)}$. In fact, since $G_t^{KZ,\mathbb{C}}|_{H^1(\varepsilon^2)}$ is the complex conjugate of $G_t^{KZ,\mathbb{C}}|_{H^1(\varepsilon^4)}$, it suffices to study the latter cocycle. 

Here, we will take advantage of the fact that $G_t^{KZ,\mathbb{C}}|_{H^1(\varepsilon^4)}$ acts via $U(3,1)$ to get $2=3-1$ ``automatic'' zero Lyapunov exponents coming from general linear algebra:
\begin{proposition}\label{p.U31}$G_t^{KZ,\mathbb{C}}|_{H^1(\varepsilon^4)}$ has $2$ zero Lyapunov exponents (at least).
\end{proposition} 

\begin{proof} Since the cocycle $C_t:=G_t^{KZ,\mathbb{C}}|_{H^1(\varepsilon^4)}$ preserves the Hodge intersection form $(.,.)$, we have 
$$(v,w)=(C_t(v), C_t(w))\,, \quad \text{ \rm for all } t\in\mathbb{R} \,.$$

Let $v,w$ be two vectors in some Oseledets subspaces associated to Lyapunov exponents $\lambda,\mu$ with $\lambda+\mu\neq0$. By definition of Lyapunov exponents, one has that 
$$|(C_t(v), C_t(w))|\leq C e^{(\lambda+\mu)t}\|v\| \|w\|\to 0$$
as $t\to+\infty$ or $-\infty$ (depending on whether $\lambda+\mu<0$ or $\lambda+\mu>0$). 

Therefore, we conclude that $(v,w)=0$ whenever $v,w$ belong to Oseledets subspaces associated to Lyapunov exponents $\lambda,\mu$ with $\lambda+\mu\neq0$.
In particular, denoting by $E^u$, resp. $E^s$, the \emph{unstable}, resp. \emph{stable}, Oseledets subspace associated to the \emph{positive}, resp. \emph{negative}, Lyapunov exponents of $G_t^{KZ,\mathbb{C}}|_{H^1(\varepsilon^4)}$, we obtain that $E^u$ and $E^s$ are \emph{isotropic} vector subbundles of the Hodge bundle with respect to the Hodge form $(.,.)$, i.e., 
$$E^u, E^s\subset\mathcal{C}\,,$$
where $\mathcal{C}=\{v\in H^1(\varepsilon^4):(v,v)=0\}$ is the \emph{light-cone} of the pseudo-Hermitian form $(.,.)|_{H^1(\varepsilon^4)}$ of signature $(3,1)$. 
At this point, the following general linear algebra fact is useful. 
\begin{lemma}\label{l.Upq} A vector space $V$ inside the light-cone of a pseudo-Hermitian form of signature $(p,q)$ has dimension $\min\{p,q\}$ at most.
\end{lemma} 

\begin{proof}By taking adequate coordinates, we may assume that $(.,.)$ is the standard Hermitian form of signature $(p,q)$ in $\mathbb{C}^{p+q}$:
$$(z,w)=z_1\overline{w_1}+\dots+z_p\overline{w_p} - z_{p+1}\overline{w_{p+1}} - z_{p+q}\overline{w_{p+q}}\,.$$
By symmetry, we can assume that $p\leq q$, i.e., $\min\{p,q\}=p$.

Suppose that $V$ is a vector space of dimension $\geq p+1$ inside the light-cone $\mathcal{C}$. Let's choose  a collection $v^{(1)},\dots,v^{(p+1)}\in V$ of $p+1$ linearly independent vectors and let's define $w^{(1)}:=(v^{(1)}_1,\dots,v^{(1)}_p), \dots, w^{(p+1)}:=(v^{(p+1)}_1,\dots, v^{(p+1)}_p) \in \mathbb{C}^p$, where $v^{(j)}_i\in\mathbb{C}$ is the $i$th coordinate of the vector $v^{(j)}$.
Since $w^{(1)},\dots,w^{(p+1)}$ is a collection of $p+1$ vectors in $\mathbb{C}^p$, we can find a non-trivial collection of coefficients $(a_1,\dots,a_{p+1})\in 
\mathbb{C}^{p+1}-\{0\}$ with 
$$\sum\limits_{j=1}^{p+1}a_jw^{(j)}=0\,.$$

Since $w^{(j)}$ were built from the $p$ first coordinates of $v^{(j)}$, we deduce that 
$$V\ni v:=\sum\limits_{j=1}^{p+1}a_j v^{(j)}=(\underbrace{0,\dots,0}_{p},v_{p+1},\dots,v_{p+q})\,.$$

However, since $v\in V\subset\mathcal{C}$, one would have 
$$0=(v,v)=-|v_{p+1}|^2-\dots-|v_{p+q}|^2$$
that is, $v=0$, a contradiction with the linear independence of $v^{(1)},\dots, v^{(p+1)}\in V$ because $(a_1,\dots,a_{p+1})\in\mathbb{C}^{p+1}-\{0\}$. This shows that $\textrm{dim} V\leq p$ whenever $V\subset\mathcal{C}$, as desired.
\end{proof}

By applying this lemma in the context of the cocycle $G_t^{KZ,\mathbb{C}}|_{H^1(\varepsilon^4)}$, we get that $E^u$ and $E^s$ have dimension $1$ \emph{at most} because they are in the light-cone of a pseudo-Hermitian form of signature $(3,1)$. Since $H^1(\varepsilon^4)=E^u\oplus E^c\oplus E^s$ and $H^1(\varepsilon^4)$ has dimension $4$, we deduce that 
$$\textrm{dim}(E^c)=4-\textrm{dim}(E^u)-\textrm{dim}(E^s)\geq 4-1-1=2,$$
i.e., $G_t^{KZ,\mathbb{C}}|_{H^1(\varepsilon^4)}$ has 2 zero Lyapunov exponents at least.
\end{proof} 

In fact, by computing the restriction of form $B_{\omega}$ to $H^{1,0}(\varepsilon^2)\oplus H^{1,0}(\varepsilon^4)$ in the basis of holomorphic differentials 
$$\{dx/y^2, dx/y^4, (x-x_1)dx/y^4, (x-x_1)^2dx/y^4\}$$
one can show that it has rank 2. By combining this with an analog of Theorem~\ref{t.F02} to $G_t^{KZ,\mathbb{C}}|_{H^1(\varepsilon^4)}$, one can show that $G_t^{KZ,\mathbb{C}}|_{H^1(\varepsilon^4)}$ has \emph{exactly} $2$ zero Lyapunov exponents, a \emph{positive} Lyapunov exponent $\lambda$ and a \emph{negative} Lyapunov exponent $-\lambda$. Furthermore, the natural isomorphism between $\mathcal{Z}$ and $\mathcal{H}^{(1)}(2)$ allows to compute the \emph{Siegel--Veech constant} (cf. Subsection~\ref{ss.EKZ}) of the measure $\mu_{\mathcal{Z}}$ and by Theorem~\ref{t.EKZ-formula} this information   leads to the explicit value $\lambda=4/9$. We refer the reader to \cite{FMZ2} and \cite{FMZ3} where this analysis is discussed in detail.

After this discussion, we understand \emph{completely} the Lyapunov spectrum of the KZ cocycle with respect to the ``Masur--Veech'' measure  $\mu_{\mathcal{Z}}$ of the locus 
$\mathcal{Z}$:

\begin{proposition}The non-negative part of the Lyapunov spectrum of the KZ cocycle with respect to 
the $SL(2,\mathbb R)$-invariant, $g_t$-ergodic probability measure $\mu_{\mathcal{Z}}$ on 
${\mathcal H}(8, 2^5)$  is 
$$\{1>4/9=4/9>1/3>0=0=0=0=0=0\}\,.$$
\end{proposition}  

Now, we pass to the study of the neutral Oseledets subspace of the KZ cocycle over $\mathcal{Z}$, or, more precisely\footnote{Here we ``excluded'' the part of the neutral Oseledets bundle coming from the blocks $H^1(\varepsilon)$ and $H^1(\varepsilon^5)$ because the complex KZ cocycle acts via $U(0,4)$ and $U(4,0)$, and it is not hard to show from this that the corresponding components of the neutral Oseledets bundle are ``geometrically explained'' in terms of  the annihilator of the second fundamental form $B^{\mathbb R}_{\omega}$ on these blocks as in the case of square-tiled cyclic covers.}, $G_t^{KZ,\mathbb{C}}|_{H^1(\varepsilon^4)}$. Here, it is worth to notice that the neutral Oseledets subspace of $G_t^{KZ,\mathbb{C}}|_{H^1(\varepsilon^4)}$ and the intersection of the annihilator of $B_{\omega}$ with $H^{1,0}(\varepsilon^2)\oplus H^{1,0}(\varepsilon^4)$ have the \emph{same} rank (namely, $2$). In particular, it is natural to ask whether these subspaces coincides, or equivalently, the neutral Oseledets subspace of $G_t^{KZ,\mathbb{C}}|_{H^1(\varepsilon^4)}$ has a nice geometrical explanation. 

This was shown \emph{not} to be true in \cite{FMZ3} along the following lines. Since the neutral Oseledets subspace is $g_t$-invariant and the annihilator of $B_{\omega}$ is continuous and $SO(2,\mathbb{R})$-invariant, the coincidence of these subspaces would imply that the neutral Oseledets subspace of $G_t^{KZ,\mathbb{C}}|_{H^1(\varepsilon^4)}$ is a (rank $2$) continuous $SL(2,\mathbb{R})$-invariant subbundle of $H^1(\varepsilon^4)$. This property imposes \emph{severe} restrictions on the behavior $G_t^{KZ,\mathbb{C}}|_{H^1(\varepsilon^4)}$: for instance, by considering two periodic (i.e., pseudo-Anosov) orbits of the Teichm\"uller flow in the \emph{same} $SL(2,\mathbb{R})$ associated to two Abelian differentials on the \emph{same} Riemann surface, we get that the matrices $A$ and $B$ representing $G_t^{KZ,\mathbb{C}}|_{H^1(\varepsilon^4)}$ along these periodic orbits must \emph{share} a common subspace of dimension $2$, and this last property can be \emph{contradicted} by \emph{explicitly} computing with some periodic orbits. Unfortunately, while this idea is very simple, the calculations needed to implement it are somewhat long and we will not try to reproduce them here. Instead, we refer to Appendix A of \cite{FMZ3} where the calculation is largely detailed (and illustrated with several pictures). Another consequence of this result is that the neutral Oseledets bundle is
\emph{not }$SL(2, \mathbb R)$-invariant, hence it is not $SO(2, \mathbb R)$-invariant (see also Corollary 1.1 of \cite{FMZ3}). In fact, by Theorem 3 of \cite{FMZ2} if it were $SL(2, \mathbb R)$-invariant, it would coincide with the annihilator of the second fundamental form and we have just seen that 
this is not true.

\begin{remark} During an exposition of this topic by one of us (C.~Matheus), Y.~Guivarch asked whether $G_t^{KZ,\mathbb{C}}|_{H^1(\varepsilon^4)}$ still acts \emph{isometrically} on its neutral subspace. This question is very natural and interesting because now that one \emph{can't} use variational formulas involving $B_{\omega}$ to deduce this property (as we did in the case of \emph{square-tiled cyclic covers}. As it turns out, the answer to Guivarch's question is positive by the following argument: one has that the neutral Oseledets subspace $E^c$ is \emph{outside} the light-cone $\mathcal{C}$ because the stable Oseledets subspace $E^s$ has dimension $1$, and so, if $E^c\cap\mathcal{C}$ were non-trivial, we would get a subspace $(E^c\cap\mathcal{C})\oplus E^s\subset \mathcal{C}$ of  dimension at least $2$ inside the light-cone $\mathcal{C}$ of a Hermitian form of signature $(3,1)$, a contradiction with Lemma~\ref{l.Upq} above. In other words, the light-cone provides a geometric mechanism to produce neutral Oseledets subbundles with isometric behavior \emph{genuinely different} from the (also geometric) mechanism based on the $SL(2,\mathbb R)$-invariance of the annihilator of the second fundamental form $B_\omega$ of the Gauss-Manin connection of the Hodge bundle.
\end{remark}

Of course, the fact that the neutral Oseledets subspace \emph{doesn't} coincide with the annihilator of the second fundamental form $B_{\omega}$ is not the ``end of the road'': indeed, by carefully inspecting the arguments of the previous paragraph one notices that it leaves \emph{open} the possibility that the neutral Oseledets subspace maybe \emph{continuous} despite the fact that it is not the annihilator of $B_{\omega}$.

Heuristically, one strategy to ``prove'' that the neutral Oseledets subspace is not very smooth goes as follows: as we know, the Lyapunov exponents of the Teichm\"uller flow can be deduced from the ones of the KZ cocycle by shifting them by $\pm1$; in this way, the smallest non-negative Lyapunov exponent of the Teichm\"uller flow is $5/9=1-4/9$; therefore, the generic points tend to be separated by Teichm\"uller flow by $\geq e^{5t/9}$ after time $t\in\mathbb{R}$; on the other hand, the largest Lyapunov exponent on the fiber $H^1(\varepsilon^4)$ is $4/9$, so that the angle between the neutral Oseledets bundle over two generic points grows by $\leq e^{4t/9}$ after time $t\in\mathbb{R}$; hence, in general, one \emph{can't} expect the neutral Oseledets bundle to be \emph{better} than $\alpha=(4/9)/(5/9)=4/5$ H\"older continuous. 

Of course, there are several details missing in this heuristic, and currently we don't know how to 
make it into a formal argument. However, in a recent work still in progress \cite{AMY}, A.~Avila, 
C.~Matheus and J.-C.~Yoccoz proved (among other things) that the neutral Oseledets subspace $E^c$ of $G_t^{KZ,\mathbb{C}}|_{H^1(\varepsilon^4)}$ is \emph{not} continuous at all (and hence \emph{only} measurable by Oseledets theorem). In the sequel, we provide a brief sketch of this proof 
of the non-continuity of $E^c$.

As we mentioned a few times in this text, the Teichm\"uller flow and the Kontsevich--Zorich cocycle over (connected components of) strata can be efficiently \emph{coded} by means of the so-called \emph{Rauzy--Veech induction}. Roughly speaking, given a (connected component of a) stratum $\mathcal{C}$ of Abelian differentials of genus $g\geq 1$, the Rauzy--Veech induction associates the following objects: a finite oriented \emph{graph} $\mathcal{G}(\mathcal{C})$ (``Rauzy graph''), a finite collection of \emph{simplices} (``Rauzy--Veech boxes''), a finite number of copies of the Euclidean space $\mathbb{C}^{2g}$ over each vertex of $\mathcal{G}(\mathcal{C})$, and, for each arrow of $\mathcal{G}(\mathcal{C})$, a (expanding) projective map between simplices over the vertices connected by this arrow, and a linear map (a matrix) between the copies of $\mathbb{C}^{2g}$ over the vertices connected by this arrow. We strongly recommend J.-C.~Yoccoz's survey \cite{Y} for more details on the Rauzy--Veech induction.

In this language, the simplices (Rauzy--Veech boxes) over the vertices of the Rauzy graph represent admissible paramaters determining translations surfaces (Abelian differentials on Riemann surfaces $M$) in $\mathcal{C}$, the (expanding) projective maps between simplices (associated to vertices connected by a given arrow) correspond to the action of the Teichm\"uller flow on the parameter space (after running this flow for an adequate amount of time), a copy of the Euclidean space $\mathbb 
C^{2g}$ at a vertex of the Rauzy graph corresponds to the complex cohomology  $H^1(M,\mathbb{C})$ of the surface $M$, and the linear maps on the Euclidean spaces $\mathbb{C}^{2g}$, that is, the matrices, attached to the arrows correspond to the action of the Kontsevich--Zorich cocycle on $H^1(M,\mathbb{C})\simeq\mathbb{C}^{2g}$. 

Among the main properties of the Rauzy--Veech induction, we can highlight the fact that it alllows to ``simulate'' almost every (\emph{with respect to Masur--Veech measures}) orbit of the Teichm\"uller flow on on $\mathcal{C}$, in the sense that these trajectories correspond to (certain) infinite paths on the Rauzy graph $\mathcal{G}(\mathcal{C})$. In order words, the Rauzy--Veech induction allows to code the Teichm\"uller flow as a subshift of a \emph{Markov shift} on \emph{countably many} symbols (as one can use loops on $\mathcal{G}(\mathcal{C})$ based on an arbitrarily fixed vertex as basic symbols / letters of the alphabet of our Markov subshift). Moreover, the KZ cocycle over these trajectories of Teichm\"uller flow can be computed by simply multiplying the matrices attached to the arrows one sees while following the corresponding infinite path on $\mathcal{G}(\mathcal{C})$. Equivalently, we can think the KZ cocycle as a \emph{monoid} of (countably many) matrices (as we can \emph{only} multiply the matrices precisely when our \emph{oriented} arrows can be concatened, but in principle we don't dispose of the inverses of our matrices because we don't have the right to ``revert'' the orientation of the arrows). 

In the particular case of $\mathcal{H}(2)$, the associated graph $\mathcal{G}(\mathcal{H}(2))$ is depicted below: 
\begin{center}
\begin{figure}[htb!]
\includegraphics[scale=0.6]{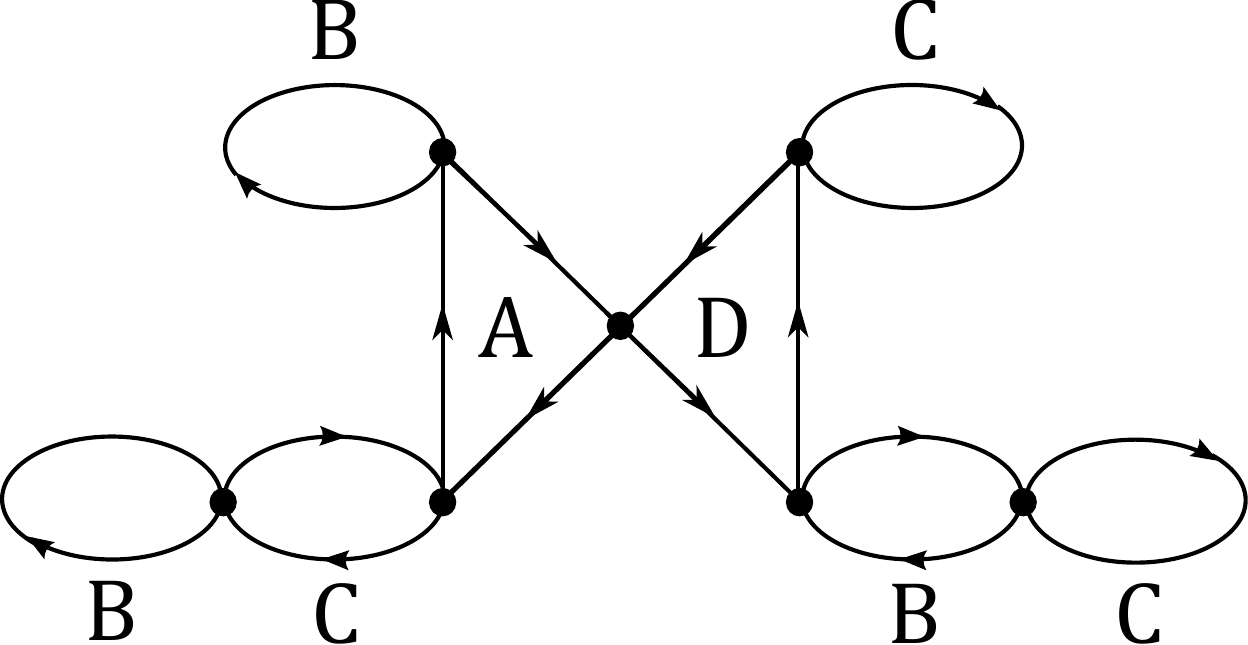}
\caption{Schematic representation of the Rauzy diagram associated to $\mathcal{H}(2)$. The letters near the arrows are not important here, only the $7$ vertices (black dots) and the arrows between them.}\label{f.3}
\end{figure}
\end{center}

Now, we observe that $\mathcal{Z}$ was defined by taking certain triple covers of Abelian differentials of $\mathcal{H}(2)$, so that it is also possible to code the Teichm\"uller flow and KZ cocycle on $\mathcal{Z}$ by the same graph and the same simplices over its vertices, but by changing the matrices attached to the arrows: in the case of $\mathcal{H}(2)$, these matrices acted on $\mathbb{C}^4$, but in the case of $\mathcal{Z}$ they act on $\mathbb{C}^{20}$ and they contain the matrices for $\mathcal{H}(2)$ as  blocks in a certain position. After these preliminaries, one can prove that the neutral Oseledets subspace $E^c$ of $G_t^{KZ,\mathbb{C}}|_{H^1(\varepsilon^4)}$ is not continuous as follows. 

Firstly, one computes the restriction of the KZ cocycle (or rather the matrices of the monoid) to $E^c$ on certain ``\emph{elementary}'' loops and one checks that they have \emph{finite order}. In particular, every time we can get the inverses of the matrices associated to these elementary loops by simply repeating these loops an appropriate number of times (namely, the order of the matrix minus 1). On the other hand, since these elementary loops are set up so that any infinite path (coding a Teichm\"uller flow orbit) is a concatenation of elementary loops, one conclude that the action (on $E^c\subset\mathbb{C}^{20}$) of our \emph{monoid} of matrices is through a \emph{group}! In particular, given any loop $\gamma$ (not necessarily an elementary one), we can find another loop $\delta$ such that the matrix attached to $\delta$ (i.e., the matrix obtained by multiplying the matrices attached to the arrows forming $\delta$ ``in the order they show up'' with respect to their natural orientation of $\delta$) is \emph{exactly} the inverse of the matrix attached to $\gamma$.

Secondly, by computing with a pair of ``sufficiently random'' loops $\gamma_A$ and $\gamma_B$, it is not hard to see that we can choose them such that their attached matrices $A$ and $B$ have distinct and/or transverse central eigenspaces $E^c_A$ and $E^c_B$ (associated to eigenvalues of modulus $1$). 

In this way, the periodic orbits (\emph{pseudo-Anosov} orbits) of the Teichm\"uller flow coded by the infinite paths $\dots\gamma_A\gamma_A\gamma_A\dots$ and $\dots\gamma_B\gamma_B\gamma_B\dots$, obtained by infinite concatenation of the loops $\gamma_A$ and $\gamma_B$, have distinct and/or transverse neutral Oseldets bundle, but this is no contradiction to continuity since the base points of these periodic orbits are not very close. However, we can use $\gamma_A$ and $\gamma_B$ to produce a contradiction as follows. Let $k\gg 1$ a large integer. Since our monoid acts by a group, we can find a loop $\gamma_{C,k}$ such that the matrix attached to it is $A^{-k}$. It follows that the matrix attached to the loop $\gamma_{A,B,k}:=\underbrace{\gamma_A\dots\gamma_A}_{k}\gamma_{C,k}\gamma_B$ is $A^k\cdot A^{-k}\cdot B$, i.e., $B$. Therefore, the infinite paths $\dots\gamma_A\gamma_A\gamma_A\dots$ and $\dots\gamma_{A,B,k}\gamma_{A,B,k}\gamma_{A,B,k}\dots$ correspond to periodic orbits whose neutral Oseledets bundle still are $E^0_A$ and $E^0_B$ (and hence, distinct and/or transverse), but this time their basepoints are arbitrarily close (as $k\to\infty$) because the first $k$ ``symbols'' (loops) of the paths coding them are equal (to $\gamma_A$).  

\begin{remark} Actually, this argument is part of more general considerations in \cite{AMY} on certain cyclic covers obtained by taking $2n$ copies of a regular polygon with $m$ sides, and cyclically gluing the sides of these polygons in such a way that their middle points become ramification points: indeed, $\mathcal{Z}$ corresponds to the case $n=3$ and $m=5$ of this construction. 
\end{remark}

\begin{remark} It is interesting to notice that the real version of Kontsevich--Zorich cocycle over $\mathcal{Z}$ on $(H^1(\varepsilon^2)\oplus H^1(\varepsilon^4))\cap H^1(M_{10},\mathbb{R})$ is an \emph{irreducible} symplectic cocycle with non-continuous neutral Oseledets bundle. In principle, this irreducibility at the real level makes it difficult to see the presence of zero exponents, so that the passage to its complex version (where we can decompose it as a sum of two complex conjugated monodromy representations by matrices in $U(1,3)$ and $U(3,1)$) reveals a ``hidden truth'' not immediately detectable from the real point of view (thus confirming the famous quotation of J.~Hadamard: ``\emph{the shortest route between two truths in the real domain passes through the complex domain}''). We believe this example has some independent interest because, to the best of our knowledge, most examples of symplectic cocycles and/or diffeomorphisms exhibiting some zero Lyapunov exponents usually have smooth neutral Oseldets bundle due to some sort of ``invariance principle'' (see this article of A.~Avila and M.~Viana \cite{AV-InvPrinc} for some illustrations of this phenomenon).
\end{remark}

We state below two ``optimistic guesses'' on the features of the KZ cocycle over the support of general $SL(2,\mathbb{R})$-invariant probability measures mostly based on our experience so far with cyclic covers. Notice that we call these ``optimistic guesses'' instead of ``conjectures'' because we think they're shared (to some extent) by others working with Lyapunov exponents of the KZ cocycle (and so it would be unfair to state them as ``our'' conjectures).

\medskip 

\noindent \textbf{Optimistic Guess 1.} \textit{Let $\mu$ be a $SL(2,\mathbb{R})$-invariant probability on some connected component of a stratrum of Abelian differentials and denote by $\mathcal{L}$ its support. Then, there exists a finite (ramified) cover $\widehat{\mathcal{L}}$ such that (the lift of) the Hodge bundle $H^1_{\mathbb{C}}$ over 
$\widehat{\mathcal{L}}$ can be decomposed into a direct sum of continuous subbundles as follows:
$$H^1_{\mathbb{C}} = L\oplus (A_1\otimes W_1 \oplus\dots\oplus A_m\otimes W_m) \oplus (B_1\otimes(U_1\oplus \overline{U_1})\oplus\dots\oplus B_1\otimes(U_n\oplus \overline{U_n}))\,,$$
where $W_1,\dots,W_m$, $U_1,\dots, U_n$ are distinct $SL(2,\mathbb{R})$-irreducible representations admiting Hodge filtrations $W_i=W_i^{1,0}\oplus W_i^{0,1}$, $U_j= U_j^{1,0}\oplus U_j^{0,1}$ such that $W_i=\overline{W_i}$, $U_j\cap\overline{U_j}=\{0\}$, $A_i$, $B_j$ are complex vector spaces (taking into account the multiplicities of the irreducible factors $W_i$, $U_j$), and $L$ is the tautological bundle $L=L^{1,0}\oplus L^{0,1}$, $L^{1,0}=\mathbb{C}\omega$, $L^{0,1}=\mathbb{C}\overline{\omega}$, $\omega\in\widehat{\mathcal{L}}$. Moreover, this decomposition is unique and it can't be refined after passing to any further finite cover.}

\begin{remark} Whenever $\mu$ is the (unique) $SL(2,\mathbb{R})$-invariant probability supported on 
a Teichm\"uller curve $\mathcal{L}$ (i.e., a closed $SL(2,\mathbb{R})$-orbit, coming from a Veech surface), the  Optimistic Guess 1 is a consequence of \emph{Deligne's semisimplicity theorem} 
\cite{Del}.
\end{remark}

\noindent \textbf{Optimistic Guess 2.} \textit{In the setting of Optimistic Guess 1, denote by 
$$p_i=q_i=r_i=\textrm{dim}_{\mathbb{C}}W_i^{1,0}=\textrm{dim}_{\mathbb{C}}W_i^{0,1}$$
and 
$$p_j=\textrm{dim}_{\mathbb{C}}U_j^{1,0}, q_j =\textrm{dim}_{\mathbb{C}}W_j^{0,1}, r_j=\min\{p_j,q_j\}$$
Then, the Lyapunov spectrum of the KZ cocycle on $W_i$ is simple, i.e., 
$$\lambda_{i,1}>\dots>\lambda_{i,r_i}>-\lambda_{i,r_i}>\dots>-\lambda_{i,1}$$
and the Lyapunov spectrum of the KZ cocycle on $U_j$ is ``as simple as possible'', i.e., 
$$\lambda_{j,1}>\dots>\lambda_{j,r_j}>\underbrace{0=\dots=0}_{|q_j-p_j|}>-\lambda_{j,r_j}>\dots>-\lambda_{j,1}$$}

\begin{remark} This ``guess'' is based on the general philosophy (supported by work on Lyapunov exponents for random cocycles such as the papers of A.~Raugi and Y.~Guivarch \cite{GR}, and I.~Goldscheid and G.~Margulis \cite{GM}) that, after reducing our cocycle to irreducible pieces, if the cocycle restricted to such a piece is ``sufficiently generic'' inside a certain Lie group of matrices $G$, then the Lyapunov spectrum on this piece should look like the ``Lyapunov spectrum'' (i.e., collection of the logarithms of the norms of eigenvalues) of the ``generic'' matrix of $G$. For instance, since a generic matrix inside the group 
$U(p,q)$ has spectrum 
$$\lambda_1>\dots>\lambda_r>\underbrace{0=\dots=0}_r>-\lambda_r>\dots>-\lambda_1\,,$$
where $r=\min\{p,q\}$, the above guess essentially claims that, once one reduces the KZ cocycle to irreducible pieces, its Lyapunov spectrum on each piece must be as generic as possible.
\end{remark}

\begin{remark} Notice that our Optimistic Guess 2 above doesn't make any attempt to compare Lyapunov exponents within distinct irreducible factors: indeed, in general non-isomorphic representations may lead to the same exponent by ``pure chance'' (as it happens in the case of 
certain genus $5$ Abelian differentials associated to the ``wind-tree model'', cf. \cite{DHL}).
\end{remark}

We close this section by mentioning that in Appendix~\ref{s.finalremarks} below we present some 
recent results on both the non-simplicity and simplicity of the Lyapunov spectrum of the KZ cocycle in the context of \emph{square-tiled surfaces}.


\section{Arithmetic Teichm\"uller curves with complementary series}\label{s.complementary-series}

During some conversations of one if us (C.~Matheus) with Artur Avila and Jean-Christophe Yoccoz about the $SL(2,\mathbb{R})$ action on the moduli space of Abelian differentials and spectral gap of the corresponding $SL(2,\mathbb{R})$ unitary representations, we showed the following result:

\begin{theorem}\label{t.AMY}There are Teichm\"uller curves\footnote{Recall that \emph{Teichm\"uller curves} are a shorthand for \emph{closed} $SL(2,\mathbb{R})$-orbits in $\mathcal{H}_g$. See Subsection~\ref{ss.Shimura-Teich} of Section~\ref{s.Veech} for more comments on Teichm\"uller curves.} (actually $SL(2,\mathbb{R})$ orbits of square-tiled surfaces) with complementary series. 
\end{theorem}

The proof of this result is not very long assuming previous knowledge of the theory of unitary representations of $SL(2,\mathbb{R})$. In particular, we'll borrow the notations from Appendix~\ref{a.Rt} where the reader is quickly reminded of main results in this theory (e.g., Bargmann's classification) and its connection to Ratner's work~\cite{Rt} on the rates of mixing of geodesic flows. Then, we'll combine this knowledge with a recent theorem of J. Ellenberg and D. B. McReynolds~\cite{EM} to construct the desired square-tiled surfaces by a certain \emph{cyclic cover construction} and a \emph{reverse Ratner estimate}.

\begin{remark} The algebraic part of the proof of Theorem~\ref{t.AMY} was already known to A.~Selberg: in fact, as pointed out to us by N.~Bergeron and P.~Hubert, the same cyclic covering construction giving arbitrarily small first eigenvalue of $S=\Gamma\backslash\mathbb{H}$ (i.e., arbitrarily small spectral gap) was first found by Selberg. The reader can find an exposition of this argument in the subsection 3.10.1 of Bergeron's book~\cite{B}. In particular, despite the ``difference'' between Selberg argument and ours, that is, while Selberg focuses on the first eigenvalue of the Laplacian $\Delta_S$ on the hyperbolic surface $S$, we rely on  the dynamical properties of the geodesic flow (more precisely the rates of mixing) on the unit tangent bundle $T_1(S)$ and a reverse Ratner estimate, it is clear that both arguments are essentially the same
\end{remark} 

\subsection{A theorem  of J. Ellenberg and D. McReynolds}

Recall that a \emph{square-tiled surface} is a Riemann surface $M$ obtained by gluing the sides of a finite collection of unit squares of the plane, so that a left side (resp., bottom side) of one square is always glued with a right side (resp., top side) of another square, together with the Abelian differential 
$\omega$ on $M$ induced by the quotient of $dz$ under these identifications. As we know, square-tiled surfaces are dense in the moduli space of Abelian differentials (because $(M,\omega)$ is square-tiled iff the periods of $\omega$ are rational) and the $SL(2,\mathbb{R})$-orbit of any square-tiled surface is a closed submanifold of $\mathcal{H}_g$. Such a submanifold can be identified with $\Gamma\backslash SL(2,\mathbb{R})$, where $\Gamma$ is the \emph{Veech group} of our square-tiled surface (i.e., the finite-index subgroup $\Gamma$ of $SL(2,\mathbb{Z})$ stabilizing the $SL(2,\mathbb{R})$ orbit of our square-tiled surface in the moduli space). Furthermore, the Teichm\"uller geodesic flow on the 
$SL(2,\mathbb{R})$-orbit of a square-tiled surface corresponds to the geodesic flow of  $\Gamma\backslash\mathbb{H}$, where $\Gamma$ is a finite-index subgroup of $SL(2,\mathbb{Z})$ (hence $\Gamma$ is a lattice of $SL(2,\mathbb{R})$). In the converse direction, J.~Ellenberg and D.~McReynolds~\cite{EM} recently proved that:

\begin{theorem}[Ellenberg and McReynolds] Any finite-index subgroup $\{\pm1\}\subset \Gamma$ of the congruence subgroup\footnote{Recall that the (principal) congruence subgroup $\Gamma(m)$ of $SL(2,\mathbb{Z})$ is $\Gamma(m):=\left\{\left(\begin{array}{cc}a&b\\ c&d\end{array}\right)\in SL(2,\mathbb{Z}):a\equiv d\equiv 1 (\textrm{mod } m), b\equiv c\equiv 1 (\textrm{mod } m)\right\}$.} $\Gamma(2)\subset SL(2,\mathbb{Z})$ is the Veech group of some square-tiled surface.
\end{theorem}

\subsection{Teichm\"uller curves with complementary series}

We are ready to prove Theorem~\ref{t.AMY} claiming that there are square-tiled surfaces such that the representation $\rho_{\mathcal{S}}$ associated to its $SL(2,\mathbb{R})$-orbit $\mathcal{S}$ has irreducible factors in the complementary series.

Observe that the natural identification between $SL(2,\mathbb{R})$-orbits $\mathcal{S}$ of a square-tiled surface and the unit cotangent bundle of $\Gamma\backslash\mathbb{H}$, where $\Gamma$ is the corresponding Veech group allows to think of $\rho_{\mathcal{S}}$ as the regular unitary $SL(2,\mathbb{R})$-representation $\rho_{\Gamma}$ on the space $L^2_0(\Gamma\backslash\mathbb{H},\nu_{\Gamma})$ of zero-mean $L^2$-functions with respect to the natural measure $\nu_{\Gamma}$ on $\Gamma\backslash\mathbb{H}$.

In view of the above theorem of Ellenberg and McReynolds, it suffices to find a finite-index subgroup $\Gamma\subset\Gamma(2)$ such that $\rho_{\Gamma}$ has complementary series. As we promised, this will be achieved by a cyclic covering procedure. 

Firstly, we fix a congruence subgroup $\Gamma(m)$ such that the corresponding modular curve $\Gamma(m)\backslash\mathbb{H}$ has genus $g\geq1$, e.g., $\Gamma(6)$. Next, we fix a homotopically non-trivial closed geodesic $\beta$ of $\Gamma(m)\backslash\mathbb{H}$ \emph{after the compactification of its cusps} and we construct a cyclic covering of $\Gamma(m)\backslash\mathbb{H}$ (i.e., we choose a subgroup $\Gamma\subset\Gamma(m)$) of high degree $N$ such that a lift $\beta_N$ of $\beta$ satisfies $\ell(\beta_N)=N\cdot\ell(\beta)$. This construction is illustrated in
Figure~\ref{f.selberg} below.

\begin{figure}[htb!]
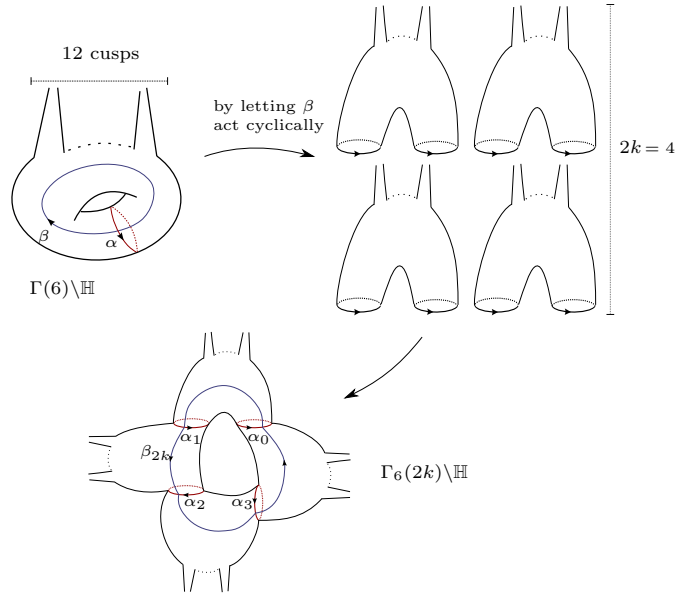
\caption{Cyclic cover construction of a subgroup $\Gamma_{6}(2k)$ of $\Gamma(6)$ of index $N=2k$.}\label{f.selberg}
\end{figure}

We claim that $\rho_{\Gamma}$ has complementary series, i.e., $\sigma(\Gamma)>-1$ (i.e., $\beta(\Gamma)>-1/4$) for a sufficiently large $N$. Actually, we will show a little bit more: $\sigma(\Gamma)$ is \emph{arbitrarily close} to $0$ for large $N$ (i.e., the spectral gap of $\Gamma$ can be made arbitrarily small). 

Let us select two small open balls $U$ and $V$ of area $1/N$ whose respective centers are located at two points of $\beta_N$ belonging to fundamental domains of the cyclic covering $\Gamma\backslash\mathbb{H}$ very far apart, so that the distance between the centers of $U$ and $V$ is $\sim N/2$. 
\begin{center}

\begingroup
  \makeatletter
  \providecommand\color[2][]{%
    \errmessage{(Inkscape) Color is used for the text in Inkscape, but the package 'color.sty' is not loaded}
    \renewcommand\color[2][]{}%
  }
  \providecommand\transparent[1]{%
    \errmessage{(Inkscape) Transparency is used (non-zero) for the text in Inkscape, but the package 'transparent.sty' is not loaded}
    \renewcommand\transparent[1]{}%
  }
  \providecommand\rotatebox[2]{#2}
  \ifx\svgwidth\undefined
    \setlength{\unitlength}{264.43117676pt}
  \else
    \setlength{\unitlength}{\svgwidth}
  \fi
  \global\let\svgwidth\undefined
  \makeatother
  \begin{picture}(1,0.92611776)%
    \put(0,0){\includegraphics[width=\unitlength]{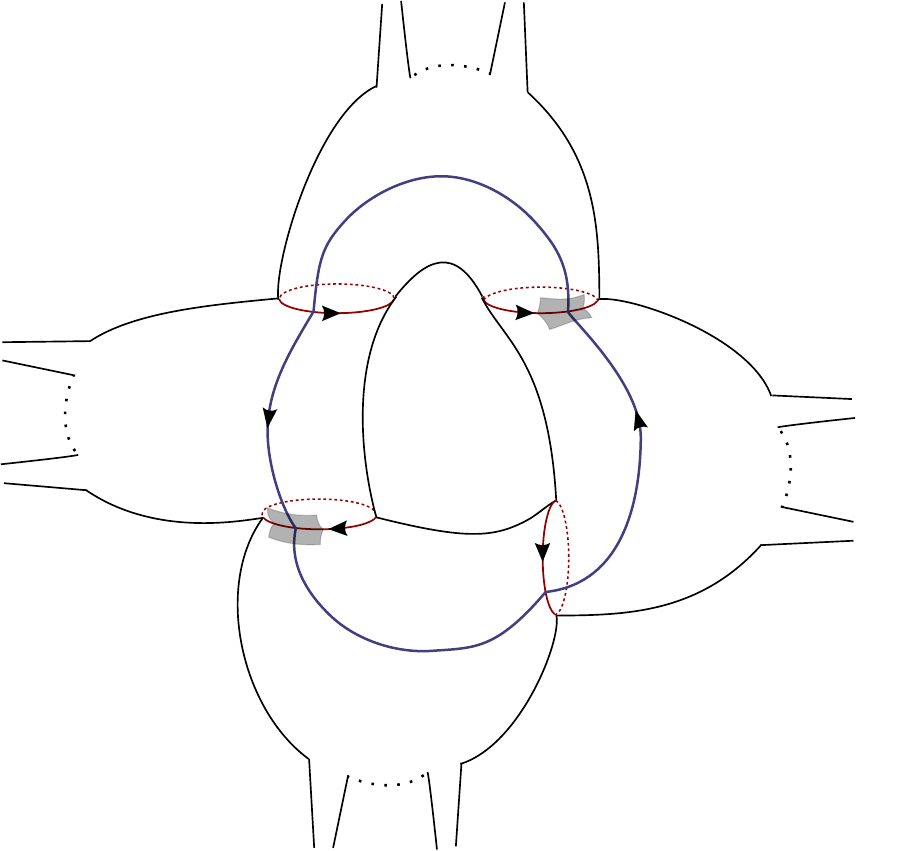}}%
    \put(0.55766226,0.55181399){\small\color[rgb]{0,0,0}\makebox(0,0)[lb]{\smash{$\alpha_0$}}}%
    \put(0.33960562,0.55288361){\small\color[rgb]{0,0,0}\makebox(0,0)[lb]{\smash{$\alpha_1$}}}%
    \put(0.34870267,0.32091464){\small\color[rgb]{0,0,0}\makebox(0,0)[lb]{\smash{$\alpha_2$}}}%
    \put(0.53785047,0.32519314){\small\color[rgb]{0,0,0}\makebox(0,0)[lb]{\smash{$\alpha_3$}}}%
    \put(0.07882695,0.84036496){\small\color[rgb]{0,0,0}\makebox(0,0)[lb]{\smash{$\Gamma_6(2k)\backslash\mathbb{H}$}}}%
    \put(0.22390509,0.47952997){\small\color[rgb]{0,0,0}\makebox(0,0)[lb]{\smash{$\beta_{2k}$}}}%
    \put(0.27390509,0.29905502){\small\color[rgb]{0,0,0}\makebox(0,0)[lb]{\smash{$U$}}}%
    \put(0.64383479,0.55559003){\small\color[rgb]{0,0,0}\makebox(0,0)[lb]{\smash{$V$}}}%
  \end{picture}%
\endgroup

\end{center}
Let us define $u=\sqrt{N}\cdot\chi_U$, $v=\sqrt{N}\cdot\chi_V$. Take $f=u-\int u$ and $g=v-\int v$ the zero mean parts of $u$ and $v$. Assume that there exists some $\varepsilon_0>0$ such that $\sigma(\Gamma)<-\varepsilon_0$ for every $N\in \mathbb N$. By Ratner's theorem~\ref{t.Ratner}, it follows 
that 
$$|\langle f, \rho_{\Gamma}(a_t)g\rangle|\leq C(\varepsilon_0)\cdot e^{\sigma(\Gamma)\cdot t}\|f\|_{L^2}\|g\|_{L^2}\,, \quad \text{ \rm for any }\, |t|\geq 1\,.$$
However, since the distance between the centers of $U$ and $V$ is $\sim N/2$, the support of $u$ is disjoint from the image of the support of $v$ under the geodesic flow $a(t_N)$ for a time $t_N\sim N/2$. Thus, $\langle u,\rho_{\Gamma}(a(t_N))v\rangle = \int u\cdot v\circ a(t_N) = 0$, and, \emph{a fortiori}, 
$$|\langle f, \rho_{\Gamma}(a_t)g\rangle| = \left|\int u\cdot v\circ a(t_N) - \int u\cdot\int v\right| = \int u\cdot\int v\sim 1/N.$$
By putting these two estimates together, since by construction we have that $\|f\|_{L^2}\leq\|u\|_{L^2}\leq 1$ and $\|g\|_{L^2}\leq\|v\|_{L^2}\leq 1$, we derive the inequality
$$1/N\leq C(\varepsilon_0) e^{-\varepsilon_0\cdot N/2}\,.$$
In particular, $\varepsilon_0\leq C(\varepsilon_0)\cdot\frac{\ln N}{N}$, a contradiction for a sufficiently large $N$. 

\subsection{Explicit square-tiled surfaces with complementary series} The curious reader may ask whether Theorem~\ref{t.AMY} admits an \emph{explicit} (or \emph{effective}) version, that is, whether it is possible to actually \emph{exhibit} square-tiled surfaces with complementary series. 

Of course, the naive strategy is to make the arguments of the previous two subsections as explicit as possible. By trying to do so, we notice that there are essentially two places where one needs to pay attention: 
\begin{itemize}
\item firstly, in Ellenberg-McReynolds theorem, given a group $\Gamma$, we need to know explicitly a square-tiled surface with Veech group $\Gamma$;
\item secondly, for the explicit construction of a group $\Gamma$ with complementary series, we need explicit constants in Ratner's theorem~\ref{t.Ratner}: indeed, in terms of the notation of the previous subsection, by taking $\varepsilon_0=1$, we need to know the constant $C(\varepsilon_0)=C(1)$ in order to determine a value of $N$ (and hence $\Gamma$) violating the inequality $1\leq C(1)\frac{\ln N}{N}$ (imposed by a hypothetical absence of the complementary series).
\end{itemize}

This indicates that a straightforward implementation of the naive strategy might be \emph{tricky}:
\begin{itemize}
\item a closer inspection of the methods of J.~Ellenberg and D.~McReynolds reveals that the construction of a square-tiled surface $M$ with a prescribed Veech group $\Gamma\subset \Gamma(2)$ passes by several covering processes, and, in particular, the total number of squares of $M$ tend to grow very fast as the index $[\Gamma(2):\Gamma]$ increases;  
\item even though the constant $C(1)$ in Ratner's theorem is rather explicit (after a tedious bookkeeping of constants in Ratner's original argument one can check\footnote{See the blog post ``\emph{Explicit constants in RatnerÕs estimates on rates of mixing of geodesic flows of hyperbolic surfaces}'' at the second author's mathematical blog \cite{DM}.} that $C(1)<25$), since the function $\ln N/N$ decays ``slowly'', the first values of $N$ violating the inequality $1\leq C(1)\ln N/N$ are likely to be large; because $N$ is directly related to the index of $\Gamma$ in $\Gamma(2)$, this indicates that $\Gamma$ has large index in $\Gamma(2)$.
\end{itemize}

However, one can \emph{slightly} improve the implementation of this strategy by recalling that the presence of complementary series is detected by the first eigenvalue $\lambda_1(\Gamma)$ of the Laplacian on $\Gamma\backslash\mathbb{H}$ and the so-called \emph{Cheeger--Buser inequality} provides fairly good bounds on 
$\lambda_1(\Gamma)$ in terms of the geometry of $\Gamma\backslash\mathbb{H}$. More precisely, by denoting by $\Gamma=\Gamma_{2k}(6)\subset\Gamma(6)$ the subgroup constructed by the method indicated in Figure~\ref{f.selberg}, the Cheeger-Buser inequality says that 
\begin{equation}\label{e.Buser}
\sqrt{10\lambda_1(\Gamma_6(2k))+1}\leq 10 h(\Gamma_6(2k))+1\,,
\end{equation}
where $h(\Gamma_6(2k))$ is the Cheeger constant of the hyperbolic surface $\Gamma_6(2k)\backslash\mathbb{H}$:
$$h(\Gamma_6(2k)):=\inf\limits_{\substack{\gamma \textrm{ multicurve separating }\Gamma_6(2k)\backslash\mathbb{H} \\ \textrm{ into two connected open regions }A, B}}\frac{\textrm{length}(\gamma)}{\min\{\textrm{area}(A), \textrm{area}(B)\}}$$
By numbering in Figure~\ref{f.selberg} the preimages $\alpha_0,\dots,\alpha_{2k-1}\subset \Gamma_6(2k)\backslash\mathbb{H}$ of the curve $\alpha\subset\Gamma(6)\backslash\mathbb{H}$ in the order we ``see'' them along $\beta$ and by taking $\gamma=\alpha_0\cup\alpha_k$, we obtain a multicurve such that: 
\begin{itemize}
\item $\textrm{length}(\gamma)=2\cdot\textrm{length}(\alpha)=4\cdot\textrm{arc cosh}(17)$, 
\item $\Gamma_6(2k)\backslash\mathbb{H}-\gamma=A\cup B$ and $\textrm{area}(A)=\textrm{area}(B)=k\cdot\textrm{area}(\Gamma(6)\backslash\mathbb{H})=k\cdot 24\pi$.
\end{itemize}
Therefore, $h(\Gamma_6(2k))\leq \textrm{arc cosh}(17)/(k\cdot 6\pi)$ and one can use Cheeger--Buser inequality~\eqref{e.Buser} to conclude that 
$$\lambda_1(\Gamma_6(2k))\leq 1/(2k)<1/4\,, \quad \text{ \rm for }  k\geq 3\,.$$
By Appendix~\ref{a.Rt}, this implies that, e.g., $\Gamma_6(6)\backslash\mathbb{H}$ has complementary series. 

This reduces the problem of construction of explicit square-tiled surfaces with complementary series to find some square-tiled surface with Veech group $\Gamma_6(6)$. At this point, one can improve \emph{again} over the naive strategy above by a \emph{partial} application of the methods of J.~Ellenberg and D.~McReynolds. More precisely, since $\lambda_1(\widetilde{\Gamma}\backslash\mathbb{H})\leq\lambda_1(\Gamma\backslash\mathbb{H})$ whenever $\widetilde{\Gamma}\subset\Gamma$ (i.e., $\widetilde{\Gamma}\backslash\mathbb{H}$ covers $\Gamma\backslash\mathbb{H}$), it suffices to construct a square-tiled surface with Veech group $\Gamma\subset \Gamma_6(6)$ and, for this purpose, we don't have to follow \cite{EM} until the end: by doing so, we ``save'' a few ``covering steps'' needed when one wants the Veech group to be \emph{exactly} $\Gamma_6(6)$. In other words, by ``stopping'' the arguments in \cite{EM} \emph{earlier}, we get ``only'' a square-tiled surface $M$ with Veech group $\Gamma\subset \Gamma_6(6)$ but we ``reduce'' the total number of squares of $M$ because we \emph{don't} insist on taking further ``coverings steps'' to get Veech group \emph{equal} to $\Gamma_6(6)$. In fact, this ``mildly improved'' strategy was pursued in the article~\cite{MSch} by Gabriela Schmith\"usen and the second author were it is proved that:

\begin{theorem} There exists an explicit pair of permutations $h,v\in S_{576}$ in $576$ elements determining\footnote{Cf. Appendix~\ref{s.Eskin-Okounkov} for more comments on this construction.} a square-tiled surface $(M,\omega)\in\mathcal{H}(1,5,5,5,\underbrace{2,\dots,2}_{138})$ with $576$ squares, genus $147$, and Veech group $\Gamma\subset\Gamma_6(6)$. In particular, the Teichm\"uller curve $SL(2,\mathbb{R})\cdot(M,\omega)$ has complementary series.
\end{theorem}

We close this section by referring the reader to \cite{MSch} for the explicit pair of permutations $h,v$ quoted above and a complete proof of this result.






\appendix

\section{Representation theory of $SL(2,\mathbb{R})$ and Ratner's work}\label{a.Rt}
Let $\rho:SL(2,\mathbb{R})\to U(\mathcal{H})$ be a unitary \emph{representation} of $SL(2,\mathbb{R})$, i.e., $\rho$ is a homomorphism from $SL(2,\mathbb{R})$ into the group $U(\mathcal{H})$ of \emph{unitary} transformations of the \emph{complex} separable Hilbert space $\mathcal{H}$. We say that a vector $v\in\mathcal{H}$ is a $C^k$-vector of $\rho$ if $g\mapsto\rho(g)v$ is a $C^k$ function on $SL(2,\mathbb{R})$ . Recall that the subset of $C^{\infty}$-vectors is \emph{dense} in $\mathcal{H}$.

The \emph{Lie algebra} $sl(2,\mathbb{R})$ of $SL(2,\mathbb{R})$ (i.e., the tangent space of $SL(2,\mathbb{R})$ at the identity) is the set of all $2\times2$ matrices with zero trace. Given a $C^1$-vector $v$ of the representation $\rho$ and $X\in sl(2,\mathbb{R})$, the \emph{Lie derivative} $L_X v$ is
$$L_X v := \lim\limits_{t\to0}\frac{\rho(\exp(tX))\cdot v - v}{t}\,,$$
where $\exp(X)$ is the \emph{exponential map} (of matrices).
\begin{exercise}Show that $\langle L_Xv,w\rangle = -\langle v,L_Xw\rangle$ for any pair of $C^1$-vectors  $v,w\in\mathcal{H}$ of the representation $\rho$ and for any $X\in sl(2,\mathbb{R})$.
\end{exercise}

An important basis of $sl(2,\mathbb{R})$ is
$$W:=\left(\begin{array}{cc}0&1\\-1&0\end{array}\right), \quad Q:=\left(\begin{array}{cc}1&0\\0&-1\end{array}\right), \quad V:=\left(\begin{array}{cc}0&1\\1&0\end{array}\right)$$
\begin{exercise}Show that $\exp(tW)=\left(\begin{array}{cc}\cos t&\sin t\\-\sin t&\cos t\end{array}\right)$, $\exp(tQ) = \left(\begin{array}{cc}e^t&0\\0&e^{-t}\end{array}\right)$ and $\exp(tV) = \left(\begin{array}{cc}\cosh t&\sinh t\\-\sinh t&\cosh t\end{array}\right)$. Furthermore, $[Q,W]=2V$, $[Q,V]=2W$ and $[W,V]=2Q$ where $[.,.]$ is the Lie bracket of $sl(2,\mathbb{R})$ (i.e., $[A,B]:= AB-BA$ is the commutator).
\end{exercise}
The \emph{Casimir operator} $\Omega_{\rho}$  is $\Omega_{\rho}:=(L_V^2+L_Q^2-L_W^2)/4$ on the dense subspace of $C^2$-vectors of $\rho$. It is known that $\Omega_{\rho}$ is symmetric, that is, 
$\langle \Omega_{\rho}v,w\rangle = \langle v,\Omega_{\rho}w\rangle$ for any $C^2$-vectors $v,w\in\mathcal{H}$, that the closure of $\Omega_{\rho}$ is \emph{self-adjoint}, and that $\Omega_{\rho}$ commutes with $L_X$ on $C^3$-vectors, for any $X\in sl(2,\mathbb{R})$, and  with $\rho(g)$ on 
$C^2$-vectors, for any $g\in SL(2,\mathbb{R})$. 

In addition, when the representation $\rho$ is \emph{irreducible}, $\Omega_{\rho}$ is a scalar multiple of the identity operator, i.e., $\Omega_{\rho}v = \lambda(\rho)v$ for some $\lambda(\rho)\in\mathbb{R}$ and for any $C^2$-vector $v\in\mathcal{H}$ of $\rho$. In general, as we're going to see below, the spectrum $\sigma(\Omega_{\rho})$ of the Casimir operator $\Omega_{\rho}$ is a fundamental object.

\subsection{Bargmann's classification}

We introduce the following notation:
\begin{equation*}r(\lambda):=\left\{\begin{array}{cc}-1 & \quad\quad\textrm{if } \lambda\leq -1/4, \\ -1+\sqrt{1+4\lambda} & \quad\quad\quad\,\,\,\,\textrm{ if } -1/4<\lambda<0 \\ -2 & \textrm{if } \lambda\geq 0\end{array}\right.
\end{equation*}
Note that $r(\lambda)$ satisfies the quadratic equation $x^2+2x-4\lambda=0$ when $-1/4<\lambda<0$.

Bargmann's classification of \emph{irreducible} unitary $SL(2,\mathbb{R})$ says that the eigenvalue $\lambda(\rho)$ of the Casimir operator $\Omega_{\rho}$ has the form
$$\lambda(\rho) = (s^2-1)/4$$
where $s\in\mathbb{C}$ falls into one of the following three categories:
\begin{itemize}
\item \emph{Principal series}: $s$ is \emph{purely imaginary}, i.e., $s\in\mathbb{R}i$;
\item \emph{Complementary series}: $s\in (0,1)$ and $\rho$ is \emph{isomorphic} to the representation 
$$\rho_s\left(\begin{array}{cc}a&b\\c&d\end{array}\right) f(x):= (cx+d)^{-1-s} f\left(\frac{ax+b}{cx+d}\right),$$ where $f$ belongs to the Hilbert space $\mathcal{H}_s:=\left\{f:\mathbb{R}\to\mathbb{C}: \iint\frac{f(x)\overline{f(y)}}{|x-y|^{1-s}}dx\,dy<\infty\right\}$;
\item \emph{Discrete series}: $s\in\mathbb{N}-\{0\}$.
\end{itemize}
In other words, $\rho$ belongs to the \emph{principal series} when $\lambda(\rho)\in(-\infty,-1/4]$, $\rho$ belongs to the \emph{complementary series} when $\lambda(\rho)\in (-1/4,0)$ and $\rho$ belongs to the \emph{discrete series} when $\lambda(\rho)=(n^2-1)/4$ for some natural number $n\geq 1$.
Note that, when $-1/4<\lambda(\rho)<0$ (i.e., $\rho$ belongs to the complementary series), we have $r(\lambda(\rho))=-1+s$.

\subsection{Some examples of $SL(2,\mathbb{R})$ unitary representations}

Given a dynamical system consisting of a $SL(2,R)$ action (on a certain space $X$) preserving some probability measure ($\mu$), we have a naturally associated unitary $SL(2,\mathbb{R})$ representation on the Hilbert space $L^2(X,\mu)$ of $L^2$ functions of the probability space $(X,\mu)$. More concretely, we'll be interested in the following two examples.

\textbf{Hyperbolic surfaces of finite volume.} It is well-known that $SL(2,\mathbb{R})$ is naturally identified with the unit cotangent bundle of the upper half-plane $\mathbb{H}$. Indeed, the quotient $SL(2,\mathbb{R})/SO(2,\mathbb{R})$ is diffeomorphic to $\mathbb{H}$ via
$$\left(\begin{array}{cc}a&b\\c&d\end{array}\right)\cdot SO(2,\mathbb{R})\mapsto \frac{ai+b}{ci+d}$$
Let $\Gamma$ be a lattice of $SL(2,\mathbb{R})$, i.e., a discrete subgroup such that $M:=\Gamma\backslash SL(2,\mathbb{R})$ has finite volume with respect to the natural measure $\mu$ induced from the Haar measure of $SL(2,\mathbb{R})$. In this situation, our previous identification shows that $M:=\Gamma\backslash SL(2,\mathbb{R})$ is naturally identified with the unit cotangent bundle $T_1 S$ of the hyperbolic surface $S:=\Gamma\backslash SL(2,\mathbb{R})\slash SO(2,\mathbb{R}) = \Gamma\backslash \mathbb{H}$ of finite volume with respect to the natural measure $\nu$.

Since the action of $SL(2,\mathbb{R})$ on $M:=\Gamma\backslash SL(2,\mathbb{R})$ and $S:=\Gamma\backslash\mathbb{H}$ preserves the respective probability measures $\mu$ and $\nu$ (induced from the Haar measure of $SL(2,\mathbb{R})$), we obtain the following (\emph{regular}) unitary $SL(2,\mathbb{R})$ representations:
$$\rho_M(g)f(\Gamma z)=f(\Gamma z\cdot g) \quad \forall\, f\in L^2(M,\mu)$$
and
$$\rho_S(g)f(\Gamma z SO(2,\mathbb{R})) = f(\Gamma z\cdot g SO(2,\mathbb{R})) \quad \forall\, f\in L^2(S,\nu).$$
Observe that $\rho_S$ is a subrepresentation of $\rho_M$ because the space $L^2(S,\nu)$ can be identified with the subspace $\mathcal{H}_{\Gamma}:=\{f\in L^2(M,\mu): f \textrm{ is constant along } SO(2,\mathbb{R})-\textrm{orbits}\}$. Nevertheless, it is possible to show that the Casimir operator $\Omega_{\rho_M}$ restricted to $C^2$-vectors of $\mathcal{H}_{\Gamma}$ \emph{coincides} with the Laplacian $\Delta=\Delta_S$ on $L^2(S,\nu)$. Also, we have that a number $-1/4<\lambda<0$ belongs to the spectrum of the Casimir operator $\Omega_{\rho_M}$ (on $L^2(M,\mu)$) if and only if $-1/4<\lambda<0$ belongs to the spectrum of the Laplacian $\Delta=\Delta_S$ on $L^2(S,\nu)$.

\textbf{Moduli spaces of Abelian differentials.} Of course, an interesting space philosophically related to the hyperbolic surfaces of finite volumes are the moduli spaces $\mathcal{H}_g$ of Abelian differentials on Riemann surfaces of genus $g\geq 1$.


As we saw in the last Section~\ref{ss.1.torii} of Section~\ref{s.intro}, the case of $\mathcal{Q}_1$ is particularly clear: it is well-known that $\mathcal{Q}_1$ is isomorphic to the unit cotangent bundle $SL(2,\mathbb{Z})\backslash SL(2,\mathbb{R})$ of the \emph{modular curve}. In this nice situation, the $SL(2,\mathbb{R})$ action has a natural absolutely continuous (w.r.t. Haar measure) invariant \emph{probability} $\mu_{(1)}$, so that we have a natural unitary representation of  the group 
$SL(2,\mathbb{R})$ on the Hilbert space $L^2(\mathcal{Q}_1,\mu_{(1)})$.

After the work of H.~Masur and W.~Veech, we know that the general case has some similarities with the genus $1$ situation, in the sense that connected components $\mathcal{C}$ of strata $\mathcal{H}_{\kappa}$ of $\mathcal{H}_g$ come equipped with a natural Masur--Veech invariant probability measure $\mu_{\mathcal{C}}$. In particular, we get also an unitary $SL(2,\mathbb{R})$ representation on $L^2(\mathcal{C},\mu_{\mathcal{C}})$. More generally, there are plenty of $SL(2,\mathbb{R})$-invariant probability measures $\mu$ on $\mathcal{C}$ (e.g., coming from square-tiled surfaces) and evidently all of them yield unitary $SL(2,\mathbb{R})$ representations (on $L^2(\mathcal{C},\mu)$).

\subsection{Rates of mixing and size of the spectral gap}

Once we have introduced two examples (coming from Dynamical Systems) of unitary $SL(2,\mathbb{R})$ representations, what are the possible series (in the sense of Bargmann classification) appearing in the decomposition of our representation into its irreducible factors.

In the case of hyperbolic surfaces of finite volume, we understand precisely the global picture: the possible irreducible factors are described by the rates of mixing of the geodesic flow on our hyperbolic surface. More precisely, let
$$A:=\{a(t) := \textrm{diag}(e^t,e^{-t})\in SL(2,\mathbb{R})\}$$
be the 1-parameter subgroup of diagonal matrices of $SL(2,\mathbb{R})$. It is not hard to check that the geodesic flow on a hyperbolic surface of finite volume $\Gamma\backslash\mathbb{H}$ is identified with the action of the diagonal subgroup $A$ on $\Gamma\backslash SL(2,\mathbb{R})$.

Ratner showed that the Bargmann's series of the irreducible factors of the regular representation $\rho_{\Gamma}$ of $SL(2,\mathbb{R})$ on $L^2(\Gamma\backslash SL(2,\mathbb{R}))$ can be deduced from the \emph{rates of mixing} of the geodesic flow $a(t)$ along a certain class of observables. In order to keep the exposition as elementary as possible, we will state a very particular case of Ratner's results (referring the reader to~\cite{Rt} for more general statements). We define $\mathcal{H}_{\Gamma}:=\{f\in L^2(\Gamma\backslash SL(2,\mathbb{R})): f \textrm{ is constant along } SO(2,\mathbb{R})-\textrm{orbits and}\int f = 0\}$ equipped with the usual $L^2$ inner product $\langle.,.\rangle$. In the sequel, we denote by
$$\mathcal{C}(\Gamma)=\sigma(\Delta_S)\cap (-1/4,0)$$
the intersection of the spectrum of the Laplacian $\Delta_S$ with the open interval $(-1/4,0)$,
$$\beta(\Gamma) = \sup\mathcal{C}(\Gamma)$$
with the convention $\beta(\mathcal{C}(\Gamma))=-1/4$ when $\mathcal{C}(\Gamma)=\emptyset$ and
$$\sigma(\Gamma)=r(\beta(\Gamma)):=-1+\sqrt{1+4\beta(\Gamma)}\,.$$
We remind the reader that the subset $\mathcal{C}(\Gamma)$ detects the presence of complementary series in the decomposition of $\rho_{\Gamma}$ into irreducible representations. Also, since $\Gamma$ is a lattice, it is possible to show that $\mathcal{C}(\Gamma)$ is finite and, \emph{a fortiori}, $\beta(\Gamma)<0$. Since $\beta(\Gamma)$ essentially measures the distance between zero and the first eigenvalue of $\Delta_S$ on $\mathcal{H}_{\Gamma}$, it is natural to call $\beta(\Gamma)$ the \emph{spectral gap}.

\begin{theorem}\label{t.Ratner}For any $f,g\in\mathcal{H}_{\Gamma}$ and $|t|\geq 1$, we have
\begin{itemize}
\item $|\langle f, \rho_{\Gamma}(a(t))g\rangle|\leq C_{\beta(\Gamma)}\cdot e^{\sigma(\Gamma)t}\cdot \|f\|_{L^2}\|g\|_{L^2}$ when $\mathcal{C}(\Gamma)\neq\emptyset$;
\item $|\langle f, \rho_{\Gamma}(a(t))g\rangle|\leq C_{\beta(\Gamma)}\cdot e^{\sigma(\Gamma)t}\cdot \|f\|_{L^2}\|g\|_{L^2} = C_{\beta(\Gamma)}\cdot e^{-t}\cdot \|f\|_{L^2}\|g\|_{L^2}$ when $\mathcal{C}(\Gamma)=\emptyset$, $\sup(\sigma(\Delta_S)\cap(-\infty,-1/4))<-1/4$ and $-1/4$ is not an eigenvalue of the Casimir operator $\Omega_{\rho_{\Gamma}}$;
\item $|\langle f, \rho_{\Gamma}(a(t))g\rangle|\leq C_{\beta(\Gamma)}\cdot t\cdot e^{\sigma(\Gamma)t}\cdot \|f\|_{L^2}\|g\|_{L^2} = C_{\beta(\Gamma)}\cdot t\cdot e^{-t}\cdot \|f\|_{L^2}\|g\|_{L^2}$ otherwise, i.e., when $\mathcal{C}(\Gamma)=\emptyset$ and either $\sup(\sigma(\Delta_S)\cap(-\infty,-1/4))=-1/4$ or $-1/4$ is an eigenvalue of the Casimir operator $\Omega_{\rho_{\Gamma}}$.
\end{itemize}
The above constants $C_{\mu}$ are uniformly bounded when $\mu$ varies on compact subsets of $(-\infty,0)$.
\end{theorem}

In other words, Ratner's theorem relates the (exponential) rate of mixing of the geodesic flow $a(t)$ with the spectral gap: indeed, the quantity $|\langle f,\rho_{\Gamma}(a(t))g\rangle|$ roughly measures how fast the geodesic flow $a(t)$ mixes different places of phase space (actually, this is more clearly seen when $f$ and $g$ are characteristic functions of Borelian sets), so that Ratner's result says that the exponential rate $\sigma(\Gamma)$ of mixing of $a(t)$ is an explicit function of the spectral gap $\beta(\Gamma)$ of $\Delta_S$.

In the case of moduli spaces of Abelian differentials, our knowledge is less complete than the previous situation: as far as we know, the best results about the ``\emph{spectral gap}" of the $SL(2,\mathbb{R})$ representation $\rho_{\mu}$ on the space $L^2_0(\mathcal{C},\mu)$ of zero-mean $L^2$-functions with respect to a given $SL(2,\mathbb{R})$-invariant probability measures $\mu$ on a connected component $\mathcal{C}$ of the moduli space $\mathcal{H}_g$ are the two results discussed in Subsection~\ref{s.AGY} of Section~\ref{s.Teich-MV}, namely:
\begin{theorem}[A. Avila, S. Gou\"ezel, J.-C. Yoccoz]In the case of the Masur--Veech measure $\mu=\mu_{\mathcal{C}}$, the unitary $SL(2,\mathbb{R})$ representation $\rho_{\mu}=\rho_{\mu_{\mathcal{C}}}$ has spectral gap in the sense that it is isolated from the trivial representation, i.e., there exists some $\varepsilon>0$ such that all irreducible factors $\rho_{\mathcal{C}}^{(s)}$ of $\rho_{\mathcal{C}}$ in the complementary series are isomorphic to the representation $\rho_s$ with $s<1-\varepsilon$.
\end{theorem}

\begin{theorem}[A. Avila and S. Gou\"ezel]Let $\mu$ be an \emph{algebraic}\footnote{Recall that, roughly speaking, a $SL(2,\mathbb{R})$-invariant probability is \emph{algebraic} whenever its support is an affine suborbifold (i.e., a suborbifold locally described, in periodic coordinates, by affine subspaces) such that, in period coordinates, $\mu$ is absolutely continuous with respect to Lebesgue measure and its density is locally constant.} $SL(2,\mathbb{R})$-invariant probability measure, and let  $L^2_0(\mathcal{C},\mu)=\int \mathcal{H}_{\xi}d\lambda(\xi)$ be the decomposition of the unitary $SL(2,\mathbb{R})$ representation on $L^2_0(\mathcal{C},\mu)$ into an integral of irreducible components $\mathcal{H}_{\xi}$. Then, for any $\delta>0$, the components $\mathcal{H}_{\xi}$ of the complementary series with parameter $s(\mathcal{H}_{\xi})\in [\delta,1]$ appear for at most \emph{finitely many parameters} (i.e., $\{s\in [\delta,1]: s=s(\mathcal{H}_{\xi}) \textrm{ for some }\xi\}$ is finite) and with \emph{finite multiplicity} (i.e., for each $s\in[\delta,1]$, $\{\xi: s(\mathcal{H}_{\xi})=s\}$ is finite). In particular, the Teichm\"uller geodesic flow $(g_t)_{t\in \mathbb R}$ is exponentially mixing with respect to the probability measure $\mu$ on $\mathcal C$.

\end{theorem}

Observe that, generally speaking, the results of Avila, Gou\"ezel and Yoccoz say that $\rho_{\mu}$ doesn't contain all possible irreducible representations of the complementary series, but it is doesn't give any hint about quantitative estimates of the ``spectral gap", i.e., how small $\varepsilon>0$ can be in general. In fact, at the present moment, it seems that the only situation which is more precisely known is the case of the moduli space $\mathcal{H}_1$ of surfaces of genus $1$:
\begin{theorem}[Selberg/Ratner]The representation $\rho_{\mathcal{H}_1}$ has no irreducible factor in the complementary series, hence the bound $|\langle f,\rho_{\mathcal{H}_1}g\rangle|\leq C\|f\|_{L^2} 
\|g\|_{L^2}\cdot t\cdot e^{-t}$ holds for any $SO(2,\mathbb R)$-invariant functions $f, g \in 
L^2(\mathcal{H}_1, \mu_{\mathcal{H}_1})$ of zero average and for all $|t|\geq 1$.
\end{theorem}
In fact, using the notation of Ratner's theorem, Selberg proved that $\mathcal{C}(SL(2,\mathbb{Z}))=\emptyset$. Since we already saw that $\mathcal{H}_1 = SL(2,\mathbb{Z}) \backslash SL(2,\mathbb{R})$, the first part of the theorem is a direct consequence of Selberg's result, while the second part is a direct consequence of Ratner's result.

In view of the previous theorem, it is natural to make the following conjecture:
\medskip

\noindent\textbf{Conjecture} (J.-C. Yoccoz) The representations $\rho_{\mu_\mathcal{C}}$ \emph{don't} have complementary series (where $\mu_{\mathcal{C}}$ are the Masur--Veech measures).

\medskip
This conjecture is currently open (to the best of the authors' knowledge). In any case, it is worth to recall that in Section~\ref{s.complementary-series} we saw that this conjecture becomes false if the invariant natural measure 
$\mu_{\mathcal{C}}$ is replaced by other invariant measures supported on smaller loci.

\section{A pseudo-Anosov in genus $2$ with vanishing second Lyapunov exponent}\label{a.pAjc}

In this (short) appendix, we will \emph{sketch} the construction of a periodic orbit $\gamma$ of the Teichm\"uller flow on $\mathcal{H}(2)$ such that the second Lyapunov exponent of the KZ cocycle over $\gamma$ vanishes. For this purpose, recall that typical orbits of the Teichm\"uller flow on $\mathcal{H}(2)$ are coded by $\infty$-complete paths in the Rauzy diagram schematically depicted below:

\begin{figure}[htb!]
\includegraphics[scale=0.5]{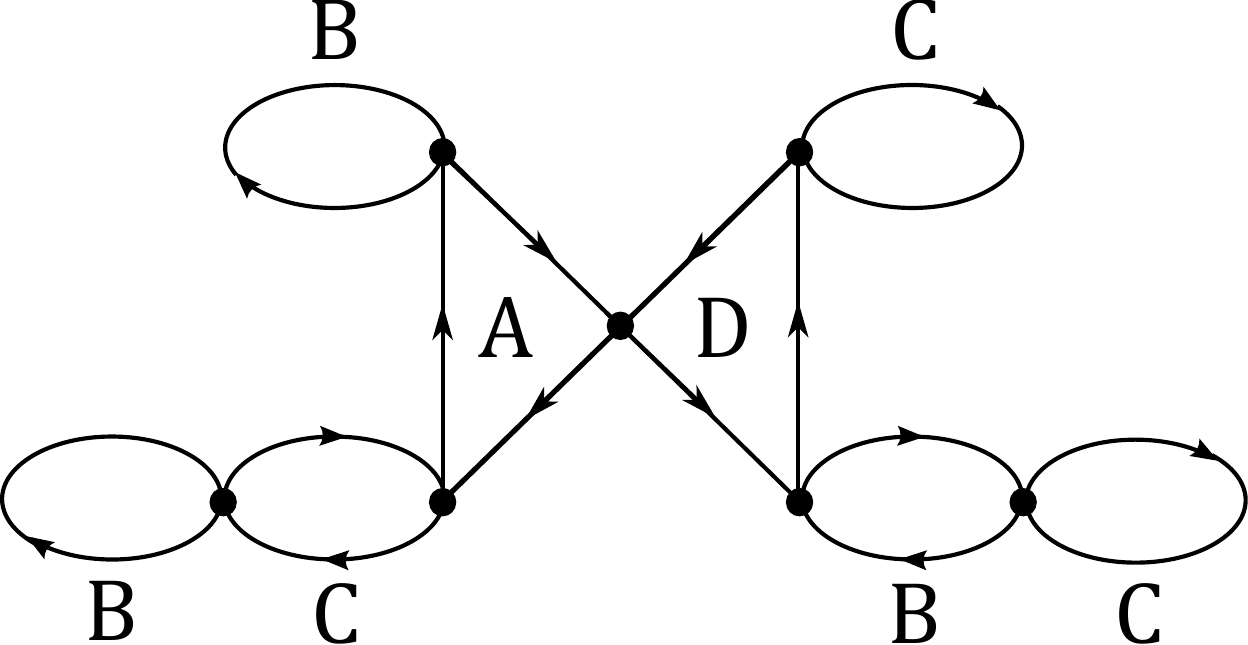}
\end{figure}

In addition, it is possible to attach matrices to the arrows of Rauzy diagrams such that the KZ cocycle over a Teichm\"uller flow orbit represented by a certain concatenation of arrows of a $\infty$-complete path is simply given by the product of the matrices associated to the arrows in the order they are concatenated. The reader can find detailed explanation of this construction in J.-C.~Yoccoz's survey
\cite{Y}.

For our current task, starting from the vertex at the center of the Rauzy we take the following concatenation $\gamma$ of arrows $D\to B\to B\to D\to C\to D\to A\to A\to A$, i.e., 
$\gamma:=DB^2DCDA^3$. Here, using the language of \cite{Y}, we're coding arrows by the associated winning letter. The fact that the four letters $A, B, C, D$ appear in the construction of $\gamma$ means that it is $\infty$-complete, so that $\gamma$ represents a periodic orbit of the Teichm\"uller flow.

A direct calculation (with the formulas presented in \cite{Y}) shows that the KZ cocycle over $\gamma$ is represented by a matrix $B_{\gamma}$ with characteristic polynomial 
$$x^4-7x^3+11x^2-7x+1$$
By performing the substitution $y=x+1/x$, we obtain the quadratic polynomial
$$y^2-7y+9$$
whose discriminant is $\Delta=\sqrt{13}$. Since
$$\frac{7-\sqrt{13}}{2}<2\,,$$
we see that $B_{\gamma}$ has a pair of complex conjugated eigenvalues of modulus $1$, i.e., 
the KZ cocycle over the closed Teichm\"uller orbit $\gamma$ in $\mathcal{H}(2)$ (which corresponds to a pseudo-Anosov map  on a genus $2$ surface) has vanishing second Lyapunov exponent.

\section{Volumes of strata (after A.~Eskin $\&$ A.~Okounkov)}\label{s.Eskin-Okounkov} 
This appendix contains some comments on the work \cite{EO} of A.~Eskin and A.~Okounkov about the computation of the volumes $\lambda_{\kappa}^{(1)}(\mathcal{H}^{(1)}(\kappa))$ of strata of Abelian differentials (cf. Section~\ref{s.general-dynamics}). 

In the $0$th order of approximation, the idea of A.~Eskin and A.~Okounkov: by analogy with the case of $\mathbb{R}^n$, one can hope to compute the volume of a stratum 
$\mathcal{H}^{(1)}(\kappa)$ by counting \emph{integral}/\emph{rational} points. More precisely, let us recall that the volume of the unit sphere $S^{n-1}$ of $\mathbb{R}^n$ can be calculated by the following method. Denoting by $B(0,R)$ the ball of radius $R$ in $\mathbb{R}^n$, let 
$$N(R):=\#(B(0,R)\cap\mathbb{Z}^n)$$ 
That is, $N(R)$ is the number of integral points of $\mathbb{R}^n$ in the ball $B(0,R)$ of radius $R$:

\begin{figure}[htb!]
\includegraphics[scale=0.5]{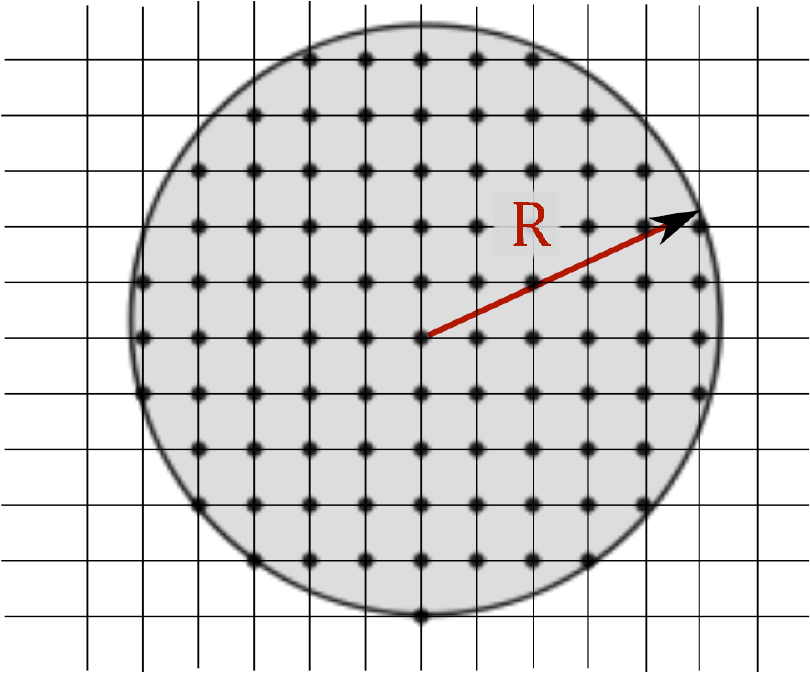}
\end{figure}

Since $N(R)$ is a good approximation of the volume of $B(0,R)$ for large $R$, we see that the knowledge of the asymptotic behavior of $N(R)$, i.e., 
$$N(R)\sim c(n)R^n$$
for a constant $c(n)>0$ allows to deduce the volume of the unit sphere $S^{n-1}$ by homogeneity, i.e., 
$$\textrm{vol}(S^{n-1})=\frac{d\textrm{vol}(B(0,R))}{dR}|_{R=1}=n\cdot c(n)$$ 

In the case of volume of strata of moduli spaces of Abelian differentials, the strategy is ``similar'': 
\begin{itemize}
\item firstly, one realizes that the role of \emph{integral} points is played by \emph{square-tiled surfaces}, so that the volume of the ``ball'' $\mathcal{H}^{(R)}(\kappa)$ of translation surfaces in the stratum $\mathcal{H}(\kappa)$ with total area at most $R$ is reasonably approximated by the number $N_{\kappa}(R)$ of square-tiled surfaces in the stratum $\mathcal{H}(\kappa)$ composed of $R$ unit squares at most;
\item secondly, one computes the asymptotics $N_{\kappa}(R)\sim c(\kappa)\cdot R^{2g+s-1}$ (recall that the stratum $\mathcal{H}(\kappa)$ has \emph{complex} dimension $2g+s-1$ when $\kappa=(k_1,\dots,k_s)$, $2g-2=\sum\limits_{j=1}^s k_j$);
\item finally, by homogeneity, one deduces that $\lambda_{\kappa}^{(1)}(\mathcal{H}^{(1)}(\kappa))=(4g+2s-2)\cdot c(\kappa)$.
\end{itemize} 

Evidently, the most difficult step here is the calculation of $c(\kappa)$. In rough terms, the main point is that one can reduce the computation of $c(\kappa)$ to a combinatorial problem about \emph{permutations} which can be attacked by methods based on the representation theory of the symmetric group. However, the implementation of this idea is a hard task and it is out of the scope of these notes to present the arguments of A.~Eskin and A.~Okounkov~\cite{EO}. In particular, we will content ourselves to reduce the calculation of $c(\kappa)$ to a combinatorial problem and then we will simply state some of the main results of~\cite{EO}. Finally, we will conclude this appendix by showing how the action of $SL(2,\mathbb{Z})$ on square-tiled surfaces translates in terms of combinatorics of permutations, so that it will be ``clear'' that the $SL(2,\mathbb{Z})$-action on square-tiled surfaces with a ``low'' number of squares can be calculated with the aid of computer programs. 

The computation of the constant $c(\kappa)$ essentially amounts to count the number of square-tiled surfaces with $N$ squares inside a given stratum $\mathcal{H}(\kappa)$. From a combinatorial point of view, a square-tiled surface with $N$ squares can be encoded by numbering its squares from $1$ to $N$ and then considering a pair of permutations $h,v\in S_N$ such that 
\begin{itemize}
\item $h(i)$ is the number of the square to the right of the square $i$;
\item $v(i)$ is the number of the square on the top of the square $i$.
\end{itemize}

For example, the $L$-shaped square-tiled surface below is coded by the pair of permutations\footnote{In what follows, we will represent permutations by their cycles.} $h=(1,2)(3)$ and $v=(1,3)(2)$.
\begin{figure}[htb!]
\includegraphics[scale=0.3]{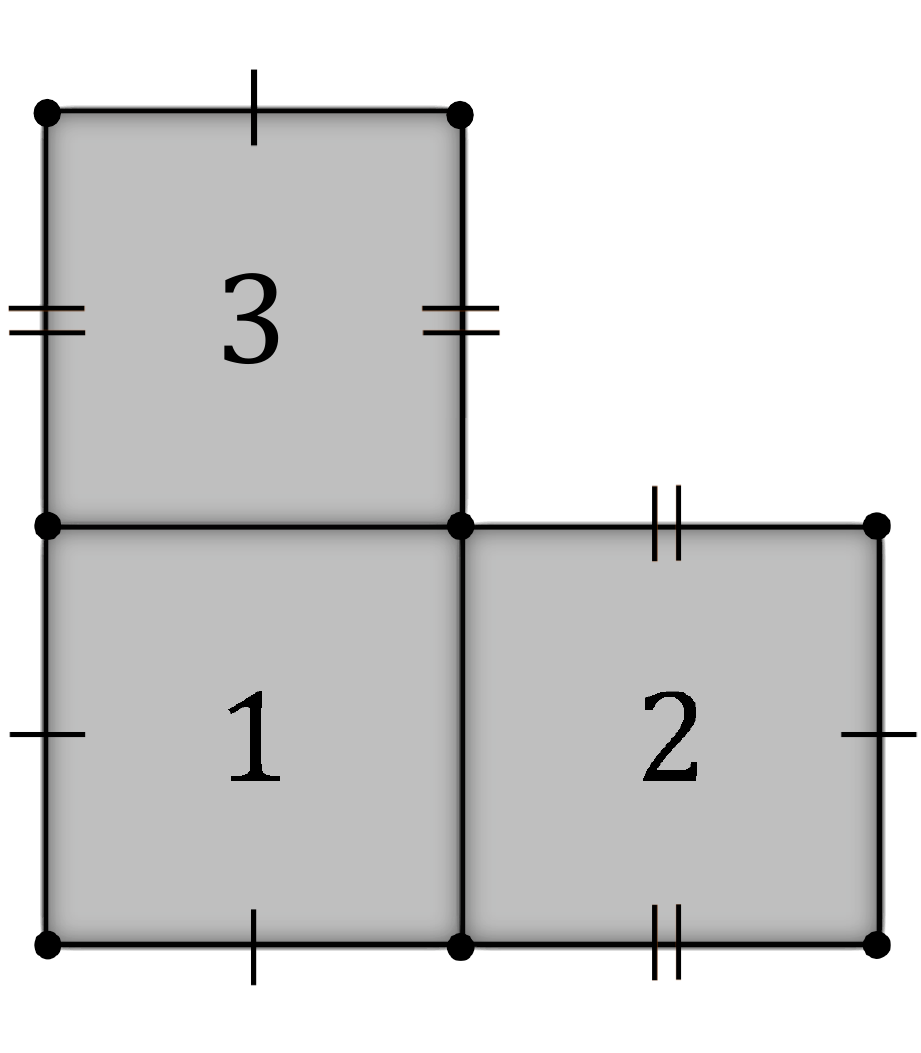}
\end{figure}

Logically, the codification by a pair of permutations is \emph{not} unique because we can always \emph{renumber} the squares \emph{without} changing the square-tiled surface:
\begin{figure}[htb!]
\includegraphics[scale=0.3]{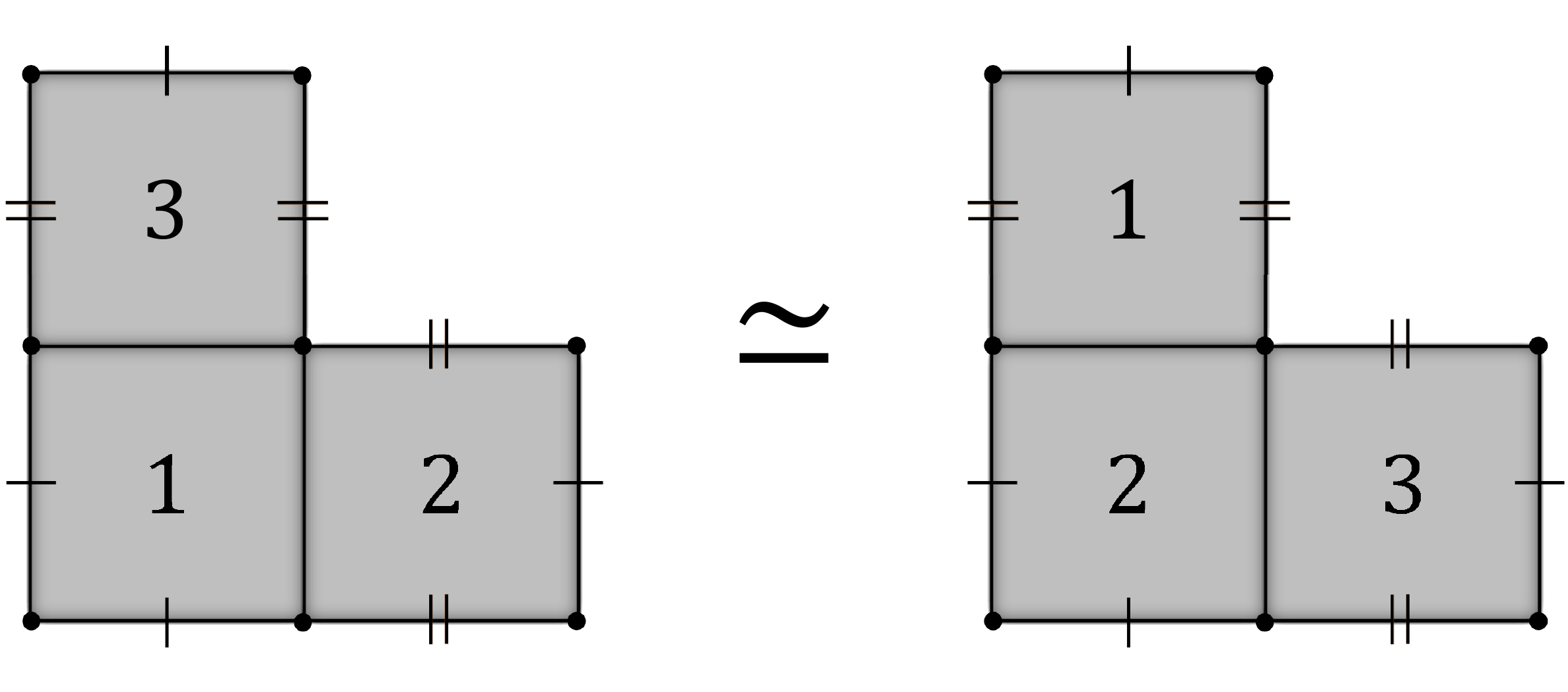}
\end{figure}

In combinatorial terms, the operation of renumbering corresponds to perform a \emph{simultaneous} conjugation of $h$ and $v$, i.e., we replace the pair of permutations $(h,v)$ by $(\phi h\phi^{-1},\phi v\phi^{-1})$ for some $\phi\in S_N$. Since we are interested in the square-tiled surface itself (but not on the particular codification), we will declare that 
$$(h,v)\sim (h',v') \iff h'=\phi h\phi^{-1}  \textrm{ and } v'=\phi v\phi^{-1}\,, \quad 
 \textrm{ for some }\phi\in S_N\,.$$

Moreover, we see that a (connected) square-tiled surface with $N$ squares is coded by a pair of permutations $(h,v)\in S_N\times S_N$ acting \emph{transitively} on the set $\{1,\dots, N\}$ of squares.

In this language, we just saw that a connected square-tiled surface with $N$ squares is the same as the \emph{equivalence classes} of a pair of permutations acting transitively on $\{1,\dots,N\}$ \emph{modulo simultaneous conjugation}. 
Next, we observe that the \emph{stratum} $\mathcal{H}(\kappa)$ of a square-tiled surface can be read off from a pair of permutations $(h,v)$ coding the surface. Indeed, we note that the \emph{commutator} $[h,v]:=vhv^{-1}h^{-1}$ is the permutation that can be obtained geometrically by turning around (in the counterclockwise sense) the leftmost bottom corners of each square:
\begin{figure}[htb!]
\includegraphics[scale=0.35]{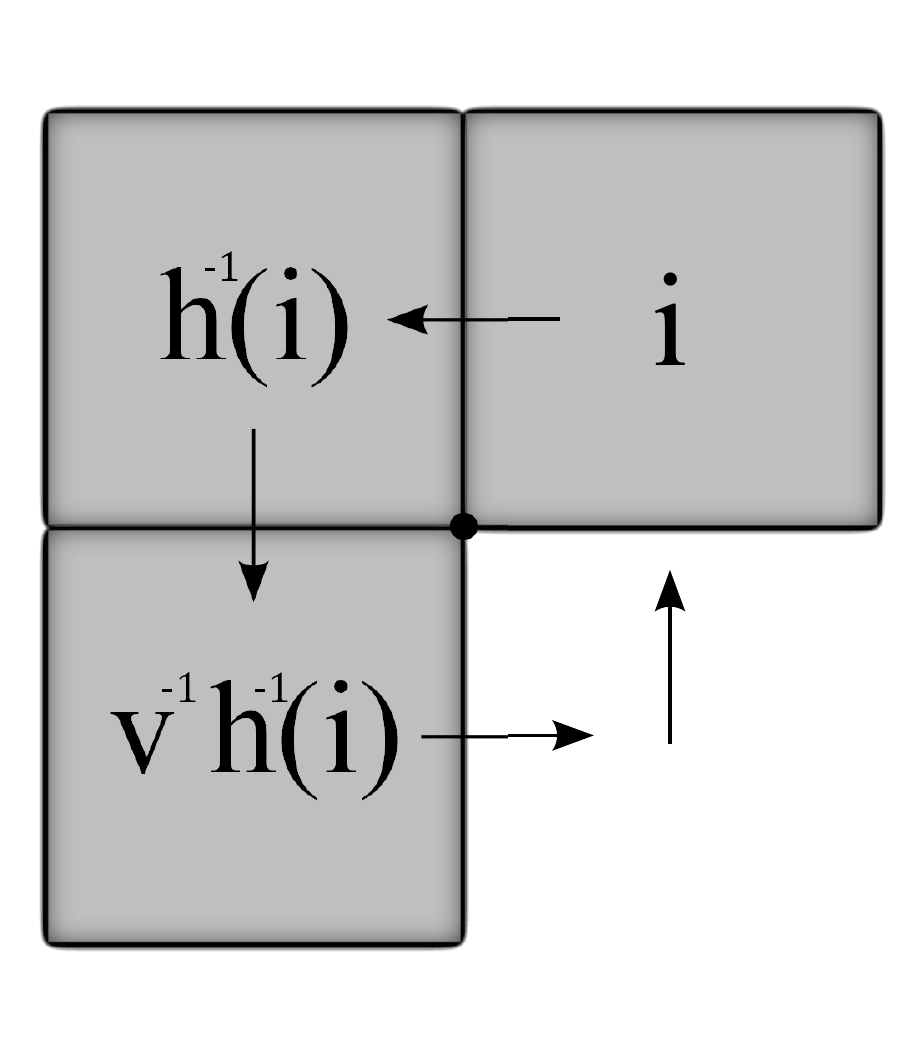}
\end{figure}

Therefore, a square-tiled surface coded by $(h,v)$ belong to the stratum $\mathcal{H}(d_1-1,\dots,d_s-1)$ where $d_1, \dots, d_s$ are the \emph{lengths} of the (non-trivial) cycles of the commutator $[h,v]$ of $h$ and $v$. 

In other words, we see that counting (connected) square-tiled surfaces with $N$ squares in $\mathcal{H}(k_1,\dots,k_s)$ is the same as the combinatorial problem of counting equivalence classes, modulo simultaneous conjugations, of pairs of permutations $(h,v)\in S_N\times S_N$ such that 
\begin{itemize}
\item the pair $(h,v)$ acts transitively on $\{1,\dots,N\}$ ;
\item the commutator $[h,v]$ has $s$ (non-trivial) cycles of lengths $(k_1+1), \dots, (k_s+1)$. 
\end{itemize}

By solving this combinatorial problem, A.~Eskin and A.~Okounkov \cite{EO} proved that 
\begin{theorem} The number $c(\kappa)$ and, \emph{a fortiori}, the volume $\lambda_{\kappa}^{(1)}(\mathcal{H}^{(1)}(\kappa))$ is a \emph{rational} multiple of $\pi^{2g}$. Moreover, the \emph{generating function}
$$\sum\limits_{N=1}^{\infty}q^N\cdot 
\sum\limits_{\substack{S\in\mathcal{H}(\kappa) \\ \textrm{square-tiled surface} \\ \textrm{with } N \textrm{ squares}}}\frac{1}{\#\textrm{Aut}(S)}$$
is \emph{quasi-modular}: indeed, it is a polynomial in the \emph{Eisenstein} series $G_{2}(q)$, $G_4(q)$ and $G_6(q)$.
\end{theorem}

\begin{remark}The (very) attentive reader may recall that strata are \emph{not} connected \emph{in general} and they may have \emph{at most} $3$ connected components, distinguished by \emph{hyperellipticity} and \emph{parity of spin structure}. As it turns out, the volume of \emph{individual} connected components can be translated into a combinatorial problem of counting certain equivalence classes of permutations, but the new counting problem becomes slightly harder because parity of spin structure is not easy to read off from pairs of permutations: they are related to the so-called \emph{theta characteristics}. Nevertheless, this computation was successfully performed by A.~Eskin, A.~Okounkov and R.~Pandharipande \cite{EOP} to determine explicit formulas for volumes of connected components of strata. 
\end{remark} 

Closing this appendix, let's write the action of $SL(2,\mathbb{Z})$ in terms of pairs of permutations. For this sake, recall that $SL(2,\mathbb{Z})$ is generated by 
$S=\left(\begin{array}{cc}1&0\\1&1\end{array}\right)$ and $T=\left(\begin{array}{cc}1&1\\0&1\end{array}\right)$. Therefore, it suffices to write the action of $S$ and $T$ in terms of pairs of permutations $(h,v)$, and this is not hard: for instance, note that $T$ acts as
\begin{figure}[htb!]
\includegraphics[scale=0.35]{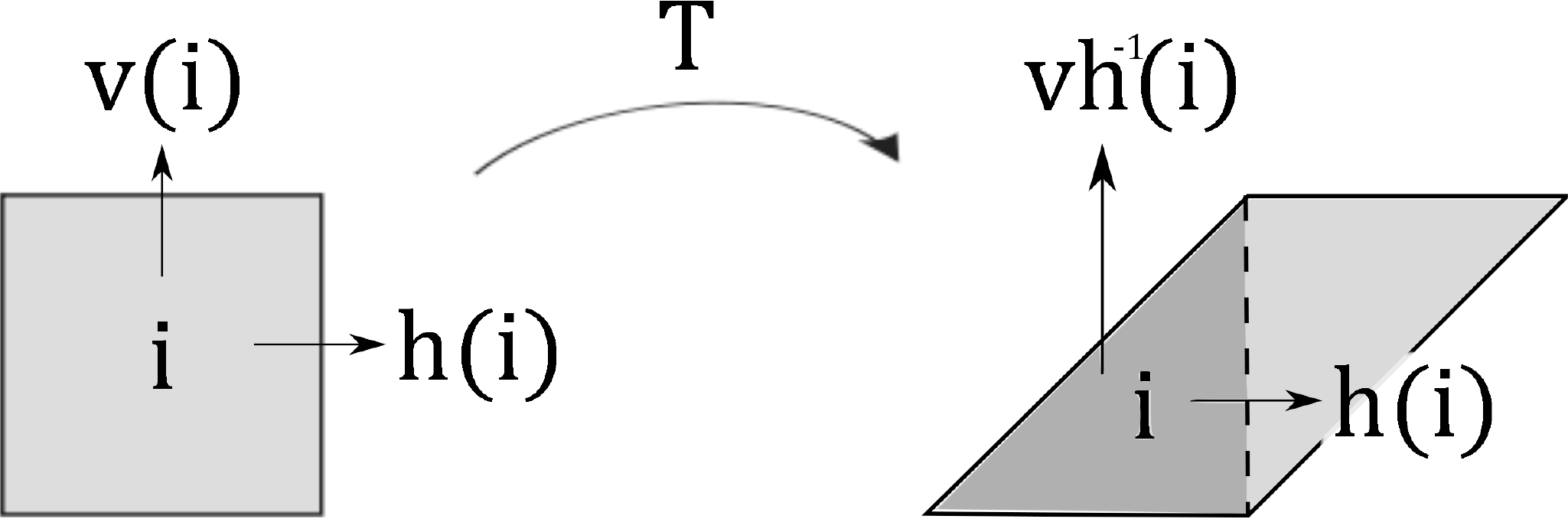}
\end{figure} 
so that, in the language of permutations, $T(h,v)=(h,vh^{-1})$. Similarly, one can convince himself/herself that $S(h,v)=(hv^{-1},v)$. 

In order words, by definition $T$ and $S$ act on pair of permutations $(h,v)$ via the so-called \emph{Nielsen transformations}. Therefore, this combinatorial description is particularly effective to compute (by hands or with a computer program) $SL(2,\mathbb{Z})$-orbits of square-tiled surfaces: we consider a pair of permutation $(h,v)$ and we successively apply $T$ and $S$ by paying attention to the fact that we're interested in pairs of permutations \emph{modulo simultaneous conjugations}. 

For instance, we invite the reader to use this approach to solve the following exercise:

\begin{exercise} Show that the $SL(2,\mathbb{Z})$-orbit of the ``Swiss cross''  (see the picture below) given by the pair of permutations $h=(1,2,3)(4)(5), v=(1)(2,4,5)(3)$ has has cardinality $9$.

\begin{figure}[htb!]
\includegraphics[scale=0.35]{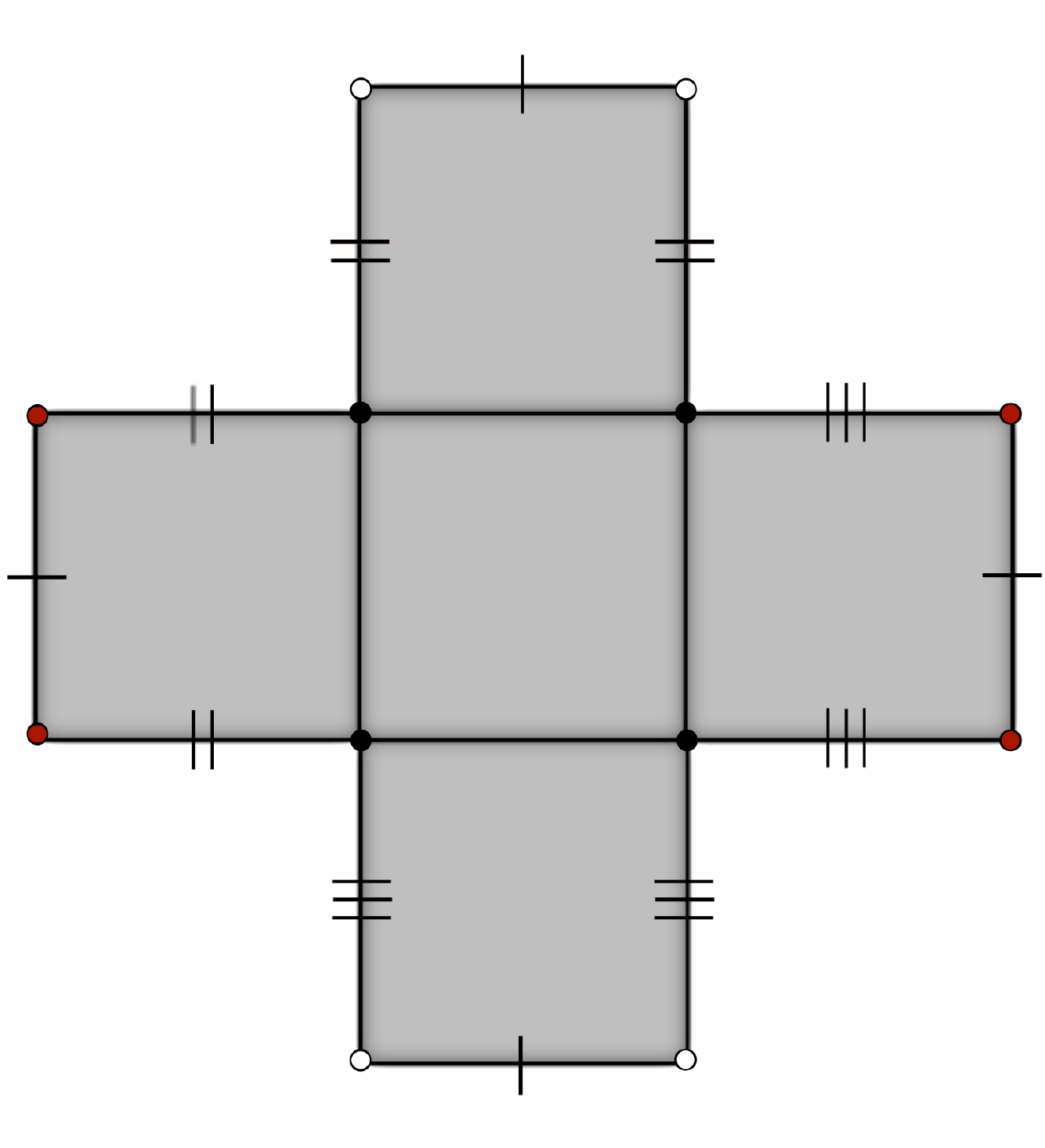}
\end{figure} 
\end{exercise}

\section{Some comments on the Lyapunov spectrum of square-tiled surfaces}\label{s.finalremarks} 

The goal of this appendix is to \emph{briefly} survey some recent results of C.~Matheus, J.-C.~Yoccoz and D.~Zmiaikou~\cite{MYZ} and C.~Matheus, M.~M\"oller and J.-C.~Yoccoz ~\cite{MMY} on \emph{non-simplicity} and \emph{simplicity} of Lyapunov exponents of the KZ cocycle over $SL(2,\mathbb{R})$-orbits of square-tiled surfaces. 

\subsection{Square-tiled surfaces with symmetries and multiplicity of Lyapunov exponents}\label{ss.MYZ} In this subsection we'll follow closely the paper~\cite{MYZ}. Let us consider a square-tiled surface $M$ represented as a pair of permutations $(h,v)\in S_N\times S_N$ (see the previous appendix). The subgroup $G$ of $S_N$ generated by $h$ and $v$ is called the \emph{monodromy group} of the square-tiled surface. Note that the stabilizers of the squares of $M$ form a conjugacy class of subgroups of $G$ whose intersection is trivial. Conversely, given a finite group $G$ generated by two elements $h$ and $v$, and a subgroup $H$ of $G$ whose intersection with any of its conjugated subgroups is trivial (i.e., $H$ doesn't contain non-trivial normal subgroups of $G$), we recover an origami whose squares are labelled by the elements of $H\backslash G$ such that $Hgh$ is the neighbor to the right of $Hg\in H\backslash G$ and $Hgv$ is the neighbor on the top of $Hg\in H\backslash G$. For the sake of this subsection, we'll think of a square-tiled surface $M$ as the data of $G, H, h, v$ as above.

As we explained in Section~\ref{s.MY-JMD}, the study of the non-tautological Lyapunov exponents of the KZ cocycle over the $SL(2,\mathbb{R})$-orbits of a square-tiled surface $M$ amounts to understand the action of the affine group $\textrm{Aff}(M)$ on $H_1^{(0)}(M,\mathbb{R})$. Actually, by technical (linear algebra) reasons, we start with the action of $\textrm{Aff}(M)$ on $H_1^{(0)}(M,\mathbb{C})$.
In this direction, we'll warm up with the action of the group of automorphisms $\textrm{Aut}(M)$ on $H_1^{(0)}(M,\mathbb{C})$. By thinking of $M$ as the data $(G,H,h,v)$, it is possible to check that $\textrm{Aut}(M)$ is naturally \emph{isomorphic} to $N/H$ where $N$ is the \emph{normalizer} of $H$ in $G$. By taking this point of view, we have that $H_1^{(0)}(M,\mathbb{C})$ is a $N/H$-\emph{module} and we can ask what's the \emph{multiplicity} $\ell_{\alpha}$ of a given \emph{irreducible} representation 
$\chi_{\alpha}$ of the finite group $N/H$ inside $H_1^{(0)}(M,\mathbb{C})$. 
In \cite{MYZ}, J.-C.~Yoccoz, D.~Zmiaikou and the second author show the following formula:
\begin{theorem}\label{t.MYZ-l} One has the formula
$$\ell_{\alpha}=\frac{\#G}{\#N}\textrm{dim}(\chi_{\alpha}) - \sum\limits_{g\in G}\frac{1}{n(g)}\textrm{dim}(\textrm{Fix}_{\alpha}(gc^{n(g)}g^{-1}))\,$$
where $c=[h,v] $ is the commutator of the permutations $h$ and $v$, $n(g)>0$ denotes the smallest integer such that $gc^{n(g)}g^{-1}\in N$, and $\textrm{Fix}_{\alpha}(n)$ is the subspace of 
$H_1^{(0)}(M,\mathbb{C})$ fixed by $\chi_{\alpha}(n)$.
\end{theorem}

\begin{remark}\label{r.commutator-multiplicity}
An interesting consequence of this formula is the fact that the multiplicity $\ell_{\alpha}$ depend on $h$ and $v$ \emph{only} by means of its commutator $c=[h,v]$.
\end{remark}

This formula is one of the ingredients towards the following result:

\begin{corollary} The multiplicity $\ell_{\alpha}$ is never equal to $1$, i.e., either $\ell_{\alpha}=0$ or $\ell_{\alpha}>1$.
\end{corollary} 

Once we understand the decomposition of the  $N/H$-module $H_1^{(0)}(M,\mathbb{C})$ recall
that $\textrm{Aut}(M)\simeq N/H$) into irreducible pieces, we can pass to the analysis of the 
$N/H$-module $H_1^{(0)}(M,\mathbb{R})$. 

By general representation theory, the Galois group $\textrm{Gal}(\mathbb{C}|\mathbb{R})$ acts naturally on the set $\textrm{Irr}_{\mathbb{C}}(N/H)$ of $\mathbb{C}$-irreducible representations of $N/H$, and the $\mathbb{R}$-irreducible representations $a\in\textrm{Irr}_{\mathbb{R}}(N/H)$ are precisely the $\textrm{Gal}(\mathbb{C}|\mathbb{R})$-orbits of $\alpha\in\textrm{Irr}_{\mathbb{C}}(N/H)$. Furthermore, given such an orbit $a\in\textrm{Irr}_{\mathbb{R}}(N/H)$, the character $\chi_a$ has the form 
$$\chi_a=m_a\sum\limits_{\alpha\in a}\chi_{\alpha}$$
where $m_a$ is the so-called \emph{Schur index}. 

Concerning the multiplicities inside the $N/H$-modules $H_1^{(0)}(M,\mathbb{R})$ and $H_1^{(0)}(M,\mathbb{C})=\mathbb{C}\otimes H_1^{(0)}(M,\mathbb{R})$, this means that $\ell_a=m_a\ell_{\alpha}$ for any $\alpha\in a$. In particular, one can use Theorem~\ref{t.MYZ-l} to determine the multiplicities $\ell_a$ of $\mathbb{R}$-irreducible $N/H$-representations in $H_1^{(0)}(M,\mathbb{R})$. 

Moreover, given $V_a$ an irreducible $\mathbb{R}(N/H)$-module, the \emph{commuting algebra} $D_a$ of $\mathbb{R}(N/H)$ in $\textrm{End}_{\mathbb{R}}(V_a)$ is a skew-field of degree $m_a^2$ over its center $K_a$. 

This means that by denoting $W_a\simeq V_a^{\ell_a}$ the \emph{isotypical} component of $V_a$ in $H_1^{(0)}(M,\mathbb{R})$, the commuting algebra of $\mathbb{R}(N/H)$ in $\textrm{End}_{\mathbb{R}}(W_a)$ is isomorphic to the matrix algebra $M(\ell_a,D_a)$.

Actually, it is possible to show that there are only three types of $a\in\textrm{Irr}_{\mathbb{R}}(N/H)$:
\begin{itemize}
\item $a$ is \emph{real}, i.e., $a=\{\alpha\}$, $m_a=1$ and $D_a\simeq\mathbb{R}$;
\item $a$ is \emph{complex}, i.e., $a=\{\alpha,\overline{\alpha}\}$, $m_a=1$ and $D_a\simeq\mathbb{C}$;
\item $a$ is \emph{quaternionic}, i.e., $a=\{\alpha\}$, $m_a=2$ and $D_a\simeq\mathbb{H}$ (Hamilton's quaternions).
\end{itemize}

In summary, the action of $\textrm{Aut}(M)\simeq N/H$ on $H_1^{(0)}(M,\mathbb{R})$ is completely determined: $H_1^{(0)}(M,\mathbb{R})$ is decomposed into isotypical components $W_a\simeq V_a^{\ell_a}$ (where the multiplicity $\ell_a$ is explicitly computable) of real, complex or quaternionic type. 

After this warm up, let's consider the action of the affine group $\textrm{Aff}(M)$ on $H_1^{(0)}(M,\mathbb{R})$. Note that $\textrm{Aff}(M)$ acts by conjugation on $\textrm{Aut}(M)$, i.e., $AgA^{-1}\in\textrm{Aut}(M)$ whenever $A\in\textrm{Aff}(M)$ and $g\in\textrm{Aut}(M)$. In particular, the elements of $\textrm{Aff}(M)$ \emph{permute} the isotypical components $W_a$ of the $\textrm{Aut}(M)$-module $H_1^{(0)}(M,\mathbb{R})$, so that one can pass to an adequate \emph{finite index} subgroup $\textrm{Aff}_{**}(M)$ of $\textrm{Aff}(M)$ such that the elements of $\textrm{Aff}_{**}(M)$ \emph{preserve} each isotypical component $W_a$ and they act on the $\mathbb{R}(\textrm{Aut}(M))$-module $W_a$ via automorphisms.

In fact, we can say a little bit more about the action of $\textrm{Aff}_{**}(M)$ on isotypical components $W_a$. Recall that $H_1^{(0)}(M,\mathbb{R})$ carries a symplectic intersection form $\{.,.\}$ preserved by $\textrm{Aff}(M)$. Moreover, it is not hard to see that the restriction $\{.,.\}|_{W_a}$ of $\{.,.\}$ to any isotypical component $W_a$ of $H_1^{(0)}(M,\mathbb{R})$ is non-degenerate. In other words, $\textrm{Aff}_{**}(M)$ acts on $W_a$ \emph{via} the group 
$$Sp(W_a):=\{\textrm{automorphisms of the }\mathbb{R}(N/H)-\textrm{module }W_a \textrm{ preserving } \{.,.\}_{W_a}:=\{.,.\}|_{W_a}\}\,.$$
By studying each possibility for $a\in\textrm{Irr}_{\mathbb{R}}(N/H)$, one can show that:
\begin{itemize}
\item if $a$ is \emph{real}, $\ell_a$ is even and $Sp(W_a)$ is isomorphic to the symplectic group $Symp(\ell_a,\mathbb{R})$;
\item if $a$ is \emph{complex}, there are integers $p_a$, $q_a$ with $\ell_a=p_a+q_a$ such that $Sp(W_a)$ is isomorphic to the group $U_{\mathbb{C}}(p_a,q_a)$ of matrices with complex coefficients preserving a pseudo-Hermitian form of signature $(p_a,q_a)$;
\item if $a$ is \emph{quaternionic}, there are integers $p_a$, $q_a$ with $\ell_a=p_a+q_a$ such that $Sp(W_a)$ is isomorphic to the group $U_{\mathbb{H}}(p_a,q_a)$ of matrices with quaternionic coefficients preserving a pseudo-Hermitian form of signature $(p_a,q_a)$.
\end{itemize}

From this discussion, we can already  derive some consequences for the Lyapunov exponents of the KZ cocycle: indeed, the fact that $\textrm{Aff}_{**}(M)$ acts on complex and quaternionic isotypical components $W_a$ via the groups $U_{\mathbb{C}}(p_a,q_a)$ and $U_{\mathbb{H}}(p_a,q_a)$ can be used to ensure\footnote{In fact, we already met this phenomenon during the proof of Proposition~\ref{p.U31}} the presence of $|p_a-q_a|$ zero Lyapunov exponents (at least). See e.g. \cite{MYZ} and/or Appendix A of \cite{FMZ3} for more details. Moreover, by looking at the definitions it is not hard to show that Oseledets subspaces $W_a(\theta,x)$ associated to a Lyapunov exponent $\theta$ of the restriction of the KZ cocycle (or, equivalently $\textrm{Aff}_{**}(M)$) to a isotypical component $W_a\simeq 
V_a^{\ell_a}$ at a point $x$ in the $SL(2,\mathbb{R})$-orbit of $M$ are $\textrm{Aut}(M)$-invariant. Therefore, these Oseledets subspaces $W_a(\theta,x)$ is a $\mathbb{R}(\textrm{Aut}(M))$-module obtained as a finite sum of copies of $V_a$, and, \emph{a fortiori}, the multiplicity of the Lyapunov exponent $\theta$ (i.e., the dimension of $W_a(\theta,x)$) is a \emph{multiple} of $\textrm{dim}_{\mathbb{R}}(V_a)$. 

In a nutshell, we can summarize our discussion so far as follows: starting with a square-tiled surface $M$ with a \emph{non-trivial} group of automorphism $\textrm{Aut}(M)$ (in the sense that $\textrm{Aut}(M)$ has a rich representation theory), \emph{usually} one finds: 
\begin{itemize}
\item several vanishing Lyapunov exponents, mostly coming from complex and/or quaternionic isotypical components, and 
\item \emph{high multiplicity}, i.e., non-simplicity, of general Lyapunov exponents.
\end{itemize} 

Closing this subsection, let's point out that our discussion so far depend only on the knowledge of $G, H$ and of the \emph{commutator} $c=[h,v]$, cf. Remark~\ref{r.commutator-multiplicity}. In particular, one may ask whether the Lyapunov exponents $\theta$ depend only on $c$. As it turns out, the answer to this question is \emph{negative}: for instance, for $G=A_6$ and $H=\{id\}$, it is possible to construct two pairs of permutations $(h_0,v_0)$ and $(h_1,v_1)$ generating $G=A_6$ with the \emph{same} commutator $c=[h_0,v_0]=[h_1,v_1]$ such that the sums of the non-negative Lyapunov exponents of the KZ cocycle over the $SL(2,\mathbb{R})$-orbits of the  square-tiled surfaces encoded by $(h_0,v_0)$ 
and $(h_1,v_1)$  are respectively $278/5$ and $54$, hence the Lyapunov spectra are \emph{different}.

\subsection{A criterion for the simplicity of Lyapunov exponents of square-tiled surfaces} In this subsection we will follow the paper~\cite{MMY} by C.~Matheus, M.~M\"oller and J.-C.~Yoccoz to give a criterion for the \emph{simplicity} of the Lyapunov exponents of square-tiled surfaces. Then, we will see some applications of this criterion to square-tiled surfaces in $\mathcal{H}(2)$ (genus 2) 
and $\mathcal{H}(4)$ (genus 3).

\subsubsection{Avila--Viana simplicity criterion for cocycles over countable shifts} Let's start by studying the simplicity of Lyapunov exponents in the \emph{abstract} setting of cocycles over countable shifts. 

Let $\Lambda$ be a finite or countable alphabet. Define $\Sigma = \Lambda^{\mathbb{N}}$ and denote by $f:\Sigma\to\Sigma$ the natural (left) shift map on $\Sigma$. Let 
$\Omega = \bigcup\limits_{n\geq 0}\Lambda^n$ the set of words of the alphabet $\Lambda$. Given $\underline{\ell}\in\Omega$, let
$$\Sigma(\underline{\ell}):=\{x\in\Sigma: x \textrm{ starts by }\underline{\ell}\}$$

\begin{definition}
We say that a probability measure $\mu$ on $\Sigma$ has \emph{bounded distortion} whenever there exists a constant $C(\mu)>0$ such that, for any $\underline{\ell}_1, \underline{\ell}_2\in\Omega$,
$$\frac{1}{C(\mu)}\mu(\Sigma(\underline{\ell}_1))\mu(\Sigma(\underline{\ell}_2)) \leq \mu(\Sigma(\underline{\ell}_1\underline{\ell}_2)) \leq C(\mu) \mu(\Sigma(\underline{\ell}_1)) \mu(\Sigma(\underline{\ell}_2))\,.$$
\end{definition}

The bounded distortion assumption says that, in some sense, $\mu$ is ``not very far'' from a Bernoulli measure. As an exercise, the reader can check that bounded distortion implies that $\mu$ is $f$-ergodic.

From now on, we will assume that $\mu$ has bounded distortion and we think of $(f,\mu)$ as our base dynamical system. Next, we discuss some assumptions concerning the class of cocycles we want to investigate over this base dynamics.

\begin{definition} We say that a cocycle $A:\Sigma\to Sp(2d,\mathbb{R})$ is
\begin{itemize} 
\item \emph{locally constant} if $A(\underline{x})=A_{x_0}$, where $A_{\ell}\in Sp(2d,\mathbb{R})$ for $\ell\in\Lambda$, and $\underline{x}=(x_0,\dots)\in\Sigma$;
\item ($\log$-)\emph{integrable} if $\int_{\Sigma}\log\|A^{\pm1}(\underline{x})\|\,d\mu(\underline{x}) = \sum \mu(\Sigma(\ell))\log\|A^{\pm1}_{\ell}\|<\infty$.
\end{itemize}
\end{definition}

\begin{remark} Following the work~\cite{AV} of Avila and Viana, we'll focus here in the case $A_\ell\in Symp(d,\mathbb{R})$, $d$ even, because we want to apply their criterion to a symplectic cocycle closely related to the Kontsevich--Zorich cocycle. However, it is not hard to see that Avila--Viana simplicity criterion below can be \emph{extended} to the groups $U_{\mathbb{C}}(p,q)$ and $U_{\mathbb{H}}(p,q)$, and this extension is particularly useful because KZ cocycle may act via these groups in some examples (as we already saw above). For more details on this, see \cite{MMY}.
\end{remark}

Given $\underline{\ell}=(\ell_0,\dots,\ell_{n-1})\in\Omega$, we write
$$A^{\underline{\ell}}:=A_{\ell_{n-1}}\dots A_{\ell_0}\,.$$

By definition, it satisfies $(f,A)^n(x)=(f^n(x),A^{\underline{\ell}})$ for any $x\in\Sigma(\underline{\ell})$.

The ergodicity of $\mu$ (coming from the \emph{bounded distortion} property) and the integrability of the cocycle $A$ allow us to apply the Oseledets theorem to deduce the existence of Lyapunov exponents
$$\theta_1\geq \dots\geq \theta_d\geq -\theta_d\geq\dots\geq-\theta_1\,.$$

We denote by $G(k)$ the Grassmanian of 
\begin{itemize}
\item \emph{isotropic} $k$-planes if $1\leq k\leq d$, and 
\item \emph{coisotropic} $k$-planes if $d\leq k< 2d$.
\end{itemize}

At this point, we are ready to introduced the main assumptions on our cocycle A:
\begin{definition}\label{d.pinching/twisting} We say that the cocycle $A$ is
\begin{itemize}
\item \emph{pinching} if there exists $\underline{\ell}^*\in\Omega$ such that the spectrum of the matrix $A^{\underline{\ell}^*}$ is simple.
\item \emph{twisting} if for each $k$ there exists $\underline{\ell}(k)\in\Omega$ such that 
$$A^{\underline{\ell}(k)}(F)\cap F'=\{0\}\,,$$
for any $A^{\underline{\ell}^*}$-\emph{invariant} subspaces $F\in G(k)$, $F'\in G(2d-k)$. 
\end{itemize} 
\end{definition}

\begin{remark}\label{r.twisting-wrt-pinching}Later on, we will refer to the matrices $B$ with the same property as $A^{\underline{\ell}(k)}$ above as ($k$-)\emph{twisting} with respect to (the pinching matrix) $A:=A^{\underline{\ell}^*}$.
\end{remark}

In this language, the following version of Avila and Viana's simplicity criterion \cite{AV} holds:

\begin{theorem}\label{t.AV-simplicity} Let $A$ be a locally constant $\log$-integrable cocycle over a base dynamics $(f,\mu)$ consisting of a countable shift $f$ and a $f$-invariant probability measure $\mu$ with bounded distortion. Suppose that the cocycle $A$ is pinching and twisting. Then, the Lyapunov spectrum of A is simple.
\end{theorem} 

In the sequel, we wish to apply this result to produce a simplicity criterion for the KZ cocycle over $SL(2,\mathbb{R})$-orbits of square-tiled surfaces. For this sake, we will need to briefly discuss how to \emph{code} the Teichm\"uller geodesic flow on $SL(2,\mathbb{R})$-orbits of square-tiled surfaces via a countable shift (closely related to the \emph{continued fraction algorithm}).

The $SL(2,\mathbb{R})$-orbit of a square-tiled surface $(M,\omega)$ is a finite cover of the \emph{modular surface} $SL(2,\mathbb{R})/SL(2,\mathbb{Z})$: indeed, one has $SL(2,\mathbb{R})\cdot(M,\omega)\simeq SL(2,\mathbb{R})/SL(M,\omega)$ in moduli space, where $SL(M,\omega)$ is the Veech group, and $SL(M,\omega)$ is a finite-index subgroup of $SL(2,\mathbb{Z})$ when $(M,\omega)$ is a square-tiled surface. These considerations suggest  to start our discussion by reviewing how the geodesic flow on the modular surface is coded by the continued fraction algorithm. 

\subsubsection{Coding the geodesic flow on the modular surface} We will think of $SL(2,\mathbb{R})/SL(2,\mathbb{Z})$ as the space of normalized (i.e., unit covolume) lattices of $\mathbb{R}^2$, and we will select an appropriate \emph{fundamental domain}. Here, it is worth to point out that we're \emph{not} going to consider the lift to $SL(2,\mathbb{R})$ of the ``classical'' fundamental domain $\mathcal{F}=\{z\in\mathbb{H}: |z|\geq 1, |\textrm{Re}z|\leq 1/2\}$ of the action of $SL(2,\mathbb{Z})$ on the hyperbolic plane $\mathbb{H}$. Indeed, as we will see below, our choice of fundamental domain is \emph{not} $SO(2,\mathbb{R})$-invariant, while any fundamental domain obtained by lifting to $SL(2,\mathbb{R})$ a fundamental domain of $\mathbb{H}/SL(2,\mathbb{Z})$ must be $SO(2,\mathbb{R})$-invariant (as $\mathbb{H}/SL(2,\mathbb{Z})=SO(2,\mathbb{R})\backslash SL(2,\mathbb{R})/SL(2,\mathbb{Z})$).

\begin{definition} A lattice $L\subset\mathbb{R}^2$ is \emph{irrational} if $L$ intersect the coordinate axis $x$ and $y$ only at the origin $0\in\mathbb{R}^2$. Equivalently, 
$L$ is irrational if and only if the orbit $g_t(L)$ doesn't diverge (neither in the past nor in the future) to the cusp of $SL(2,\mathbb{R})/SL(2,\mathbb{Z})$.
\end{definition}

Our choice of fundamental domain will be guided by the following fact:

\begin{proposition}\label{p.0} Let $L$ be a normalized irrational lattice. Then, there exists an \emph{unique} basis $\{v_1=(\lambda_1,\tau_1), v_2=(\lambda_2,\tau_2)\}$ of $L$ such that \emph{exactly} one of the two possibilities below occur:
\begin{itemize}
\item \emph{top case} -- $\lambda_2\geq 1>\lambda_1>0$ and $0<\tau_2<-\tau_1$;
\item \emph{bottom case} -- $\lambda_1\geq 1>\lambda_2>0$ and $0<-\tau_1<\tau_2$.
\end{itemize}
\end{proposition}

The proof of this result is based on Minkowski's theorem (ensuring the existence of vectors of the irrational lattice $L$ in $Q^+=(0,1)\times (0,1)$ or $Q^-=(0,1)\times (-1,0)$) and some elementary computations. See, e.g., \cite{MMY} for more details. 

Below, we illustrate irrational lattices of top and bottom types:

\begin{figure}[htb!]
\includegraphics[scale=0.6]{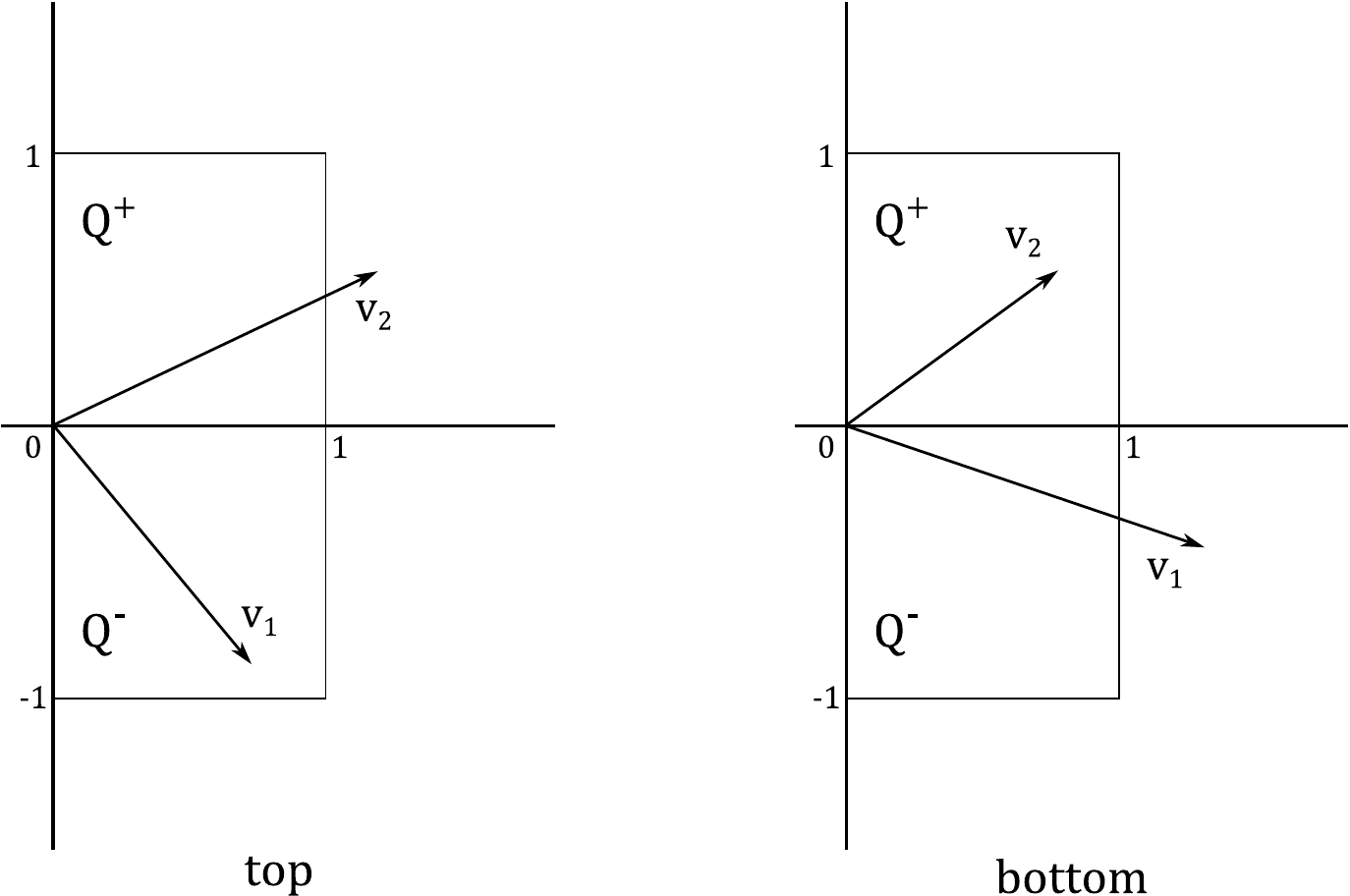}
\end{figure}

Using this proposition, we can describe the Teichm\"uller geodesic flow $g_t=\left(\begin{array}{cc}e^t & 0 \\ 0 & e^{-t} \end{array}\right)$ on the space $SL(2,\mathbb{R})/SL(2,\mathbb{Z})$ of normalized lattices as follows. Let $L_0$ be a normalized {\em irrational} lattice, and let $(v_1,v_2)$ be the basis of $L_0$ given by the proposition above, i.e., the top, resp. bottom, condition. Then, we see that the basis $(g_t v_1, g_t v_2)$ of $L_t:=g_t L_0$ satisfies the top, resp. bottom condition for all $t<t^*$, where $\lambda_1 e^{t^*}=1$ in the top case, resp. $\lambda_2 e^{t^*}=1$ in the bottom case.

However, at time $t^*$, the basis $\{v_1^*=g_{t^*} v_1, v_2^*=g_{t^*} v_2\}$ of $L_0$ {\em ceases} to fit the requirements of the above proposition, but we can remedy this problem by {\em changing} the basis: for instance, if the basis $\{v_1,v_2\}$ of the initial lattice $L_0$ has {\em top} type, then it is not hard to check that the vectors
$$v_1'=v_1^* \quad \textrm{ and } \quad v_2'=v_2^*-a v_1^* \,,$$
with the choice of $a=\lfloor\lambda_2/\lambda_1\rfloor$,  form a basis of $L_{t^*}$ of {\em bottom} type. 

This is illustrated in the picture below:

\begin{figure}[htb!]
\includegraphics[scale=0.7]{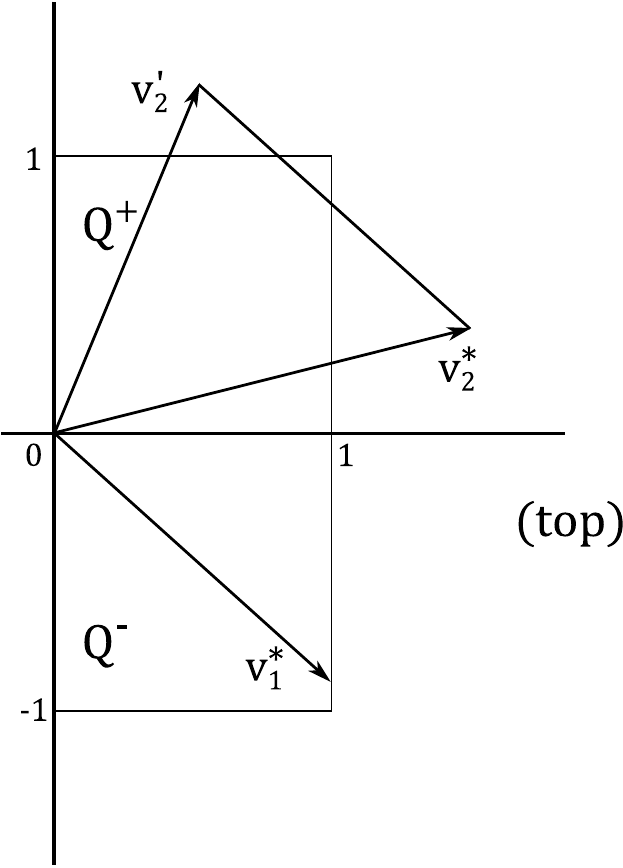}
\end{figure}

We observe that the quantity $\alpha:=\lambda_1/\lambda_2\in (0,1)$ giving the ratios of the first coordinates of the vectors $g_t v_1, g_t v_2$, forming a  basis of top type of the lattice $L_t$ for all 
$0\leq t<t^*$, is related to the integer $a=\lfloor\lambda_2/\lambda_1\rfloor$ by the formula
$$a=\lfloor 1/\alpha\rfloor$$
Also, the new quantity $\alpha'$ giving the ratio of the first coordinates of the vectors $v_1', v_2'$ forming a bottom type basis of $L_{t^*}$ is related to $\alpha$ by the formula
$$\alpha'=\lambda_2'/\lambda_1' = \{1/\alpha\}:=G(\alpha)\,,$$
where $G$ is the so-called \emph{Gauss map}. In this way, we find the classical relationship between the geodesic flow on the modular surface $SL(2,\mathbb{R})/SL(2,\mathbb{Z})$ and the continued fraction algorithm.

At this stage, we're ready to code the Teichm\"uller flow over the unit tangent bundle of the Teichm\"uller surface $SL(2,\mathbb{R})/SL(M)$ associated to a square-tiled surface.

\subsubsection{Coding the Teichm\"uller flow on $SL(2,\mathbb{R})$-orbits of square-tiled surfaces} Let $\Gamma(M)$ be the following \emph{graph}: the set of its vertices is
$$\textrm{Vert}(\Gamma(M)) = \{SL(2,\mathbb{Z})\textrm{-orbit of } M \}\times \{t,b\} = \{M=M_1,\dots, M_r\}\times \{t,b\}$$
and its arrows are
$$(M_i,c)\stackrel{\gamma_{a,i,c}}{\to} (M_j,\overline{c})\,,$$
where $a\in\mathbb{N}$, $a\geq 1$, $c\in\{t,b\}$, $\overline{c}= b$ (resp. $t$) if $c=t$ (resp. $b$), and
$$M_j=\left\{\begin{array}{cl}\left(\begin{array}{cc} 1 & a \\ 0 & 1\end{array}\right) M_i, & \textrm{if } c=t\, ; \\ \left(\begin{array}{cc} 1 & 0 \\ a & 1\end{array}\right) M_i, & \textrm{if } c=b\,. \end{array}\right.$$

Notice that this graph has finitely many vertices but countably many arrows. Using this graph, we can code irrational orbits of the flow $(g_t)_{t\in \mathbb R}$ on $SL(2,\mathbb{R})/SL(M)$ as follows. Given any element $m_0\in SL(2,\mathbb{R})$, let $L_{st}=\mathbb{Z}^2$ be the {\em standard} lattice and put $m_0 L_{st} = L_0$. Also, let us denote $m_t=g_t m_0$.
By Proposition~\ref{p.0}, there exists an {\em unique} $h_0\in SL(2,\mathbb{Z})$ such that the pair
$$v_1=m_0 h_0^{-1}(e_1)\,, \quad v_2 = m_0 h_0^{-1}(e_2)$$ 
satisfies the conditions of the proposition (here, $\{e_1, e_2\}$ is the canonical basis 
of $\mathbb{R}^2$).

Let us denote by $c$ the type (top or bottom) of the basis $\{v_1,v_2\}$ of $L_0$. We assign to $m_0$ the vertex $(M_i, c)\in\textrm{Vert}(\Gamma(M))$ with $M_i:=h_0 M$.
For the sake of concreteness, let's assume that $c=t$ (top case). Following the notations introduced after the proof of Proposition~\ref{p.0}, we notice that the lattice $L_{t^*}$ associated to $m_{t^*}$ has a basis of {\em bottom} type formed by the vectors
$$v_1' = g_{t^*} m_0 h_0^{-1}(e_1) = g_{t^*} m_0 h_1^{-1}(e_1) \quad \textrm{ and } \quad v_2' = g_{t^*} m_0 h_0^{-1}(e_2-a e_1) = g_{t^*} m_0 h_1^{-1}(e_2)$$
with $h_1=h_* h_0$ and 
$$h_* = \left(\begin{array}{cc}1 & a \\ 0 & 1\end{array}\right)\,.$$

In other words, starting from the vertex $(M_i,t)$ associated to the initial point $m_0$, after running the geodesic flow for a time $t^*$, we end up with the vertex $(M_j,b)$ where $M_j= h_* M_i$. Equivalently, the piece of trajectory from $m_0$ to $g_{t^*} m_0$ is coded by the arrow 
$$(M_i,t)\stackrel{\gamma_{i,a,t}}{\to}(M_j,b)\,.$$

Evidently, we can iterate this procedure (by replacing $L_0$ by $L_{t^*}$) in order to code the {\em entire} (forward) orbit $\{g_t m_0\vert t\geq 0\}$ by a succession of arrows. However, this coding has the  ``inconvenient'' (with respect to the setting of Avila--Viana simplicity criterion) that it is {\em not} associated to a {\em complete shift} but only to a {\em subshift} (as we do not have the right to concatenate two arrows $\gamma$ and $\gamma'$ unless the endpoint of $\gamma$ coincides with 
the start of $\gamma'$).  Fortunately, this little difficulty is easy to overcome: in order to get a coding by a complete shift, it suffices  consider {\em exclusively} concatenations of {\em loops} based at a \emph{fixed} vertex $p^*\in\textrm{Vert}(\Gamma(M))$. Of course, we pay a price here: since there may be some orbits of the flow $(g_t)_{t\in \mathbb R}$ whose coding is not a concatenation of loops based on $p^*$, thahus we are {\em throwing} away some orbits in this new way of coding. But, it is not hard to see that the (unique, Haar) $SL(2,\mathbb{R})$-invariant probability measure $\mu$ on $SL(2,\mathbb{R})/SL(M)$ gives zero weight to the orbits that we are throwing away, so that this new coding still captures most orbits of $(g_t)_{t\in \mathbb R}$ (from the point of view of $\mu$). In any case, this allows to code $(g_t)_{t\in \mathbb R}$ by a complete shift whose (countable) alphabet is constituted of (minimal) loops based at $p^*$.

Once we know how to code our flow $(g_t)_{t\in \mathbb R}$ by a complete shift, the next natural step (in view of Avila--Viana criterion) is the verification of the \emph{bounded distortion condition} for the invariant measure induced by $\mu$ on the complete shift described above. 

According to the above discussion, the geodesic flow ({\em modulo} the stable manifolds, that is, the 
``$\tau$-coordinates'' [vertical coordinates]) is coded by the dynamical system
$$\textrm{Vert}(\Gamma(M))\times ((0,1)\cap(\mathbb{R}-\mathbb{Q}))\to\textrm{Vert}(\Gamma(M))\times ((0,1)\cap(\mathbb{R}-\mathbb{Q}))$$
given by $(p,\alpha)\mapsto (p',G(\alpha))$, where $G(\alpha)=\{1/\alpha\}=\alpha'$ is the Gauss map and $p\stackrel{\gamma_{a,p}}{\to}p'$ with $a=\lfloor1/\alpha\rfloor$. In this language, $\mu$ becomes (up to normalization) the \emph{Gauss measure} $dt/(1+t)$ on each copy $\{p\}\times (0,1)$, $p\in \textrm{Vert}(\Gamma(M))$, of the unit interval $(0,1)$. 

Now, for the sake of concreteness, let us fix $p^*$ a vertex of {\em top} type. Given $\gamma$ a loop based on $p^*$, i.e., a word on the letters of the alphabet of the coding leading to a complete shift, we denote by $I(\gamma)\subset (0,1)$ the interval corresponding to $\gamma$, that is, the interval $I(\gamma)$ consisting of $\alpha\in (0,1)$ such that the concatenation of loops (based at $p^*$) coding the orbit of $(p^*,\alpha)$ starts by the word $\gamma$.

In this setting, the measure induced by $\mu$ on the complete shift is easy to express: by definition, the measure of the cylinder $\Sigma(\gamma)$ corresponding to concatenations of loops (based at $p^*$) starting by $\gamma$ is the Gauss measure of the interval $I(\gamma)$ up to normalization. Since the Gauss measure is equivalent to the Lebesgue measure (as its density $1/(1+t)$ satisfies $1/2\leq1/(1+t)\leq 1$ in $(0,1)$), we conclude that the measure of $\Sigma(\gamma)$ is equivalent  (up to a multiplicative constant) to 
$$|I(\gamma)|:=\textrm{Lebesgue measure of } I(\gamma)\,.$$

In particular, it follows that the bounded distortion condition for the measure induced by $\mu$ on the complete shift is {\em equivalent} to the existence of a constant $C>0$ such that
\begin{equation}\label{e.bd}
C^{-1}|I(\gamma_0)|\cdot|I(\gamma_1)|\leq |I(\gamma)|\leq C|I(\gamma_0)|\cdot |I(\gamma_1)|\,,
\quad \text{ \rm for all }\gamma=\gamma_0\gamma_1\,.
\end{equation}

In summary, this reduces the analysis of the bounded distortion condition to the problem of understanding the interval $I(\gamma)$. Here, by the usual properties of the continued fraction algorithm, it is not hard to show that $I(\gamma)$ is a \emph{Farey} interval 
$$I(\gamma)=\left(\frac{p}{q}, \frac{p+p'}{q+q'}\right)$$
with the matrix
$$\left(\begin{array}{cc}p' & p \\ q' & q\end{array}\right)\in SL(2,\mathbb{Z})$$
 {\em t-reduced}, in the sense that $0<p'\leq p,q'<q$ (see Definition~\ref{def:reduced} below).

Consequently, from this description, we recover the classical fact that 
\begin{equation}\label{e.cf}
\frac{1}{2q^2}\leq |I(\gamma)|=\frac{1}{q(q+q')}\leq \frac{1}{q^2}
\end{equation}

Given $\gamma=\gamma_0\gamma_1$, and denoting by $\left(\begin{array}{cc}p_0' & p_0 \\ q_0' & q_0\end{array}\right)$, $\left(\begin{array}{cc}p_1' & p_1 \\ q_1' & q_1\end{array}\right)$ and $\left(\begin{array}{cc}p' & p \\ q' & q\end{array}\right)$ the matrices associated, respectively, to $\gamma_0$, $\gamma_1$ and $\gamma$, it is not hard to check that 
$$\left(\begin{array}{cc}p' & p \\ q' & q\end{array}\right)=\left(\begin{array}{cc}p_0' & p_0 \\ q_0' & q_0\end{array}\right)\left(\begin{array}{cc}p_1' & p_1 \\ q_1' & q_1\end{array}\right)\,,$$
so that $q=q_0'p_1+q_0q_1$. Since these matrices are {\em t-reduced}, we have that 
$$q_0q_1\leq q\leq 2q_0q_1\,.$$

Therefore, in view of \eqref{e.bd} and \eqref{e.cf}, the bounded distortion condition follows.

Once we have established that the base dynamics (the geodesic flow on $SL(2,\mathbb{R})/SL(M)$) is coded by a complete shift equipped with a probability measure with bounded distortion, we can pass to the study of the Kontsevich--Zorich cocycle in terms of the coding. 

\subsubsection{Coding KZ cocycle over $SL(2,\mathbb{R})$-orbits of square-tiled surfaces} Let 
$$(M_i,[\textrm{t, resp. b}])\stackrel{\gamma_{a,i,t}}{\to}(M_j,[\textrm{b, resp. t}])$$ 
be an arrow of $\Gamma(M)$ and denote by $A:M_i\to M_j$ an {\em affine} map of derivative $$\left(\begin{array}{cc}1 & a \\ 0 & 1\end{array}\right)\,, \quad \text{ \rm resp.} \quad \left(\begin{array}{cc}1 & 0 \\ a & 1\end{array}\right).$$ Of course, $A$ is only well-defined up to automorphisms of $M_i$ and
 $M_j$. In terms of translation structures,  the identity map $\textrm{id}:(M,\zeta)\to (M,g\zeta)$ is an affine map of derivative $g$ for any matrix $g\in SL(2,\mathbb{R})$ and for any translation structure $\zeta$ on $M$.

Given $\gamma$ a path in $\Gamma(M)$ obtained by concatenation $\gamma=\gamma_1\dots\gamma_{\ell}$, and starting at $(M_i,c)$ and ending at $(M_j,c')$, one has, by functoriality, an affine map $A_{\gamma}:M_i\to M_j$ given by $A_{\gamma}=A_{\gamma_{\ell}}\dots A_{\gamma_1}$. 

Suppose now that $\gamma$ is a loop based at $(M,c)$. Then, by definition, the derivative $A_{\gamma}\in SL(M)$. For our subsequent discussions, an important question is:  {\em what matrices of the Veech group $SL(M)$ can be obtained in this way?} The answer to this question can be formulated in terms of the following definition (already encountered in the previous section):

\begin{definition} \label{def:reduced} We say that $A=\left(\begin{array}{cc}a & b \\ c & d\end{array}\right)\in SL(2,\mathbb{Z})$ is 
\begin{itemize}
\item t-reduced if $0<a\leq b,c<d$;
\item b-reduced if $0<d\leq b,c<a$.
\end{itemize}
\end{definition}

Observe that the product of two t-reduced (resp. b-reduced) matrices is also t-reduced (resp. b-reduced), i.e., these conditions are stable under products. 

The following statement is the answer to the question above:

\begin{proposition} The matrices associated to the loops $\gamma$ based at the vertex $(M,c)$ are \emph{precisely} the c-reduced matrices of $SL(M)$.
\end{proposition} 

\subsubsection{Simplicity criterion for KZ cocycle over $SL(2,\mathbb{R})$-orbits of square-tiled surfaces}

At this point, our discussion so far implies that it suffices to check the above pinching and twisting conditions to obtain simplicity of the Lyapunov spectrum of square-tiled surfaces. In this direction, C.~Matheus, M.~M\"oller and J.-C.~Yoccoz~\cite{MMY} showed that in the context of square-tiled surfaces the pinching and twisting conditions (and, \emph{a fortiori}, the simplicity of the Lyapunov spectrum) can be obtained from certain \emph{Galois theory} conditions:

\begin{theorem}\label{t.MMY}
Let $M$ be a square-tiled surface. Assume that there are 
two affine diffeomorphisms $\varphi_A$ and $\varphi_B$, whose linear parts $D\varphi_A$ and $D\varphi_B$ are either both t-reduced or both b-reduced, such that the action of $\varphi_A$ 
and $\varphi_B$ on the subspace $H_1^{(0)}(M,\mathbb{R})$ is given by two matrices 
$A, B\in Sp(2g-2,\mathbb{Z})$ with the following properties:
\begin{itemize}
\item[i)] The eigenvalues of $A$ are real;
\item[ii)] The splitting field of the characteristic polynomial $P$ of $A$;
has degree $2^{g-1}(g-1)!$, i.e., the Galois group is as large as possible;
\item[iii)]  $A$ and $B^2$ don't share a common proper invariant subspace. 
\end{itemize}
Then the Lyapunov spectrum of $M$ is simple.

Finally, the condition iii) above can be verified by checking that i) and ii) hold, as well as the 
disjointness of the splitting fields of $A$ and $B$ (see Remark~\ref{r.MMY-Galois} below). 
\end{theorem}

In what follows, we'll give a sketch of proof of this theorem. We begin by noticing that the matrix $A$ verifies the \emph{pinching condition} (cf. Theorem~\ref{t.AV-simplicity} and Definition~\ref{d.pinching/twisting}): indeed, since the Galois group $G$ of $P$ is the largest possible, we have that $P$ is irreducible, and thus its roots are simple. By the assumption $(i)$, all roots $\lambda_i, \lambda_i^{-1}$, $1\leq i\leq d$, of $P$ are real, so that the pinching condition is violated by $A$ precisely when there are $i\neq j$ such that $\lambda_i=-\lambda_j^{\pm}$. However, this is \emph{impossible} because $G$ is the \emph{largest} possible: for instance, since $i\neq j$, we have an element of $G$ fixing $\lambda_i$ and exchanging $\lambda_j$ and $\lambda_j^{-1}$; applying this element to the relation $\lambda_i=-\lambda_j^{\pm}$, we would get that $\lambda_i=-\lambda_j$ \emph{and} $\lambda_i=-\lambda_j^{-1}$, so that $\lambda_j=\pm1$ a contradiction with the fact that $P$ is irreducible.

\begin{remark}\label{r.MMY-Galois} Concerning the applications of this theorem to the case of origamis, we observe that item $(iii)$ is satisfied whenever the \emph{splitting fields} $\mathbb{Q}(P_B)$ and $\mathbb{Q}(P)$ of the characteristic polynomials of $B$ and $A$ are \emph{disjoint} as extensions of 
$\mathbb{Q}$, i.e., $\mathbb{Q}(P_B)\cap\mathbb{Q}(P)=\mathbb{Q}$. Indeed, if $E\subset\mathbb{R}^{2d}$ is invariant by $A$ and $B^2$, one has that 
\begin{itemize}
\item $E$ is generated by eigenvectors of $A$ (as $A$ is pinching, i.e., $A$ has simple spectrum), so that $E$ is defined over $\mathbb{Q}(P)$, and 
\item $E$ is invariant by $B^2$, so that $E$ is also defined over $\mathbb{Q}(P_B)$.
\end{itemize}
Since $\mathbb{Q}(P)$ and $\mathbb{Q}(P_B)$ are disjoint, it follows that $E$ is defined over $\mathbb{Q}$. But this is impossible as $A$ doesn't have rational invariant subspaces (by $(i)$ and $(ii)$).
\end{remark}

Once we know that the matrix $A$ satisfies the pinching condition, the proof of Theorem~\ref{t.MMY} is reduced to checking the twisting condition with respect to $A$ (Remark~\ref{r.twisting-wrt-pinching}). Keeping this goal in mind, we introduced the following notations. 

We denote by $\widetilde{R}$ the set of roots of the polynomial $P$ (so that $\#\widetilde{R}=2d$), for each $\lambda\in\widetilde{R}$, we put $p(\lambda)=\lambda+\lambda^{-1}$, and we define $R=p(\widetilde{R})$ (so that $\# R=d$).

Given $1\leq k\leq d$, let $\widetilde{R}_k$, resp. $R_k$ be the set whose elements are subsets $\underline{\lambda}$ of $\widetilde{R}$, resp. $R$ with $k$ elements, and let $\widehat{R}_k$ be the set whose elements are subsets $\underline{\lambda}$ of $\widetilde{R}$ with $k$ elements such that $p|_{\underline{\lambda}}$ is \emph{injective}. In other words, $\widehat{R}_k$ consist of those $\underline{\lambda}\in\widetilde{R}_k$ such that if $\lambda\in\underline{\lambda}$, then $\lambda^{-1}\notin\underline{\lambda}$.

Next, we make a choice of basis of $\mathbb{R}^{2d}$ as follows. For each $\lambda\in\widetilde{R}$, we select an eigenvector $v_{\lambda}$ of $A$ associated to $\lambda$, i.e., $Av_\lambda=\lambda v_{\lambda}$. In particular, $v_{\lambda}$ is defined over $\mathbb{Q}(\lambda)\subset \mathbb{Q}(P)$). Then, we assume that the 
choices of $v_{\lambda}$'s are \emph{coherent} with the action of the Galois group $G$, i.e., $v_{g\lambda}=gv_{\lambda}$ (and thus $A(v_{g\lambda}) = (g\lambda) v_{g\lambda}$) for each $g\in G$. In this way, for each $\underline{\lambda}\in\widetilde{R}_k$, we can associated a multivector $v_{\underline{\lambda}} = v_{\lambda_1}\wedge\dots\wedge v_{\lambda_k}\in\bigwedge^k\mathbb{R}^{2d}$ (using the natural order of the elements of $\underline{\lambda}=\{\lambda_1<\dots<\lambda_k\}$).

By definition, $(\bigwedge^k A)(v_{\underline{\lambda}}) = N(\underline{\lambda})v_{\underline{\lambda}}$ where $N(\underline{\lambda}):=\prod\lambda_i$.

By our assumptions $(i)$ and $(ii)$ on $A$, we have that: 
\begin{itemize}
\item the set $\{v_{\underline{\lambda}}\vert \underline{\lambda}\in\widetilde{R}_k\}$ is a basis of $\bigwedge^k\mathbb{R}^{2d}\,;$
\item the subspace generated by $v_{\lambda_1}, \dots, v_{\lambda_k}$ is \emph{isotropic} if and only if $\underline{\lambda}=\{\lambda_1,\dots,\lambda_k\}\in\widehat{R}_k\,.$
\end{itemize}
Also, by an elementary (linear algebra) computation, it is not hard to check that a matrix $C$ is twisting with respect to $A$ if and only if  the coefficients $C^{(k)}_{\underline{\lambda}, \underline{\lambda}'}$
of the matrix $\bigwedge^k C$ in the basis $\{v_{\underline{\lambda}}\}$ satisfy the condition
\begin{equation}\label{e.condition-k}
C^{(k)}_{\underline{\lambda}, \underline{\lambda}'}\neq 0\,, \quad \text{ \rm for all }
\underline{\lambda}, \underline{\lambda}'\in\widehat{R}_k\,.
\end{equation}

In order to organize our discussions, we observe that the condition \eqref{e.condition-k} can be used to define an \emph{oriented graph} $\Gamma_k(C)$ as follows.The set $\textrm{Vert}(\Gamma_k(C))$ 
of vertices of $\Gamma_k(C)$ is $\widehat{R}_k$, and we have an arrow from $\underline{\lambda}_0$ to $\underline{\lambda}_1$ if and only if $C^{(k)}_{\underline{\lambda}_0,\underline{\lambda}_1}\neq 0$. In this language, \eqref{e.condition-k} corresponds to the fact that $\Gamma_k(C)$ is a \emph{complete} graph. Unfortunately, the verification of the completeness of the graph 
$\Gamma_k(C)$ is not simple in general, and hence it could be interesting to look for \emph{softer} properties of $\Gamma_k(C)$ ensuring completeness of $\Gamma_k(D)$ for some matrix $D$ constructed as a product of powers of $C$ and $A$. Here, we take our inspiration from Dynamical Systems and we introduce the following classical notion:
\begin{definition}\label{d.mixing} The graph $\Gamma_k(C)$ is \emph{mixing} if there exists $m\geq 1$ such that for all $\underline{\lambda}_0,\underline{\lambda}_1\in\widehat{R}_k$ we can find an oriented path in $\Gamma_k(C)$ of length $m$ going from $\underline{\lambda}_0$ to $\underline{\lambda}_1$.
\end{definition}  
Here, we note that it is important in this definition that we can connect two arbitrary vertices by a path of length \emph{exactly} $m$ (and not of length $\leq m$). For instance, the figure below shows a connected graph that is not mixing because all paths connecting $A$ to $B$ have \emph{odd} length while all paths connecting $A$ to $C$ have \emph{even} length. 

\begin{figure}[htb!]
\includegraphics[scale=0.7]{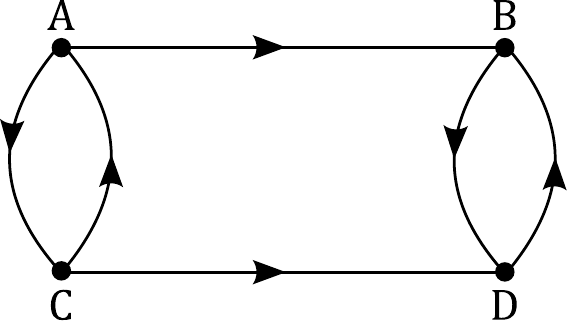}
\end{figure}

As the reader can guess by now, \emph{mixing} is a soft property ensuring completeness of a ``related'' graph. This is the content of the following proposition:
\begin{proposition}\label{p.mixing/twisting} Let us assume that the graph $\Gamma_k(C)$ is mixing with respect to an integer $m\geq 1$.  Then there exists a finite family of hyperplanes $V_1,\dots,V_t$ of 
$\mathbb{R}^{m-1}$ such that the following holds. For any $\underline{\ell}=(\ell_1,\dots,\ell_{m-1})\in \mathbb{Z}^{m-1}-(V_1\cup\dots\cup V_{m-1})$, the matrix 
$$D(n):=C A^{n\ell_1}\dots C A^{n\ell_{m-1}} C$$
satisfies \eqref{e.condition-k} for all sufficiently large $n\in \mathbb N$.
\end{proposition}

\begin{proof} For simplicity, let's adopt the notation $D:=D(n)$. By definition, 
$$
D^{(k)}_{\underline{\lambda}_0,\underline{\lambda}_m} =
\sum_{ \substack{\gamma \textrm{ path of length }m \\  \textrm{ in } \Gamma_k(C) \textrm{ from } \underline{\lambda}_0  \textrm{ to }\underline{\lambda}_m}} 
C^{(k)}_{\underline{\lambda}_0,\underline{\lambda}_1}N(\underline{\lambda}_1)^{n\ell_1} C^{(k)}_{\underline{\lambda}_1,\underline{\lambda}_2}\dots 
N(\underline{\lambda}_{m-1})^{n\ell_{m-1}} C^{(k)}_{\underline{\lambda}_{m-1},
\underline{\lambda}_m} \,.
$$
Let  then $L_\gamma$ be the linear forms defined as follows:
$$
L_\gamma(\underline{\ell}) = \sum\limits_{i=1}^{n-1} \ell_i\left(\sum\limits_{\lambda\in\underline{\lambda}_i} \log|\lambda|\right)\,.
$$
By the above formulas there exist coefficients $c_\gamma \not =0$ such that
$$
D^{(k)}_{\underline{\lambda}_0,\underline{\lambda}_m}  = \sum\limits_{\gamma} c_{\gamma} \exp(n L_{\gamma}(\underline{\ell}))\,.
$$
Our goal is to prove that  $D^{(k)}_{\underline{\lambda}_0,\underline{\lambda}_m}\neq0$. Of course, even though $D^{(k)}_{\underline{\lambda}_0,\underline{\lambda}_m}$ was expressed as a sum of exponentials with non-vanishing coefficients, there is a risk of getting non-trivial cancelations so that the resulting expression vanishes. The idea is to show that $\underline{\ell}$ can be chosen suitably to avoid such cancelations, and the heart of this argument is the observation that, for $\gamma\neq\gamma'$, the linear forms $L_{\gamma}$ and $L_{\gamma'}$ are \emph{distinct}. Indeed, given $\underline{\lambda}\in\widehat{R}_k$ and $\underline{\lambda}'\in\widetilde{R}_k$, $\underline{\lambda}'\neq\underline{\lambda}$, we claim that the following coefficients of $L_{\gamma}$ and $L_{\gamma'}$ differ:
$$\sum\limits_{\lambda\in\underline{\lambda}} \log|\lambda| \neq \sum\limits_{\lambda'\in\underline{\lambda}'} \log|\lambda'|\,.$$
Otherwise, we would have a relation 
$$\prod\limits_{\lambda\in\underline{\lambda}}\lambda = \pm\prod\limits_{\lambda'\in\underline{\lambda}'}\lambda' :=\phi\,.$$
However, since $\underline{\lambda}\in\widehat{R}_k$, we have that if $\lambda\in\underline{\lambda}$ then $\lambda^{-1}\notin\underline{\lambda}$. In particular, by taking an element $\lambda(0)\in\underline{\lambda}-\underline{\lambda}'$, and by considering an element $g$ of the Galois group $G$ with $g(\lambda(0))=\lambda(0)^{-1}$ and $g(\lambda)=\lambda$ otherwise, one would get on one hand that
$$g\phi = \prod\limits_{\lambda\in\underline{\lambda}}g\lambda = \lambda(0)^{-2}\phi\,,$$
but, on the other hand, 
$$g\phi=\pm\prod\limits_{\lambda'\in\underline{\lambda}'}g\lambda'=\pm\left\{\begin{array}{cc}\lambda(0)^2\phi & \textrm{ if } \lambda(0)^{-1}\in\underline{\lambda}'  \,, \\ \phi & \textrm{ otherwise, }\end{array}\right.$$
so that $\lambda(0)^{-2}\phi=\pm\lambda(0)^2\phi$ or $\pm\phi$, a contradiction in any event (as $\lambda(0)$ is real and $\lambda(0)\neq\pm1$ by the pinching 
conditions on the matrix $A$).

Now, we define $V(\gamma,\gamma')=\{\underline{\ell}: L_{\gamma}(\underline{\ell}) = L_{\gamma'}(\underline{\ell})\}$. Since $L_{\gamma}$ and $L_{\gamma'}$ are distinct linear forms for $\gamma\neq\gamma'$, it follows that $V(\gamma,\gamma')$ is a \emph{hyperplane}. Since there are only finitely many paths 
$\gamma,\gamma'$ of length $m$ on $\Gamma_k(C)$, the collection of $V(\gamma,\gamma')$ corresponds to a finite family of hyperplanes $V_1,\dots,V_t$.
Finally, we complete the proof by noticing that if $\underline{\ell}\notin V_1\cup\dots\cup V_t$, then 
$$D^{(k)}_{\underline{\lambda}_0,\underline{\lambda}_m}=\sum\limits_{\gamma} c_{\gamma} \exp(n L_{\gamma}(\underline{\ell}))\neq 0$$
for $n\to\infty$ sufficiently large because the coefficients $L_{\gamma}(\underline{\ell})$ are mutually distinct.  
\end{proof}

At this point, the proof of Theorem~\ref{t.MMY} goes along the following lines:

\begin{itemize}
\item{\em Step 0}: We will show that the graphs $\Gamma_k(C)$ are always non-trivial, i.e., there is at least one arrow starting at each of its vertices.
\item {\em Step 1:} Starting from $A$ and $B$ as above, we will show that $\Gamma_1(B)$ is mixing and hence, by Proposition~\ref{p.mixing/twisting}, there exists $C$ twisting $1$-dimensional (isotropic) $A$-invariant subspaces.
\item By Step 1, the treatment of the case $d=1$ is complete, so that we have to consider $d\geq 2$. Unfortunately, there is no ``unified'' argument to deal with all cases and we are obliged to separate the case $d=2$ from $d\geq 3$.
\item {\em Step 2:} In the case $d\geq 3$, we will show that $\Gamma_k(C)$ (with $C$ as in Step 1) is mixing for all $1\leq k<d$. Hence, by Proposition~\ref{p.mixing/twisting}, we can find $D$ twisting $k$-dimensional isotropic $A$-invariant subspaces for all $1\leq k<d$. Then, we will prove that $\Gamma_d(D)$ is mixing and, by 
Proposition~\ref{p.mixing/twisting}, we have $E$ twisting with respect to $A$, so that this completes the argument in this case.
\item {\em Step 3:} In the special case $d=2$, we will show that either $\Gamma_2(C)$ or a closely related graph $\Gamma_2^*(C)$ are mixing and we will see that this is sufficient to construct $D$ twisting $2$-dimensional isotropic $A$-invariant subspaces. 
\end{itemize} 

In the sequel, the following easy remarks will be repeatedly used:

\begin{remark}\label{r.galois-invariance} If $C\in Sp(2d,\mathbb{Z})$, then the graph $\Gamma_k(C)$ is invariant under the action of Galois group $G$ on the set $\widehat{R}_k\times \widehat{R}_k$ (parametrizing all possible arrows of $\Gamma_k(C)$). In particular, since the Galois group $G$ is the largest possible, whenever an arrow 
$\underline{\lambda}\to\underline{\lambda}'$ belongs to $\Gamma_k(C)$, the inverse arrow $\underline{\lambda}'\to\underline{\lambda}$ also belongs to $\Gamma_k(C)$. Consequently, $\Gamma_k(C)$ always contains loop of even length. 
\end{remark}

\begin{remark}\label{r.mixing} A connected graph $\Gamma$ is not mixing if and only if there exists an integer $m\geq 2$ such that the lengths of all of its loops 
are multiples of $m$.
\end{remark}

\noindent\textbf{Step 0: $\Gamma_k(C)$ is nontrivial.}

\begin{lemma}Let $C\in Sp(2d,\mathbb{R})$. Then, each $\underline{\lambda}\in\widehat{R}_k$ is the start of at least one arrow of $\Gamma_k(C)$.
\end{lemma}

\begin{remark} Notice that we allow symplectic matrices with real (not necessarily integer) coefficients in this lemma. However, the fact that $C$ is symplectic is important here and the analogous lemma for general invertible (i.e., $GL$) matrices is false.
\end{remark}

\begin{proof} For $k=1$, since every $1$-dimensional subspace is isotropic, $\widehat{R}_1=\widetilde{R}$ and the lemma follows in this case from the fact that $C$ is invertible. So, let's assume that $k\geq 2$ (and, in particular, $\widehat{R}_k$ is a proper subset of $\widetilde{R}_k$). Since $C$ is invertible, for each $\underline{\lambda}\in\widehat{R}_k$, there exists $\underline{\lambda}'\in\widetilde{R}_k$ with 
$$C^{(k)}_{\underline{\lambda}, \underline{\lambda}'}\neq 0\,.$$
Of course, one may have a priori that $\underline{\lambda}'\in\widetilde{R}_k - \widehat{R}_k$, i.e., $\# p(\underline{\lambda})<k$, and, in this case, our task is to ``convert'' $\underline{\lambda}'$ into some $\underline{\lambda}''\in\widehat{R}_k$ with $C^{(k)}_{\underline{\lambda}, \underline{\lambda}''}\neq 0$. 

Evidently, in order to 
accomplish this task it suffices to show that if $\#p(\underline{\lambda}')<k$ and $C^{(k)}_{\underline{\lambda}, \underline{\lambda}'}\neq 0$, then there exists $\underline{\lambda}''$ with 
$C^{(k)}_{\underline{\lambda}, \underline{\lambda}''}\neq 0$ and $\#p(\underline{\lambda}'')=\#p(\underline{\lambda}')+1$. Keeping this goal in mind, we observe that $\underline{\lambda}'\notin\widehat{R}_k$ implies that we can write $\underline{\lambda}'=\{\lambda_1',\lambda_2',\dots,\lambda_k'\}$ with $\lambda_1'\cdot \lambda_2'=1$. Also, the fact that $C^{(k)}_{\underline{\lambda}, \underline{\lambda}'}\neq0$ is equivalent to say that the $k\times k$ minor of $C$ associated to $\underline{\lambda}$ and $\underline{\lambda}'$ is invertible, and hence, by writing $\underline{\lambda}=\{\lambda_1,\lambda_2,\dots,\lambda_k\}$, we can find $w_1,\dots, w_k\in\mathbb{R}^{2d}$ such that $\textrm{span}\{w_1,\dots,w_k\}=\textrm{span}\{v_{\lambda_1},\dots, v_{\lambda_k}\}$ and 
$$
C(w_i)=v_{\lambda_i'}+\sum\limits_{\lambda\notin\underline{\lambda}'}C_{i\lambda}^* v_{\lambda}\,.
$$
In other words, we can make a change of basis to convert the invertible minor of $C$ into the $k\times k$ identity matrix. 

Now, denoting by $\{,.\}$ the symplectic form, we observe that $\{w_1, w_2\}=0$ because $\underline{\lambda}\in\widehat{R}_k$, i.e., the span of $v_{\lambda_i}$ is an isotropic subspace, and $w_1,w_2\in\textrm{span}\{v_{\lambda_1},v_{\lambda_2},\dots, v_{\lambda_k}\}$. On the other hand, since $C$ is symplectic, we get that 
$$0=\{w_1,w_2\}=\{C(w_1), C(w_2)\} = \{v_{\lambda_1'}, v_{\lambda_2'}\} + \sum_{\substack{\lambda', \lambda''\notin \underline{\lambda}' \\ \lambda'\cdot\lambda''=1}} C_{1\lambda'}^* C_{2\lambda''}^* \{v_{\lambda'}, v_{\lambda''}\}\,.$$
Since $\{v_{\lambda_1'}, v_{\lambda_2'}\}\neq 0$ (as $\lambda_1'\cdot\lambda_2'=1$),  there exists 
$\lambda', \lambda''\notin \underline{\lambda}'$ with $C^*_{1\lambda'}\neq 0$ and $C_{2\lambda''}^*\neq 0$. 

Then we define $\underline{\lambda}'':=(\underline{\lambda}'-\{\lambda_1'\})\cup\{\lambda'\}$. We have that $\#p(\underline{\lambda}'')=\#p(\underline{\lambda}')+1$. Furthermore, the minor $C[\underline{\lambda}, \underline{\lambda}'']$ of $C$ associated to $\underline{\lambda}$ and $\underline{\lambda}''$ is obtained from the minor $C[\underline{\lambda}, \underline{\lambda}']$ of $C$ associated to $\underline{\lambda}$ and $\underline{\lambda}'$ by removing the line associated to $v_{\lambda_1'}$ and replacing it by the line associated to $v_{\lambda'}$. By looking in the basis $w_1,\dots, w_k$, this means that the minor $C[\underline{\lambda}, \underline{\lambda}'']$ differs from the identity minor $C[\underline{\lambda}, \underline{\lambda}']$ by the fact that the line associated to $v_{\lambda_1'}$ was replaced by the line associated to $v_{\lambda'}$. In other words, in the basis $w_1, \dots, w_k$, one of the entries $1$ of  $C[\underline{\lambda}, \underline{\lambda}']$ was replaced by the coefficient $C^*_{1\lambda'}\neq 0$. Thus, we conclude that the determinant $C^{(k)}_{\underline{\lambda}, \underline{\lambda}''}$ of the the minor $C[\underline{\lambda}, \underline{\lambda}'']$ is
$$C^{(k)}_{\underline{\lambda}, \underline{\lambda}''} = C^*_{1\lambda'}\neq 0\,.$$
Therefore, $\underline{\lambda}''$ satisfies the desired properties and the argument is complete.
\end{proof}

\bigskip 

\noindent\textbf{Step 1: $\Gamma_1(B)$ is mixing.} For $d=1$, the set $\widehat{R}_1$ consists of exactly one pair $=\{\lambda,\lambda^{-1}\}$, so that the possible {\em Galois invariant} graphs are: 

\begin{figure}[htb!]
\includegraphics[scale=0.6]{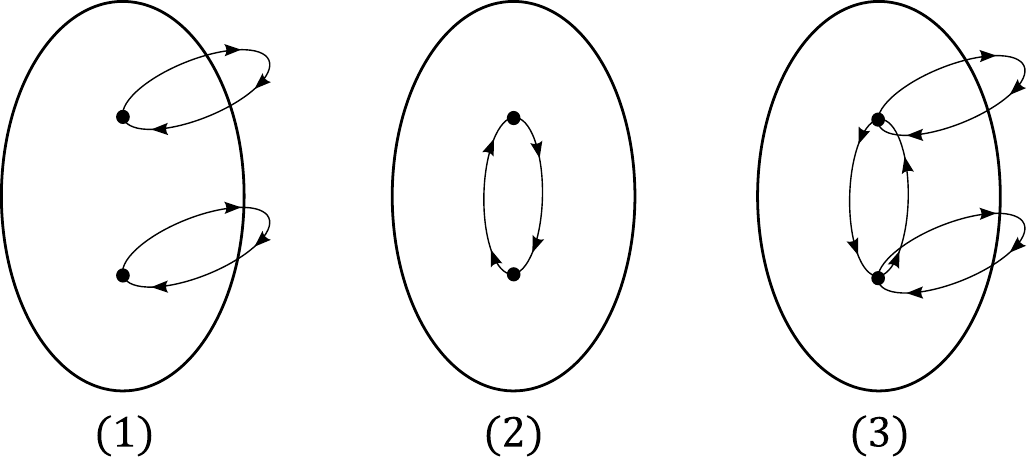}
\end{figure}

In the first case, by definition, we have that $B(\mathbb{R}v_{\lambda})=\mathbb{R}v_{\lambda}$ (and $B(\mathbb{R}v_{\lambda^{-1}}) = \mathbb{R}v_{\lambda^{-1}}$), so that 
$B$ and $A$ share a common subspace, a contradiction with our hypothesis in Theorem~\ref{t.MMY}. 

In the second case, by definition, we have that $B(\mathbb{R}v_{\lambda})=\mathbb{R}v_{\lambda^{-1}}$ and $B(\mathbb{R}v_{\lambda^{-1}}) = \mathbb{R}v_{\lambda}$, so that $B^2(\mathbb{R}v_{\lambda})=\mathbb{R}v_{\lambda}$ and thus $B^2$ and $A$ share a common subspace, a contradiction with our assumptions in Theorem~\ref{t.MMY}. 

Finally, in the third case, we have that the graph $\Gamma_1(B)$ is {\em complete}, and hence $B$ is $1$-twisting with respect to $A$. 

Now, after this ``warm up'', we pass to the general case $d\geq 2$. Firstly, suppose that the {\em sole} arrows in $\Gamma_1(B)$ are of the form $\lambda\to\lambda^{\pm1}$. Then, 
$B(\mathbb{R}v_{\lambda}\oplus\mathbb{R}v_{\lambda^{-1}}) = \mathbb{R}v_{\lambda}\oplus\mathbb{R}v_{\lambda^{-1}}$, and, {\em since} $d\geq 2$, the subspace 
$\mathbb{R}v_{\lambda}\oplus\mathbb{R}v_{\lambda^{-1}}$ is {\em non-trivial}. In particular, in this case, $B$ and $A$ share a common non-trivial subspace, a contradiction. Of course, this arguments breaks up for $d=1$ (and this is why we had a separate argument for this case). 

Therefore, we may assume that $\Gamma_1(B)$ has some arrow 
$\lambda\to\lambda'$ with $\lambda'\neq\lambda^{\pm1}$. Since the Galois group $G$ is the largest possible and $\Gamma_1(B)$ is invariant under the action of $G$ (see Remark~\ref{r.galois-invariance}), we have that {\em all} arrows of this type belong to $\Gamma_1(B)$. In view of Remarks~\ref{r.galois-invariance} and~\ref{r.mixing}, it suffices to construct a loop of {\em odd} length in $\Gamma_1(B)$. 

Since we dispose of all arrows $\lambda\to\lambda'$ with $\lambda'\neq\lambda^{\pm1}$, if $d\geq 3$, we can easily construct a loop of length $3$: 

\begin{figure}[htb!]
\includegraphics[scale=0.6]{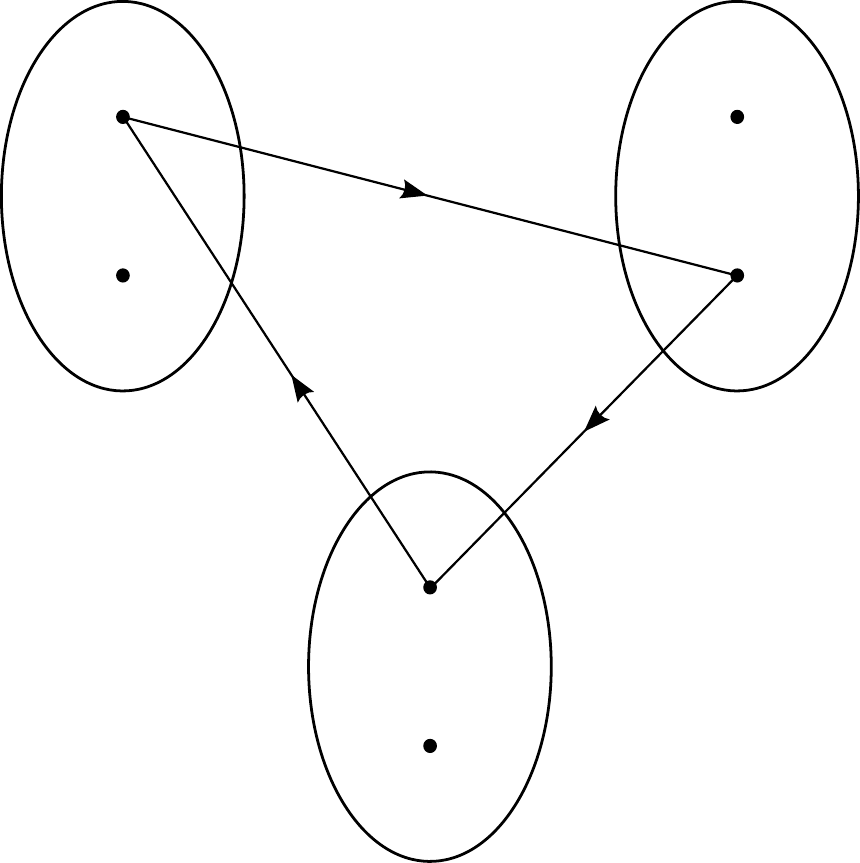}
\end{figure}

On the other hand, for $d=2$, we have two possibilities. If $\Gamma_1(B)$ is the non-mixing graph invariant under the Galois group:

\begin{figure}[htb!]
\includegraphics[scale=0.6]{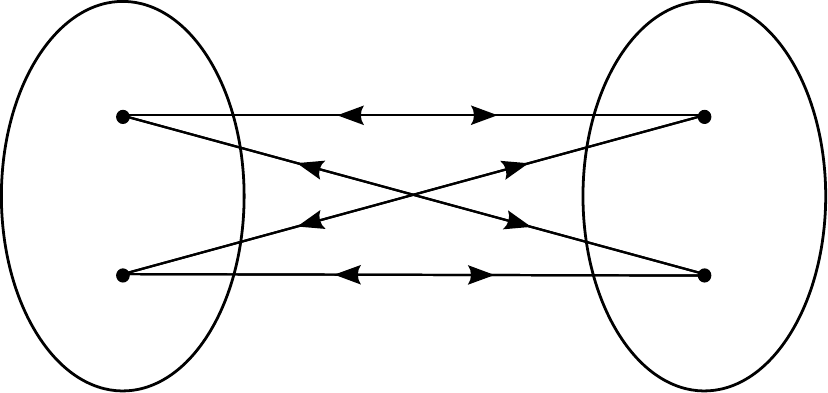}
\end{figure}

we get that $B(\mathbb{R}v_{\lambda_1}\oplus \mathbb{R}v_{\lambda_1}^{-1}) = \mathbb{R}v_{\lambda_2}\oplus \mathbb{R}v_{\lambda_2}^{-1}$ and $B(\mathbb{R}v_{\lambda_2}\oplus \mathbb{R}v_{\lambda_2}^{-1}) = \mathbb{R}v_{\lambda_1}\oplus \mathbb{R}v_{\lambda_1}^{-1}$, so that $B^2$ and $A$ share a common invariant subspace, a contradiction.

So, we have some extra arrow in the previous picture, say:

\begin{figure}[htb!]
\includegraphics[scale=0.6]{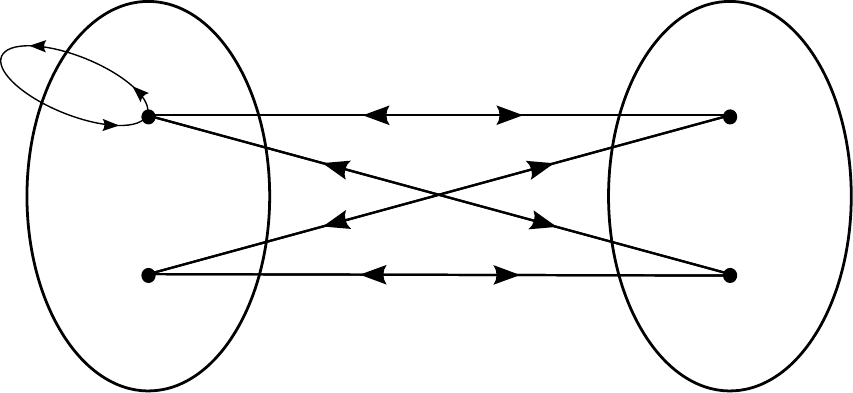}
\end{figure}

In this case, it is not hard to see that the addition of any extra arrow allows to build up loops of lenght $3$, so that, by Remarks~\ref{r.galois-invariance} and~\ref{r.mixing}, the argument is complete. 

Therefore, in any event, we proved that $\Gamma_1(B)$ is mixing. 

\bigskip

\noindent\textbf{Step 2: For $d\geq 3$, $\Gamma_k(C)$ is mixing for $2\leq k<d$, and $\Gamma_d(D)$ is mixing.} Given $C\in Sp(2d,\mathbb{Z})$ twisting $1$-dimensional $A$-invariant subspaces, we wish to prove that $\Gamma_k(C)$ is mixing for all $2\leq k<d$ whenever $d\geq 3$. Since $\Gamma_k(C)$ is invariant under the Galois group $G$ (see Remark~\ref{r.galois-invariance}), we start by considering the orbits of the action of $G$ on $\widehat{R}_k\times \widehat{R}_k$. 

\begin{proposition}\label{p.galois-orbits} The orbits of the action of $G$ on $\widehat{R}_k\times \widehat{R}_k$ are 
$$\mathcal{O}_{\widetilde{\ell}, \ell} = \{(\underline{\lambda}, \underline{\lambda}')\in \widehat{R}_k\times \widehat{R}_k: \#(\underline{\lambda}\cap \underline{\lambda}')=\widetilde{\ell}, \#(p(\underline{\lambda}) \cap p(\underline{\lambda}'))=\ell\}\,,$$
where 
\begin{equation}\label{e.etoile}
0\leq\widetilde{\ell}\leq\ell\leq k, \ell\geq 2k-d\,.
\end{equation}
\end{proposition}

We leave the proof of this proposition as an exercise to the reader. This proposition says that the orbits of the action $G$ on $\widehat{R}_k\times \widehat{R}_k$ are naturally parametrized by 
$$\widetilde{I}=\{(\widetilde{\ell}, \ell) \textrm{ satisfying } \eqref{e.etoile}\}\,.$$
In particular, since $\Gamma_k(C)$ is $G$-invariant, we can write $\Gamma_k(C)=\Gamma_k(\widetilde{J})$ for some $\widetilde{J}:=\widetilde{J}(C)\subset\widetilde{I}$, where $\Gamma_k(J)$ is the graph whose vertices are $\widehat{R}_k$ and whose arrows are 
$$\bigcup\limits_{(\widetilde{\ell}, \ell)\in\widetilde{J}} \mathcal{O}_{\widetilde{\ell}, \ell}\,.$$

\begin{proposition} The graph $\Gamma_k(\widetilde{J})$ is not mixing if and only if 
\begin{itemize}
\item either $k\neq d/2$ and $\widetilde{J}\subset\{(\widetilde{\ell},k): 0\leq \widetilde{\ell}\leq k\}$\,,
\item or $k=d/2$ and $\widetilde{J}\subset\{(\widetilde{\ell},k): 0\leq \widetilde{\ell}\leq k\}\cup\{(0,0)\}$\,.
\end{itemize}
\end{proposition}

\begin{proof} Let $\widetilde{J}\subset\{(\widetilde{\ell},k): 0\leq\widetilde{\ell}\leq k\}$ for $k\neq d/2$ or $\widetilde{J}\subset\{(\widetilde{\ell},k): 0\leq\widetilde{\ell}\leq k\}\cup\{(0,0)\}$ for $k=d/2$. Then, one can see that, {\em since} $k<d$, $\Gamma_k(\widetilde{J})$ is not mixing simply because it is not {\em connected}! For the proof of the converse statement, due to the usual space-time limitations, we're going only to say a few words on (referring to the forthcoming article by M. M\"oller, J.-C. Yoccoz and C. M. for formal arguments). Essentially, one starts by converting pairs $\{\lambda,\lambda^{-1}\}$ into a single point $p(\lambda)=p(\lambda^{-1})$, so that $\Gamma_k(\widetilde{J})$ becomes a new graph $\overline{\Gamma}_k(\widetilde{J})$. Then, one proves that, if $\widetilde{J}\not\subset \{(\widetilde{\ell},k): 0\leq \widetilde{\ell}\leq k\}\cup\{(0,0)\}$, then $\overline{\Gamma}_k(\widetilde{J})$ is connected. Using that $\overline{\Gamma}_k(\widetilde{J})$ is connected, it is possible to prove that $\Gamma_k(\widetilde{J})$ is connected and from this one can construct loops of odd length, thus getting the mixing property.
\end{proof}

Coming back to the study of $\Gamma_k(C)$, $1\leq k<d$, $d\geq 3$, we set $\widetilde{J}:=\widetilde{J}(C)$. By the previous proposition, if $\Gamma_k(C)$ is not mixing, then 
$\widetilde{J}(C)\subset \{(\widetilde{\ell}, k): 0\leq \widetilde{\ell}\leq k\}$ for $k\neq d/2$ or $\widetilde{J}\subset\{(\widetilde{\ell},k): 0\leq\widetilde{\ell}\leq k\}\cup\{(0,0)\}$ for $k=d/2$. For the sake of concreteness, we will deal ``only'' with the case $\widetilde{J}\subset\{(\widetilde{\ell},k): 0\leq\widetilde{\ell}\leq k\}\cup\{(0,0)\}$ (leaving the particular case 
$\widetilde{J}=\{(0,0)\}$ when $k=d/2$ as an exercise to the reader). In this situation, we have an arrow $\{\lambda_1, \dots, \lambda_k\}=\underline{\lambda}\to\underline{\lambda}'=\{\lambda_1', \dots, \lambda_k'\}$ of $\Gamma_k(C)$ with $p(\underline{\lambda})=p(\underline{\lambda}')$. This means that $C^{(k)}_{\underline{\lambda}, \underline{\lambda}'}\neq 0$, and hence we can find $w_1, \dots, w_k$ such that $\textrm{span}\{w_1,\dots,w_k\} = \textrm{span}\{v_{\lambda_1}, \dots, v_{\lambda_k}\}$ and 
$$C(w_i)=v_{\lambda_i'}+\sum\limits_{\lambda\notin\underline{\lambda}'} 
C_{i\lambda}^* v_{\lambda}\,.$$
In other words, as we also did in Step 0, we can use $w_1,\dots, w_k$ to ``convert'' the minor of $C$ associated to $\underline{\lambda}, \underline{\lambda}'$ into the identity. 

We claim that if $\lambda, \lambda^{-1}\notin \underline{\lambda}'$, then $C_{i\lambda}^*=0$ for all $i=1,\dots,k$. Indeed, by the same discussion around minors and replacement of lines, if this were not true, say $C_{i\lambda}^*\neq 0$, we could find an arrow from $\underline{\lambda}$ to $\underline{\lambda}''=(\underline{\lambda}'-\{\lambda_i'\})\cup\{\lambda\}$. Since $p(\underline{\lambda})=p(\underline{\lambda}')$, we have that $\#(p(\underline{\lambda})\cap p(\underline{\lambda}''))=k-1$, so that, for 
some $\widetilde{\ell}_0$, one has $(\widetilde{\ell}_0, k-1)\in \widetilde{J}\subset \{(\widetilde{\ell}, k): 0\leq \widetilde{\ell}\leq k\}$, a contradiction showing that the claim is true.

From the claim above we deduce that e.g. $C(v_{\lambda_1})$ is a linear combination of $v_{\lambda_i'}$, $i=1,\dots, k$, a contradiction with the fact that $C$ twists $1$-dimensional $A$-invariant subspaces. In other words, we proved that $\Gamma_k(C)$ is mixing for each $2\leq k<d$ whenever $C$ twists $1$-dimensional $A$-invariant subspaces. 
By Proposition~\ref{p.mixing/twisting}, it follows that we can construct a matrix $D$ twisting $k$-dimensional isotropic $A$-invariant subspaces for $1\leq k<d$, and we wish to show that $\Gamma_d(D)$ is mixing. 

In this direction, we consider the orbits of the action of the Galois group $G$ on $\widehat{R}_d\times\widehat{R}_d$. By Proposition~\ref{p.galois-orbits}, the orbits are 
$$\mathcal{O}_{\widetilde{\ell},\ell} = \{(\underline{\lambda}, \underline{\lambda}')\in \widehat{R}_k\times \widehat{R}_k: \#(\underline{\lambda}\cap \underline{\lambda}')=\widetilde{\ell}, \#(p(\underline{\lambda}) \cap p(\underline{\lambda}'))=\ell\}$$
with $\ell\leq k$, $\ell\geq 2k-d$ and $k=d$. In particular, $\ell=d$ in this case, and the orbits are parametrized by the set 
$$I=\{0\leq\widetilde{\ell}\leq d\}\,.$$
For the sake of simplicity, we will denote the orbits of $G$ on $\widehat{R}_d\times\widehat{R}_d$ by 
$$\mathcal{O}(\widetilde{\ell})=\{(\underline{\lambda}, \underline{\lambda}')\in \widehat{R}_d\times \widehat{R}_d: 
\#(\underline{\lambda}\cap \underline{\lambda}')=\widetilde{\ell}\}$$
and we write 
$$\Gamma_d(D)=\Gamma_d(J)=\bigcup\limits_{\widetilde{\ell}\in J}\mathcal{O}(\widetilde{\ell})\,,$$
where $J=J(D)\subset I=\{0\leq\widetilde{\ell}\leq d\}$.

It is possible to show (again by the arguments with ``minors'' we saw above) that if $D$ is $k$-twisting with respect to $A$, then $J$ contains two consecutive 
integers say $\widetilde{\ell}, \widetilde{\ell}+1$. 

We claim that $\Gamma_d(D)$ is {\em mixing} whenever $J$ contains two consecutive integers. 

Indeed, we start by showing that $\Gamma_d(J)$ is {\em connected}. Notice that it suffices to connect two vertices $\underline{\lambda}_0$ and $\underline{\lambda}_1$ with $\#(\underline{\lambda}_0\cap\underline{\lambda}_1)=d-1$ (as the general case of two general vertices $\underline{\lambda}$ and $\underline{\lambda}'$ follows by producing a series of vertices $\underline{\lambda}=\underline{\lambda}_0$, $\underline{\lambda}_1$, $\dots$, $\underline{\lambda}_a=\underline{\lambda}'$ with $\#(\underline{\lambda}_i\cap \underline{\lambda}_{i+1})=d-1$, $i=0,\dots, a-1$). Given $\underline{\lambda}_0$ and $\underline{\lambda}_1$ with $\#(\underline{\lambda}_0\cap\underline{\lambda}_1)=d-1$, we select $\underline{\lambda}'\subset \underline{\lambda}_0\cap \underline{\lambda}_1$ with $\#\underline{\lambda}'=d-\widetilde{\ell}-1$. Then, we consider $\underline{\lambda}''$ obtained 
from $\underline{\lambda}_0$ by replacing the elements of $\underline{\lambda}'$ by their inverses. By definition, $\#(\underline{\lambda}''\cap\underline{\lambda}_0) = \widetilde{\ell}+1$ and $\#(\underline{\lambda}''\cap\underline{\lambda}_1)=\widetilde{\ell}$ (because $\#(\underline{\lambda}_0\cap\underline{\lambda}_1)=d-1$). 
By assumption, $J$ contains $\widetilde{\ell}+1$ and $\widetilde{\ell}$, so that we have the arrows $\underline{\lambda}_0\to\underline{\lambda}''$ and 
$\underline{\lambda}''\to\underline{\lambda}_1$ in $\Gamma_d(J)$. Thus, the connectedness of $\Gamma_d(J)$ follows. 

Next, we show that $\Gamma_d(J)$ is mixing. Since $\Gamma_d(J)$ is invariant under the Galois group, it contains loops of length $2$ (see Remark~\ref{r.galois-invariance}). By Remark~\ref{r.mixing}, it suffices to construct some loop of {\em odd} length in $\Gamma_d(J)$. We fix an arrow 
$\underline{\lambda}\to\underline{\lambda}'\in\mathcal{O}(\widetilde{\ell})$ of $\Gamma_d(J)$. By the construction ``$\underline{\lambda}_0\to\underline{\lambda}''\to\underline{\lambda}_1$'' whenever $\#(\underline{\lambda}_0\cap\underline{\lambda}_1)=d-1$, performed in the proof of the connectedness of 
$\Gamma_d(J)$, we can connect $\underline{\lambda}'$ to $\underline{\lambda}$ by a path of length $2\widetilde{\ell}$ in $\Gamma_d(J)$. In this way, 
we have a loop (based on $\underline{\lambda}_0$) in $\Gamma_d(J)$ of length $2\widetilde{\ell}+1$. 

\bigskip

\textbf{Step 3: Special case $d=2$.} We consider the symplectic form $\{.,.\}:\wedge^2\mathbb{R}^4\to\mathbb{R}$. Since $\wedge^2\mathbb{R}^4$ has dimension $6$ and $\{.,.\}$ is non-degenerate, 
$K:=\textrm{Ker}\{.,.\}$ has dimension $5$. 

By denoting by $\lambda_1> \lambda_2>\lambda_2^{-1}> \lambda_1^{-1}$ the eigenvalues of $A$, we have the following basis of $K$
\begin{itemize}
\item $v_{\lambda_1}\wedge v_{\lambda_2}$, $v_{\lambda_1}\wedge v_{\lambda_2^{-1}}$, $v_{\lambda_1^{-1}}\wedge v_{\lambda_2}$, 
$v_{\lambda_1^{-1}}\wedge v_{\lambda_2^{-1}}$;
\item $v_*=\frac{v_{\lambda_1}\wedge v_{\lambda_1^{-1}}}{\omega_1} - \frac{v_{\lambda_2}\wedge v_{\lambda_2^{-1}}}{\omega_2}$ where $\omega_i = \{v_{\lambda_i}, v_{\lambda_i^{-1}}\}\neq 0$.
\end{itemize}

In general, given $C\in Sp(4,\mathbb{Z})$, we can use $\wedge^2C|_K$ to construct a graph $\Gamma_2^*(C)$ whose vertices are 
$\widehat{R}_2\simeq \{ v_{\lambda_1}\wedge v_{\lambda_2}, \dots, v_{\lambda_1^{-1}}\wedge v_{\lambda_2^{-1}}\}$ and $v_*$, and whose arrows connect vertices associated to non-zero entries of $\wedge^2C|_K$. By definition, $\wedge^2A(v_*)=v_*$, so that $1$ is an eigenvalue of $\wedge^2A|_K$. In principle, this poses a problem to apply Proposition~\ref{p.mixing/twisting} 
(to deduce $2$-twisting properties of $C$ from $\Gamma_2^*(C)$ is mixing), but, as it turns out, the fact that the eigenvalue $1$ of $\wedge^2A|_K$ is {\em simple} can be exploited to rework the proof of Proposition~\ref{p.mixing/twisting} to check that $\Gamma_2^*(C)$ is mixing implies the existence of adequate products $D$ of powers of $C$ and $A$ satisfying the $2$-twisting condition (i.e., $D$ twists $2$-dimensional $A$-invariant isotropic subspaces). Therefore, it ``remains'' to show that either $\Gamma_2(C)$ or $\Gamma_2^*(C)$ is mixing to complete this step. 

We write $\Gamma_2(C)=\Gamma_2(J)$ with $J\subset\{0, 1, 2\}$. 
\begin{itemize}
\item If $J$ contains $2$ two consecutive integers, then one can check that the arguments of the end of the previous sections work and $\Gamma_2(C)$ is mixing.
\item Otherwise, since $J\neq\emptyset$ (see Step 0), we have $J=\{0\}$, $\{2\}, \{1\}$ or $\{0,2\}$. As it turns out, the cases $J=\{0\}, \{2\}$ are ``symmetric'', as well as the cases $J=\{1\}, \{0,2\}$. 
\end{itemize}

For the sake of concreteness, we will consider the cases $J=\{2\}$ and $J=\{1\}$ (leaving the treatment of their ``symmetric'' as an exercise). We will show that the case $J=\{2\}$ is impossible, while the case $J=\{1\}$ implies that $\Gamma_2^*(C)$ is mixing. 

We begin by $J=\{2\}$. This implies that we have an arrow $\underline{\lambda}\to\underline{\lambda}$ with $\underline{\lambda}= \{\lambda_1,\lambda_2\}$. Hence, we can find $w_1, w_2$ with $\textrm{span}\{w_1,w_2\}=\textrm{span}\{v_{\lambda_1}, v_{\lambda_2}\}$ and 
$$C(w_1) = v_{\lambda_1}+ C^*_{11} v_{\lambda_1^{-1}} + C^*_{12} v_{\lambda_2^{-1}}\,, $$
$$C(w_2) = v_{\lambda_2}+ C^*_{21} v_{\lambda_1^{-1}} + C^*_{22} v_{\lambda_2^{-1}}\,.$$
Since $J=\{2\}$, the arrows $\underline{\lambda}\to\{\lambda_1^{-1}, \lambda_2\}$, $\underline{\lambda}\to\{\lambda_1, \lambda_2^{-1}\}$, and 
$\underline{\lambda}\to\{\lambda_1^{-1}, \lambda_2^{-1}\}$ do not belong $\Gamma_2(J)$. Thus, $C_{11}^*=C_{22}^*=0=C_{12}^*C_{21}^*$. On the other hand, because $C$ is symplectic, $\omega_1 C_{21}^* - \omega_2 C_{12}^* = 0$ (with $\omega_1,\omega_2\neq 0$). It follows that $C^*_{ij}=0$ for all $1\leq i,j\leq 2$, that is, $C$ preserves 
the $A$-subspace spanned by $v_{\lambda_1}$ and $v_{\lambda_2}$, a contradiction with the fact that $C$ is $1$-twisting with respect to $A$.

Now we consider the case $J=\{1\}$ and we wish to show that $\Gamma_2^*(C)$. We claim that, in this situation, it suffices to construct arrows from the vertex $v_*$ to $\widehat{R}_2$ {\em and} vice-versa. Notice that the action of the Galois group can't be used to revert arrows of $\Gamma_2^*(C)$ involving the vertex $v_*$, so that the two previous statements are ``independent''. Assuming the claim holds, we can use the Galois action to see that once $\Gamma_2^*(C)$ 
contains {\em some} arrows from $v_*$ and {\em some} arrows to $v_*$, it contains {\em all such arrows}. In other words, if the claim is true, we have the following situation:

\begin{figure}[htb!]
\includegraphics[scale=0.7]{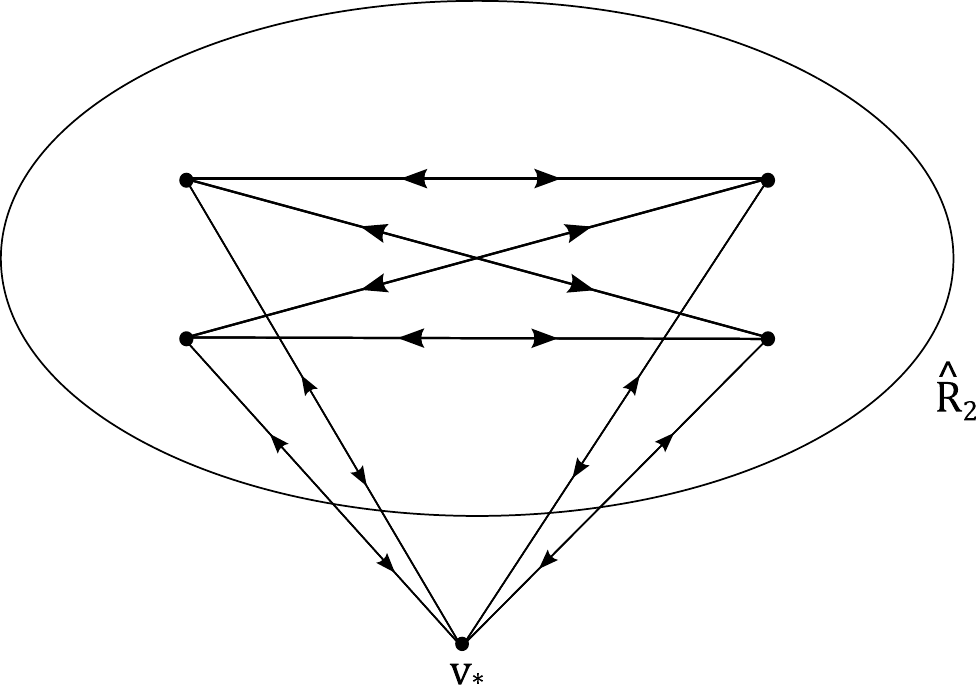}
\end{figure}

Thus, we have loops of length $2$ (in $\widehat{R}_2$), and also loops of length $3$ (based on $v_*$), so that $\Gamma_2^*(C)$ is mixing. In particular, our task is reduced to show the claim above. 

The fact that there are arrows from $\widehat{R}_2$ to $v_*$ follows from the same kind of arguments involving ``minors'' (i.e., selecting $w_1, w_2$ as above, etc.) and we will not repeat it here. 

Instead, we will focus on showing that there are arrows from $v_*$ to $\widehat{R}_2$. The proof is by contradiction: otherwise, one has $\wedge^2C(v_*)\in\mathbb{R}v_*$. Then, we invoke the following elementary lemma (whose proof is a straightforward computation):
\begin{lemma} Let $H\subset\mathbb{R}^4$ be any symplectic $2$-plane. Given any basis $e, f$ 
 of $H$ with $\{e,f\}=1$ (recall that $\{\cdot, \cdot\}$ denotes the symplectic form), we define
$$i(H):=e\wedge f\,.$$
The above definition is well-posed since the bi-vector $i(H)$ is independent of the choice of $e, f$ as above. Let us denote by $H^\perp$ the symplectic orthogonal of $H$ and set
$$v(H):=i(H)-i(H^\perp)\in K\,.$$
Then, $v(H)$ is collinear to $v(H')$ if and only if $H'=H$ or $H'=H^\perp$.
\end{lemma}

Since $v_*:=i(\textrm{span}(v_{\lambda_1}, v_{\lambda_1^{-1}}))$, by the above lemma
$\wedge^2C(v_*)\in\mathbb{R}v_*$ implies that either
$$C(\textrm{span}(v_{\lambda_1}, v_{\lambda_1^{-1}})) = \textrm{span}(v_{\lambda_1}, 
v_{\lambda_1^{-1}})$$ 
or 
$$\textrm{span}(v_{\lambda_2}, v_{\lambda_2^{-1}}),$$ 
a contradiction with the fact that $C$ is $1$-twisting with respect to $A$.

This completes the sketch of proof of Theorem~\ref{t.MMY}. Now, we will conclude this appendix with some applications of this theorem to the first two ``minimal'' strata 
$\mathcal{H}(2)$ and $\mathcal{H}(4)$.

\subsection{Some applications of the simplicity criterion}

In the sequel we will need the following lemma about the ``minimal'' strata:

\begin{lemma}\label{l.minimal} A translation surface in the stratum $\mathcal{H}(2g-2)$ has no non-trivial automorphisms.
\end{lemma}

\begin{proof} Let $(M,\omega)\in \mathcal{H}(2g-2)$ and denote by $p_M\in M$ the unique zero of the Abelian differential $\omega$.
Any automorphism $\phi$ of the translation structure $(M,\omega)$ satisfies $\phi(p_M)=p_M$. Suppose that there exists $\phi$ a non-trivial automorphism of $(M,\omega)$. We have that $\phi$ has finite order, say $\phi^{\kappa}=\textrm{id}$ where $\kappa\geq 2$ is the order of $\phi$. Since $\phi$ fixes $p_M$ and $\phi$ is non-trivial, $p_M$ is the {\em unique} fixed point of $\phi$. Hence, the quotient $N$ of $M$ by the
cyclic group $\langle\phi\rangle\simeq\mathbb{Z}/\kappa\mathbb{Z}$ can be viewed as a normal covering $\pi:(M,p_M)\to (N,p_N)$ of degree $\kappa$ such that $\pi$ is not ramified outside $p_N:=\pi(p_M)$, $\pi^{-1}(p_N)=\{p_M\}$ and $\pi$ is ramified of index $\kappa$ at $p_N$.

Now let us fix $\ast\in N-\{p_N\}$ a base point, so that the covering $\pi$ is given by a homomorphism $h:\pi_1(N-\{p_N\},\ast)\to\mathbb{Z}/\kappa\mathbb{Z}$. In this language, we see that a loop $\gamma\in \pi_1(N-\{p_N\}, \ast)$ around $p_N$ based at $\ast$ is a product of {\em commutators} in $\pi_1(N-\{p_N\},\ast)$, so that $h(\gamma)=1$ in  $\mathbb{Z}/\kappa\mathbb{Z}$, a contradiction with the fact that $\pi$ is ramified of order $\kappa\geq 2$ at $p_N$.
\end{proof}

\begin{remark} In a certain sense this lemma implies that origamis of minimal strata $\mathcal{H}(2g-2)$ are not concerned by the discussions of Subsection~\ref{ss.MYZ}: indeed, there the idea was to use representation theory of the automorphism group of the origami to detect multiple (i.e., non-simple) and/or zero exponents; in particular, we needed rich groups of automorphisms and hence, from this point of view, origamis without non-trivial automorphisms are ``uninteresting''.
\end{remark}

Next, we recall that M.~Kontsevich and A.~Zorich classified all connected components of strata 
$\mathcal{H}(k_1-1,\dots, k_s-1)$, and, as an outcome of this classification, any stratum has at most $3$ connected components. We will come back to this latter when discussing the stratum $\mathcal{H}(4)$. For now, let us mention that the two strata $\mathcal{H}(2)$ and $\mathcal{H}(1,1)$ of genus $2$ translation surfaces are {\em connected}.

Once we fix a connected component $\mathcal{C}$ of a stratum, we can ask about the classification of $SL(2,\mathbb{Z})$-orbits of {\em primitive}\footnote{By definition, this means that the square-tiled surface is \emph{not} obtained by taking a covering of an intermediate square-tiled tiled surface.}
square-tiled surfaces in $\mathcal{C}$. In this direction, it is important to dispose of {\em invariants} to distinguish $SL(2,\mathbb{Z})$-orbits. As it turns out, the \emph{monodromy group} of a square-tiled surface (cf. Subsection~\ref{ss.MYZ}) is such an invariant, and the following result (from the PhD thesis) of D.~Zmiaikou \cite{Zm} shows that this invariant can take only two values when the number of squares of the origami is sufficiently large:

\begin{theorem}[D. Zmiaikou] Given any stratum $\mathcal{H}(k_1,\dots, k_s)$, there exists an integer $N_0:=N_0(k_1,\dots,k_s)$ such that any primitive origami of $\mathcal{H}(k_1,\dots,k_s)$ with $N\geq N_0$ squares has monodromy group isomorphic to $A_N$ or $S_N$.
\end{theorem}

\begin{remark} The integer $N_0=N_0(k_1,\dots,k_s)$ has explicit upper bounds (as it was shown by 
D.~Zmiaikou), but we did not include it in the previous statement because it is believed that the current upper bounds are not sharp.
\end{remark}

\begin{remark}
In order to simplify the terminology, we will refer to square-tiled surfaces as \emph{origamis} in what follows.
\end{remark}

After this (brief) general discussion, let's specialize to the case of genus $2$ origamis.

\subsubsection{Classification of $SL(2,\mathbb{Z})$-orbits of square-tiled surfaces in $\mathcal{H}(2)$}

Denote by $N\geq 3$ the number of squares of a origami in $\mathcal{H}(2)$. By the results of P.~Hubert and S.~Leli\`evre \cite{HL}, and C.~McMullen \cite{Mc-H2}, it is possible to show that the 
$SL(2,\mathbb{Z})$-orbits of origamis are organized as follows
\begin{itemize}
\item if $N\geq 4$ is even, then there is exactly one $SL(2,\mathbb{Z})$-orbit and the monodromy group is $S_N$.
\item if $N\geq 5$ is odd, then there are exactly two $SL(2,\mathbb{Z})$-orbits distinguished by the monodromy group being $A_N$ or $S_N$.
\item if $N=3$, there is exactly one $SL(2,\mathbb{Z})$-orbit and the monodromy group is $S_3$.
\end{itemize}

Concerning the Lyapunov exponents of the KZ cocycle, recall that by the work of M.~Bainbridge \cite{Ba}, and A.~Eskin, M.~Kontsevich and A.~Zorich~\cite{EKZ} the second (non-negative) exponent 
$\lambda_2$ is always $\lambda_2=1/3$ in $\mathcal{H}(2)$ (independently of the $SL(2,\mathbb{R})$-invariant probability measure considered), so that the Lyapunov spectrum of the KZ cocycle is always simple in $\mathcal{H}(2)$. 

Of course, the above results say much more as the explicit value of $\lambda_2$ is found, while the simplicity in $\mathcal{H}(2)$ amounts only to say that $\lambda_2>0(>-\lambda_2)$. But, as it turns out, the arguments employed by Bainbridge and Eskin, Kontsevich and Zorich are {\em sophisticated} (involving the Deligne--Mumford compactification of the moduli space of curves, etc.) and {\em long} (both papers have $\geq100$ pages).  Another, much simpler and shorter proof that $\lambda_2>0$, for any $SL(2,\mathbb R)$-invariant $g_t$-ergodic probability measures on $\mathcal{H}_2$, was given by G.~Forni in unpublished work, recently developed by D.~Aulicino~\cite{Au}. Aulicino's work includes a proof of Forni's unpublished result that no Teichm\"uller disk in $\mathcal{H}_2$ is contained in the {\em determinant locus}, that is, the locus where the second fundamental form of the Gauss-Manin connection does not have maximal rank (see Remark~\ref{remark:Aulicino}). The above general non-vanishing statement then follows from Forni's lower bound on the second exponent \cite{F02} (see Corollary~\ref{c.rankB0exp0}). This proof, however simpler, still involves the Deligne--Mumford compactification. In fact, the main argument establishes that the closure of any given  Teichm\"uller disk of $\mathcal H_2$ 
in the Deligne--Mumford compactification contains  ``boundary points''  of a certain type.

We will discuss below yet another proof of the non-vanishing of the second exponent  for 
square-tiled surfaces in $\mathcal{H}(2)$ as an application of the simplicity criterion of Theorem~\ref{t.MMY}: evidently, this is a weaker result than the above-mentioned theorems, but it has the advantage of relying only on the {\em elementary} methods discussed above.

We begin by selecting the following $L$-shaped origami $L(m,n)$:

\begin{figure}[htb!]
\includegraphics[scale=0.6]{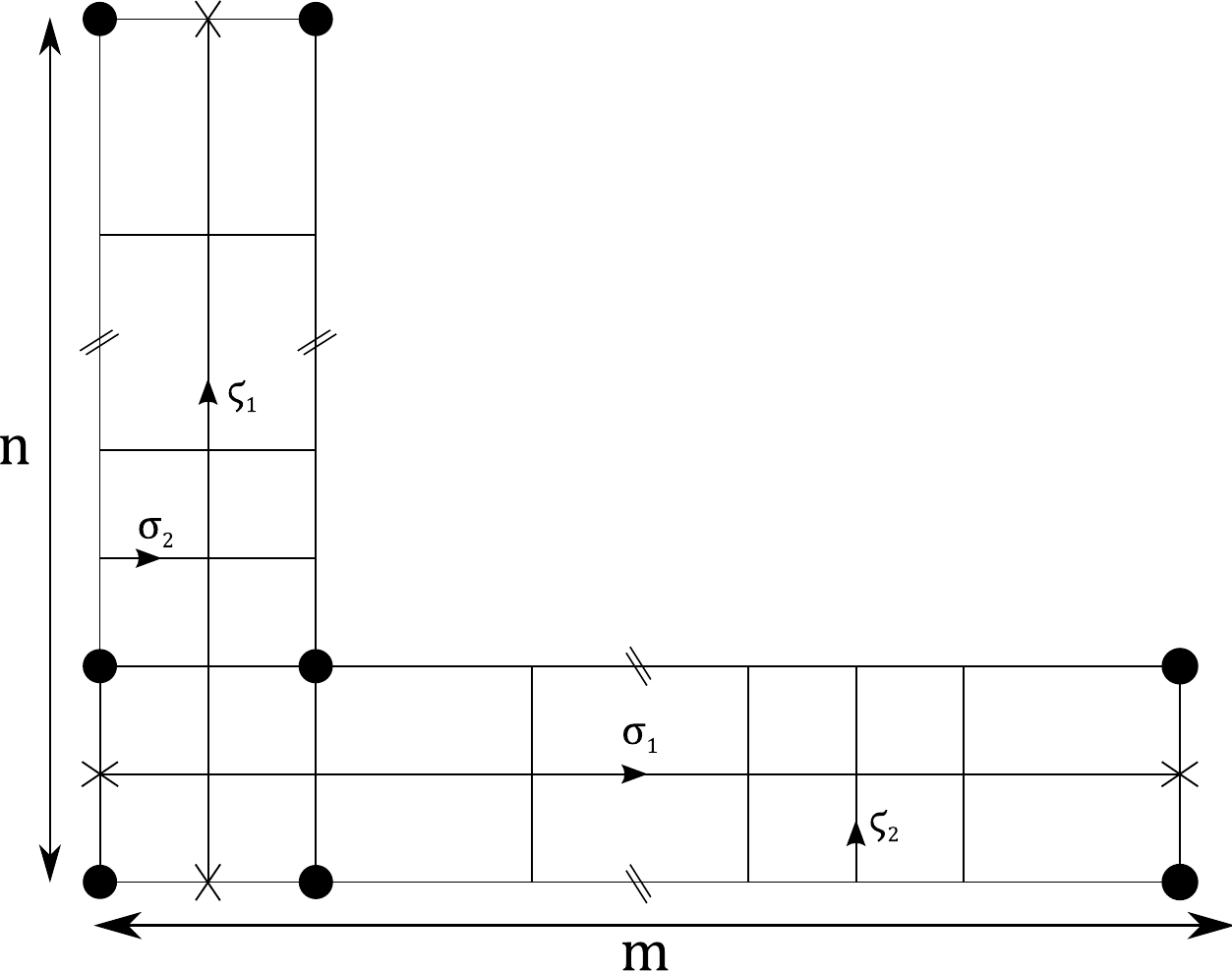}
\end{figure}

In terms of the parameters $m,n\geq 2$, the total number of squares of $L(m,n)$ is $N=n+m-1\geq 3$. Notice that the horizontal permutation $r$ associated $L(m,n)$ is a $m$-cycle, while the vertical permutation $u$ is a $n$-cycle. Thus, by our discussion so far above, one has that:
\begin{itemize}
\item if $m+n$ is odd, the monodromy group is $S_N$,
\item if $m$ and $n$ are odd, the monodromy group is $A_N$, and
\item if $m$ and $n$ are even, the monodromy group is $S_N$.
\end{itemize}

As we already know, given an origami $p:M\to\mathbb{T}^2$, we have a decomposition
$$H_1(M,\mathbb{Q}) = H_1^{st}(M,\mathbb{Q})\oplus H_1^{(0)}(M,\mathbb{Q})$$
such that the {\em standard} part $H_1^{st}(M,\mathbb{Q})$ has dimension $2$ and associated Lyapunov exponents $\lambda_1=1$ and $-\lambda_1=-1$, while
the part $H_1^{(0)}(M,\mathbb{Q})$, consisting of homology classes projecting to $0$ under the
map $p_\ast:H_1(M,\mathbb{Q}) \to H_1(\mathbb T^2,\mathbb{Q})$, has dimension $2g-2$ and associated Lyapunov exponents $\lambda_2,\dots,\lambda_g$ and their opposites.
In the particular case of $L(m,n)$, we can construct the following basis of $H_1^{(0)}(M,\mathbb{Q})$  
(in terms of the homology cycles $\sigma_1, \sigma_2, \zeta_1, \zeta_2$ showed in the previous picture):
$$\sigma:=\sigma_1-m\sigma_2, \quad\zeta:=\zeta_1-n\zeta_2.$$

Now, we consider the elements $S=\left(\begin{array}{cc}1 & m \\ 0 & 1\end{array}\right)$ and $T=\left(\begin{array}{cc}1 & 0 \\ n & 1\end{array}\right)$ of the Veech group of
$L(m,n)$. Since $L(m,n)\in \mathcal{H}(2)$ has no automorphisms (see Lemma~\ref{l.minimal}), we can also think of $S$ and $T$ as elements of the affine group of $L(m,n)$.
As the reader can check in the figures, the action of $S$ on the homology cycles $\sigma_1, \sigma_2, \zeta_1, \zeta_2$ is given by the following formulas:
\begin{itemize}
\item $S(\sigma_i)=\sigma_i$ for $i=1, 2$,
\item $S(\zeta_1)=\zeta_1+\sigma_1+(n-1)m\sigma_2$,
\item $S(\zeta_2)=\zeta_2+\sigma_1$.
\end{itemize}
Therefore, $S(\sigma)=\sigma$ and $S(\zeta)=\zeta+(n-1)\sigma$. Actually, the same computation with $S$ replaced by any power $S^a$, $a\in\mathbb{N}$, gives the ``same'' result, that is,
$$S^a(\sigma)=\sigma, \quad S^a(\zeta)=\zeta+a(1-n)\sigma\,.$$

By the natural symmetry between the horizontal and vertical directions in $L(m,n)$, there is no need to compute twice to write down the corresponding formulas for $T$ and $T^b$:
$$T^b(\sigma)=\sigma+b(1-m)\zeta, \quad T^b(\zeta)=\zeta\,.$$

By combining $S$ and $T$, we can find {\em pinching} elements acting on $H_1^{(0)}(M,\mathbb{Q})$: for instance, 
$$A:=ST=\left(\begin{array}{cc}1+(m-1)(n-1) & (1-n) \\ (1-m) & 1 \end{array}\right)$$
has eigenvalues given by the roots of the characteristic polynomial $\lambda^2-t_A\lambda+1=0$, where $t_A:=2+(m-1)(n-1)$ is the trace of $A$; since $|t_A|>2$ (as $m,n\geq 2$) and 
$t_A^2 -4 = (m-1)(n-1)[(m-1)(n-1)+4]$ is {\em not} a square, $A$ has two real irrational eigenvalues, and hence $A$ is a {\em pinching} matrix whose eigenspaces are defined over $\mathbb{Q}(\sqrt{t_A})$.
Thus, we can apply the simplicity criterion as soon as one finds a twisting matrix $B$ with respect to $A$. We take 
$$B:=S^2T^2\,.$$
By applying the previous formulas for the powers $S^a$ and $T^b$ of $S$ and $T$ (with $a=b=2$), we find that the trace of $B$ is
$$t_B=2+4(m-1)(n-1)\,,$$
so that $t_B^2-4 = 16(m-1)(n-1)[(m-1)(n-1)+1]$ is also not a square. Hence, $B$ is also a pinching matrix whose eigenspaces are defined over $\mathbb{Q}(\sqrt{t_B})$. Furthermore, the above formulas for $t_A$ and $t_B$ show that the quadratic fields $\mathbb{Q}(\sqrt{t_A})$ and $\mathbb{Q}(\sqrt{t_B})$ are {\em disjoint} in the sense that $\mathbb{Q}(\sqrt{t_A})\cap\mathbb{Q}(\sqrt{t_B})=\mathbb{Q}$. So, $B$ is twisting with respect to $A$, and therefore $1=\lambda_1>\lambda_2>0$.

\begin{remark} For the ``other'' stratum of genus $2$, namely $\mathcal{H}(1,1)$, the results of Bainbridge and EKZ show that $\lambda_2=1/2$. In principle, the (weaker) fact that $\lambda_2>0$ in this situation can be derived for {\em particular} origamis in $\mathcal{H}(1,1)$ along the lines sketched above for $\mathcal{H}(2)$, but it is hard to treat such origamis in a {\em systematic} way because currently there is only a {\em conjectural} classification of $SL(2,\mathbb{Z})$-orbits.
\end{remark}

This concludes our (preliminary) discussion of the application of the simplicity criterion in the (well-established) case of $\mathcal{H}(2)$. Now, we pass to the case of the stratum $\mathcal{H}(4)$.

\subsubsection{$SL(2,\mathbb{Z})$-orbits of origamis in $\mathcal{H}(4)$} Let's start with some general comments about the stratum $\mathcal{H}(4)$ of Abelian differential with a single zero on genus $3$ surfaces.

\noindent\textbf{Connected components of the stratum $\mathcal{H}(4)$.} By the work of M.~Kontsevich and A.~Zorich~\cite{KZ}, we know that $\mathcal{H}(4)$ has two connected components, a {\em hyperelliptic} component $\mathcal{H}(4)^{hyp}$, and an {\em odd spin} component 
$\mathcal{H}(4)^{odd}$. In this particular case, they can be distinguished as follows:
\begin{itemize}
\item $-\textrm{id}=\left(\begin{array}{cc}-1 & 0 \\ 0 & -1 \end{array}\right)$ {\em belongs} to the Veech group of {\em all} translation surfaces in the connected component $\mathcal{H}(4)^{hyp}$.
\item $-\textrm{id}=\left(\begin{array}{cc}-1 & 0 \\ 0 & -1 \end{array}\right)$ {\em doesn't belong} to the Veech group of the {\em generic} translation surface of the connected component 
$\mathcal{H}(4)^{odd}$.
\end{itemize}

\noindent\textbf{Weierstrass points and $SL(2,\mathbb{Z})$-orbits of ``symmetric'' origamis: an invariant of E.~Kani, and P.~Hubert $\&$ S.~Leli\`evre.}  In general, we have the exact sequence 
$$1\to \textrm{Aut}(M)\to \textrm{Aff}(M)\to SL(M)\to 1$$
If the origami $M$ belongs to a mininal stratum $\mathcal{H}(2g-2)$ (e.g., $\mathcal{H}(4)$), by Lemma~\ref{l.minimal}, one has $\textrm{Aut}(M)=1$ and therefore $\textrm{Aff}(M)\simeq SL(M)$. In other words, given $g\in SL(M)$, there exists an {\em unique} affine diffeomorphism of $M$ with derivative $g$. 

Suppose now that $p:M\to\mathbb{T}^2$ is a (reduced) origami such that $-\textrm{id}=\left(\begin{array}{cc}-1 & 0 \\ 0 & -1 \end{array}\right)$ belongs to the Veech group $SL(M)$ of $M$. Thus, if $M\in \mathcal{H}(2g-2)$, there exists an unique affine diffeomorphism $\phi$ of $M$ with derivative $-\textrm{id}$. Of course, $\phi$ is an {\em involution} and it corresponds to a lift under $p$ of the involution $x\mapsto -x$ of $\mathbb{T}^2$. It follows that the fixed points of $\phi$ project to the points of $\mathbb{T}^2$ of order $1$ and $2$. (Actually, it is also possible to show that the fixed points of $\phi$ are \emph{Weierstrass points} of $M$, but we will not use this here). In the figure below we indicated the 4 points of $\mathbb{T}^2$ of orders $1$ and $2$, and we denote by 
\begin{itemize}
\item $\ell_0$ the number of fixed points of $\phi$ over the integer point $0=(0,0)\in\mathbb{T}^2$, 
\item $\ell_1$ the number of fixed points of $\phi$ over the half-integer point $(0,1/2)\in\mathbb{T}^2$,
\item $\ell_2$ the number of fixed points of $\phi$ over the half-integer point $(1/2,1/2)\in\mathbb{T}^2$,
\item$\ell_3$ the number of fixed points of $\phi$ over the half-integer point $(1/2,0)\in\mathbb{T}^2$.
\end{itemize}

\begin{figure}[htb!]
\includegraphics[scale=0.8]{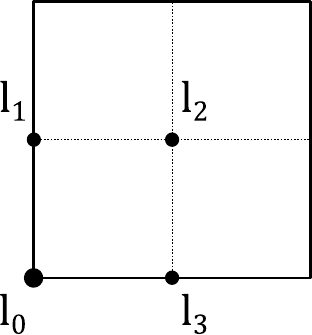}
\end{figure}

In addition, the action of $SL(2,\mathbb{Z})$ {\em conjugates} the involutions of the origamis in the $SL(2,\mathbb{Z})$-orbit of $M$ (by an element of $SL(2,\mathbb{Z})$). Since $SL(2,\mathbb{Z})$ fixes the integer point $0$ while it permutes (arbitrarily) the 3 points of order $2$, we have that the list 
$$[\ell_0, \{\ell_1,\ell_2,\ell_3\}]$$
is an invariant of the $SL(2,\mathbb{Z})$-orbit of $M$ if  the sublist $\{\ell_1,\ell_2,\ell_3\}$ is taken 
{\em up to permutations}. 

\begin{remark} The unique zero of $M\in \mathcal{H}(2g-2)$ is fixed under the involution $\phi$, so that $\ell_0\geq 1$.
\end{remark}

\begin{remark} For $i>0$, $\ell_i\equiv N (\textrm{mod }2)$: for instance, the involution permute the squares forming $M$ and, e.g. $\ell_2$ is the number of squares fixed by 
$\phi$; since $\phi$ is an involution, the $\phi$-orbits of squares have size $1$ or $2$, so that $\ell_2\equiv N (\textrm{mod } 2)$; finally, since $SL(2,\mathbb{Z})$ acts transitively on the set of half-integer points of $\mathbb{T}^2$, we can ``replace'' $\ell_2$ by $\ell_i$, $i>0$, to get the same conclusion for all $i>0$.
\end{remark}

The invariant $[\ell_0, \{\ell_1,\ell_2,\ell_3\}]$ was introduced by E.~Kani \cite{Ka}, and P.~Hubert and 
S.~Leli\`evre \cite{HL} to study $SL(2,\mathbb{Z})$-orbits of origamis of $\mathcal{H}(2)$ (historically speaking this invariant came before the monodromy invariant). Since any genus $2$ Riemann surface is {\em hyperelliptic}, $-\textrm{id}$ belongs to the Veech group of any 
translation surface $M\in \mathcal{H}(2)$, and the quotient of $M$ by involution $\phi$ has genus $0$. In particular, by Riemann-Hurwitz formula, the involution $\phi$ has $6$ fixed points (the $6$ Weierstrass points of $M$). For the origami $L(m,n)$ with $N=n+m-1$ squares shown during the discussion in $\mathcal{H}(2)$, one can compute the invariant $[\ell_0, \{\ell_1,\ell_2,\ell_3\}]$:
$$\ell_2=\left\{\begin{array}{cc}1, & \textrm{if } n,m \textrm{ are odd,} \\ 2, & \textrm{if } n+m \textrm{ is odd,} \\ 3, & \textrm{if } n,m \textrm{ are even;}\end{array}\right.$$
$$\ell_3 \, (\textrm{resp. } \ell_1)=\left\{\begin{array}{cc}0 \, (\textrm{resp. } 2), & \textrm{if } m \textrm{ is odd and } n \textrm{ is even,} \\ 1, & \textrm{if } n+m \textrm{ is even,} \\ 2 \, (\textrm{resp. } 0), & \textrm{if } m \textrm{ is even and } n \textrm{ is odd.}\end{array}\right.$$
In summary, we get that 
\begin{itemize}
\item if $N=n+m-1$ is even, there is exactly one $SL(2,\mathbb{Z})$-orbit of origamis of $\mathcal{H}(2)$ with invariant $[2,\{2,2,0\}]$.
\item if $N=n+m-1$ is odd, there are exactly two $SL(2,\mathbb{Z})$-orbit of origamis of $\mathcal{H}(2)$:
\subitem - if $m,n$ are odd, the monodromy group is $A_N$ and the invariant is $[3,\{1,1,1\}]$;
\subitem - if $m,n$ are even, the monodromy group is $S_N$ and the invariant is $[1,\{3,1,1\}]$.
\end{itemize}

Coming back to $\mathcal{H}(4)$, after some numerical experiments of V.~Delecroix and S.~Leli\`evre (using SAGE), currently we dispose of a {\em conjectural} 
classification of $SL(2,\mathbb{Z})$ orbits of origamis in $\mathcal{H}(4)$.

\noindent\textbf{Conjecture of Delecroix and Leli\`evre on origamis in $\mathcal{H}(4)$.}
\begin{itemize}
\item For the hyperelliptic component $\mathcal{H}(4)^{hyp}$, the quotient of the involution $\phi$ of any translation surface $M\in \mathcal{H}(4)^{hyp}$ has genus $0$ and (by Riemann--Hurwitz formula) $8$ fixed points. The (experimentally observed) Kani-Hubert-Leli\`evre invariants for the $SL(2,\mathbb{Z})$-orbits of reduced origamis in $\mathcal{H}(4)$ with $N$ squares are:
\begin{itemize}
\item For $N\geq 9$ odd, $[5,\{1,1,1\}]$, $[3,\{3,1,1\}]$, $[1,\{5,1,1\}]$ and $[1,\{3,3,1\}]$;
\item For $N\geq 10$ even, $[4, \{2,2,0\}]$, $[2,\{4,2,0\}]$, $[2,\{2,2,2\}]$.
\end{itemize}
\item For the ``odd spin'' connected component $\mathcal{H}(4)^{odd}$ and $N\geq 9$:
\begin{itemize}
\item For the generic origami $M\in\mathcal{H}(4)^{odd}$ (i.e., $-\textrm{id}\notin SL(M)$), there are exactly two $SL(2,\mathbb{Z})$-orbits distinguished by the monodromy group $A_N$ or $S_N$;
\item For the so-called ``{\em Prym}'' case, i.e., $M\in\mathcal{H}(4)^{odd}$ and $-\textrm{id}\in SL(M)$, the quotient of the involution $\phi$ has genus $1$, and hence, by Riemann-Hurwitz formula, $\phi$ has $4$ fixed points. In this situation, the $SL(2,\mathbb{Z})$-orbits were classified by a {\em recent theorem} of E. Lanneau and D.-Manh Nguyen \cite{LaNg} whose statement we recall below. 
\end{itemize}
\end{itemize}

\begin{theorem}[E.~Lanneau and D.-Mahn~Nguyen] In the Prym case in $\mathcal{H}(4)^{odd}$: 
\begin{itemize}
\item if $N$ is odd, there exists precisely one $SL(2,\mathbb{Z})$-orbit whose associated Kani-Hubert-Leli\`evre invariant is $[1,\{1,1,1\}]$;
\item if $N\equiv 0\, (\textrm{mod }4)$, there exists precisely one $SL(2,\mathbb{Z})$-orbit whose associated Kani-Hubert-Leli\`evre invariant is $[2,\{2,0,0\}]$;
\item if $N\equiv 2 \,(\textrm{mod }4)$, there are precisely two $SL(2,\mathbb{Z})$-orbits whose associated Kani--Hubert--Leli\`evre invariant are $[2, \{2,0,0\}]$ and $[4,\{0,0,0\}]$.
\end{itemize}
\end{theorem}

\noindent\textbf{Simplicity of Lyapunov spectrum of origamis in $\mathcal{H}(4)$.}
Concerning Lyapunov exponents of the KZ cocycle, it is possible to prove (see, e.g., \cite{Mo-Prym}) that the Prym case of $\mathcal{H}(4)^{odd}$ always lead to Lyapunov exponents 
$\lambda_2=2/5$ and $\lambda_3=1/5$. 

In particular, the Lyapunov spectrum of the KZ cocycle in the Prym case is simple. As a side remark, we would like to observe that, interestingly enough, the simplicity criterion discussed above can't be applied in the Prym case: indeed, the presence of the involution permits to decompose or, more precisely, {\em diagonalize by blocks}, the KZ cocycle so that the twisting property is never satisfied!
However, the non-uniform hyperbolicity of the KZ cocycle for all Prym origamis (that is, the qualitative statement that $\lambda_2\geq \lambda_3 >0$) can be  immediately derived from the geometric 
non-uniform hyperbolicity criterion of G.~Forni (see \cite{F11}, \S 6.1).

In any event, in \cite{MMY} the authors apply the simplicity criterion of Theorem~\ref{t.MMY} to {\em all} (non-Prym) $SL(2,\mathbb{Z})$-orbits in $\mathcal{H}(4)$ described by the Delecroix--Leli\`evre conjecture. The outcome is that for such $SL(2,\mathbb{Z})$-orbits, the Lyapunov spectrum of the KZ cocycle is simple {\em except for at most for finitely many} $SL(2,\mathbb{Z})$-orbits (corresponding to ``small'' values of $N$, the number of squares). 

In other words, in \cite{MMY} it is shown that 
\begin{theorem} If the conjecture of Delecroix-Leli\`eevre is true, then all but at most finitely many 
$SL(2,\mathbb{Z})$-orbits of origamis of $\mathcal{H}(4)$ have simple Lyapunov spectrum. 
\end{theorem}

Unfortunately, since there are several $SL(2,\mathbb{Z})$-orbits in the Delecroix-Leli\`evre conjecture above, we will not be able to discuss all of them here. Moreover, as we're going to see in a moment, the proof of the simplicity statement above involves long {\em computer-assisted} calculations, so that our plan for the rest of this discussion will be the following: we will select a family of $SL(2,\mathbb{Z})$-orbits and we will {\em start} some computations by hand towards the simplicity statement; then, we will see why we run into ``trouble'' with our calculation (essentially it is ``too naive'') and why we need some computer assistance to complete the argument; however, we will not present the computer assisted computations as we prefer to leave them for the forthcoming article, and we will end this series with this ``incomplete'' argument as it already contains the main ideas and it is not as technical as the complete argument. 

In the case of an origami $M\in \mathcal{H}(4)$, we have that $\textrm{dim}(H_1^{(0)}(M,\mathbb{R}))=4$, so that the KZ cocycle takes its values in $Sp(4,\mathbb{Z})$. 
Given $A, B\in Sp(4,\mathbb{Z})$, denote by 
$$P_A(x)=x^4+ax^3+bx^2+ax+1$$
and 
$$P_B(x)=x^4+a'x^3+b'x^2+a'x+1$$
their characteristic polynomials. 

Recall that if the roots of $P_A$ are all real, and the Galois group is $G$ is the largest possible, i.e., $G\simeq S_2\times(\mathbb{Z}/2\mathbb{Z})^2$, then the matrix $A$ is {\em pinching}. 

Following the usual trick to determine the roots of $P_A$, we pose $y=x+1/x$, so that we get the
quadratic polynomial 
$$Q_A(y)=y^2+ay+(b-2)$$
with $a,b\in\mathbb{Z}$. The roots of $Q_A$ are 
$$y_{\pm}=\frac{-a\pm\sqrt{\Delta_1}}{2}$$
where $\Delta_1:=a^2-4(b-2)$. Then, we recover the roots of $P_A$ by solving the equation
$$x^2-y_{\pm}x+1=0$$
of discriminant $\Delta_{\pm}:=y_{\pm}^2-4$. 

A simple calculation shows that 
$$\Delta_2:=\Delta_+\Delta_-=(b+2)^2-4a^2 = (b+2+2a)(b+2-2a)\,.$$

For later use, we denote by $K_A$, resp. $K_B$, the splitting field of $P_A$, resp. $P_B$. 

We leave the following proposition as an exercise in Galois theory for the reader:

\begin{proposition}\label{p.galois} The following holds:
\begin{itemize}
\item[(a)] If $\Delta_1$ is a square (of an integer number), then $Q_A$ and $P_A$ are not irreducible.
\item[(b)] If $\Delta_1$ is not a square and  $\Delta_2$ is a square, then $[K_A:\mathbb{Q}]=4$ and the Galois group $G$ is $(\mathbb{Z}/2\mathbb{Z})\times(\mathbb{Z}/2\mathbb{Z})$.
\item[(c)] If $\Delta_1$, $\Delta_2$ are not squares and $\Delta_1\Delta_2$ is a square, then $[K_A:\mathbb{Q}]=4$ and the Galois group is $\mathbb{Z}/4\mathbb{Z}$.
\item[(d)] If $\Delta_1$, $\Delta_2$ and $\Delta_1\Delta_2$ are not squares, then $[K_A:\mathbb{Q}]=8$, the Galois group is $G\simeq D_4\simeq 
S_2\times (\mathbb{Z}/2\mathbb{Z})^2$, $K_A$ contains exactly $3$ quadratic fields $\mathbb{Q}(\sqrt{\Delta_1}), \mathbb{Q}(\sqrt{\Delta_2})$ and 
$\mathbb{Q}(\sqrt{\Delta_1\Delta_2})$, and each intermediate field $\mathbb{Q}\subset L\subset K_A$ contains one of these quadratic fields.
\item[(e)] If $\Delta_1$, $\Delta_2$ and $\Delta_1\Delta_2$ are not squares, and $\Delta_1\Delta_2>0$, then the roots of $P_A$ are real and the Galois group of $P_A$ is the largest possible.
\end{itemize}
\end{proposition}

This proposition establishes a sufficient criterion for a pair of matrices $A, B\in Sp(4,\mathbb{Z})$ to satisfy the pinching and twisting conditions. Indeed, we can use item (e) to get pinching matrices $A$, and we can apply item (d) to produce twisting matrices $B$ with respect to $A$ as follows: recall that (cf. Remark~\ref{r.MMY-Galois}) the twisting condition is true if the splitting fields $K_A$ and $K_B$ are ``disjoint'' (i.e., $K_A\cap K_B=\mathbb{Q}$), and, by item (d), this disjointness can be checked by computing the quantities $\Delta_1, \Delta_2, \Delta_1\Delta_2$ associated to $A$, and $\Delta_1', \Delta_2', \Delta_1'\Delta_2'$ associated to $B$, and verifying that they generated distinct quadratic fields. 

Let us  consider the following specific family $M_n$, $n\geq 5$, of origamis in $\mathcal{H}(4)^{odd}$:

In this family, the total number of squares of $M=M_n$ is $N=n+4$ and $-\textrm{id}$ doesn't belong to the Veech group (i.e., this origami is ``generic'' in $\mathcal{H}(4)^{odd}$). The horizontal permutation is the product of a $n$-cycle with a $2$-cycle (so its parity equals the parity of $n$), and the vertical permutation is the product of a $4$-cycle with a $2$-cycle (so that its parity is even). In particular, the corresponding monodromy group is $S_N$ if $N$ is odd, and $A_N$ if $N$ is even. It is not hard
to derive from the geometric criterion of \cite{F11} that the KZ cocycle is non-uniformly hyperbolic with respect to the probability measures arising from the origamis $M_n$. In fact, as it is clear from the
figure below, both the horizontal and vertical foliations of $M_n$ have $3$ homologically independent
cylinders (that is, cylinders whose waist curves are linearly independent in homology). By the
criterion of \cite{F11} this implies that the KZ cocycle has (at least) $3$ non-zero Lyapunov exponents,
hence it is non-uniformly hyperbolic  (all surfaces in $\mathcal{H}(4)$ have genus $3$).
The proof of simplicity of the KZ spectrum for all but finitely many $n\in \mathbb N$, discussed here following the paper by C.~Matheus, M.~M\"oller and J.-C.~Yoccoz ~\cite{MMY},  is much more
delicate and, as we will see below, relies on a deep theorem from number theory.

\begin{figure}[htb!]
\includegraphics[scale=0.8]{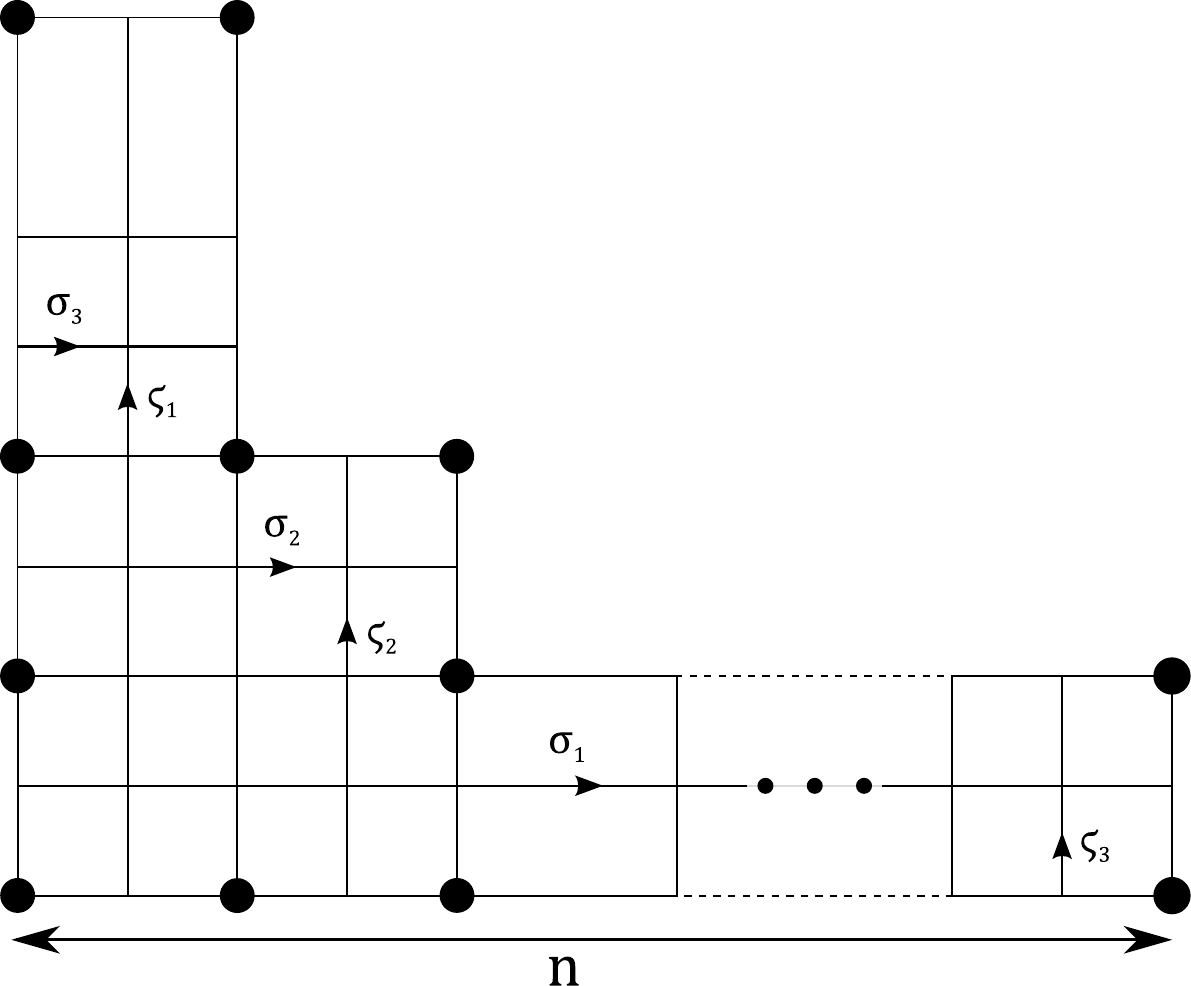}
\end{figure}

In terms of the homology cycles $\sigma_1, \sigma_2, \sigma_3, \zeta_1, \zeta_2, \zeta_3$ showed in the figure above, we can  consider the following basis of $H_1^{(0)}(M,\mathbb{Q})$:
\begin{itemize}
\item $\Sigma_1=\sigma_1-n\sigma_3$, $\Sigma_2=\sigma_2-2\sigma_3$,
\item $Z_1=\zeta_1-4\zeta_3$, $Z_2=\zeta_2-2\zeta_3$.
\end{itemize}

Then, we take the following elements of the Veech group of $M$
$$S=\left(\begin{array}{cc}1 & 2n \\ 0 & 1\end{array}\right), \quad T=\left(\begin{array}{cc}1 & 0 \\ 4 & 1\end{array}\right)\,.$$
Notice that when $n$ is even we can replace $2n$ by $n$ in the definition of $S$, but we prefer to consider $S$ directly (to avoid similar but separate  discussions depending on the parity of $n$).

By direct inspection of the figure above we deduce the following formulas:
\begin{itemize}
\item $S(\sigma_i)=\sigma_i$, for $i=1, 2, 3$,
\item $S(\zeta_1)=\zeta_1+2\sigma_1+n\sigma_2+4n\sigma_3$,
\item $S(\zeta_2)=\zeta_1+2\sigma_1+n\sigma_2$, 
\item $S(\zeta_3)=\zeta_1+2\sigma_1$.
\end{itemize}
Hence, 
\begin{itemize}
\item $S(\Sigma_i)=\Sigma_i$, for $i=1, 2$,
\item $S(Z_1)=Z_1-6\Sigma_1+n\Sigma_2$,
\item $S(Z_2)=\zeta_1-2\Sigma_1+n\sigma_2$, 
\end{itemize}
Thus, in terms of the basis $\{\Sigma_1,\Sigma_2, Z_1, Z_2\}$ of $H_1^{(0)}(M,\mathbb{Q})$, the matrix of $S|_{H_1^{(0)}}$ has the following 
decomposition into $2\times 2$ blocks:
$$S|_{H_1^{(0)}}=\left(\begin{array}{cc}1 & \overline{S} \\ 0 & 1\end{array}\right), \quad \textrm{where } \overline{S}=\left(\begin{array}{cc}-6 & n \\ -2 & n\end{array}\right)\,.$$

By symmetry, the matrix of  $T|_{H_1^{(0)}}$ is: 
$$T|_{H_1^{(0)}}=\left(\begin{array}{cc}1 & 0 \\ \overline{T} & 1\end{array}\right), \quad \textrm{where } \overline{T}=\left(\begin{array}{cc}1-n & -1 \\ 2 & 2\end{array}\right)\,.$$

In particular, 
$$(ST)|_{H_1^{(0)}}=\left(\begin{array}{cc}1+\overline{S}\overline{T} & \overline{S} \\ \overline{T} & 1\end{array}\right)$$
Of course one can extract the characteristic polynomial $x^4+ax^3+bx^2+ax+1$ of $(ST)|_{H_1^{(0)}}$ directly from these formulas, but in our particular case we can take the following ``shortcut'': observe that $ST(v_1,v_2)=\lambda(v_1,v_2)$ (where $v_1,v_2\in\mathbb{R}^2$) if and only if $\overline{S}\overline{T}(v_1) = (\lambda+\lambda^{-1}-2)v_1=(y_{\pm}-2)v_1$. By computing with $\overline{S}$ and $\overline{T}$, and using that $y_{\pm}$ are the solutions of $y^2-ay+(b-2)=0$, one finds 
that $a=-7n+6$, $b=8n^2-2n-14$ and 
$$\Delta_1=17n^2-76n+100$$

By Proposition~\ref{p.galois}, we wish to know, for instance, how often $\Delta_1$ is a square. Evidently, one can maybe produce some {\em ad hoc} method here, but since we want to verify simplicity conditions in a systematic way, it would be nice to answer this type of question in ``general''. 

The main idea to accomplish our goal is very simple: the fact that $\Delta_1=17n^2-76n+100$ is a square is equivalent to get integer and/or rational solutions to 
$$z^2=17t^2-76t+100,$$ 
that is, we need to understand integer/rational points on this curve. 

This suggests the following solution to our problem: if we can {\em replace} $A=ST$ by more {\em complicated} products of (powers of) $S$ and $T$ chosen more or less at random, it happens that 
$\Delta_1$ becomes a polynomial $P(n)$ of degree $\geq 5$ without square factors. In this situation, the problem of knowing whether $\Delta_1=P(n)$ is a square of an integer number becomes the problem of finding integer/rational points on the curve
$$z^2=P(n).$$ 
Since this a non-singular curve (as $P(n)$ has no square factors) of genus $g\geq 2$ (as $\textrm{deg}(P)\geq 5$), we know by 
\emph{Faltings' theorem} (previously known as Mordell's conjecture) that the number of rational solutions is {\em finite}.  In other words, by the end of the day, we have that $\Delta_1=P(n)$ is not the square of an integer number for all but finitely many values of $n$. 

\begin{remark} In the case we get polynomials $\widetilde{P}(n)$ of degree $3$ or $4$ after removing square factors of $P(n)$, we still can apply (and do apply in~\cite{MMY}) \emph{Siegel's theorem} stating that the genus $1$ curve $z^2=\widetilde{P}(n)$ has finite many integer points, but we 
preferred to discuss the case of higher degree polynomials because it is the ``generic situation''  
(in some sense) of the argument.
\end{remark}

Of course, at this point, the general idea to get the pinching and twisting conditions (and hence simplicity of the Lyapunov spectrum of the KZ cocycle) for the origamis $M=M_n$ for all but finitely many $n\in\mathbb{N}$ is clear: one produces ``complicated'' products of $S$, $T$ (and also a third auxiliary parabolic matrix $U$) leading to elements $A, B\in Sp(4,\mathbb{Z})$ such that the quantities $\Delta_i=\Delta_i(n), i=1,2$, $\Delta_3=\Delta_1\Delta_2(n)$ associated to $A$, $\Delta_i'=\Delta_i'(n), \Delta_3'=\Delta_1'\Delta_2'(n)$ associated to $B$ (and also the ``mixed products'' $\Delta_i\Delta_j'$ for $i,j=1,2,3$) are {\em polynomials} of the variable $n$ of {\em high} degree (and without square factors {\em if possible}), so that these quantities don't take square of integers as their values for all but finitely many 
$n\in\mathbb{N}$.

Unfortunately, it seems that there is no ``systematic way'' of choosing these ``complicated'' products of $S$, $T$, etc. and, as a matter of fact, the calculations are too long. For this reason  we use computer-assisted calculations at this stage. Again, as mentioned by the beginning of this section, these computer-assisted computations are not particularly inspiring (i.e., there are no mathematical ideas behind them), so we will end this discussion by giving an example of what kind of complicated coefficient $b=b(n)$ of the characteristic polynomial of $A$ we get in our treatment of the family $M=M_n$:
$$b(n)=3840 n^5 - 17376 n^4 + 14736 n^3 + 25384n^2 -28512n-7066\,.$$

\bibliographystyle{amsplain}

\end{document}